%% file: main.tex
\documentclass[a4paper, 10pt, reqno]{amsart}
\usepackage[T1]{fontenc}
\usepackage{lmodern}
\usepackage[utf8]{inputenc}

\usepackage{amsmath,amssymb,graphicx,mathtools,amsthm, bm} 
\usepackage{mathrsfs}
\usepackage[foot]{amsaddr}
\usepackage[a4paper,left=2.6cm,right=2.6cm,top=3cm,bottom=3.5cm]{geometry}
\usepackage{enumerate}
\usepackage[colorlinks, linkcolor=red, urlcolor=blue, citecolor=blue]{hyperref}
\hypersetup{breaklinks=true}
\usepackage{url}
\usepackage{caption}
\usepackage{subcaption}
\usepackage{comment}

\input{bibliography_setup}
\bibliography{refs}

\DeclareMathOperator{\Tr}{Tr}
\DeclareMathOperator{\tr}{tr}

\DeclareMathOperator{\supp}{supp}
\DeclareMathOperator{\sym}{Sym}

\newcommand{\R}{\mathbb{R}} 
\newcommand{\C}{\mathbb{C}} 
\newcommand{\N}{\mathbb{N}} 
\newcommand{\E}{\mathbb{E}}

\theoremstyle{plain}
\newtheorem{thm}{Theorem}[section]
\newtheorem{cor}[thm]{Corollary}
\newtheorem{lem}[thm]{Lemma}
\newtheorem{prop}[thm]{Proposition}

\theoremstyle{definition}
\newtheorem{defn}[thm]{Definition}

\newtheorem{rmk}[thm]{Remark}
\newtheorem{hyp}{Assumption}

\numberwithin{equation}{section}

\makeatletter
\renewcommand\subsection{\@startsection{subsection}{2}%
  \z@{-.5\linespacing\@plus-.7\linespacing}{.5\linespacing}%
  {\normalfont\scshape}}
\renewcommand\subsubsection{\@startsection{subsubsection}{3}%
  \z@{.5\linespacing\@plus.7\linespacing}{-.5em}%
  {\normalfont\scshape}}
\makeatother

\makeatletter
\@namedef{subjclassname@2020}{%
  \textup{2020} Mathematics Subject Classification}
\makeatother

\usepackage[format=plain,
            labelfont=sc]{caption}
\usepackage[toc,page]{appendix} 
\usepackage{scalerel,stackengine}
\stackMath%
\newcommand\reallywidehat[1]{%
\savestack{\tmpbox}{\stretchto{%
  \scaleto{%
    \scalerel*[\widthof{\ensuremath{#1}}]{\kern.1pt\mathchar"0362\kern.1pt}%
    {\rule{0ex}{\textheight}}
  }{\textheight}%
}{2.4ex}}%
\stackon[-6.9pt]{#1}{\tmpbox}%
}

\makeatletter
\newsavebox\myboxA
\newsavebox\myboxB
\newlength\mylenA
\newcommand*\xoverline[2][0.75]{%
    \sbox{\myboxA}{$\m@th#2$}%
    \setbox\myboxB\null
    \ht\myboxB=\ht\myboxA%
    \dp\myboxB=\dp\myboxA%
    \wd\myboxB=#1\wd\myboxA
    \sbox\myboxB{$\m@th\overline{\copy\myboxB}$}
    \setlength\mylenA{\the\wd\myboxA}
    \addtolength\mylenA{-\the\wd\myboxB}%
    \ifdim\wd\myboxB<\wd\myboxA%
       \rlap{\hskip 0.5\mylenA\usebox\myboxB}{\usebox\myboxA}%
    \else
        \hskip -0.5\mylenA\rlap{\usebox\myboxA}{\hskip 0.5\mylenA\usebox\myboxB}%
    \fi}
\makeatother    

\usepackage{fancyhdr}

\newcommand\shortitle{Spectral phase transitions in Gaussian multi-index models}
\newcommand\name{Florent Krzakala, Pierre Mergny, and Vanessa Piccolo}

\begin{document}
\title{Spectral phase transitions in Gaussian multi-index models}
\author{Florent Krzakala\(^1\)}
\author{Pierre Mergny\(^{1,2}\)}
\address{\(^1\)Information, Learning and Physics Laboratory. Ecole Polytechnique Fédérale de Lausanne (EPFL), Lausanne, Switzerland.}
\address{\(^2\)Département d’Informatique, École Normale Supérieure (ENS), Paris, France.}
\email{florent.krzakala@epfl.ch}
\email{pierre.mergny@epfl.ch}
\author{Vanessa Piccolo\(^1\)}
\email{vanessa.piccolo@epfl.ch}
\subjclass[2020]{60B20, 62H12, 68Q87} 
\keywords{Gaussian multi-index models, matrix-valued spectral methods, matrix-weighted covariance matrices, BBP transitions, weak subspace recovery, approximate message passing, matrix-valued self-consistent equations}
\date{\today}

\begin{abstract}
Recovering a low-dimensional latent subspace from nonlinear observations of Gaussian covariates in high dimensions is a fundamental problem in feature learning. Here, we consider Gaussian multi-index models in which the covariates $\bm{x}_i \stackrel{\mathrm{i.i.d.}}{\sim} \mathcal{N}(0,\bm{I}_d)$ and the responses $\bm{y}_i$ depend on $\bm{x}_i$ only through its projection onto an unknown $r$-dimensional subspace. Earlier work based on approximate message passing (AMP) identified a sharp threshold for weak recovery~\cite{troianifundamental}, raising the question of whether it can be attained, without side information, by a spectral method. We answer this affirmatively and develop a general random matrix theory for matrix-valued spectral estimators of the form
$$
\bm{D}_n=\frac{1}{n}\sum_{i=1}^n\bm{T}(\bm{y}_i)\otimes\bm{x}_i\bm{x}_i^\top,
$$
where $\bm{T}$ is an arbitrary bounded symmetric matrix-valued preprocessing map of fixed dimension. As $n,d \to \infty$ with $n/d\to\alpha$, we prove that the empirical spectral measure of $\bm{D}_n$ converges almost surely to a deterministic compactly supported distribution characterized by a matrix-valued self-consistent equation. We then establish a spectral phase transition for the largest eigenvalue: below threshold it sticks to the bulk edge, while above threshold an outlier emerges. We characterize the outlier location through a finite-dimensional deterministic equation and show that the associated spectral estimator achieves weak recovery of the latent subspace. Finally, we prove that the AMP-derived preprocessing of~\cite{defilippis2025} is optimal among all bounded matrix-valued preprocessing maps of any fixed dimension. Its transition coincides with the AMP weak-recovery threshold, proving the general spectral conjecture of~\cite{defilippis2025}.
\end{abstract}

\maketitle	
{
\hypersetup{linkcolor=black}
\tableofcontents
}
\section{Introduction} \label{section:introduction}

Recovering low-dimensional structure from high-dimensional random data is a recurring problem in probability, statistics, machine learning, and mathematical physics. A particularly rich instance is provided by Gaussian single- and multi-index models, in which the observations depend on a high-dimensional covariate only through a finite number of latent linear projections. Concretely, given Gaussian covariates $\bm{x}_i\sim\mathcal{N}(\bm{0}_d,\bm{I}_d)$, the conditional distribution of the response $\bm{y}_i \in \R^q$ depends on $\bm{x}_i$ through $\bm{W}_\ast^\top\bm{x}_i$, where the columns of $\bm{W}_\ast\in \R^{d\times r}$ span an unknown $r$-dimensional subspace. This broad class includes generalized linear models, phase retrieval, polynomial ridge functions, committee machines, and finite-width two-layer teacher networks, and has become a canonical setting for studying how learning algorithms recover task-relevant features rather than merely fit a fixed linear representation~\cite{aubin2019,damian2022,mousaviHosseini2023,collinswoodfin2024,bietti2025,dandi2024,simsek2025,zhangWang2025, defilippisHierarchical2026}. For Gaussian covariates, rotational invariance makes the latent subspace $\mathcal{S}_\ast=\operatorname{span}(\bm{W}_\ast)$ the natural inferential object.

We study the \emph{weak recovery} of this latent subspace in the proportional high-dimensional regime, where $n,d\to\infty$ with $n/d\to\alpha\in(0,\infty)$ while the intrinsic dimension $r$ and $q$ remain fixed. In this regime, the competition between the number of samples \(n\) and the ambient dimension \(d\) gives rise to nontrivial statistical and computational phase transitions. Recent analyses based on approximate message passing (AMP) have identified a sharp threshold for weak recovery~\cite{troianifundamental}. This raises a natural question: can the same threshold be attained, without side information or an informative initialization, by a spectral method?

The spectral operators considered in this work arise naturally from two complementary viewpoints. The first comes from AMP. Linearizing an AMP iteration around its uninformative fixed point leads to a block matrix of the form
\[
\bm{D}_n=\frac{1}{n}\sum_{i=1}^n \bm{T}(\bm{y}_i)\otimes \bm{x}_i\bm{x}_i^\top \in \R^{pd \times pd},
\]
where $\bm{T} \colon \mathbb{R}^q \to \R^{p \times p}$ is a symmetric matrix-valued preprocessing map and $p\geq1$ is a fixed working dimension, which need not coincide with the intrinsic dimension $r$. This spectral method was derived by Defilippis et al.~\cite{defilippis2025} and its spectral transition was predicted from state evolution. The same operator also has a direct optimization interpretation. For $\bm{W} \in \R^{d\times p}$, consider the label-dependent quadratic empirical objective
\[
\mathcal{L}_n(\bm{W}) =\frac{1}{2n}\sum_{i=1}^n (\bm{W}^{\top}\bm{x}_i)^{\top}\bm{T}(\bm{y}_i) (\bm{W}^{\top}\bm{x}_i),
\]
which may be nonconvex when the matrix weights \(\bm{T}(\bm{y}_i)\) have negative eigenvalues. Under column-wise vectorization, its Hessian is constant and given exactly by
\[
\nabla^2_{\mathrm{vec}(\bm{W})}\mathcal{L}_n(\bm{W})=\bm{D}_n.
\]
The spectral transition of $\bm{D}_n$ can therefore also be viewed as the emergence of informative curvature in a class of label-dependent quadratic objectives. Closely related connections between spectral geometry, planted signal directions, and gradient-based dynamics have recently been studied in high-dimensional neural-network and classification models~\cite{benarous2025,benarous2026,annesi2026,bocchi2026}. These two viewpoints---the linear instability of an uninformative AMP fixed point and the appearance of informative spectral directions in an empirical objective---provide complementary motivations for understanding the spectrum and eigenvectors of $\bm{D}_n$.

The spectral mechanism we analyze is an eigenvalue-separation phenomenon: whether the largest eigenvalue of $\bm{D}_n$ sticks to the limiting bulk edge or separates from it, and whether the resulting spectral estimator develops a nonvanishing overlap with the planted subspace. This mechanism is analogous to the celebrated Baik--Ben Arous--P\'ech\'e (BBP) transition for spiked random matrices~\cite{baik2005}. The goal of this work is to develop such a transition theory for the matrix-valued operator $\bm{D}_n$ and to determine whether the resulting spectral method can attain the AMP weak-recovery threshold.

\medskip

To place this question in context, we briefly review the most directly related spectral results; a broader discussion of related work is deferred to Subsection~\ref{subsection:related_work}. The single-index case $r=1$ provides the natural benchmark. In this setting, spectral estimators take the form of scalar-weighted sample covariance matrices, whose bulk edge, outlying eigenvalues, and eigenvector overlap admit exact high-dimensional characterizations~\cite{lu2020,mondelli2019}. Moreover, optimizing the scalar preprocessing function yields a spectral method that attains the weak-recovery threshold without requiring an informative initialization~\cite{mondelli2019}.

For $r>1$, the situation is substantially more delicate. Troiani et al.~\cite{troianifundamental} characterized the sharp weak-recovery threshold for first-order methods through an AMP analysis, using an arbitrarily small informed initialization to probe the stability of the uninformative fixed point. This raised the question of whether the same threshold could be attained spectrally, without such an initialization. Kova{\v{c}}evi{\'c}, Zhang, and Mondelli~\cite{kovavcevic2025} obtained precise asymptotics for scalar-preprocessed covariance estimators in general multi-index models and optimized the scalar preprocessing function. They further showed that the resulting optimal threshold coincides with the AMP threshold when the relevant conditional second-moment matrices are simultaneously diagonalizable. Defilippis et al.~\cite{defilippis2025} introduced the matrix-valued spectral estimator considered here by linearizing AMP and predicted its transition through state evolution. They established the corresponding spectral characterization under the same common-eigenbasis assumption, where the matrix-valued problem decomposes into single-index-type blocks, and conjectured its validity in the general case. Outside this simultaneously diagonalizable setting, the matrix weights need not commute, the scalar decomposition is unavailable, and a rigorous characterization of the transition remained open.

\medskip

This gap reflects a genuine random-matrix difficulty: in the general case, the matrix-valued problem no longer decomposes into scalar blocks. The matrices $\bm{T}(\bm{y})$ need not commute or admit a common eigenbasis, and the blocks of $\bm{D}_n$ are built from the same Gaussian covariates and are therefore strongly dependent. At the same time, the matrix weights are statistically coupled to the signal coordinates of the covariates, so that the planted structure enters through this dependence rather than as an independent additive spike. The resulting ensemble is therefore neither a classical spiked covariance model nor a direct extension of the scalar weighted-covariance setting. Establishing a BBP-type transition requires simultaneously characterizing the limiting bulk spectrum and its edge, ruling out spurious extreme eigenvalues, identifying signal-induced outliers, and controlling the associated spectral estimator.

In this work, we develop such a theory for arbitrary bounded symmetric matrix-valued preprocessing maps of any fixed working dimension. First, we prove that the empirical spectral distribution converges almost surely to a deterministic compactly supported distribution, whose Stieltjes transform is given by the normalized trace of the solution to a matrix Dyson equation, and establish upper-edge confinement of the noise spectrum. Second, we derive a finite-dimensional deterministic equation governing possible outliers and establish a sharp phase transition for the largest eigenvalue: below the transition, it sticks to the bulk edge, whereas above it, an outlier separates from the bulk and the associated spectral estimator achieves weak recovery of the latent subspace. Finally, we optimize the spectral transition over all bounded matrix-valued preprocessing maps and all fixed working dimensions. The AMP-derived preprocessing introduced in~\cite{defilippis2025} attains the optimal spectral threshold, thereby proving the general spectral conjecture formulated there.

\medskip

The proof requires overcoming two main random-matrix difficulties. The first is to characterize the noise spectrum and control it all the way to its upper edge. We derive a deterministic matrix-valued limit for the partial trace of the resolvent through a self-consistent equation. Unlike the standard matrix Dyson equation with a linear self-energy operator~\cite{ajanki2019,alt2020}, the equation arising here involves the nonlinear self-energy map 
\[
\bm{M} \mapsto \E_{\bm{y}} \left[ \bm{T} (\bm{y}) \left( \bm{I}_p + \frac{1}{\alpha} \bm{M} \bm{T}(\bm{y}) \right)^{-1} \right],
\]
so the existing MDE theory does not apply directly. We establish existence and uniqueness of the relevant matrix-valued Stieltjes transform using fixed-point and positivity arguments in the spirit of~\cite{helton2007}, and derive the limiting empirical spectral distribution using resolvent identities and leave-one-out estimates.

Identifying the limiting distribution is not sufficient for the outlier analysis: one must additionally prove \emph{spectral confinement}, ruling out a vanishing fraction of noise eigenvalues above the limiting bulk edge. After rotating into signal and noise coordinates and conditioning on the responses, the noise block becomes a matrix-weighted Gaussian covariance ensemble with random, noncommuting coefficients. Since the preprocessing matrices are not assumed to be positive semidefinite, this ensemble is in general a signed covariance model rather than a standard sample-covariance matrix. We introduce a self-adjoint linearization and apply the spectral comparison theorem of Bandeira, Cipolloni, Schr\"{o}der, and van Handel~\cite{bandeira2026}, conditionally on the responses, to compare its spectrum with that of an associated free Gaussian operator. To control the upper edge of the free model, we adapt a Fock-space argument of Lehner~\cite{lehner1999} to the present matrix-weighted setting and evaluate the resulting finite-dimensional variational bound at the solution of the limiting MDE. This yields the upper-edge confinement required for the subsequent outlier analysis. In the positive-semidefinite case, related variational edge formulas follow from the recent work by Parmaksiz and van Handel~\cite{parmaksiz2025}; our sign-indefinite setting instead requires a more general self-adjoint treatment. The same analysis yields an exact variational characterization of the limiting upper spectral edge, independently obtained by Montanari and Saeed~\cite{montanarisaeed2026}.

The second difficulty is to determine when the finite-dimensional signal creates an eigenvalue above the confined noise spectrum and to relate this separation to recovery. A Schur complement reduction shows that eigenvalues outside the bulk are governed by the zeros of a random matrix-valued function of dimension $pr$. We prove that this function converges uniformly, away from the bulk, to a finite-dimensional deterministic limit. The largest eigenvalue of the limiting matrix is strictly decreasing above the bulk edge, so its zero crossing yields the sharp transition and determines the asymptotic location of the largest eigenvalue. In the supercritical regime, the corresponding block eigenvector equation shows that separation from the bulk forces a nonvanishing component in the signal subspace, and hence weak recovery. In this way, control of the infinite-dimensional noise spectrum reduces the recovery transition to a finite-dimensional deterministic problem.

\subsection{Notation}
We first fix our notation. For a positive integer \(n\), we write \([n] \coloneqq \{1,\ldots,n\}\). We denote by \(\C_+ \coloneqq \{z\in\C \colon \Im z>0\}\) the complex upper half-plane. Let
\[
\sym_n(\R) \coloneqq \{\bm{X}\in\R^{n\times n}\colon \bm{X}=\bm{X}^{\top}\}
\]
denote the space of real symmetric \(n\times n\) matrices. We write \(\sym_n^+(\R)\) and \(\sym_n^{++}(\R)\) for the spaces of real positive semidefinite and positive definite \(n\times n\) matrices. Let
\[
\operatorname{Herm}_n(\C) \coloneqq \{\bm{X}\in\C^{n\times n}\colon \bm{X}=\bm{X}^\ast \}
\]
denote the space of complex Hermitian \(n\times n\) matrices. For \(\bm{X} \in \operatorname{Herm}_n(\C)\), we denote its spectrum, that is, the set of its eigenvalues, by \(\mathrm{sp} (\bm{X})\). We denote its eigenvalues, counted with multiplicity and arranged in nonincreasing order, by
\[
\lambda_1(\bm{X})\geq\cdots\geq\lambda_n(\bm{X}),
\]
and its empirical spectral measure by
\[
\hat{\mu}_{\bm{X}} \coloneqq \frac{1}{n}\sum_{i=1}^n\delta_{\lambda_i(\bm{X})}.
\]
We write \(\Tr\) for the unnormalized matrix trace and
\[
\tr_n (\bm{X}) \coloneqq \frac{1}{n}\Tr(\bm{X})
\]
for the normalized trace of an \(n\times n\) matrix \(\bm{X}\). 

For matrices \(\bm{X}=(X_{\mu\nu})_{\mu,\nu \in [k]}\in\C^{k \times k}\) and \(\bm{Y}\in\C^{n\times n}\), we denote their Kronecker product by 
\[ 
\bm{X} \otimes\bm{Y} \coloneqq \left (X_{\mu\nu} \bm{Y} \right)_{\mu,\nu\in[k]} \in\C^{kn \times kn}. 
\] 
Thus, \(\bm{X}\otimes\bm{Y}\) is viewed as a \(k\times k\) block matrix whose blocks have size \(n\times n\). More generally, let \(\bm{X} = (\bm{X}_{\mu\nu})_{\mu,\nu\in[k]} \in\C^{kn\times kn}\), \(\bm{X}_{\mu\nu}\in\C^{n\times n}\), be viewed as a \(k\times k\) block matrix. We define its normalized partial trace over the \(n\)-dimensional factor by
\begin{equation} \label{eq:partial_trace}
\tr_n^{(k)}(\bm{X}) \coloneqq (\bm{I}_k\otimes\tr_n)(\bm{X}) = \left( \frac{1}{n}\Tr(\bm{X}_{\mu \nu}) \right)_{\mu,\nu\in[k]} \in\C^{k\times k}.
\end{equation}

For \(\bm{X},\bm{Y} \in \operatorname{Herm}_n(\C)\), we use the Loewner order and write \(\bm{X}\preceq\bm{Y}\) if \(\bm{Y} - \bm{X}\) is positive semidefinite, and \(\bm{X}\prec\bm{Y}\) if \(\bm{Y}-\bm{X}\) is positive definite. For \(\bm{X} \in \C^{n \times n}\), we define its imaginary part by
\[
\Im \bm{X} \coloneqq \frac{\bm{X} - \bm{X}^\ast}{2i}.
\]

\subsection{Model} \label{subsection_model}

We observe \(n\) independent samples \(\{(\bm{x}_i, \bm{y}_i)\}_{i=1}^n\), where 
\[
\bm{x}_i  \stackrel{\mathrm{i.i.d.}}{\sim}\mathcal{N}(\bm{0}_d, \bm{I}_d),
\]
and, conditionally on \(\bm{x}_i\), the responses \(\bm{y}_i \in \R^q\) follow the multi-index model 
\begin{equation} \label{eq:multi-index-model}
\bm{y}_i \mid \bm{x}_i \sim P_\ast \left ( \cdot \mid \bm{W}_\ast^\top \bm{x}_i \right ).
\end{equation}
Here, \(\bm{W}_\ast \coloneqq [\bm{w}_{\ast,1},\ldots,\bm{w}_{\ast,r}] \in \mathbb{R}^{d\times r}\) has orthonormal columns and \(P_\ast (\, \cdot \mid \bm{s})\), \(\bm{s} \in \R^r\), is a conditional probability distribution on \(\R^q\). Thus, the response depends on the covariate only through the \(r\)-dimensional projection \(\bm{W}_\ast^\top\bm{x}_i\). 

The matrix \(\bm{W}_\ast\) defines an \(r\)-dimensional signal subspace 
\[
\mathcal{S}_\ast \coloneqq \mathrm{span}(\bm{W}_\ast)\subset \R^d. 
\]
The parametrization is invariant under orthogonal changes of basis within \(\mathcal{S}_\ast\): for any orthogonal matrix \(\bm{O} \in \R^{r \times r}\), replacing \(\bm{W}_\ast\) by \(\bm{W}_\ast \bm{O}\) and \(P_\ast (\, \cdot \mid \bm{s})\) by \(P_\ast (\, \cdot \mid \bm{O} \bm{s})\) leaves the conditional distribution of the response unchanged. Our inferential target is therefore the subspace \(\mathcal{S}_\ast\), rather than a particular orthonormal basis.

We consider estimators \(\widehat{\bm{W}} \in \R^{d \times p}\), where \(p\geq1\) is a fixed working dimension, and seek a nonvanishing asymptotic alignment of \(\widehat{\bm{W}}\) with \(\mathcal{S}_\ast\). We work throughout in the proportional high-dimensional regime. 
\begin{hyp} \label{hyp:prop_limit} 
As \(n,d\to\infty\),
\[
\frac{n}{d} \to \alpha \in (0,\infty), 
\]
while \(r\), \(q\), and \(p\) remain fixed.
\end{hyp}

Motivated by the linearization of an AMP iteration, we consider the following class of matrix-valued spectral estimators. The AMP derivation is deferred to Appendix~\ref{appendix:AMP_matrix}.

\begin{defn} \label{def:spectral_matrix}
Let \(\bm{T} \colon \R^q  \to  \sym_p (\R)\) be a measurable preprocessing function. We define the spectral matrix by
\begin{equation*}
\bm{D}_n \coloneqq  \frac{1}{n} \sum_{i=1}^n \bm{T}(\bm{y}_i) \otimes \bm{x}_i \bm{x}_i^\top \in \R^{pd \times pd}.
\end{equation*}
\end{defn}
The results below apply to any fixed \(p\) and to any preprocessing map satisfying the stated assumptions. In the AMP derivation of Appendix~\ref{appendix:AMP_matrix}, we instead use the true dimension \(r\) and the true response channel \(P_\ast\). This leads to the optimal preprocessing map introduced later in Definition~\ref{defn:optimal_preprocessing_map}.

\medskip

Our spectral estimator is constructed from a nonzero eigenvector \(\bm{v}_1 \in \R^{dp}\) associated with the largest eigenvalue of \(\bm{D}_n\). Define
\begin{equation} \label{eq:estimator}
\widehat{\bm{W}} \coloneqq \sqrt{d} \frac{\operatorname{mat}_{d \times p} (\bm{v}_1)}{\|\bm{v}_1\|_2} \in \R^{d \times p},
\end{equation}
where \(\operatorname{mat}_{d \times p}\) denotes the inverse of column-wise vectorization. Then \(\|\widehat{\bm{W}}\|_{\mathrm{F}}^2 = d\). Since the column space of \(\widehat{\bm{W}}\) has dimension at most \(p\), choosing \(p<r\) limits the number of signal directions it can represent, whereas \(p>r\) allows for additional uninformative directions. To quantify the alignment of \(\widehat{\bm{W}}\) with the signal subspace, let \(\bm{\Pi}_\ast \coloneqq \bm{W}_\ast\bm{W}_\ast^\top\) denote the orthogonal projector onto \(\mathcal{S}_\ast \). We define the normalized overlap by
\begin{equation} \label{eq:normalized_overlap} 
\operatorname{Ov}(\widehat{\bm{W}},\mathcal{S}_\ast) \coloneqq \frac{ \|\bm{\Pi}_\ast\widehat{\bm{W}}\|_{\mathrm{F}}^2 }{ \|\widehat{\bm{W}}\|_{\mathrm{F}}^2 } = \frac{ \|\bm{W}_\ast^\top\widehat{\bm{W}}\|_{\mathrm{F}}^2 }{ d }. 
\end{equation} 
Since \(\bm{\Pi}_\ast\) is an orthogonal projector, this overlap is invariant under an orthogonal change of basis \(\bm{W}_\ast\mapsto\bm{W}_\ast\bm{O}\). We formalize the notion of weak recovery as follows.
\begin{defn}
The spectral estimator \(\widehat{\bm{W}}\) achieves \emph{weak recovery} of the signal subspace \(\mathcal{S}_\ast\) if there exists a constant \(\epsilon>0\) such that
\[
\mathbb{P} \left ( \operatorname{Ov}(\widehat{\bm{W}},\mathcal{S}_\ast) \ge \epsilon \right) \to 1,
\]
as \(n,d\to \infty\) under Assumption~\ref{hyp:prop_limit}.
\end{defn}

\subsection{Main results} \label{subsection_main_results}

We now state our main results. We first characterize the deterministic limiting bulk spectrum of the spectral matrix \(\bm{D}_n\). We then establish a BBP-type transition for its leading eigenvalue and prove weak recovery of the signal subspace. Finally, we optimize the transition threshold over all fixed working dimensions and admissible preprocessing maps. 

Because our spectral construction permits a general matrix-valued preprocessing map \(\bm{T}\), we begin with the following regularity assumption.

\begin{hyp} \label{hyp:preprocessing}
The map \(\bm{T} \colon\R^q \to \sym_p(\R)\) is measurable and nontrivial, in the sense that 
\[
\mathbb{P} \left( \bm{T}(\bm{y})\neq\bm{0}_{p\times p} \right)>0.
\]
Moreover, there exists a finite constant \(C_T >0\) such that
\[
\|\bm{T}(\bm{y})\|_{\mathrm{op}} \leq C_T
\]
almost surely, where
\[ 
\bm{s}\sim\mathcal{N}(\bm{0}_r,\bm{I}_r), \qquad \bm{y}\mid\bm{s} \sim P_\ast(\,\cdot\mid\bm{s}). 
\]
\end{hyp}

Our first result characterizes the limiting bulk spectrum of the spectral matrix \(\bm{D}_n\) of Definition~\ref{def:spectral_matrix}. 

\begin{thm} \label{thm:bulk}
Suppose that Assumptions~\ref{hyp:prop_limit} and~\ref{hyp:preprocessing} hold. Then there exists a unique compactly supported probability measure \(\mu_\alpha\) on \(\R\) such that the empirical spectral measure of \(\bm{D}_n\), 
\[
\hat{\mu}_{\bm{D}_n} \coloneqq \frac{1}{pd} \sum_{i=1}^{pd} \delta_{\lambda_i(\bm{D}_n)},
\]
converges weakly to \(\mu_\alpha\) almost surely. Let \(\lambda_+ (\alpha) \coloneqq \sup \supp (\mu_\alpha)\) denote the upper edge of \(\mu_\alpha\). Then, for every fixed \(\varepsilon > 0\),
\[
\mathbb{P} \left( \# \left\{ i\in[pd] \colon \lambda_i (\bm{D}_n)> \lambda_+ (\alpha) +\varepsilon \right\} \leq pr \right) \to 1.
\]
Equivalently,
\[
\mathbb{P} \left( \lambda_{pr+1}(\bm{D}_n) \leq \lambda_+ (\alpha) +\varepsilon \right) \to 1.
\]
\end{thm}

The limiting measure \(\mu_\alpha\) is characterized through a self-consistent equation. More precisely, for \(z \in \C_+\), Proposition~\ref{prop:bulk} characterizes a deterministic matrix-valued function \(\bm{M}_\alpha(z)\in\C^{p\times p}\) as the unique solution of a Dyson-type matrix-valued self-consistent equation with the appropriate Stieltjes-transform properties. Its normalized trace \(m_\alpha (z) \coloneqq \frac{1}{p}\Tr \bm{M}_\alpha (z)\) is the Stieltjes transform of \(\mu_\alpha\). We refer to Proposition~\ref{prop:bulk} for the precise characterization. 

We next study the largest eigenvalue of \(\bm{D}_n\) and the associated spectral phase transition. To this end, we first introduce the deterministic quantities entering the transition criterion. Define 
\begin{equation} \label{eq:cond_second_moment}
\bm{C}(\bm{y}) \coloneqq \E [\bm{s} \bm{s}^\top \mid \bm{y}],
\end{equation}
where \(\bm{s}\sim\mathcal{N}(\bm{0}_r,\bm{I}_r)\) and \( \bm{y}\mid\bm{s} \sim P_\ast(\,\cdot \mid\bm{s})\). In particular, 
\[
\E_{\bm{y}} [\bm{C}(\bm{y})] = \E [\bm{s} \bm{s}^\top ] =\bm{I}_r.
\]
By Lemma~\ref{lem:continuation_off_support}, \(\bm{M}_\alpha\) admits an analytic continuation to \(\C \setminus \supp (\mu_\alpha)\). For \(x \in \R \setminus \supp (\mu_\alpha)\), we use \(\bm{M}_\alpha(x)\) to denote this continuation, which is Hermitian. Moreover, in Section~\ref{section:MDE}, we show that
\[
\bm{I}_p+\frac{1}{\alpha} \bm{M}_\alpha(x) \bm{T} (\bm{y})
\]
is invertible almost surely for every \(x>\lambda_+(\alpha)\). We may therefore define
\begin{equation} \label{eq:Phi_alpha} 
\Phi_\alpha(x) \coloneqq \lambda_1 \left( \E_{\bm{y}} \left[ \bm{T}(\bm{y}) \left( \bm{I}_p+\frac{1}{\alpha} \bm{M}_\alpha(x) \bm{T} (\bm{y}) \right)^{-1} \otimes \bm{C}(\bm{y}) \right] \right), \qquad x>\lambda_+(\alpha).
\end{equation}
We also define
\begin{equation} \label{eq:h_alpha} 
h_\alpha(x) \coloneqq \Phi_\alpha (x) - x.
\end{equation} 
Finally, we define the edge gap of the phase transition by 
\begin{equation} \label{eq:Delta_alpha} 
\Delta_{\bm{T}}(\alpha) \coloneqq  \sup_{x>\lambda_+(\alpha)} h_\alpha(x) \in \R\cup\{+\infty\}. 
\end{equation} 
The sign of \(\Delta_{\bm{T}}(\alpha)\) determines whether the largest eigenvalue remains at the upper bulk edge or separates from it.

\begin{thm} \label{thm:BBP_transition} 
Suppose that Assumptions~\ref{hyp:prop_limit} and~\ref{hyp:preprocessing} hold. Then \(h_\alpha \colon (\lambda_+(\alpha),\infty) \to \R\) is continuous and strictly decreasing, and satisfies \(\lim_{x \to\infty} h_\alpha(x)=-\infty\). Consequently,
\[
\Delta_{\bm{T}}(\alpha) = \lim_{x \downarrow\lambda_+(\alpha)} h_\alpha(x). 
\]
Furthermore, the largest eigenvalue \(\lambda_1(\bm{D}_n)\) of \(\bm{D}_n\) undergoes the following phase transition.
\begin{enumerate} 
\item If \(\Delta_{\bm{T}}(\alpha)\leq 0\), then the equation \(h_\alpha(x)=0\) has no solution in \((\lambda_+(\alpha),\infty)\) and no eigenvalue separates above the upper bulk edge:
\[
\lambda_1(\bm{D}_n) \stackrel{\mathbb{P}}{\to} \lambda_+(\alpha).
\]
\item If \(\Delta_{\bm{T}}(\alpha)>0\), then there exists a unique solution \(\theta_\alpha>\lambda_+(\alpha)\) such that \(h_\alpha(\theta_\alpha)=0\). Moreover, an outlier emerges at the point \(\theta_\alpha\):
\[
\lambda_1(\bm{D}_n) \stackrel{\mathbb{P}}{\to} \theta_\alpha .
\]
\end{enumerate}
\end{thm} 

Eigenvalue separation also has an estimation consequence: whenever the leading eigenvalue separates from the bulk, the corresponding spectral estimator has a nonvanishing asymptotic overlap with the signal subspace.

\begin{thm} \label{thm:weak_recovery}
Suppose that Assumptions~\ref{hyp:prop_limit} and~\ref{hyp:preprocessing} hold and that \(\Delta_{\bm{T}}(\alpha)>0\). Let \(\bm{v}_1\in\R^{pd} \setminus \{\bm{0}\}\) be any eigenvector associated with \(\lambda_1(\bm{D}_n)\), and let \(\widehat{\bm{W}}\) be the corresponding spectral estimator defined in~\eqref{eq:estimator}. Then there exists a constant \(\epsilon>0\) such that
\[
\mathbb{P}\left( \operatorname{Ov}(\widehat{\bm{W}},\mathcal{S}_\ast) \geq \epsilon\right)\to1.
\]
In particular, any leading eigenvector of \(\bm{D}_n\) yields, through~\eqref{eq:estimator}, a spectral estimator achieving weak recovery of the signal subspace \(\mathcal{S}_\ast\) throughout the supercritical regime.
\end{thm}

Theorem~\ref{thm:BBP_transition} gives a sharp characterization of the largest-eigenvalue transition at every fixed sampling ratio \(\alpha\). We next study how this transition varies with \(\alpha\). For a fixed preprocessing map \(\bm{T} \colon \R^q \to \sym_p (\R)\), define
\begin{equation} \label{eq:crit_thresholds}
\begin{split}
\alpha_{\mathrm{c},\min}(\bm{T}) & \coloneqq \inf \left\{ \alpha>0 \colon  \Delta_{\bm{T}}(\alpha)>0 \right\}, \\
\alpha_{\mathrm{c},\max}(\bm{T}) &\coloneqq \inf\left\{ \bar{\alpha}>0 \colon \Delta_{\bm{T}}(\alpha)>0 \text{ for every }\alpha>\bar{\alpha} \right\},
\end{split}
\end{equation}
with the convention \(\inf \emptyset=+\infty\). By construction, 
\[
\alpha_{\mathrm{c},\min}(\bm{T}) \leq \alpha_{\mathrm{c},\max}(\bm{T}). 
\]
The lower threshold is the infimum of the sampling ratios at which the supercritical regime can occur, whereas the upper threshold is the infimum of the cutoff values beyond which the spectral method is supercritical for every larger sampling ratio. The two thresholds need not coincide unless additional monotonicity properties of \(\alpha \mapsto \Delta_{\bm{T}}(\alpha)\) are established. 

To ensure that the supercritical regime is eventually reached, we impose the following sufficient condition.

\begin{hyp} \label{hyp:population_separation}
The leading population eigenvalue in the signal component is strictly larger than the leading population eigenvalue in the noise component:
\[
\lambda_1 \left( \E_{\bm{y}} \left[\bm{T} (\bm{y})\otimes\bm{C}(\bm{y}) \right] \right) > \lambda_1 \left( \E_{\bm{y}}[\bm{T}(\bm{y})] \right).
\]
\end{hyp}

\begin{rmk}
Assumption~\ref{hyp:population_separation} has a natural population-level interpretation. Using the orthogonal decomposition
\[
\R^p \otimes\R^d = (\R^p\otimes\mathcal{S}_\ast) \oplus (\R^p\otimes\mathcal{S}_\ast^\perp),
\]
Gaussian rotational invariance gives
\[
\E[\bm{D}_n] = (\bm{I}_p \otimes \bm{W}_\ast) \E_{\bm{y}} \left[\bm{T}(\bm{y})\otimes\bm{C}(\bm{y}) \right] (\bm{I}_p\otimes\bm{W}_\ast)^\top + \E_{\bm{y}}[\bm{T}(\bm{y})]\otimes(\bm{I}_d-\bm{\Pi}_\ast).
\]
Thus, the population matrix is block diagonal, with signal block
\[
\E_{\bm{y}} \left[\bm{T}(\bm{y}) \otimes \bm{C}(\bm{y})\right]
\]
and noise block
\[
\E_{\bm{y}} [\bm{T}(\bm{y})]\otimes\bm{I}_{d-r}.
\]
Assumption~\ref{hyp:population_separation} therefore ensures that the leading population eigenspace lies in \(\R^p\otimes\mathcal{S}_\ast\). This should only be viewed as population-level intuition in the proportional regime, where operator-norm fluctuations cannot in general be neglected. The exact fixed-\(\alpha\) criterion for eigenvalue separation is instead \(\Delta_{\bm{T}}(\alpha)>0\); Assumption~\ref{hyp:population_separation} is used only as a sufficient condition ensuring that the supercritical regime is eventually reached as \(\alpha\to\infty\). In the scalar rank-one case \(p=r=1\), it reduces to
\[
\E[T(\bm{y}) s^2]>\E[T(\bm{y})],
\]
the population-separation condition of Lu and Li~\cite{lu2020}.
\end{rmk}

We are now ready to state the spectral phase-transition result as a function of the sampling ratio \(\alpha\).

\begin{prop} \label{prop:crit_thresholds}
Suppose that Assumptions~\ref{hyp:prop_limit},~\ref{hyp:preprocessing}, and~\ref{hyp:population_separation} hold. Then
\[
0 < \alpha_{\mathrm{c},\min}(\bm{T}) \leq \alpha_{\mathrm{c},\max}(\bm{T}) <\infty.
\]
More precisely, there exist constants \(0<\alpha_-<\alpha_+<\infty\) such that \(\Delta_{\bm{T}}(\alpha)\leq0\) for every \(0<\alpha<\alpha_-\) and \(\Delta_{\bm{T}}(\alpha)>0\) for every \(\alpha>\alpha_+\). Thus, the following phase-transition statements hold.
\begin{enumerate}
\item If \(0<\alpha< \alpha_{\mathrm{c},\min}(\bm{T})\), then
\[
\lambda_1(\bm{D}_n) \stackrel{\mathbb{P}}{\to} \lambda_+(\alpha).
\]
\item If \(\alpha> \alpha_{\mathrm{c},\max}(\bm{T})\), then there exists a unique \(\theta_\alpha>\lambda_+(\alpha)\) satisfying \(h_\alpha(\theta_\alpha)=0\) and   
\[
\lambda_1(\bm{D}_n) \stackrel{\mathbb{P}}{\to} \theta_\alpha .
\]
\end{enumerate}
\end{prop}

Proposition~\ref{prop:crit_thresholds} does not characterize the entire interval \( [\alpha_{\mathrm{c},\min}(\bm{T}), \alpha_{\mathrm{c},\max}(\bm{T})]\). Without an additional monotonicity result for \(\alpha \mapsto \Delta_{\bm{T}} (\alpha)\), the supercritical set need not be an interval: the method may, in principle, alternate between subcritical and supercritical behavior inside this range. If the two thresholds coincide, however, they define a single sharp sampling transition.

We finally turn to the optimal choice of the preprocessing map. Our goal is to identify the smallest spectral threshold (in terms of the sample complexity $\alpha$) achievable over all matrix-valued preprocessing maps $\bm{T}$. At the spectral level, 
the optimal preprocessing rule is determined by the second-order fluctuations of the conditional second moment \(\bm{C}(\bm{y})\). We impose the following assumption.

\begin{hyp} \label{hyp:second_order_channel}
Let \(\bm{s} \sim\mathcal{N}(\bm{0}_r,\bm{I}_r)\) and \(\bm{y}\mid\bm{s} \sim P_\ast(\,\cdot\mid\bm{s})\). Recall \(\bm{C}(\bm{y}) \coloneqq \E [ \bm{s}\bm{s}^\top \mid\bm{y}]\) from~\eqref{eq:cond_second_moment}. We assume that the conditional second moment is nontrivial: 
\[ 
\mathbb{P}\left( \bm{C}(\bm{y})\neq\bm{I}_r \right)>0,
\]
and that there exists a constant \(c>0\) such that
\[
\bm{C}(\bm{y}) \succeq c\bm{I}_r \quad\text{a.s.}
\]
\end{hyp}

To describe the optimal preprocessing map, consider the linear operator \(\mathcal{A} \colon \sym_r (\R) \to \sym_r(\R) \) given by
\begin{equation} \label{eq:operator_A}
\mathcal{A} [\bm{H}] \coloneqq \E_{\bm{y}} \left[ (\bm{C} (\bm{y})-\bm{I}_r) \bm{H} (\bm{C} (\bm{y}) -\bm{I}_r) \right].
\end{equation}
We equip \(\sym_r(\R)\) with the Frobenius norm and denote the induced operator norm of \(\mathcal{A}\) by
\[
\|\mathcal{A}\|_{\mathrm{op}} \coloneqq \sup_{\substack{\bm{H} \in\sym_r(\R)\\ \|\bm{H}\|_{\mathrm{F}}=1}} \|\mathcal{A}[\bm{H}]\|_{\mathrm{F}}.
\]
The operator \(\mathcal{A}\) is self-adjoint with respect to the Frobenius inner product and preserves the cone of positive semidefinite matrices. Hence, by the Perron--Frobenius theorem for positive linear maps, its operator norm is an eigenvalue admitting a nonzero positive semidefinite eigenmatrix. We therefore define
\[
\mathscr{P}_\ast \coloneqq \left \{ \bm{H} \succeq \bm{0}_r \colon \Tr (\bm{H})=1, \enspace \mathcal{A}[\bm{H}] = \|\mathcal{A}\|_ {\mathrm{op}}\bm{H} \right\}.
\]
The set \(\mathscr{P}_\ast\) is nonempty and convex, and thus its relative interior \(\operatorname{relint}(\mathscr{P}_\ast)\) is nonempty. 

\begin{defn}
\label{defn:optimal_preprocessing_map}
Choose any \(\bm{H}_\ast \in \operatorname{relint}(\mathscr{P}_\ast)\) and set \(p_\ast \coloneqq \operatorname{rank}(\bm{H}_\ast).\) Let \(\bm{U}_\ast \in \R^{r\times p_\ast}\) have orthonormal columns spanning \(\operatorname{range} (\bm{H}_\ast)\).  Define the restricted conditional second moment
\[
\bm{C}_\ast (\bm{y}) \coloneqq \bm{U}_\ast^\top \bm{C} (\bm{y}) \bm{U}_\ast ,
\]
and the preprocessing map
\begin{equation} \label{eq:optimal_preprocessing}
\bm{T}_\ast (\bm{y}) \coloneqq \bm{I}_{p_\ast} - \bm{C}_\ast (\bm{y})^{-1} .
\end{equation}
We call \(\bm{T}_\ast\) the \emph{optimal preprocessing map}.
\end{defn}
The lower bound in Assumption~\ref{hyp:second_order_channel} implies \(\bm{C}_\ast (\bm{y}) \succeq c \bm{I}_{p_\ast}\) almost surely, and therefore \(\bm{T}_\ast\) is well defined and uniformly bounded. Moreover, \(\bm{T}_\ast\) is nontrivial and hence defines an admissible preprocessing map.

\begin{rmk}
The construction above has a simple interpretation. The matrix \(\bm{H}_\ast\) selects the part of the signal space in which the strongest second-order information is present, while the columns of \(\bm{U}_\ast\) provide an orthonormal basis for this subspace. The matrix \(\bm{C}_\ast(\bm{y})\) is therefore just the conditional second moment \(\bm{C}(\bm{y})\) restricted to the selected subspace. The particular choice of the orthonormal basis \(\bm{U}_\ast\) only changes the representation by an orthogonal change of coordinates. In particular, if \(\operatorname{rank}(\bm{H}_\ast)=r\), then the selected subspace is all of \(\R^r\), and we may take \(\bm{U}_\ast=\bm{I}_r\). In this case,
\[
\bm{T}_\ast (\bm{y}) = \bm{I}_r-\bm{C}(\bm{y})^{-1}.
\]
\end{rmk}

The structural properties of this Perron reduction that are used in the optimality argument are collected in Lemma~\ref{lem:perron_reduction_properties}. In Appendix~\ref{appendix:AMP_matrix}, we provide a heuristic derivation of this preprocessing map from Bayesian and TAP considerations. We now turn to the rigorous characterization of its optimal spectral threshold.

\begin{thm} \label{thm:opt}
Suppose that Assumptions~\ref{hyp:prop_limit} and~\ref{hyp:second_order_channel} hold. Then the lower and upper spectral thresholds of the preprocessing map \(\bm{T}_\ast\) coincide:
\[
\alpha_{\mathrm{c},\min}(\bm{T}_\ast) = \alpha_{\mathrm{c},\max}(\bm{T}_\ast) \eqqcolon \alpha_{\mathrm{c}}^\ast.
\]
Moreover, \(0<\alpha_{\mathrm{c}}^\ast<\infty\) and 
\begin{align}
\label{eq:def_alpha_cast}
\frac{1}{\alpha_{\mathrm c}^\ast}= \left\| \E_{\bm{y}}\left[ \left(\bm{C}(\bm{y})-\bm{I}_r \right) \otimes \left(\bm{C}(\bm{y})-\bm{I}_r \right) \right] \right\|_{\mathrm{op}}.
\end{align}
Furthermore, for every fixed \(p\geq1\) and every measurable preprocessing map \(\bm{T}\colon \R^q\to\sym_p(\R) \) satisfying Assumption~\ref{hyp:preprocessing}, it holds that
\[
\alpha_{\mathrm{c},\min}(\bm{T}) \geq \alpha_{\mathrm c}^\ast .
\]
\end{thm}

Theorem~\ref{thm:opt} identifies the smallest spectral threshold achievable by any admissible finite-dimensional matrix-valued preprocessing rule. The preprocessing map \(\bm{T}_\ast\), whose working dimension satisfies \(p_\ast\leq r\), attains this threshold, and no choice of a larger working dimension can improve upon \(\alpha_{\mathrm{c}}^\ast\). In particular, the optimal threshold depends only on the strength of the second-order fluctuations of \(\bm{C}(\bm{y})\), as measured by the operator appearing in~\eqref{eq:def_alpha_cast}. Theorem~\ref{thm:opt} proves Conjecture~3.10 of~\cite{defilippis2025}.

\subsection{Related work} \label{subsection:related_work}

\textbf{Gradient-based feature learning in multi-index models.}
Multi-index models provide a natural teacher setting for finite-width two-layer neural networks: learning the target requires identifying a low-dimensional collection of features hidden in a high-dimensional covariate space. The committee machine already exhibits statistical-to-computational gaps in the proportional Gaussian regime~\cite{aubin2019}. A growing line of work studies how gradient-based algorithms recover this latent structure, including representation and subspace learning by gradient descent, gradient flow, or SGD~\cite{damian2022,mousaviHosseini2023,bietti2025,simsek2025}, high-dimensional SGD dynamics and sequential learning phenomena~\cite{abbe2023,collinswoodfin2024}, and the gains obtained by reusing batches or individual samples across gradient steps~\cite{dandi2024,arnaboldi2024}. Recent results further investigate hierarchical and increasingly general multi-index targets, achieving near-linear sample complexity for subspace recovery in broad settings~\cite{zhangWang2025,defilippisHierarchical2026}. Together, these works motivate weak subspace recovery as a natural first stage of feature learning. We refer to~\cite{brunahsu2025} for a broader survey of algorithms for multi-index models.

\textbf{Single-index spectral methods.}
The single-index setting provides the classical benchmark for spectral weak recovery in nonlinear Gaussian models. For generalized linear models with Gaussian covariates, spectral estimators take the form of scalar-weighted sample covariance matrices. Lu and Li~\cite{lu2020} obtained a precise high-dimensional characterization of their spectral phase transition, while Mondelli and Montanari~\cite{mondelli2019} identified the fundamental weak-recovery threshold and a preprocessing function attaining it, with phase retrieval as a canonical example. The design of optimal spectral preprocessing was further developed from optimization, message-passing, and statistical-physics perspectives~\cite{luo2019,ma2021,maillard2022}, while the joint behavior of linear and spectral estimators was characterized in~\cite{mondelli2022b}. Spectral initialization for subsequent AMP iterations was rigorously analyzed in~\cite{mondelli2022a}. More recent work has extended this spectral theory beyond i.i.d.\ Gaussian designs to correlated Gaussian and, subsequently, general orthogonally invariant designs~\cite{zhang2025,zhang2026orthogonal}. The scalar structure reduces both the limiting spectral law and the outlier equation to scalar self-consistent equations.

\textbf{Multi-index weak recovery and spectral methods.}
Estimating the index subspace in multi-index models has a long history in sufficient dimension reduction, including sliced inverse regression and related inverse-moment methods~\cite{li1991,cook2000}; see~\cite{brunahsu2025} for a recent survey. In the proportional high-dimensional regime, Troiani et al.~\cite{troianifundamental} characterized the instability threshold of the uninformative AMP fixed point and the resulting computational threshold for weak recovery with infinitesimal side information. Related precise spectral asymptotics for problems with multiple latent signals were obtained for mixed generalized linear models with independent signal components~\cite{zhang2026}. Kova{\v{c}}evi{\'c}, Zhang, and Mondelli~\cite{kovavcevic2025} studied scalar-preprocessed spectral estimators for general multi-index models, deriving precise eigenvalue and eigenvector asymptotics and optimizing the scalar preprocessing function. When the relevant conditional second-moment matrices are simultaneously diagonalizable, they showed that the resulting optimal spectral threshold coincides with the AMP threshold. In parallel, Defilippis et al.~\cite{defilippis2025} derived a genuinely matrix-valued spectral estimator by linearizing AMP and predicted its transition through state evolution. They established the corresponding spectral characterization under the same common-eigenbasis assumption, where the matrix-valued problem decomposes into scalar single-index-type blocks, and conjectured its validity in the general noncommutative setting. Beyond this second-order proportional regime, Damian, Lee, and Bruna~\cite{damian2025} developed higher-order spectral estimators and characterized sharp sample-complexity exponents for efficient subspace recovery in general Gaussian multi-index models. Our results remove the simultaneous-diagonalizability assumption, analyze the full matrix-valued operator, and prove the conjectured transition in general.

\textbf{AMP- and TAP-inspired spectral methods.}
A related line of work derives spectral estimators by linearizing message-passing or TAP equations around an uninformative fixed point. For block-structured spiked Wigner models, an optimal spectral estimator arising from linearized AMP was rigorously analyzed in~\cite{mergny24a}, where its sharp BBP transition was established. The matrix-valued estimator of Defilippis et al.~\cite{defilippis2025}, discussed above, arises from the same AMP-linearization principle. More recently, two concurrent works developed related constructions for correlated multi-view inference. Yang, Sen, and Lu~\cite{yang2026} derive a spectral estimator for multi-view spiked Wigner models by linearizing AMP and establish its sharp weak-recovery transition, while Du, Hu, and Lepsveridze~\cite{du2026} develop TAP-inspired optimal spectral algorithms for correlated two-view models, including high-dimensional canonical correlation analysis and correlated spiked Wigner and Wishart models. These works provide further examples in which linearized message-passing or statistical-physics equations lead to spectral operators attaining sharp inference thresholds, although the resulting random-matrix structures differ from the matrix-weighted covariance ensemble considered here.

\textbf{Spectral geometry and gradient-based dynamics.}
Spectral transitions in loss Hessians have long been connected to the ability of gradient dynamics to escape uninformative regions of high-dimensional nonconvex landscapes. In spiked matrix--tensor models and phase retrieval, BBP-type transitions in the Hessian identify signal-bearing unstable directions and predict the success of gradient flow~\cite{saraomannelli2019,saraomannelli2020}. More recently, BBP transitions in the Hessian at initialization have been studied in overparameterized teacher--student neural networks, revealing how overparameterization can modify the emergence of informative curvature~\cite{annesi2026}.

More broadly, high-dimensional SGD dynamics in planted inference models have been characterized through effective low-dimensional limits, including critical scaling regimes in which stochastic corrections modify the deterministic dynamics~\cite{benarous2024}. Subsequent work has investigated how spectral geometry itself evolves along these dynamics. In phase retrieval, the time-dependent Hessian exhibits finite windows of informative negative curvature~\cite{bonnaire2025}, while in~\cite{benarous2025,benarous2026} the authors establish connections between high-dimensional SGD dynamics, emerging Hessian outliers, and alignment with low-dimensional signal eigenspaces. Related work on quadratic two-layer networks connects one-pass SGD escape dynamics with the geometry of the population-loss Hessian~\cite{bocchi2026}, while solvable random-matrix models of gradient flow exhibit transient BBP outliers driven by covariance anisotropy~\cite{coeurdoux2026}.

For multi-index learning, related spectral mechanisms have recently appeared in analyses of early feature discovery. Defilippis et al.~\cite{defilippisHierarchical2026} interpret a scalar spectral estimator as the Hessian of a modified objective and as the small-learning-rate limit of first-layer gradient descent. Zhang et al.~\cite{zhangWang2025} show that, near initialization, layer-wise gradient descent can implement a power-iteration-like mechanism on a local Hessian and recover the hidden subspace for generic multi-index targets. Complementary proportional-regime analyses study signal-bearing Hessian outliers along two-layer-network training~\cite{montanariWang2026}, while iterative spectral mechanisms have also been proposed as models of hierarchical feature construction across layers~\cite{dandi2026}. These works analyze spectral operators tied to specific models or learning dynamics, whereas the present paper isolates a fixed matrix-valued operator and develops its noncommutative random-matrix theory.

\textbf{Random matrix context.}
Scalar-weighted covariance matrices are closely related to classical sample-covariance models~\cite{silverstein1995,bai2010}, while matrix Dyson equations provide a broad framework for correlated random matrices~\cite{ajanki2019,alt2020}. Sharp control of spectral edges in covariance- and Gram-type models is often obtained through local laws and eigenvalue rigidity; see, for instance,~\cite{alt2017}. Such results provide substantially finer spectral information than is needed here. In our setting, moreover, a direct local-law approach would have to accommodate noncommuting matrix weights, correlated covariance blocks, and the nonstandard nonlinear self-consistent equation described above. Instead, we only require confinement of the noise spectrum at its upper edge. We obtain this through the spectral comparison approach of~\cite{bandeira2026}: after conditioning on the responses and introducing a self-adjoint linearization, the noise ensemble is compared with an associated free Gaussian operator, whose upper edge is controlled through a Lehner-type variational argument~\cite{lehner1999}. Thus, rather than establishing a full local law for the matrix-weighted ensemble, we prove the weaker spectral confinement sufficient for the subsequent finite-dimensional outlier analysis.

\subsection{Overview} 
The remainder of the paper is organized as follows. Section~\ref{section:outline_proofs} gives a detailed outline of the proofs of our main results. In Section~\ref{section:MDE}, we introduce and analyze the deterministic self-consistent equation governing the limiting bulk spectrum. Sections~\ref{section:bulk}--\ref{section:optimal_preprocessing} contain the proofs of our main results: Section~\ref{section:bulk} establishes the bulk spectral results, Section~\ref{section:phase_transition} analyzes the spectral phase transition and weak recovery, and Section~\ref{section:optimal_preprocessing} studies the AMP-derived optimal preprocessing map. Appendix~\ref{appendix:AMP_matrix} provides the AMP heuristics and motivation underlying the optimal spectral construction, Appendix~\ref{appendix:limiting_variational_edge} derives the exact variational characterization of the limiting upper spectral edge, and Appendix~\ref{appendix:auxiliary} collects auxiliary results used throughout the proofs.
\\

\textbf{Acknowledgments.} We thank Johannes Alt, Afonso Bandeira, Yatin Dandi, Bruno Loureiro, and Subhabrata Sen for helpful discussions about this problem. We acknowledge funding from the Swiss National Science Foundation grants OperaGOST (grant number 200021 200390), DSGIANGO (grant number 225837), and from the Simons
Collaboration on the Physics of Learning and Neural Computation via the Simons Foundation grant (\# 1257412).
\section{Outline of proofs}  \label{section:outline_proofs}

In this section, we outline the proofs of the main results stated in Subsection~\ref{subsection_main_results}. The goal is to explain the main ideas and the structure of the arguments before presenting the technical estimates. It therefore also provides a roadmap for the proofs and indicates where the main technical arguments are established.

\subsection{Signal--noise decomposition} \label{subsection:preliminaries}

We begin by decomposing the spectral matrix \(\bm{D}_n\) from Definition~\ref{def:spectral_matrix} into a noise block and a finite-rank signal perturbation.

By rotational invariance of the Gaussian design, after an orthogonal change of coordinates we may assume without loss of generality that 
\[ 
\bm{W}_\ast = \begin{pmatrix} \bm{I}_r \\ \bm{0}_{(d-r)\times r} \end{pmatrix}. 
\]
We therefore write \(\bm{x}_i = (\bm{s}_i^\top, \bm{u}_i^\top)^\top\), where \(\bm{s}_i \sim \mathcal{N}(\bm{0}_r, \bm{I}_r)\) and \(\bm{u}_i \sim  \mathcal{N}(\bm{0}_{d-r}, \bm{I}_{d-r})\) are independent. The pairs \((\bm{s}_i,\bm{u}_i)\) are also independent across samples. In these coordinates, the response model becomes
\[
\bm{y}_i \mid \bm{s}_i \sim P_\ast (\, \cdot \mid \bm{s}_i),
\]
so that \(\bm{u}_i\) is independent of \((\bm{s}_i,\bm{y}_i)\), and hence of \(\bm{T}(\bm{y}_i)\). 

Recall that the spectral matrix is given by
\[
\bm{D}_n = \frac{1}{n} \sum_{i=1}^n \bm{T}(\bm{y}_i) \otimes \bm{x}_i \bm{x}_i^\top \in \R^{pd \times pd}.
\]
After permuting the coordinates to group the signal and noise components, \(\bm{D}_n\) takes the block form
\begin{equation} \label{eq:D_n_prime}
\bm{D}'_n =
\begin{pmatrix}
\bm{A}_n & \bm{Q}_n^\top \\
\bm{Q}_n & \bm{P}_n
\end{pmatrix},
\end{equation}
where
\begin{align}
\bm{A}_n & =  \frac{1}{n} \sum_{i=1}^n \bm{T}(\bm{y}_i) \otimes \bm{s}_i \bm{s}_i^\top \in \R^{pr \times pr}, \label{eq:A_n} \\
\bm{Q}_n &= \frac{1}{n} \sum_{i=1}^n \bm{T}(\bm{y}_i) \otimes \bm{u}_i \bm{s}_i^\top \in \R^{p(d-r) \times pr}, \label{eq:Q_n}\\
\bm{P}_n &= \frac{1}{n} \sum_{i=1}^n \bm{T}(\bm{y}_i) \otimes \bm{u}_i \bm{u}_i^\top \in \R^{p(d-r) \times p(d-r)}. \label{eq:P_n}
\end{align}
The matrix \(\bm{P}_n\) represents the noise component. Its limiting spectral distribution determines the bulk spectrum of \(\bm{D}_n\), whereas the remaining blocks constitute a finite-rank perturbation that may generate outlying eigenvalues.

Since the rotation and permutation of coordinates do not change eigenvalues, \(\bm{D}_n\) and \(\bm{D}_n'\) share the same spectrum. Thus, for every \( z \notin \operatorname{sp} (\bm{P}_n) \), the Schur complement formula gives
\begin{equation} \label{eq:schur_det_lemma}
\begin{split}
\det (\bm{D}_n -z \bm{I}_{pd}) & = \det(\bm{P}_n - z \bm{I}_{p(d-r)}) \det \left ( \bm{A}_n -  z \bm{I}_{pr} - \bm{Q}_n^\top (\bm{P}_n - z \bm{I}_{p(d-r)})^{-1} \bm{Q}_n \right ).
\end{split}
\end{equation}
Therefore, if \( z \notin \operatorname{sp} (\bm{P}_n) \), then \(z\) is an eigenvalue of \(\bm{D}_n\) if and only if  
\begin{equation} \label{eq:finite_schur_outlier}
 \det \left ( \bm{A}_n - z \bm{I}_{pr}  - \bm{Q}_n^\top (\bm{P}_n - z \bm{I}_{p(d-r)})^{-1} \bm{Q}_n \right )=0.
\end{equation}
In particular, on any interval disjoint from \(\operatorname{sp}(\bm{P}_n)\), the eigenvalues of \(\bm{D}_n\) are characterized by the zeros of the \(pr\)-dimensional Schur complement determinant.

The proof therefore separates into two tasks. We first characterize the noise spectrum of \(\bm{P}_n\) and establish confinement at its upper edge. We then analyze the finite-dimensional Schur complement above that edge to locate signal-induced outliers and determine the spectral transition.

\subsection{Bulk spectrum}

We begin by studying the noise component \(\bm{P}_n\), which determines the limiting bulk spectrum of \(\bm{D}_n\). We denote by \(\hat{\mu}_{\bm{P}_n}\) the empirical spectral measure of \(\bm{P}_n\):
\begin{equation*}
\hat{\mu}_{\bm{P}_n}\coloneqq \frac{1}{p(d-r)} \sum_{i=1}^{p(d-r)} \delta_{\lambda_i(\bm{P}_n)}.
\end{equation*}
For \(z\in\C_+\), define the resolvent by
\begin{equation*}
\bm{G}_n (z) \coloneqq (\bm{P}_n - z\bm{I}_{p(d-r)})^{-1},
\end{equation*}
and its Stieltjes transform by
\begin{equation*}
m_n (z) \coloneqq \int_\R \frac{1}{\lambda - z} \mathrm{d} \hat{\mu}_{\bm{P}_n}(\lambda) = \frac{1}{p(d-r)}\Tr\bm{G}_n(z).
\end{equation*}

We view \(\bm{G}_n(z)\) as a \(p\times p\) block matrix with blocks \(\bm{G}_{\mu\nu}(z)\in\C^{(d-r)\times(d-r)}\). We define its normalized partial trace by
\[
\bm{M}_n(z) \coloneqq \tr_{d-r}^{(p)} \left ( \bm{G}_n(z)\right ) \in \C^{p \times p},
\] 
where \(\tr_{d-r}^{(p)}\) is defined in~\eqref{eq:partial_trace}. By construction,
\begin{equation*}
m_n (z)=\frac{1}{p}\Tr\bm{M}_n(z).
\end{equation*}
Thus, \(\bm{M}_n(z)\) is a matrix-valued Stieltjes transform that retains the block structure of the resolvent.

Our first main result identifies the limiting matrix-valued Stieltjes transform through a self-consistent equation and establishes the almost-sure weak convergence of the empirical spectral measure in the proportional high-dimensional limit.

\begin{prop} \label{prop:bulk}
Suppose that Assumptions~\ref{hyp:prop_limit} and \ref{hyp:preprocessing} hold. Then, almost surely, \(\bm{M}_n\) converges locally uniformly on \(\C_+\) to a deterministic analytic function \(\bm{M}_\alpha \colon \C_+ \to \C^{p \times p}\). The limit belongs to the \emph{normalized matrix-valued Herglotz class}, i.e., it satisfies 
\[
\Im \bm{M}_\alpha (z) \succeq 0, \enspace z \in \C_+, \quad \mathrm{and} \quad \lim_{\eta\to\infty} i\eta\,\bm{M}_\alpha (i\eta) = -\bm{I}_p,
\]
and is the unique solution in the normalized matrix-valued Herglotz class of the self-consistent equation
\[ 
\bm{M}_\alpha (z)^{-1} = - z \bm{I}_p + \E_{\bm{y}}\left [  \bm{T} (\bm{y}) \left ( \bm{I}_p + \frac{1}{\alpha}  \bm{M}_\alpha (z) \bm{T} (\bm{y})\right)^{-1}\right ].
\]
In particular, for every fixed \(z \in \C_+\), 
\begin{equation*}
\lVert \bm{M}_n (z) - \bm{M}_\alpha (z) \rVert_{\mathrm{op}} \to 0, \quad \textnormal{a.s.}
\end{equation*}
Consequently, \(\hat{\mu}_{\bm{P}_n}\) converges weakly almost surely to a deterministic compactly supported probability measure \(\mu_\alpha\) on \(\R\), whose Stieltjes transform is \(m_{\mu_\alpha}(z) = \frac{1}{p} \Tr \,\bm{M}_\alpha(z)\).
\end{prop}

The proof of Proposition~\ref{prop:bulk} is provided in Section~\ref{section:bulk}.

\begin{rmk}
When \(p=1\), the self-consistent equation reduces to 
\[
\frac{1}{m_\alpha (z)} = - z + \E \left[ \frac{T(\bm{y})}{1 + \alpha^{-1} T(\bm{y}) m_\alpha (z)}\right],
\]
which is the usual self-consistent equation for a weighted sample covariance matrix. If \(T(\bm{y})\geq0\) almost surely, this is the generalized Marchenko--Pastur equation satisfied by the weighted sample covariance matrix \(\frac{1}{n} \bm{Z}^{1/2} \bm{U}^\top\bm{U} \bm{Z}^{1/2}\) with \(\bm{Z} = \mathrm{diag}\left(T(\bm{y}_1),\ldots,T(\bm{y}_n)\right)\) (see e.g.,~\cite{silverstein1995}). If \(T\) is sign-indefinite, writing
\[
\bm{Z} = \bm{Z}^+-\bm{Z}^-, \quad \bm{Z}^\pm\succeq 0,
\]
gives
\[
\bm{P}_n = \frac{1}{n}\bm{U} \bm{Z}^+\bm{U}^\top - \frac{1}{n} \bm{U} \bm{Z}^-\bm{U}^\top.
\]
Thus \(\bm{P}_n\) can be viewed as a difference of two positive weighted sample
covariance matrices. In this setting, the same scalar self-consistent equation remains
valid~\cite{mondelli2019}.
\end{rmk}

Recall that \(\lambda_+ (\alpha)= \sup \supp(\mu_\alpha)\) denotes the upper edge of the limiting bulk spectrum. Since \(\lambda_+(\alpha)\in\supp(\mu_\alpha)\), Proposition~\ref{prop:bulk}
implies that, for every \(\varepsilon>0\), 
\[
\mathbb{P} \left (\lambda_1 (\bm{P}_n) > \lambda_+ (\alpha) -\varepsilon\right) \to  1.
\]
The next result complements this lower bound with confinement above the upper edge, and hence identifies the asymptotic location of the largest eigenvalue of \(\bm{P}_n\).

\begin{prop} \label{prop:upper_edge}
Suppose that Assumptions~\ref{hyp:prop_limit} and \ref{hyp:preprocessing} hold. Then
\[
\lambda_1(\bm{P}_n) \stackrel{\mathbb{P}}{\to} \lambda_+(\alpha).
\]
\end{prop}

Proposition~\ref{prop:upper_edge} is completed in Section~\ref{section:upper_edge}, where we establish the matching upper bound. Furthermore, in Appendix~\ref{appendix:limiting_variational_edge} we provide an exact variational characterization of the limiting spectral edge \(\lambda_+(\alpha)\) using, in part, results established in Section~\ref{section:upper_edge}.

We readily obtain the following corollary about the smallest eigenvalue of \(\bm{P}_n\).

\begin{cor}
Let \(\lambda_-(\alpha)\coloneqq\inf\supp(\mu_\alpha)\). Then, 
\[
\lambda_{p(d-r)} (\bm{P}_n) \stackrel{\mathbb{P}}{\to} \lambda_- (\alpha).
\]
\end{cor}

\begin{proof}
Apply Proposition~\ref{prop:upper_edge} with the preprocessing map \(\bm{T}\) replaced by \(-\bm{T}\), which also satisfies Assumption~\ref{hyp:preprocessing}. Indeed, \(\bm{P}_n[-\bm{T}]=-\bm{P}_n[\bm{T}]\), and the limiting law associated with \(-\bm{T}\) is the reflection of \(\mu_\alpha\), so its upper edge is \(-\lambda_-(\alpha)\).
\end{proof}

The first assertion of Theorem~\ref{thm:bulk} follows from Proposition~\ref{prop:bulk} and the fact that \(\bm{D}_n'\) differs from the block embedding of \(\bm{P}_n\) by a matrix of rank at most \(2pr\). The second assertion follows from Cauchy's interlacing theorem together with Proposition~\ref{prop:upper_edge}. The details are given in Sections~\ref{section:bulk} and~\ref{section:upper_edge}, respectively.

\subsection{Outlier equation}

We now reduce the study of the outlying eigenvalues of \(\bm{D}_n\) to a finite-dimensional equation. Fix \(\varepsilon > 0\). By Proposition~\ref{prop:upper_edge}, with probability tending to one,
\[
\mathrm{sp} (\bm{P}_n) \cap (\lambda_+(\alpha)+\varepsilon,\infty) = \emptyset.
\]
On this event, for every \(x > \lambda_+(\alpha)+\varepsilon\), the Schur complement identity~\eqref{eq:schur_det_lemma} gives
\[
x \in \mathrm{sp}(\bm{D}_n) \quad\Longleftrightarrow\quad  \det \left ( \bm{H}_n(x)  \right ) =0,
\]
where 
\begin{equation} \label{eq:H_n}
\bm{H}_n(x) \coloneqq \bm{A}_n - \bm{Q}_n^\top (\bm{P}_n - x \bm{I}_{p(d-r)})^{-1} \bm{Q}_n - x \bm{I}_{pr}.
\end{equation}
Thus, away from the spectrum of \(\bm{P}_n\), the eigenvalues of \(\bm{D}_n\) are characterized by the points at which the \(pr \times pr\) matrix \(\bm{H}_n(x)\) is singular. 

We next identify the deterministic limit of \(\bm{H}_n(x)\). By Lemma~\ref{lem:continuation_off_support}, proved in Section~\ref{section:MDE}, \(\bm{M}_\alpha\) admits an analytic continuation to \(\C \setminus \supp(\mu_\alpha)\). For \(x \in \R \setminus \supp (\mu_\alpha)\), we use \(\bm{M}_\alpha (x)\) to denote this continuation. For \(x>\lambda_+(\alpha)\), define
\begin{equation} \label{eq:H_alpha}
\bm{H}_\alpha (x) \coloneqq \E_{\bm{y}} \left [ \bm{T}(\bm{y}) \left ( \bm{I}_p + \frac{1}{\alpha} \bm{M}_\alpha(x) \bm{T}(\bm{y}) \right)^{-1} \otimes \bm{C}(\bm{y}) \right ] - x \bm{I}_{pr},
\end{equation}
where we recall that \(\bm{C}(\bm{y}) = \E [\bm{s} \bm{s}^\top \mid \bm{y}]\) and \(\bm{M}_\alpha\) is the limiting matrix-valued Stieltjes transform from Proposition~\ref{prop:bulk}.

\begin{prop} \label{prop:outlier_equation} 
Suppose that Assumptions~\ref{hyp:prop_limit} and~\ref{hyp:preprocessing} hold. Let \( K\subset (\lambda_+(\alpha),\infty)\) be a compact interval, and define
\[
\mathcal{E}_n(K) \coloneqq \left \{ K \cap \mathrm{sp}(\bm{P}_n) = \emptyset \right \}.
\]
Then, \(\mathbb{P} \left ( \mathcal{E}_n(K)\right) \to 1\). On the event \(\mathcal{E}_n(K)\), the matrix-valued function \(\bm{H}_n(x)\) is well-defined for every \(x \in K\). Moreover, for every \(\varepsilon>0\),
\[
\mathbb{P} \left ( \mathcal{E}_n(K) \cap \left \{ \sup_{x\in K} \left\| \bm{H}_n (x) - \bm{H}_\alpha (x) \right\|_{\mathrm{op}} > \varepsilon \right \} \right) \to 0.
\]
Thus, if \(\theta_n \in \mathrm{sp}(\bm{D}_n)\) is a sequence of random eigenvalues such that \(\theta_n \stackrel{\mathbb{P}}{\to} \theta > \lambda_+ (\alpha)\), then \(\det (\bm{H}_\alpha(\theta)) =0\). 
\end{prop}

In particular, Proposition~\ref{prop:outlier_equation} shows that every possible limit of an outlying eigenvalue belongs to the deterministic set
\[
\left\{ x >\lambda_+(\alpha) \colon \det (\bm{H}_\alpha(x) ) =0 \right\}.
\]
The proposition does not yet show that every element of this set produces a sample outlier. For the largest eigenvalue, this conclusion follows from the monotonicity argument developed in the next subsection. Proposition~\ref{prop:outlier_equation} is proved in Section~\ref{section:phase_transition}.

\subsection{Spectral phase transition}

Proposition~\ref{prop:outlier_equation} identifies the possible limiting locations of eigenvalues outside the bulk, but does not by itself determine whether such an outlier actually occurs. For the largest eigenvalue, this question can be reduced to the behavior of the largest eigenvalue of the deterministic Schur complement. Recall from~\eqref{eq:h_alpha} that
\[
h_\alpha(x)= \lambda_1 \left ( \bm{H}_\alpha (x) \right) = \Phi_\alpha(x)-x, \qquad x > \lambda_+(\alpha).
\]
The function \(h_\alpha\) is continuous and strictly decreasing on \((\lambda_+(\alpha),\infty)\), and it tends to \(-\infty\) as \(x \to \infty\). Thus, whether the leading branch of the deterministic outlier equation crosses zero is determined entirely by its behavior at the upper bulk edge:
\[
\Delta_{\bm{T}}(\alpha) = \lim_{x \downarrow \lambda_+(\alpha)} h_\alpha(x).
\]
This gives a simple interpretation of the phase transition. If \(\Delta_{\bm{T}}(\alpha)\leq0\), then the right limit of \(h_\alpha\) at the edge is nonpositive, and strict monotonicity implies \(h_\alpha(x) < 0\) for every \(x > \lambda_+(\alpha)\). It therefore has no zero above the bulk, and no leading eigenvalue separates:
\[
\lambda_1(\bm{D}_n)\stackrel{\mathbb{P}}{\to} \lambda_+(\alpha).
\]
If instead \(\Delta_{\bm{T}}(\alpha)>0\), then \(h_\alpha\) is positive near the edge and negative for sufficiently large \(x\). It consequently crosses zero at a unique point \(\theta_\alpha>\lambda_+(\alpha)\), which becomes the asymptotic location of the separated eigenvalue:
\[
\lambda_1(\bm{D}_n)\stackrel{\mathbb{P}}{\to} \theta_\alpha.
\]

Viewing this criterion as a function of the sampling ratio \(\alpha\) leads naturally to the study of the transition as \(\alpha\) varies. The quantities \(\alpha_{\mathrm{c},\min}(\bm{T})\) and \(\alpha_{\mathrm{c},\max}(\bm{T})\) describe, respectively, the smallest sampling ratios at which the supercritical regime can occur and the point beyond which it persists for every larger sampling ratio. Under the population-separation condition of Assumption~\ref{hyp:population_separation}, the method is subcritical for all sufficiently small \(\alpha\) and supercritical for all sufficiently large \(\alpha\). 

The rigorous argument requires transferring this deterministic zero-crossing picture to the finite-\(n\) spectrum. This involves the uniform approximation of the Schur complement away from the bulk, additional control near the upper edge, and stability of the zero crossing. These arguments are carried out in Section~\ref{section:phase_transition}.

\section{The matrix Dyson equation} \label{section:MDE}

The purpose of this section is to introduce and analyze the deterministic self-consistent equation governing the noise component \(\bm{P}_n\) defined in~\eqref{eq:P_n}. More precisely, for \(z \in \C_+\), we seek a matrix-valued function \(\bm{M} \colon \C_+ \to \C^{p \times p}\) satisfying 
\begin{equation} \label{eq:MDE}
\bm{M} (z)^{-1} = - z \bm{I}_p + \E_{\bm{y}}\left [  \bm{T} (\bm{y}) \left ( \bm{I}_p + \frac{1}{\alpha}  \bm{M}(z) \bm{T} (\bm{y}) \right)^{-1}\right ].
\end{equation}
Throughout this section, \(\bm{y}\) has the marginal distribution induced by 
\[
\bm{s} \sim \mathcal{N}(\bm{0}_r, \bm{I}_r), \quad \bm{y} \mid \bm{s} \sim P_\ast (\, \cdot \mid \bm{s}),
\] 
and \(\bm{T}(\bm{y})\) is assumed to be symmetric and to satisfy \(\|\bm{T}(\bm{y})\|_{\mathrm{op}} \le C_T\) almost surely. 

Equation~\eqref{eq:MDE} has the structure of a \emph{matrix Dyson equation (MDE)}. Matrix Dyson equations and related operator-valued self-consistent equations play a central role in the description of deterministic spectral limits of random matrices; see, for example,~\cite{helton2007,ajanki2019,alt2020,erdos2019}. In the standard matrix Dyson equation, the self-energy operator is linear and positivity preserving. In contrast, the effective self-energy map associated with~\eqref{eq:MDE},
\[
\bm{M} \mapsto  \E_{\bm{y}}\left [  \bm{T} (\bm{y}) \left ( \bm{I}_p + \frac{1}{\alpha}  \bm{M} \bm{T} (\bm{y}) \right)^{-1}\right ]
\]
is nonlinear. Consequently, the standard existence, uniqueness, and stability results for matrix Dyson equations with linear self-energy do not apply directly. Nevertheless, because~\eqref{eq:MDE} has an analogous resolvent and matrix-valued Stieltjes-transform structure, we refer to it throughout as the matrix Dyson equation. The arguments below instead exploit its specific nonlinear form.

This section is organized as follows. We first establish existence and uniqueness of a solution to~\eqref{eq:MDE} within an appropriate class of matrix-valued analytic functions. We then prove that the matrix-valued measure associated with this solution is compactly supported and extend the solution analytically to the complement of its support. Finally, we establish stability of the matrix Dyson equation away from the limiting spectrum.

\subsection{Existence and uniqueness}

We begin by specifying the class in which the matrix Dyson equation admits a unique solution.
\begin{defn}
We say that an analytic function \(\bm{M} \colon \C_+ \to \C^{p\times p}\) belongs to the \emph{normalized matrix-valued Herglotz class} if
\[
\Im \bm{M} (z) \succeq \bm{0}_{p \times p}, \qquad z \in \C_+,
\]
and
\[
\lim_{\eta \to \infty} i \eta \bm{M} (i\eta) = - \bm{I}_p.
\]
\end{defn}

\begin{prop} \label{prop:existence_uniqueness}
Fix \(\alpha>0\) and let \(\bm{T} \colon\R^q\to\sym_p(\R)\) satisfy Assumption~\ref{hyp:preprocessing}. Then there exists a unique analytic function 
\[
\bm{M}_\alpha \colon \C_+ \to \C^{p\times p}
\]
that satisfies the matrix Dyson equation~\eqref{eq:MDE} and belongs to the normalized matrix-valued Herglotz class. In fact, \(\Im\bm{M}_\alpha(z)\succ \bm{0}_{p \times p}\) for \(z \in \C_+\). Moreover, there exists a unique positive semidefinite matrix-valued Borel measure \(\bm{\Omega}_\alpha\) on \(\R\), with total mass \(\bm{\Omega}_\alpha (\R) = \bm{I}_p\), such that, for every \(z \in \C_+\),
\[
\bm{M}_\alpha (z) =  \int_\R \frac{\bm{\Omega}_\alpha(\mathrm{d} \lambda)}{\lambda - z}.
\]
Consequently,  
\[
\mu_\alpha\coloneqq \frac{1}{p} \Tr \bm{\Omega}_\alpha
\]
is a probability measure on \(\R\) and its Stieltjes transform satisfies
\[
m_{\mu_\alpha}(z) =  \int_{\R}\frac{\mu_\alpha(\mathrm{d}\lambda)}{\lambda-z} = \frac{1}{p} \Tr \bm{M}_\alpha(z), \qquad z\in\C_+.
\]
\end{prop}

The proof of Proposition~\ref{prop:existence_uniqueness} is based on a fixed-point formulation of the matrix Dyson equation. Denote by 
\[
\mathbb{H}_p \coloneqq \{ \bm{X} \in \C^{p \times p} \colon \Im \bm{X} \succ \bm{0}_{p \times p}\}
\]
the matrix upper half-plane. For \(\bm{X} \in \mathbb{H}_p\) and \(\bm{y} \in \R^q\), define
\begin{equation} \label{eq:mathcal_Q_alpha}
\bm{\mathcal{Q}}_\alpha (\bm{X},\bm{y}) \coloneqq \bm{T} (\bm{y}) \left ( \bm{I}_p + \frac{1}{\alpha}  \bm{X} \bm{T} (\bm{y}) \right)^{-1},
\end{equation}
and
\begin{equation} \label{eq:mathcal_S_alpha}
\bm{\mathcal{S}}_\alpha (\bm{X}) \coloneqq \E_{\bm{y}}\left [  \bm{\mathcal{Q}}_\alpha (\bm{X},\bm{y})\right ].
\end{equation}
For \(z\in\C_+\), define
\begin{equation} \label{eq:bm_Psi_alpha}
\bm{\Psi}_{\alpha,z} (\bm{X}) \coloneqq \left ( - z \bm{I}_p + \bm{\mathcal{S}}_\alpha [\bm{X}]  \right)^{-1}.
\end{equation}
Then, for each fixed \(z\in\C_+\), a matrix \(\bm{M} \in \mathbb{H}_p\) solves~\eqref{eq:MDE} if and only if 
\begin{equation} \label{eq:fixed_point_eq}
\bm{M}=\bm{\Psi}_{\alpha,z}(\bm{M}). 
\end{equation}

The next lemma establishes the basic analytic and positivity properties of \(\bm{\mathcal{S}}_\alpha\) needed for the fixed-point argument.
\begin{lem} \label{lem:S_holo}
Let \(\bm{X} \in \mathbb{H}_p\). Then, for almost every \(\bm{y}\), \(\bm{I}_p + \alpha^{-1} \bm{X} \bm{T}(\bm{y})\) is invertible. Moreover, 
\begin{equation} \label{eq:nonlinear_term_bound}
\| \bm{\mathcal{Q}}_\alpha (\bm{X},\bm{y})  \|_{\mathrm{op}} \le \frac{\alpha}{\lambda_p (\Im \bm{X})},
\end{equation}
and
\begin{equation} \label{eq:nonlinear_term_imaginary_part}
\Im \bm{\mathcal{Q}}_\alpha (\bm{X},\bm{y}) = - \frac{1}{\alpha} \bm{\mathcal{Q}}_\alpha (\bm{X},\bm{y})  (\Im \bm{X}) \bm{\mathcal{Q}}_\alpha (\bm{X},\bm{y}) ^\ast \preceq \bm{0}.
\end{equation}
Therefore, \(\bm{\mathcal{S}}_\alpha\) is well-defined and holomorphic on \(\mathbb{H}_p\) and satisfies \(\Im \bm{\mathcal{S}}_\alpha (\bm{X}) \preceq \bm{0}\).
\end{lem}

\begin{proof}
Fix \(\bm{X} \in \mathbb{H}_p\). We first show that the matrix \(\bm{I}_p + \alpha^{-1} \bm{X} \bm{T} (\bm{y})\) is invertible almost surely. Suppose that
\begin{equation} \label{eq:invertibility}
(\bm{I}_p + \alpha^{-1} \bm{X} \bm{T} (\bm{y})) \bm{v}=0
\end{equation}
for some \(\bm{v} \in \C^p\), and set \(\bm{u}\coloneqq \bm{T} (\bm{y})\bm{v}\). Multiplication on the left by \(\bm{u}^\ast\) gives
\begin{equation} \label{eq:invertibility_bis}
\bm{u}^\ast \bm{v} + \alpha^{-1} \bm{u}^\ast \bm{X} \bm{u} = 0.
\end{equation}
Since \(\bm{T} (\bm{y})\) is Hermitian, \(\bm{u}^\ast \bm{v} = \bm{v}^\ast \bm{T} (\bm{y}) \bm{v} \in \R\). Taking imaginary parts on both sides in~\eqref{eq:invertibility_bis} therefore yields
\[
\bm{u}^\ast (\Im \bm{X}) \bm{u} = 0. 
\]
Since \(\Im \bm{X} \succ0\), we obtain that \(\bm{u}=0\), and~\eqref{eq:invertibility} then gives \(\bm{v}=0\). Thus
\[ 
\ker \left(\bm{I}_p + \alpha^{-1}\bm{X} \bm{T} (\bm{y})\right) = \{\bm{0}_p\}, 
\] 
and the matrix is invertible almost surely.

To prove~\eqref{eq:nonlinear_term_bound}, fix \(\bm{x}\in\C^p\) and set 
\[
\bm{v} \coloneqq \left( \bm{I}_p+\alpha^{-1}\bm{X}\bm{T}(\bm{y}) \right)^{-1}\bm{x} \quad \mathrm{and} \quad \bm{u} \coloneqq \bm{T}(\bm{y})\bm{v}. 
\]
Then \( \bm{x}=\bm{v}+\alpha^{-1}\bm{X}\bm{u} \) and \(\bm{u} = \bm{\mathcal{Q}}_\alpha (\bm{X},\bm{y}) \bm{x}\). Since \(\bm{u}^\ast \bm{v}\in\R\), 
\[ 
\Im \left ( \bm{u}^\ast \bm{x} \right) = \frac{1}{\alpha} \bm{u}^\ast (\Im \bm{X} )\bm{u}  \geq \frac{\lambda_p (\Im\bm{X})}{\alpha}\|\bm{u}\|^2,
\] 
where \(\lambda_p (\Im \bm{X})>0\) since \(\Im \bm{X}\) is positive definite. On the other hand, 
\[
\left|\Im(\bm{u}^\ast\bm{x})\right| \leq \|\bm{u}\| \|\bm{x}\|. 
\]
Therefore, 
\[
\|\bm{u}\| \leq \frac{\alpha}{\lambda_p (\Im \bm{X})}\|\bm{x}\|,
\]
which proves~\eqref{eq:nonlinear_term_bound}. Moreover, applying the resolvent identity we obtain
\[
\begin{split}
\bm{\mathcal{Q}}_\alpha (\bm{X},\bm{y})  - \bm{\mathcal{Q}}_\alpha (\bm{X},\bm{y})^\ast & =\bm{T} (\bm{y})  \left [ \left ( \bm{I}_p + \frac{1}{\alpha}  \bm{X} \bm{T} (\bm{y}) \right)^{-1} - \left ( \bm{I}_p + \frac{1}{\alpha}  \bm{X}^\ast \bm{T} (\bm{y}) \right)^{-1}\right ] \\
& =  - \frac{1}{\alpha}  \bm{\mathcal{Q}}_\alpha (\bm{X},\bm{y})  \left (\bm{X} - \bm{X}^\ast \right)  \bm{\mathcal{Q}}_\alpha (\bm{X},\bm{y})^\ast.
\end{split}
\]
Dividing by \(2i\) and using \(\Im \bm{X} \succ 0\) gives
\[
\Im \bm{\mathcal{Q}}_\alpha (\bm{X},\bm{y}) = - \frac{1}{\alpha}  \bm{\mathcal{Q}}_\alpha (\bm{X},\bm{y}) \left(\Im \bm{X} \right)  \bm{\mathcal{Q}}_\alpha (\bm{X},\bm{y})^\ast \preceq \bm{0},
\]
which proves~\eqref{eq:nonlinear_term_imaginary_part}.

For every compact set \(K \Subset \mathbb{H}_p\), there exists \(\delta_K>0\) such that \(\Im \bm{X} \succeq \delta_K \bm{I}_p\) for every \(\bm{X} \in K\). Hence~\eqref{eq:nonlinear_term_bound} is locally uniform in \(\bm{X}\), and the standard theorem on holomorphic parameter integrals implies that \(\bm{\mathcal{S}}_\alpha\) is holomorphic on \(\mathbb{H}_p\). Finally,
\[
\Im \bm{\mathcal{S}}_\alpha (\bm{X}) = \E_{\bm{y}} \left [ \Im \bm{\mathcal{Q}}_\alpha (\bm{X},\bm{y}) \right ] \preceq \bm{0}.
\]
\end{proof}

The next result shows that \(\bm{\Psi}_{\alpha,z}\) is a holomorphic self-map of \(\mathbb{H}_p\) and that its image is bounded in operator norm.

\begin{cor} \label{cor:holo}
For every \(z\in\C_+\), the map \(\bm{\Psi}_{\alpha,z}\colon\mathbb{H}_p\to\mathbb{H}_p\) is well-defined and holomorphic, and satisfies 
\[
\| \bm{\Psi}_{\alpha,z} (\bm{M}) \|_{\mathrm{op}} \le \frac{1}{\Im z}.
\]
\end{cor}

\begin{proof}
Fix \(z \in \C_+\) and \(\bm{M} \in \mathbb{H}_p\), and set \(\bm{A} \coloneqq -z\bm{I}_p + \bm{\mathcal{S}}_\alpha(\bm{M})\). By Lemma~\ref{lem:S_holo}, and in particular by~\eqref{eq:nonlinear_term_imaginary_part},
\[
\Im \bm{A} = - (\Im z) \bm{I}_p + \Im \bm{\mathcal{S}}_\alpha(\bm{M}) \preceq - (\Im z) \bm{I}_p \prec \bm{0}.
\]
In particular \(\bm{A}\) is invertible. Since \(\bm{\mathcal{S}}_\alpha\) is holomorphic on \(\mathbb{H}_p\) and matrix inversion is holomorphic on the set of invertible matrices \(\operatorname{GL}_p(\C)\), the map 
\[
\bm{\Psi}_{\alpha,z}(\bm{M}) = \bm{A}^{-1} 
\]
is well-defined and holomorphic on \(\mathbb{H}_p\). Moreover,
\[
\Im \bm{A}^{-1}  = - \bm{A}^{-\ast}  ( \Im \bm{A}) \bm{A}^{-1} \succ \bm{0}.
\]
Thus \(\bm{\Psi}_{\alpha,z} (\bm{M}) \in \mathbb{H}_p\), so \(\bm{\Psi}_{\alpha,z}\) is a holomorphic self-map of \(\mathbb{H}_p\). Finally, for every \(\bm{v} \in\C^p\),
\[
\| \bm{A} \bm{v}\|_2 \|\bm{v}\|_2 \ge \left | \Im \langle \bm{v}, \bm{A} \bm{v} \rangle \right | \ge (\Im z) \|\bm{v}\|^2_2.
\]
Therefore, 
\[
\|\bm{A} \bm{v}\|_2 \ge ( \Im z) \|\bm{v}\|_2,
\]
and consequently
\[
\| \bm{\Psi}_{\alpha,z}(\bm{M} ) \|_{\mathrm{op}} = \left \| \bm{A}^{-1} \right\|_{\mathrm{op}} \le \frac{1}{\Im z}.
\]
\end{proof}

Corollary~\ref{cor:holo} shows that the image of \( \bm{\Psi}_{\alpha,z} \) is bounded in operator norm. This alone does not imply relative compactness in \(\mathbb{H}_p\), since its imaginary part may degenerate near the boundary. We therefore consider the second iterate
\[
\bm{\Psi}_{\alpha,z}^{\circ 2} \coloneqq \bm{\Psi}_{\alpha,z} \circ \bm{\Psi}_{\alpha,z}.
\]
The next lemma shows that its image is not only bounded, but also uniformly separated from the boundary of \(\mathbb{H}_p\).

\begin{lem} \label{lem:Psi_second_iterate}
Fix \(z \in \C_+\). Then, for every \(\bm{X}\in\mathbb{H}_p\),
\begin{equation} \label{eq:first_iterate_Q_bound}
\left\| \bm{\mathcal{Q}}_\alpha \left (\bm{\Psi}_{\alpha,z}(\bm{X}),\bm{y}\right) \right \|_{\mathrm{op}}
\leq K_z \coloneqq 2C_T+\frac{C_T^2}{\alpha (\Im z)}
\end{equation}
almost surely. Consequently,
\begin{equation} \label{eq:second_iterate_norm_bound}
\left \| \bm{\Psi}_{\alpha,z}^{\circ2}(\bm{X}) \right \|_{\mathrm{op}}
\leq \frac{1}{\Im z},
\end{equation}
and
\begin{equation} \label{eq:second_iterate_imaginary_bound}
\Im \bm{\Psi}_{\alpha,z}^{\circ 2}(\bm{X}) \succeq \frac{\Im z}{(|z|+K_z)^2} \bm{I}_p.
\end{equation}
\end{lem}

\begin{proof}
We begin by proving~\eqref{eq:first_iterate_Q_bound}. This follows similar ideas as~\eqref{eq:nonlinear_term_bound} in Lemma~\ref{lem:S_holo}. For \(z \in \C_+\) and \(\bm{X} \in \mathbb{H}_p\), let \(\bm{Y} \coloneqq \bm{\Psi}_{\alpha,z} (\bm{X})\). From~\eqref{eq:bm_Psi_alpha} we have 
\[
\bm{Y}^{-1} = - z \bm{I}_p + \bm{\mathcal{S}}_\alpha (\bm{X}),
\]
and hence 
\[
- \Im \bm{Y}^{-1} = (\Im z) \bm{I}_p - \Im \bm{\mathcal{S}}_\alpha (\bm{X}) \succeq (\Im z) \bm{I}_p,
\]
which follows from Lemma~\ref{lem:S_holo}. Using \(\Im \bm{Y}= - \bm{Y}^\ast (\Im \bm{Y}^{-1}) \bm{Y}\), we obtain
\[
\Im \bm{Y} \succeq (\Im z) \bm{Y}^\ast\bm{Y}.
\]
Fix \(\bm{x}\in\C^p\), and set 
\[
\bm{v} = \left( \bm{I}_p+\alpha^{-1} \bm{Y} \bm{T}(\bm{y}) \right)^{-1}\bm{x} \quad \mathrm{and} \quad \bm{u} = \bm{T}(\bm{y})\bm{v} = \bm{\mathcal{Q}}_\alpha \left (\bm{Y},\bm{y}\right) \bm{x}. 
\]
Then \( \bm{x}=\bm{v}+\alpha^{-1}\bm{Y}\bm{u} \) and \(\bm{u}^\ast \bm{v} = \bm{v}^\ast \bm{T} (\bm{y}) \bm{v} \in \R\). Consequently,
\[ 
\Im \left ( \bm{u}^\ast \bm{x} \right) = \frac{1}{\alpha} \bm{u}^\ast (\Im \bm{Y} )\bm{u}  \geq \frac{\Im z}{\alpha}\|\bm{Yu}\|_2^2.
\] 
Therefore, 
\[
(\Im z) \| \bm{Y} \bm{u}\|_2^2 \le \alpha \|\bm{u}\|_2 \|\bm{x}\|_2.
\]
On the other hand, 
\[
\|\bm{u}\|_2 \le C_T \|\bm{v}\|_2 \le C_T \left (\|\bm{x}\|_2 + \alpha^{-1} \|\bm{Y}\bm{u}\|_2 \right) \le C_T \left ( \|\bm{x}\|_2 +  \sqrt{\frac{\|\bm{u}\|_2 \|\bm{x}\|_2}{\alpha (\Im z)}}\right ).
\]
Dividing by \(\|\bm{x}\|_2\), we get
\[
\frac{\|\bm{u}\|_2}{\|\bm{x}\|_2} \le C_T \left (1 + \sqrt{\frac{\|\bm{u}\|_2}{\alpha (\Im z)\|\bm{x}\|_2}}\right ).
\]
using Young's inequality \(ab \le (a^2 + b^2)/2\) with \(a = \sqrt{\|\bm{u}\|_2 / \|\bm{x}\|_2}\) and \(b = C_T / \sqrt{\alpha (\Im z)}\), we get 
\[
\frac{\|\bm{u}\|_2}{\|\bm{x}\|_2} \le C_T + \frac{1}{2}\frac{\|\bm{u}\|_2}{\|\bm{x}\|_2} + \frac{C_T^2}{2 \alpha (\Im z)} \le 2C_T + \frac{C_T^2}{\alpha (\Im z)} \eqqcolon K_z, 
\]
which proves~\eqref{eq:first_iterate_Q_bound}. 

It follows that
\[
\| - z \bm{I}_p + \bm{\mathcal{S}}_\alpha (\bm{Y}) \|_{\mathrm{op}} \le |z| + K_z ,
\]
and
\[
- \Im \left ( - z \bm{I}_p + \bm{\mathcal{S}}_\alpha (\bm{Y})\right ) \succeq (\Im z) \bm{I}_p.
\]
Since 
\[
\bm{\Psi}_{\alpha,z}^{\circ 2}(\bm{X}) = \bm{\Psi}_{\alpha,z} (\bm{Y}) = \left ( - z \bm{I}_p + \bm{\mathcal{S}}_\alpha (\bm{Y})\right )^{-1},
\]
it follows that
\[
\| \bm{\Psi}_{\alpha,z}^{\circ 2}(\bm{X}) \|_{\mathrm{op}} \le (\Im z)^{-1},
\]
and
\[
\Im \bm{\Psi}_{\alpha,z}^{\circ 2}(\bm{X}) \succeq \frac{\Im z}{(|z|+K_z)^2} \bm{I}_p,
\]
as desired.
\end{proof}

\begin{lem} \label{lem:pointwise_existence_uniqueness}
For every \(z \in \C_+\), the map \(\bm{\Psi}_{\alpha,z} \colon \mathbb{H}_p \to \mathbb{H}_p\) has a unique fixed point, denoted by \(\bm{M}_\alpha(z)\). It is complex symmetric \(\bm{M}_\alpha(z) = \bm{M}_\alpha(z)^\top\) and satisfies 
\[
\|\bm{M}_\alpha (z) \|_{\mathrm{op}} \le \frac{1}{\Im z}.
\]
Moreover, the function \(\bm{M}_\alpha \colon \C_+ \to \mathbb{H}_p\) is analytic and satisfies 
\[
\lim_{\eta \to \infty} i \eta \bm{M}_{\alpha} (i\eta) = - \bm{I}_p.
\]
\end{lem}

\begin{proof}
Fix \(z \in \C_+\) and choose \(b > (\Im z)^{-1}\). Define
\[
D_b \coloneqq \{\bm{X}\in\mathbb{H}_p \colon \|\bm{X}\|_{\mathrm{op}}<b\}.
\]
Then \(D_b\) is an open, bounded, and convex subset of \(\C^{p\times p}\), and hence a bounded domain. By Lemma~\ref{lem:Psi_second_iterate},
\[
\bm{\Psi}_{\alpha,z}^{\circ2}(D_b) \subseteq \left \{ \bm{X} \in \mathbb{H}_p \colon
\|\bm{X}\|_{\mathrm{op}}\leq (\Im z)^{-1}, \enspace \Im \bm{X} \succeq c_z\bm{I}_p \right \} .
\]
The set on the right is compact and lies a positive distance from both the norm boundary \(\{\|\bm{X}\|_{\mathrm{op}}  =b\}\) and the boundary of the matrix upper half-plane. Since \(b > (\Im z)^{-1}\), it is relatively compact in \(D_b\). The map \(\bm{\Psi}_{\alpha,z}^{\circ 2}\) is holomorphic, so the Earle--Hamilton fixed point theorem (see, e.g.,~\cite[Theorem 3.1]{helton2007}) implies that it has a unique fixed point \(\bm{M} \in D_b\). 

We claim that \(\bm{M}\) is in fact a fixed point of \(\bm{\Psi}_{\alpha,z}\). By Corollary~\ref{cor:holo},
\[
\bm{\Psi}_{\alpha,z}(\bm{M})\in\mathbb{H}_p,
\qquad
\|\bm{\Psi}_{\alpha,z}(\bm{M})\|_{\mathrm{op}} \leq\frac{1}{\Im z}<b,
\]
and hence \(\bm{\Psi}_{\alpha,z}(\bm{M})\in D_b\). Moreover,
\[
\bm{\Psi}_{\alpha,z}^{\circ2} \left ( \bm{\Psi}_{\alpha,z}(\bm{M}) \right ) = \bm{\Psi}_{\alpha,z}
\left (\bm{\Psi}_{\alpha,z}^{\circ2} (\bm{M} )\right) = \bm{\Psi}_{\alpha,z}(\bm{M}) .
\]
Thus \(\bm{\Psi}_{\alpha,z}(\bm{M})\) is also a fixed point of \(\bm{\Psi}_{\alpha,z}^{\circ2}\) in \(D_b\). By uniqueness,
\[
\bm{\Psi}_{\alpha,z}(\bm{M})=\bm{M}.
\]
Therefore \(\bm{M}\) is a fixed point of \(\bm{\Psi}_{\alpha,z}\). Conversely, every fixed point of \(\bm{\Psi}_{\alpha,z}\) is a fixed point of \(\bm{\Psi}_{\alpha,z}^{\circ2}\). Furthermore, every fixed point \(\bm{N}\in\mathbb{H}_p\) of \(\bm{\Psi}_{\alpha,z}\) satisfies
\[
\|\bm{N}\|_{\mathrm{op}} = \| \bm{\Psi}_{\alpha,z} (\bm{N}) \|_{\mathrm{op}} \le \frac{1}{\Im z} < b,
\]
and hence belongs to \(D_b\). The uniqueness of the fixed point of \(\bm{\Psi}_{\alpha,z}^{\circ2}\) in \(D_b\) therefore shows that \(\bm{M}\) is the unique fixed point of \(\bm{\Psi}_{\alpha,z}\) in all of
\(\mathbb{H}_p\). 

The bound \(\|\bm{M}_\alpha (z)\|_{\mathrm{op}} \le (\Im z)^{-1}\) follows from Corollary~\ref{cor:holo} and the fixed-point identity. We also record the symmetry of the fixed point. Since \(\bm{T}(\bm{y})\) is symmetric, Lemma~\ref{lem:push_through} gives
\[
\bm{\mathcal{S}}_\alpha(\bm{X})^\top = \bm{\mathcal{S}}_\alpha(\bm{X}^\top),
\]
and hence
\[
\bm{\Psi}_{\alpha,z}(\bm{X})^\top = \bm{\Psi}_{\alpha,z}(\bm{X}^\top).
\]
Moreover, \(\mathbb{H}_p\) is invariant under transposition, since \(\Im (\bm{X}^\top) = (\Im \bm{X})^\top \succ \bm{0}\). Therefore, if \(\bm{M}_\alpha(z)\) is a fixed point of \(\bm{\Psi}_{\alpha,z}\), then so is \(\bm{M}_\alpha(z)^\top\). By uniqueness of the fixed point in \(\mathbb{H}_p\),
\[
\bm{M}_\alpha(z)^\top = \bm{M}_\alpha(z).
\]

We next prove analyticity in \(z\). Fix \(\bm{M}_0 \in \mathbb{H}_p\), and define the sequence \((\bm{M}^{(k)} (z))_{k \in \N}\) iteratively by
\[
\bm{M}^{(0)} (z) \coloneqq \bm{M}_0, \quad \bm{M}^{(k+1)} (z) \coloneqq \bm{\Psi}_{\alpha,z}^{\circ 2} \left (  \bm{M}^{(k)} (z) \right).
\]
For every fixed \(z \in \C_+\), choose \(b > (\Im z)^{-1}\). Then \(\| \bm{M}^{(1)} (z)\|_{\mathrm{op}} \le (\Im z)^{-1} < b\), so \(\bm{M}^{(1)} (z) \in D_b\). The Earle--Hamilton convergence theorem therefore applies and gives 
\[
\bm{M}^{(k)} (z) \to \bm{M}_\alpha (z).
\]
Moreover, for \(k \ge 1\), the iterates are locally uniformly bounded on \(\C_+\), because \(\|\bm{M}^{(k)} (z)\|_{\mathrm{op}} \le (\Im z)^{-1} \). Since they converge pointwise for every \(z \in \C_+\), Vitali's theorem applied entrywise implies locally uniform convergence and hence analyticity of \(\bm{M}_\alpha\).

Finally, since \(\bm{M}_\alpha(i\eta)\) is a fixed point, it is itself in the image of \(\bm{\Psi}_{\alpha,i\eta}\). Hence~\eqref{eq:first_iterate_Q_bound} gives
\[
\left \| \bm{\mathcal{S}}_\alpha \left ( \bm{M}_\alpha (i\eta) \right ) \right \|_{\mathrm{op}} \le 2C_T + \frac{C_T^2}{\alpha \eta}.
\]
Using the fixed-point equation,
\[
i\eta \bm{M}_\alpha(i\eta) = -\left(\bm{I}_p -  \frac{\bm{\mathcal{S}}_\alpha \left ( \bm{M}_\alpha (i\eta) \right )}{i\eta} \right)^{-1}.
\]
The preceding bound implies
\[
\frac{\bm{\mathcal{S}}_\alpha \left ( \bm{M}_\alpha (i\eta) \right )}{i\eta} \to 0,
\]
and therefore \(i \eta \bm{M}_\alpha (i\eta) \to - \bm{I}_p\).
\end{proof}

We now prove Proposition~\ref{prop:existence_uniqueness}. 

\begin{proof}[Proof of Proposition~\ref{prop:existence_uniqueness}]
By Lemma~\ref{lem:pointwise_existence_uniqueness}, there exists an analytic solution \(\bm{M}_\alpha \colon \C_+ \to \mathbb{H}_p\) of~\eqref{eq:MDE} satisfying 
\[
i\eta\bm{M}_\alpha(i\eta)\to-\bm{I}_p.
\]
Hence \(\bm{M}_\alpha\) belongs to the normalized matrix-valued Herglotz class, with \(\Im \bm{M}_\alpha (z) \succ \bm{0}\).

To prove uniqueness in this class, let \(\widetilde{\bm{M}}_\alpha\) be another normalized Herglotz solution. By the matrix-valued Nevanlinna--Herglotz representation theorem~\cite[Theorem 5.4]{gesztesy} and the normalization at infinity, \(\widetilde{\bm{M}}_\alpha\) is the Stieltjes transform of a positive semidefinite matrix-valued measure of total mass \(\bm{I}_p\). Thus \(\Im \widetilde{\bm{M}}_\alpha (z) \succ \bm{0}\) for every \(z\in\C_+\), so \(\widetilde{\bm{M}}_\alpha (z)\in\mathbb{H}_p\). Pointwise uniqueness from Lemma~\ref{lem:pointwise_existence_uniqueness} therefore yields
\(\widetilde{\bm{M}}_\alpha (z) = \bm{M}_\alpha (z)\) for every \(z \in \C_+\).

Applying the same representation theorem to \(\bm{M}_\alpha\) gives a unique positive semidefinite matrix-valued measure \(\bm{\Omega}_\alpha\) with \(\bm{\Omega}_\alpha(\R)=\bm{I}_p\) such that
\[
\bm{M}_\alpha(z) = \int_\R \frac{\bm{\Omega}_\alpha (\mathrm{d} \lambda)}{\lambda-z}.
\]
Taking normalized traces gives the claimed probability measure \(\mu_\alpha=p^{-1} \Tr \bm{\Omega}_\alpha\) and its Stieltjes transform.
\end{proof}

From the proof of Proposition~\ref{prop:existence_uniqueness}, we also obtain the following pointwise recognition result.

\begin{cor} \label{cor:MDE_recognition}
Let \(\bm{M}\colon\C_+\to\C^{p \times p}\) be a normalized matrix-valued Herglotz function, and let \(U \subseteq \C_+\). Suppose that \(\bm{M}\) satisfies the matrix Dyson equation~\eqref{eq:MDE} for every \(z\in U\). Then 
\[
\bm{M} (z) = \bm{M}_\alpha(z), \quad z \in U.
\]
In particular, if \(U\) has an accumulation point in \(\C_+\), then \(\bm{M} = \bm{M}_\alpha\) on all of \(\C_+\), and hence \(\bm{M}\) satisfies~\eqref{eq:MDE} on all of \(\C_+\).
\end{cor}

\begin{proof}
By the matrix-valued Nevanlinna--Herglotz representation theorem~\cite[Theorem 5.4]{gesztesy}, \(\Im \bm{M}(z)\succ \bm{0}\) for every \(z\in\C_+\). Thus \(\bm{M}(z) \in \mathbb{H}_p\) for every \(z\in\C_+\). Fix \(z \in U\). Since \(\bm{M}\) satisfies the matrix Dyson equation at \(z\), \(\bm{M}(z) = \bm{\Psi}_{\alpha,z}(\bm{M}(z)\). Hence \(\bm{M}(z)\) is a fixed point of \(\bm{\Psi}_{\alpha,z}\) in \(\mathbb{H}_p\). By Lemma~\ref{lem:pointwise_existence_uniqueness}, this fixed point is unique and is equal to \(\bm{M}_\alpha(z)\). Therefore, \(\bm{M} (z) = \bm{M}_\alpha(z)\) for every \(z \in U\). If \(U\) has an accumulation point in \(\C_+\), then the analytic matrix-valued functions \(\bm{M}\) and \(\bm{M}_\alpha\) agree on a set with an accumulation point in their domain. The identity theorem, applied entrywise, therefore yields \(\bm{M} = \bm{M}_\alpha\) on all of \(\C_+\). Since \(\bm{M}_\alpha\) satisfies~\eqref{eq:MDE} on \(\C_+\), the same is true of \(\bm{M}\).
\end{proof}

\subsection{Analytic continuation and compact support}

We next study the analytic structure of the solution outside the upper half-plane. We first construct an analytic continuation for sufficiently large spectral parameters. This implies compactness of the representing measure and subsequently yields an analytic continuation of \(\bm{M}_\alpha\) to the complement of its support.

\begin{lem} \label{lem:analytic_continuation}
There exist \(R>0\) and an analytic function 
\[
\widehat{\bm{M}}_\alpha \colon \{z\in\C \colon |z|>R\} \to \C^{p\times p}
\]
such that \( \widehat{\bm{M}}_\alpha(z)=\bm{M}_\alpha(z)\) for every \(z\in\C_+\cap\{z\in\C \colon |z|>R\}\). Moreover, for every \(z\in\C\) with \(|z|>R\), 
\[ 
\widehat{\bm{M}}_\alpha(z) = \bm{\Psi}_{\alpha,z} \left (\widehat{\bm{M}}_\alpha(z)\right)
\] 
and 
\[ 
\widehat{\bm{M}}_\alpha(\bar{z}) = \widehat{\bm{M}}_\alpha(z)^\ast . 
\] 
In particular, for every \(x\in\R\) with \(|x|>R\), \(\widehat{\bm{M}}_\alpha(x) = \widehat{\bm{M}}_\alpha(x)^\ast\).
\end{lem}

\begin{proof} 
For \(z\in\C\) with \(|z|\) sufficiently large, consider the closed ball 
\[ 
\mathcal{B}_z \coloneqq \left\{ \bm{X}\in\C^{p\times p} \colon \|\bm{X}\|_{\mathrm{op}} \leq \frac{2}{|z|} \right\}. 
\] 
For every \(\bm{X}\in\mathcal{B}_z\), if \(|z|>4C_T/\alpha\), then 
\[ 
\left\| \frac{1}{\alpha}\bm{X}\bm{T}(\bm{y}) \right\|_{\mathrm{op}} \le \frac{2 C_T}{\alpha |z|} < \frac{1}{2}. 
\] 
It follows from the Neumann series that \(\bm{I}_p + \alpha^{-1} \bm{X}\bm{T}(\bm{y})\) is invertible almost surely and 
\[ 
\left\| \left( \bm{I}_p + \frac{1}{\alpha}\bm{X}\bm{T}(\bm{y}) \right)^{-1} \right\|_{\mathrm{op}} \leq 2. 
\] 
Therefore \(\|\bm{\mathcal{S}}_\alpha(\bm{X})\|_{\mathrm{op}} \leq 2C_T \). If in addition \(|z|>4C_T\), then \(-z\bm{I}_p+\bm{\mathcal{S}}_\alpha(\bm{X}) \) is invertible and 
\[ 
\left\| \left( -z\bm{I}_p+\bm{\mathcal{S}}_\alpha(\bm{X}) \right)^{-1} \right\|_{\mathrm{op}}\leq \frac{1}{ |z|-\|\bm{\mathcal{S}}_\alpha(\bm{X})\|_{\mathrm{op}} }\leq \frac{1}{|z|-2C_T} \leq \frac{2}{|z|}. 
\] 
Therefore, \(\bm{\Psi}_{\alpha,z}(\mathcal{B}_z) \subseteq \mathcal{B}_z\) for all sufficiently large \(|z|\). 

We next show that \(\bm{\Psi}_{\alpha,z}\) is a contraction on \(\mathcal{B}_z\). For \(\bm{X},\bm{Y}\in\mathcal{B}_z\), the resolvent identity gives 
\[ 
\bm{\mathcal{S}}_\alpha(\bm{X}) - \bm{\mathcal{S}}_\alpha(\bm{Y}) = -\frac{1}{\alpha} \E_{\bm{y}} \left [ \bm{T}(\bm{y}) \left( \bm{I}_p + \frac{1}{\alpha}\bm{X}\bm{T}(\bm{y}) \right)^{-1} (\bm{X}-\bm{Y}) \bm{T}(\bm{y}) \left( \bm{I}_p + \frac{1}{\alpha}\bm{Y}\bm{T}(\bm{y}) \right)^{-1} \right]. 
\]
It thus follows that 
\[ 
\| \bm{\mathcal{S}}_\alpha(\bm{X}) - \bm{\mathcal{S}}_\alpha(\bm{Y}) \|_{\mathrm{op}} \leq \frac{4C_T^2}{\alpha} \|\bm{X}-\bm{Y}\|_{\mathrm{op}}. 
\] 
Applying the resolvent identity once more, we obtain 
\[ 
\begin{aligned}  
\| \bm{\Psi}_{\alpha,z}(\bm{X}) - \bm{\Psi}_{\alpha,z}(\bm{Y}) \|_{\mathrm{op}} & \leq \left\| \left( -z\bm{I}_p+\bm{\mathcal{S}}_\alpha(\bm{X}) \right)^{-1} \right\|_{\mathrm{op}} \| \bm{\mathcal{S}}_\alpha(\bm{X}) - \bm{\mathcal{S}}_\alpha(\bm{Y}) \|_{\mathrm{op}} \left\| \left( -z\bm{I}_p+\bm{\mathcal{S}}_\alpha(\bm{Y}) \right)^{-1} \right\|_{\mathrm{op}} \\ 
&\leq \frac{4}{|z|^2} \| \bm{\mathcal{S}}_\alpha(\bm{X}) - \bm{\mathcal{S}}_\alpha(\bm{Y}) \|_{\mathrm{op}} \\
&\leq \frac{16C_T^2}{\alpha|z|^2} \|\bm{X}-\bm{Y}\|_{\mathrm{op}}. 
\end{aligned} 
\] 
We may therefore choose \(R>0\) sufficiently large so that, for every \(|z|>R\), the map \(\bm{\Psi}_{\alpha,z}\) is a strict contraction from \(\mathcal{B}_z\) into itself. By Banach's fixed-point theorem, for each \(|z|>R\), there exists a unique \(\widehat{\bm{M}}_\alpha(z)\in\mathcal{B}_z\) such that \(\widehat{\bm{M}}_\alpha(z) = \bm{\Psi}_{\alpha,z} \left (\widehat{\bm{M}}_\alpha(z)\right)\). 

We now prove that \(z\mapsto\widehat{\bm{M}}_\alpha(z)\) is analytic. Fix \(z_0\in\C\) with \(|z_0|>R\). After increasing \(R\), if necessary, there exist a neighborhood \(V\) of \(z_0\), a closed ball \(\mathcal{B}\subset\C^{p\times p}\), and a constant \(q<1\) such that, for every \(z\in V\), the map \(\bm{\Psi}_{\alpha,z}\) maps \(\mathcal{B}\) into itself and is \(q\)-Lipschitz on \(\mathcal{B}\). Starting from \(\bm{X}_0(z)=\bm{0}\), define 
\[ 
\bm{X}_{k+1}(z) \coloneqq \bm{\Psi}_{\alpha,z}\left (\bm{X}_k(z)\right). 
\] 
Each \(\bm{X}_k\) is analytic on \(V\), and the contraction estimate implies that \(\bm{X}_k\) converges locally uniformly on \(V\) to \(\widehat{\bm{M}}_\alpha\). Hence \(\widehat{\bm{M}}_\alpha\) is analytic on \(V\). Since \(z_0\) was arbitrary, \(z\mapsto\widehat{\bm{M}}_\alpha(z)\) is analytic on \(\{z\in\C \colon |z|>R\}\). 

We next identify \(\widehat{\bm{M}}_\alpha\) with \(\bm{M}_\alpha\) in the upper half-plane. Consider the nonempty open set 
\[ 
\mathcal{C}_R \coloneqq \left\{ z\in\C_+ \colon |z|>R, \quad \Im z>\frac{|z|}{2} \right\}. 
\] 
For \(z\in\mathcal{C}_R\), the Herglotz bound gives 
\[ 
\|\bm{M}_\alpha(z)\|_{\mathrm{op}} \leq \frac{1}{\Im z} < \frac{2}{|z|}. 
\] 
Therefore, \(\bm{M}_\alpha(z)\in\mathcal{B}_z\). Since \(\bm{M}_\alpha(z)\) satisfies the fixed-point equation and the fixed point in \(\mathcal{B}_z\) is unique, 
\[ 
\bm{M}_\alpha(z) = \widehat{\bm{M}}_\alpha(z), \quad z\in\mathcal{C}_R. 
\] 
Both functions are analytic on the connected domain \(\C_+\cap\{z\in\C \colon |z|>R\}\). Since they agree on the nonempty open subset \(\mathcal{C}_R\), the identity theorem yields 
\[ 
\bm{M}_\alpha(z) = \widehat{\bm{M}}_\alpha(z), \quad z\in \C_+\cap\{z\in\C\colon |z|>R\}. 
\] 
Thus \(\widehat{\bm{M}}_\alpha\) is an analytic continuation of \(\bm{M}_\alpha\). 

It remains to prove the conjugation symmetry. Lemma~\ref{lem:push_through} implies that \(\bm{\mathcal{S}}_\alpha(\bm{X})^\ast = \bm{\mathcal{S}}_\alpha(\bm{X}^\ast)\) and that \(\bm{\Psi}_{\alpha,z}(\bm{X})^\ast = \bm{\Psi}_{\alpha,\bar{z}}(\bm{X}^\ast)\). It follows that \(\widehat{\bm{M}}_\alpha(z)^\ast\) is a fixed point of \(\bm{\Psi}_{\alpha,\bar{z}}\). Moreover, 
\[ 
\|\widehat{\bm{M}}_\alpha(z)^\ast\|_{\mathrm{op}} = \|\widehat{\bm{M}}_\alpha(z)\|_{\mathrm{op}} \leq \frac{2}{|z|} = \frac{2}{|\bar{z}|}, 
\] 
so \(\widehat{\bm{M}}_\alpha(z)^\ast \in \mathcal{B}_{\bar{z}}\). By uniqueness of the fixed point in \(\mathcal{B}_{\bar{z}}\), \(\widehat{\bm{M}}_\alpha(z)^\ast = \widehat{\bm{M}}_\alpha(\bar{z})\). For real \(x\) with \(|x|>R\), this gives \(\widehat{\bm{M}}_\alpha(x) = \widehat{\bm{M}}_\alpha(x)^\ast\). 
\end{proof}

For notational simplicity, we henceforth use the same notation \(\bm{M}_\alpha\) for the analytic continuation constructed in
Lemma~\ref{lem:analytic_continuation}. We first state a standard consequence of the Stieltjes inversion formula that will allow us to identify intervals outside the support of a measure.

\begin{lem} \label{lem:analytic_continuation_support}
Let \(\mu\) be a finite positive Borel measure on \(\R\), and let
\[
m_\mu(z) \coloneqq \int_\R \frac{\mu(\mathrm{d}\lambda)}{\lambda-z}, \quad z\in\C_+,
\]
be its Stieltjes transform. Suppose that \(m_\mu\) admits an analytic continuation to a complex neighborhood of an open interval \(I \subseteq \R\), and that this continuation is real-valued on \(I\). Then
\[
I \cap \supp (\mu) = \emptyset.
\]
\end{lem}

\begin{proof}
Fix \(x_0 \in I\). Choose \(a,b \in I\) such that \(a < x_0 < b\), \([a,b] \subset I\), and such that neither \(a\) nor \(b\) is an atom of \(\mu\). This is possible because a finite measure has at most countably many atoms. Since the analytic continuation of \(m_\mu\) is real-valued on \([a,b]\), continuity of the continuation implies
\[
\sup_{x\in[a,b]} \left |  \Im m_\mu(x+i\eta) \right| \to 0 ,
\]
as \(\eta \downarrow 0\). Hence
\[
\lim_{\eta\downarrow0} \int_a^b \Im m_\mu(x+i\eta) \mathrm{d}x = 0.
\]
By the Stieltjes inversion formula and the choice of \(a\) and \(b\),
\[
\mu((a,b)) = \lim_{\eta\downarrow0} \frac{1}{\pi} \int_a^b \Im m_\mu(x+i\eta) \mathrm{d}x = 0.
\]
Thus \(x_0\) has an open neighborhood of zero \(\mu\)-mass, and hence \(x_0\notin\supp(\mu)\). Since \(x_0\in I\) was arbitrary, we conclude that
\[
I\cap\supp(\mu)=\emptyset.
\]
\end{proof}

We now apply this criterion to the scalar measures obtained by testing the matrix-valued measure \(\bm{\Omega}_\alpha\) against vectors. This turns the analytic continuation obtained for large spectral parameters into compactness of the limiting support. Once compactness is known, the Stieltjes representation yields an analytic continuation to the full complement of the support.

\begin{lem} \label{lem:continuation_off_support} 
The matrix-valued measure \(\bm{\Omega}_\alpha\) is compactly supported, and
\[ 
\supp (\bm{\Omega}_\alpha) = \supp (\mu_\alpha) \subseteq [-R,R],
\] 
where \(R\) is the constant from Lemma~\ref{lem:analytic_continuation}. In particular, \(\mu_\alpha\) is a compactly supported probability measure. Moreover, the representation 
\[ 
\bm{M}_\alpha(z) = \int_\R \frac{\bm{\Omega}_\alpha(\mathrm{d}\lambda)}{\lambda-z} 
\] 
defines an analytic continuation of \(\bm{M}_\alpha\) to \(\C\setminus \supp (\mu_\alpha)\). For every \(x \in \R \setminus \supp(\mu_\alpha)\), the matrix \(\bm{M}_\alpha(x)\) is real symmetric.
\end{lem}

\begin{proof} 
We first prove that \(\bm{\Omega}_\alpha\) is supported in \([-R,R]\). Let \(I\subseteq\R \setminus [-R,R]\) be a bounded open interval. By Lemma~\ref{lem:analytic_continuation}, the function \(\bm{M}_\alpha\) admits an analytic continuation to a complex neighborhood of \(I\), and this continuation is Hermitian on \(I\). Fix \(\bm{v}\in\C^p\), and define the finite positive scalar measure 
\[ 
\nu_{\bm{v}}(B) \coloneqq \bm{v}^\ast\bm{\Omega}_\alpha(B)\bm{v},
\] 
for every Borel \(B \subseteq \R\). Its Stieltjes transform is 
\[ 
f_{\bm{v}}(z) \coloneqq \bm{v}^\ast\bm{M}_\alpha(z)\bm{v} = \int_{\R} \frac{\nu_{\bm{v}}(\mathrm{d}\lambda)}{\lambda-z}. 
\] 
Since \(\bm{M}_\alpha\) extends analytically through \(I\), the function \(f_{\bm{v}}\) also admits an analytic continuation to a complex neighborhood of \(I\). Moreover, because \(\bm{M}_\alpha(x)\) is Hermitian for \(x \in I\), \(f_{\bm{v}} (x) = \bm{v}^\ast \bm{M}_\alpha(x) \bm{v} \in\R\). Lemma~\ref{lem:analytic_continuation_support} therefore gives
\[
I \cap \supp(\nu_{\bm{v}})=\emptyset,
\]
and in particular \(\nu_{\bm{v}}(I) = \bm{v}^\ast \bm{\Omega}_\alpha(I) \bm{v} = 0\). Since this holds for every \(\bm{v}\in\C^p\), and since \(\bm{\Omega}_\alpha(I) \succeq \bm{0}\), we conclude that
\[
\bm{\Omega}_\alpha (I) =\bm{0}.
\]
Every point of \(\R \setminus [-R,R]\) belongs to a bounded open interval of this form. Hence
\[
\supp (\bm{\Omega}_\alpha)\subseteq[-R,R],
\]
so \(\bm{\Omega}_\alpha\) is compactly supported.

We next show that the matrix-valued measure and its normalized trace have the same support. Recall that \(\mu_\alpha (B) =\frac{1}{p}\Tr\bm{\Omega}_\alpha(B)\) for every Borel set \(B\subseteq\R\). Since \(\bm{\Omega}_\alpha(B)\succeq\bm{0}\), we have
\[
\mu_\alpha(B)=0 \quad \Longleftrightarrow\quad \Tr \bm{\Omega}_\alpha(B)=0
\quad \Longleftrightarrow \quad \bm{\Omega}_\alpha(B)=\bm{0}.
\]
In particular, an open set has zero \(\mu_\alpha\)-mass if and only if the matrix-valued measure vanishes on that open set. It follows directly from the definition of the support that
\[ 
\supp (\bm{\Omega}_\alpha) = \supp (\mu_\alpha). 
\] 
Since \(\mu_\alpha\) is a probability measure by Proposition~\ref{prop:existence_uniqueness}, it is therefore a compactly supported probability measure with \(\supp(\mu_\alpha)\subseteq[-R,R]\).

It remains to establish the continuation to the complement of the support. Let \(K \Subset \C \setminus \supp(\mu_\alpha)\) be compact. Since \(\supp(\bm{\Omega}_\alpha)=\supp(\mu_\alpha)\), there exists \(\delta_K>0\) such that
\[
|\lambda-z|\geq\delta_K
\]
for every \(z\in K\) and every \(\lambda\in\supp(\bm{\Omega}_\alpha)\). Therefore, for \(z \in K\), we have
\[
\left | \frac1{\lambda-z} \right| \leq \frac{1}{\delta_K},
\]
and the integral
\[
\int_\R \frac{\bm{\Omega}_\alpha(\mathrm{d}\lambda)}{\lambda-z}
\]
converges locally uniformly on \(\C \setminus \supp(\mu_\alpha)\). Entrywise differentiation under the integral sign shows that it defines an analytic matrix-valued function on this domain. On \(\C_+\), this function agrees with the original Stieltjes representation of \(\bm{M}_\alpha\). It therefore provides the desired analytic continuation of \(\bm{M}_\alpha\) to \(\C \setminus \supp(\mu_\alpha)\).

Finally, let \(x \in \R \setminus \supp(\mu_\alpha)\). Then
\[
\bm{M}_\alpha(x) = \int_\R \frac{\bm{\Omega}_\alpha (\mathrm{d} \lambda)}{\lambda-x}.
\]
The kernel \((\lambda-x)^{-1}\) is real, while \(\bm{\Omega}_\alpha\) is a Hermitian matrix-valued measure. Hence
\[
\bm{M}_\alpha(x)^\ast = \int_\R \frac{\bm{\Omega}_\alpha(\mathrm{d}\lambda)^\ast}{\lambda-x} = \bm{M}_\alpha(x)
\]
Moreover, the symmetry \(\bm{M}_\alpha(z)^\top=\bm{M}_\alpha(z)\), established in Lemma~\ref{lem:pointwise_existence_uniqueness} on \(\C_+\), extends by analytic continuation to \(\C \setminus \supp(\mu_\alpha)\). Hence
\[
\bm{M}_\alpha(x)^\top=\bm{M}_\alpha(x).
\]
Thus, on the real axis outside the support, \(\bm{M}_\alpha(x)\) is both symmetric and Hermitian, and therefore real symmetric.
This completes the proof.
\end{proof}

The preceding lemma gives the analytic framework needed for the stability analysis. In particular, on every compact interval separated from \(\supp(\mu_\alpha)\), the solution \(\bm{M}_\alpha\) extends to a complex neighborhood and is real symmetric on the real axis.

\subsection{Stability outside the limiting support}

Proposition~\ref{prop:existence_uniqueness} establishes existence and uniqueness of the solution \(\bm{M}_\alpha\) of the matrix Dyson equation within the normalized matrix-valued Herglotz class on \(\C_+\). By Lemma~\ref{lem:continuation_off_support}, \(\bm{M}_\alpha\) admits an analytic continuation to \(\C \setminus \supp(\mu_\alpha)\), which is real symmetric on \(\R\setminus\supp(\mu_\alpha)\). We refer to \(\bm{M}_\alpha\), together with this continuation, as the \emph{physical branch}, since it is the branch selected by the normalized Herglotz condition in the upper half-plane. Moreover, for \( x \in \R \setminus\supp(\mu_\alpha)\), we call the real-symmetric matrix \(\bm{M}_\alpha(x)\) the \emph{physical solution} at \(x\).

On the real axis, the positivity of the imaginary part is lost, and the uniqueness characterization of Proposition~\ref{prop:existence_uniqueness} does not immediately extend beyond \(\C_+\). In particular, the real matrix Dyson equation may admit additional solutions that do not arise from the continuation of the Herglotz solution. The purpose of this subsection is to characterize the physical solution among regular algebraic solutions through the stability of the linearized fixed-point map. We also derive a quantitative consequence of this stability that will be used later in the analysis of the BBP transition.

\begin{defn}
Let \(x \in \R\). A matrix \(\bm{M}\in \sym_p(\R)\) is called an \emph{algebraic solution} of the real matrix Dyson equation at \(x\) if the inverses below are well defined almost surely, the expectation is finite, and
\begin{equation} \label{eq:realMDE}
\bm{M} ^{-1}  = - x \bm{I}_p + \E_{\bm{y}} \left [  \bm{T} (\bm{y}) \left ( \bm{I}_p + \frac{1}{\alpha}  \bm{M} \bm{T} (\bm{y}) \right)^{-1}\right ].
\end{equation}
An algebraic solution \(\bm{M}\) at \(x\) is called \emph{regular} if
\[
\operatorname*{ess \, sup}_{\bm{y}} \left \| \left ( \bm{I}_p + \frac{1}{\alpha} \bm{M} \bm{T}(\bm{y})\right )^{-1} \right \|_{\mathrm{op}} < \infty.
\]
\end{defn}

We restrict attention to the component to the right of the limiting spectrum. Recall that \(\lambda_+(\alpha)= \sup \supp(\mu_\alpha)\). Fix $\varepsilon>0$ and $L>\lambda_+(\alpha)+\varepsilon$, and set
\[
I_{\varepsilon,L} \coloneqq [\lambda_+(\alpha)+\varepsilon,L].
\]
Choose \(0 < \delta < \varepsilon/4\) and define the nested neighborhoods
\[
\mathcal{U}_{\varepsilon,L}  \coloneqq \{ z \in \C \colon \operatorname{dist} (z, I_{\varepsilon,L}) < \delta\},
\]
and
\[
\mathcal{U}_{\varepsilon,L}^+  \coloneqq \{ z \in \C \colon \operatorname{dist} (z, I_{\varepsilon,L}) < 2\delta\}.
\]
Since \(2\delta < \varepsilon/2\),
\[
\overline{\mathcal{U}_{\varepsilon,L}^+} \subset \C \setminus \supp (\mu_\alpha).
\]
Thus the physical branch \(\bm{M}_\alpha\) is analytic throughout \(\mathcal{U}_{\varepsilon,L}^+\).

We first establish uniform bounds on the physical branch and its inverse on the larger off-support neighborhood.

\begin{lem}\label{lem:bound_solution}
There exists \(C_{\varepsilon,L}<\infty\) such that
\[
\sup_{z\in\overline{\mathcal{U}_{\varepsilon,L}^+}} \left ( \|\bm{M}_\alpha(z) \|_{\mathrm{op}} + \|\bm{M}_\alpha(z)^{-1}\|_{\mathrm{op}} \right )\leq C_{\varepsilon,L}.
\]
\end{lem}

\begin{proof}
By Lemma~\ref{lem:continuation_off_support}, the matrix-valued Stieltjes representation
\[
\bm{M}_\alpha(z) = \int_\R \frac{\bm{\Omega}_\alpha(\mathrm{d}\lambda)}{\lambda-z}
\]
holds for every \(z\in\C \setminus\supp(\mu_\alpha)\), and \(\supp(\bm{\Omega}_\alpha) = \supp(\mu_\alpha)\). Let
\[
d_{\varepsilon,L} \coloneqq \operatorname{dist} \left ( \overline{\mathcal{U}_{\varepsilon,L}^+}, \supp (\mu_\alpha) \right) >0.
\]
For unit vectors \(\bm{u},\bm{v} \in \C^p\), the matrix-valued Cauchy--Schwarz inequality and \(\bm{\Omega}_\alpha (\R) = \bm{I}_p\) give
\[
\begin{split}
\left |\bm{u}^\ast \bm{M}_\alpha(z) \bm{v} \right | & = \left | \int_\R \frac{1}{\lambda-z} \mathrm{d} (\bm{u}^\ast \bm{\Omega}_\alpha (\lambda) \bm{v} )\right | \\
& \le \frac{1}{d_{\varepsilon,L}} |\bm{u}^\ast \bm{\Omega}_\alpha \bm{v} |(\R) \\
& \le \frac{1}{d_{\varepsilon,L}}  \sqrt{\bm{u}^\ast \bm{\Omega}_\alpha (\R) \bm{u}} \sqrt{\bm{v}^\ast \bm{\Omega}_\alpha (\R) \bm{v}} \\
& =\frac{1}{d_{\varepsilon,L}} .
\end{split}
\]
Hence 
\[
\sup_{z\in\overline{\mathcal{U}_{\varepsilon,L}^+}} \|\bm{M}_\alpha (z)\|_{\mathrm{op}} \le \frac{1}{d_{\varepsilon,L}}.
\]

We next bound the inverse. Let \(z = E + i \eta \in\overline{\mathcal{U}_{\varepsilon,L}^+}\). Since \(\delta < \varepsilon/4\), we have
\[
E \ge \lambda_+(\alpha) + \varepsilon - 2\delta \ge \lambda_+(\alpha)+\frac{\varepsilon}{2}.
\]
Moreover,
\[
E \le L + 2\delta, \qquad |\eta|\leq 2\delta.
\]
Taking the Hermitian part of the Stieltjes representation gives
\[
-\Re \bm{M}_\alpha (z) = \int_\R \frac{E-\lambda}{(E-\lambda)^2+\eta^2} \bm{\Omega}_\alpha (\mathrm{d}\lambda).
\]
Write \(\lambda_-(\alpha)= \inf \supp(\mu_\alpha)\). For \(\lambda\in\supp(\bm{\Omega}_\alpha)\), we have 
\[
E - \lambda \ge \frac{\varepsilon}{2},
\]
while 
\[
(E-\lambda)^2+\eta^2 \le (L + 2\delta -\lambda_- (\alpha))^2 + (2\delta)^2.
\]
It follows that
\[
\frac{E-\lambda}{(E-\lambda)^2+\eta^2} \ge \frac{\varepsilon/2}{\left (L+2\delta-\lambda_-(\alpha)\right)^2+(2\delta)^2} \eqqcolon c_0 .
\]
Hence
\[
-\Re \bm{M}_\alpha (z) \succeq c_0  \bm{\Omega}_\alpha(\R) = c_0  \bm{I}_p.
\]
For every \(\bm{u} \in\C^p\), the Cauchy--Schwarz inequality therefore yields
\[
\|\bm{M}_\alpha (z)\bm{u}\|_2 \| \bm{u}\|_2 \ge \left| \bm{u}^\ast \bm{M}_\alpha(z)\bm{u} \right| \ge -\bm{u}^\ast \Re\bm{M}_\alpha (z) \bm{u} \ge c_0 \|\bm{u}\|_2^2.
\]
Thus, \(\|\bm{M}_\alpha (z)\bm{u}\|_2 \ge c_0 \|\bm{u}\|_2\), which shows that the smallest singular value of \(\bm{M}_\alpha (z)\) is at least \(c_0\). Hence
\[
\|\bm{M}_\alpha(z)^{-1}\|_{\mathrm{op}} \le \frac{1}{c_0}
\]
uniformly for \(z\in\overline{\mathcal{U}_{\varepsilon,L}^+}\). 
\end{proof}

We next establish uniform regularity of the nonlinear denominators on the smaller neighborhood.

\begin{lem} \label{lem:uniform_denominator_bound}
There exists \(C_{\varepsilon,L}<\infty\) such that
\[
\operatorname*{ess \, sup}_{\bm{y}} \sup_{z\in\overline{\mathcal{U}_{\varepsilon,L}}} \left \| \left ( \bm{I}_p + \frac{1}{\alpha} \bm{M}_\alpha (z) \bm{T}(\bm{y}) \right)^{-1}\right \|_{\mathrm{op}}\leq C_{\varepsilon,L}.
\]
\end{lem}

\begin{proof}
For \(z \in \C_+\), set
\begin{equation} \label{eq:mathcal_Q_y}
\bm{\mathcal{Q}}_{\bm{y}} (z) \coloneqq \bm{\mathcal{Q}}_\alpha (\bm{M}_\alpha(z), \bm{y}) =\bm{T} (\bm{y}) \left ( \bm{I}_p + \frac{1}{\alpha} \bm{M}_\alpha(z) \bm{T} (\bm{y})\right)^{-1}.
\end{equation}
By Lemma~\ref{lem:S_holo}, 
\[
\Im \bm{\mathcal{Q}}_{\bm{y}} (z) =  - \frac{1}{\alpha}\bm{\mathcal{Q}}_{\bm{y}} (z) \left ( \Im \bm{M}_\alpha(z) \right ) \bm{\mathcal{Q}}_{\bm{y}} (z)^\ast \preceq\bm{0}.
\]
Then
\[
\bm{H}_{\bm{y}}(z) \coloneqq \bm{T} (\bm{y})- \bm{\mathcal{Q}}_{\bm{y}} (z) .
\]
is a matrix-valued Herglotz function. Using 
\[
\bm{I}_p - \left( \bm{I}_p + \frac{1}{\alpha} \bm{M}_\alpha(z) \bm{T} (\bm{y}) \right)^{-1} = \frac{1}{\alpha} \bm{M}_\alpha(z) \bm{T} (\bm{y})  \left( \bm{I}_p + \frac{1}{\alpha} \bm{M}_\alpha(z) \bm{T} (\bm{y})  \right)^{-1},
\]
we obtain 
\[
\bm{H}_{\bm{y}} (z) = \frac{1}{\alpha} \bm{T} (\bm{y}) \bm{M}_\alpha(z)  \bm{T} (\bm{y}) \left( \bm{I}_p +  \frac{1}{\alpha} \bm{M}_\alpha(z) \bm{T} (\bm{y}) \right)^{-1}.
\]
Since \(i\eta\bm{M}_\alpha(i\eta)\to -\bm{I}_p\) and
\[
\left( \bm{I}_p + \frac{1}{\alpha}\bm{M}_\alpha(i\eta) \bm{T} (\bm{y})  \right)^{-1} \to \bm{I}_p,
\]
as \(\eta \to \infty\), it follows that
\[
i\eta \bm{H}_{\bm{y}}(i\eta) \to - \frac{1}{\alpha }\bm{T} (\bm{y})^2.
\]
The matrix-valued Nevanlinna--Herglotz representation theorem~\cite[Theorem 5.4]{gesztesy} therefore gives a positive semidefinite matrix-valued measure \(\bm{\Theta}_{\bm{y}}\) on \(\R\) such that
\[
\bm{H}_{\bm{y}}(z) =  \int_\R \frac{\bm{\Theta}_{\bm{y}}(\mathrm{d}\lambda)}{\lambda-z}, \qquad \bm{\Theta}_{\bm{y}}(\R) = \frac{1}{\alpha }\bm{T} (\bm{y})^2.
\]
In particular,\(\|\bm{\Theta}_{\bm{y}}(\R)\|_{\mathrm{op}} \leq \alpha^{-1} C_T^2\) almost surely.

We now show that these measures have no mass near \(I_{\varepsilon,L}\). Averaging \(\bm{H}_{\bm{y}}\) and using the MDE~\eqref{eq:MDE} give
\[
\overline{\bm{H}}(z) \coloneqq \E_{\bm{y}} \left [\bm{H}_{\bm{y}} (z) \right] = \E_{\bm{y}} [\bm{T} (\bm{y})] - \E_{\bm{y}} [\bm{\mathcal{Q}}_{\bm{y}} (z) ] = \E_{\bm{y}} [\bm{T} (\bm{y})] - \bm{M}_\alpha(z)^{-1} - z\bm{I}_p.
\]
By Lemma~\ref{lem:bound_solution}, the right-hand side is analytic on \(\mathcal{U}_{\varepsilon,L}^{+}\). Let 
\[
J_{\varepsilon,L} \coloneqq \mathcal{U}_{\varepsilon,L}^+ \cap \R = \left ( \lambda_+(\alpha)+\varepsilon-2\delta, L+2\delta \right ).
\]
For \(x \in J_{\varepsilon,L}\), Lemma~\ref{lem:continuation_off_support} gives \(\bm{M}_\alpha (x) \in \sym_p(\R)\), and hence \(\overline{\bm{H}} (x)\) is real symmetric. Thus, for every \(\bm{v}\in\C^p\), the scalar Herglotz function \(z \mapsto \bm{v}^\ast \overline{\bm{H}} (z) \bm{v}\) admits a real-valued analytic continuation through \(J_{\varepsilon,L}\). By Lemma~\ref{lem:analytic_continuation_support}, the representing matrix-valued measure \(\overline{\bm{\Theta}} \coloneqq \E_{\bm{y}}[\bm{\Theta}_{\bm{y}}]\) has no mass on \(J_{\varepsilon,L}\), i.e., 
\[
\overline{\bm{\Theta}} (J_{\varepsilon,L}) = \bm{0}.
\]
Since \(\bm{\Theta}_{\bm{y}} (J_{\varepsilon,L}) \succeq \bm{0}\), 
\[
\E_{\bm{y}} \Tr \bm{\Theta}_{\bm{y}} (J_{\varepsilon,L}) =0
\]
implies that
\[
\bm{\Theta}_{\bm{y}} (J_{\varepsilon,L}) = \bm{0}
\]
for almost every \(\bm{y}\).

The smaller neighborhood \(\overline{\mathcal{U}_{\varepsilon,L}}\) lies at a positive distance from \(\R \setminus J_{\varepsilon,L}\). Set
\[
d_0 \coloneqq \operatorname{dist} \left( \overline{\mathcal{U}_{\varepsilon,L}}, \R\setminus J_{\varepsilon,L} \right) > 0.
\]
For almost every \(\bm{y}\), the Stieltjes representation of \(\bm{H}_{\bm{y}}\) therefore extends analytically to a neighborhood of \(\overline{\mathcal{U}_{\varepsilon,L}}\), and for \(z \in \overline{\mathcal{u}_{\varepsilon,L}}\),
\[
\|\bm{H}_{\bm{y}}(z)\|_{\mathrm{op}} \leq \frac{C_T^2}{\alpha d_0}.
\]
Indeed, for unit vectors \(\bm{u},\bm{v}\), the Cauchy--Schwarz inequality implies that
\[
\left |  \bm{u}^\ast \bm{H}_{\bm{y}}(z) \bm{v} \right| \leq \frac1{d_0} \left| \bm{u}^\ast \bm{\Theta}_{\bm{y}} \bm{v} \right| (\R) \leq \frac1{d_0} \sqrt{\bm{u}^\ast\bm{\Theta}_{\bm{y}}(\R)\bm{u}} \sqrt{\bm{v}^\ast\bm{\Theta}_{\bm{y}}(\R)\bm{v}} \leq \frac{C_T^2}{\alpha d_0}.
\]
Define 
\[
\widetilde{\bm{Q}}_{\bm{y}} (z) \coloneqq \bm{T} (\bm{y}) -\bm{H}_{\bm{y}}(z).
\]
Then \(\widetilde{\bm{Q}}_{\bm{y}} (z)\) is an analytic continuation of \(\bm{Q}_{\bm{y}} (z)\) to a neighborhood of \(\overline{\mathcal{U}_{\varepsilon,L}}\), and
\begin{equation} \label{eq:bound_tilde_Q}
\sup_{z\in\overline{\mathcal{U}_{\varepsilon,L}}} \|\widetilde{\bm{Q}}_{\bm{y}}(z) \|_{\mathrm{op}} \leq
C_T + \frac{C_T^2}{\alpha d_0}.
\end{equation}
For \(z\in\mathcal{U}_{\varepsilon,L}\), define
\[
\bm{A}_{\bm{y}} (z) \coloneqq\bm{I}_p + \frac{1}{\alpha}\bm{M}_\alpha(z)\bm{T} (\bm{y}),
\]
and
\[
\bm{B}_{\bm{y}} (z) \coloneqq \bm{I}_p - \frac{1}{\alpha} \bm{M}_\alpha(z) \bm{Q}_{\bm{y}} (z) .
\]
Both functions are analytic on \(\mathcal{U}_{\varepsilon,L}\). On \(\mathcal{U}_{\varepsilon,L}\cap\C_+\), the resolvent identity implies
\[
\bm{B}_{\bm{y}} (z)  = \bm{A}_{\bm{y}} (z)^{-1},
\]
and hence
\[
\bm{A}_{\bm{y}} (z) \bm{B}_{\bm{y}} (z)  = \bm{I}_p.
\]
By the identity theorem, this product identity holds throughout \(\mathcal{U}_{\varepsilon,L}\). Thus \(\bm{A}_{\bm{y}}(z)\) is invertible there and
\[
\bm{A}_{\bm{y}} (z)^{-1} = \bm{B}_{\bm{y}} (z) .
\]
Therefore
\[
\left \| \left( \bm{I}_p + \frac{1}{\alpha}  \bm{M}_\alpha(z)\bm{T} (\bm{y}) \right)^{-1} \right\|_{\mathrm{op}} \leq
1 + \frac{1}{\alpha} \|\bm{M}_\alpha(z)\|_{\mathrm{op}} \|\widetilde{\bm{Q}}_{\bm{y}}(z) \|_{\mathrm{op}} ,
\]
and the claimed uniform bound follows from~\eqref{eq:bound_tilde_Q} and Lemma~\ref{lem:bound_solution}.
\end{proof}

\begin{rmk} \label{rmk:physical_solution}
The preceding lemma shows that the physical solution is regular on the component to the right of the limiting support. Indeed, let $x > \lambda_+(\alpha)$.  Choose \(\varepsilon >0\) and \(L > x\) such that \(x \in I_{\varepsilon,L}\). By Lemma~\ref{lem:uniform_denominator_bound},
\[
\operatorname*{ess\,sup}_{\bm{y}} \left\lVert \left( \bm{I}_p + \frac{1}{\alpha} \bm{M}_\alpha (x) \bm{T} (\bm{y}) \right)^{-1} \right\rVert_{\mathrm{op}} <\infty.
\]
Moreover, letting $\eta\downarrow0$ in the matrix Dyson equation at $x+i\eta$, the uniform denominator bound and dominated convergence yield
\[
\bm{M}_\alpha(x)^{-1} = -x \bm{I}_p + \E_{\bm{y}} \left[ \bm{T}(\bm{y}) \left( \bm{I}_p + \frac{1}{\alpha} \bm{M}_\alpha(x)\bm{T} (\bm{y}) \right)^{-1} \right].
\]
Thus $\bm{M}_\alpha(x)$ is a regular algebraic solution for every $x>\lambda_+(\alpha)$.
\end{rmk}

Let now \(\bm{M} \in\sym_p(\R)\) be a regular algebraic solution of the real MDE at \(x\). Set
\[
\bm{Q}_{\bm{M}} (\bm{y}) \coloneqq \bm{T} (\bm{y}) \left( \bm{I}_p + \frac{1}{\alpha} \bm{M} \bm{T} (\bm{y}) \right)^{-1}.
\]
By Lemma~\ref{lem:push_through},
\[
\bm{Q}_{\bm{M}}(\bm{y}) = \left( \bm{I}_p+ \frac{1}{\alpha} \bm{T} (\bm{y} ) \bm{M} \right)^{-1} \bm{T} (\bm{y}),
\]
so $\bm{Q}_{\bm{M}} (\bm{y})$ is symmetric. The derivative of the fixed-point map at $\bm{M}$ is therefore the linear operator
\begin{equation} \label{eq:mathcal_K}
\mathcal{K}_{\bm{M}} [\bm{H}] \coloneqq \frac{1}{\alpha} \bm{M} \E_{\bm{y}} \left [ \bm{Q}_{\bm{M}} (\bm{y})  \bm{H} \bm{Q}_{\bm{M}} (\bm{y}) \right ] \bm{M}, \qquad \bm{H} \in \sym_p(\R).
\end{equation}

\begin{defn}\label{eq:def_K_M}
The \emph{stability operator} associated with the algebraic solution \(\bm{M}\) is
\[
\mathcal{L}_{\bm{M}} \coloneqq \operatorname{Id} - \mathcal{K}_{\bm{M}} .
\]
Equivalently,
\[
\mathcal{L}_{\bm{M}} [\bm{H}]\coloneqq \bm{H} - \frac{1}{\alpha} \bm{M} \E_{\bm{y}}
\left[ \bm{Q}_{\bm{M}} (\bm{y}) \bm{H} \bm{Q}_{\bm{M}}(\bm{y}) \right] \bm{M}.
\]
\end{defn}

The operator \(\mathcal{L}_{\bm{M}}\) is the linearization of the fixed-point equation at \(\bm{M}\). Indeed, 
\[
D \bm{\Psi}_{\alpha,x}(\bm{M}) = \mathcal{K}_{\bm{M}} ,
\]
and therefore
\[
\mathcal{L}_{\bm{M}} = \operatorname{Id} - D \bm{\Psi}_{\alpha,x} (\bm{M}).
\]
Whenever the physical solution \(\bm{M}_\alpha(x)\) is regular, we use the shorthand notation \(\bm{Q}_x (\bm{y}) \coloneqq \bm{Q}_{\bm{M}_\alpha(x)}(\bm{y})\), \(\mathcal{K}_x \coloneqq \mathcal{K}_{\bm{M}_\alpha(x)}\), and \(\mathcal{L}_x \coloneqq \mathcal{L}_{\bm{M}_\alpha (x)}\).

The next lemma characterizes stable regular algebraic solutions and, in particular, shows that stability forces the spectral parameter to lie outside the limiting support.

\begin{lem} \label{lem:physical_sol}
Let \(x \in \R\) and \(\bm{M} \in\sym_p(\R)\) be a regular algebraic solution of the real matrix Dyson equation at \(x\). Then the following are equivalent:
\[
\rho(\mathcal{K}_{\bm{M}})<1 \qquad \Longleftrightarrow x \notin \supp(\mu_\alpha) \quad\textnormal{and}\quad \bm{M}=\bm{M}_\alpha(x)
\]
\end{lem}

\begin{proof}
Suppose first that \(x \notin \supp (\mu_\alpha)\) and \(\bm{M}= \bm{M}_\alpha(x)\). By Lemma~\ref{lem:continuation_off_support}, the physical branch is analytic in a neighborhood of \(x\). Moreover, since \( \bm{M} = \bm{M}_\alpha(x)\) is regular, boundedness of \(\bm{T}\) implies that \((z,\bm{X}) \mapsto \bm{\Psi}_{\alpha,z} (\bm{X})\) is analytic in a neighborhood of \((x,\bm{M})\). Hence the fixed-point identity extends through \(x\). The Stieltjes representation gives
\[
\bm{M}_\alpha'(x) = \int_\R \frac{\bm{\Omega}_\alpha(\mathrm{d}\lambda)}{(\lambda-x)^2} \succ \bm{0}.
\]
Differentiating the fixed-point equation \(\bm{M}_\alpha (z) =\bm{\Psi}_{\alpha,z} (\bm{M}_\alpha(z))\) at \(z=x\) gives
\[
\left ( \operatorname{Id} - \mathcal{K}_x \right ) [\bm{M}_\alpha'(x)] = \bm{M}_\alpha ^2(x).
\]
Therefore, 
\[
\mathcal{K}_x [\bm{M}_\alpha'(x)] = \bm{M}_\alpha'(x) - \bm{M}_\alpha^2(x) \prec \bm{M}_\alpha'(x).  
\]
The operator \(\mathcal{K}_x\) preserves the positive semidefinite cone. Indeed, since \(\bm{Q}_x (\bm{y})\) and \(\bm{M}_\alpha(x)\) are symmetric, for every \(\bm{H} \in \sym_p^+(\R)\),
\[
\begin{split}
\mathcal{K}_x [\bm{H}] & = \frac{1}{\alpha}  \E_{\bm{y}} \left [\bm{M}_\alpha (x) \bm{Q}_x (\bm{y})  \bm{H} \bm{Q}_x (\bm{y})\bm{M}_\alpha (x) \right ] \\
& = \frac{1}{\alpha}  \E_{\bm{y}} \left [ \left (  \bm{Q}_x (\bm{y}) \bm{M}_\alpha (x)\right)^\ast \bm{H} \left (  \bm{Q}_x (\bm{y}) \bm{M}_\alpha (x)\right) \right ] \succeq \bm{0}.
\end{split}
\]
Set \(\bm{L} \coloneqq\bm{M}_\alpha'(x) \in \sym_p^{++} (\R)\). Defining
\[
q_x \coloneqq \lambda_1 \left (  \bm{L}^{-1/2} \mathcal{K}_x [\bm{L}] \bm{L}^{-1/2} \right ),
\]
gives
\[
0 \le q_x < 1, \qquad \mathcal{K}_x [\bm{L}] \preceq q_x \bm{L}.
\]
Now, for \(\bm{H}\in \sym_p(\R)\), introduce  
\[ 
\|\bm{H}\|_{\bm{L}} \coloneqq \inf \left\{ c>0 \colon -c\bm{L}\preceq \bm{H}\preceq c\bm{L} \right\}. 
\] 
Since \(\bm{L} \succ \bm{0}\), this is a norm. In fact, we have
\[
\|\bm{H}\|_{\bm{L}} = \| \bm{L}^{-1/2} \bm{H}\bm{L}^{-1/2} \|_{\mathrm{op}}.
\]
Since the map \(\mathcal{K}_x\) preserves the positive semidefinite cone, we have 
\[ 
-c\bm{L}\preceq \bm{H}\preceq c\bm{L} \Longrightarrow -c\mathcal{K}_x [\bm{L}]  \preceq \mathcal{K}_x  [\bm{H}] \preceq c\mathcal{K}_x [\bm{L}] .
\] 
Using \(\mathcal{K}_x [\bm{L}] \preceq q_x \bm{L}\), we get 
\[ 
-q_x c\bm{L} \preceq \mathcal{K}_x [\bm{H}] \preceq q_x c\bm{L}. 
\] 
Therefore 
\[ 
\|\mathcal{K}_x [\bm{H}] \|_{\bm{L}} \leq q_x \|\bm{H}\|_{\bm{L}}. 
\] 
Thus the operator norm of \(\mathcal{K}_x\) with respect to this norm satisfies \(\| \mathcal{K}_x \|_{\bm{L} \to \bm{L}} \le q_x \). Since the spectral radius is bounded by every operator norm, it follows that
\[
\rho (\mathcal{K}_x) \le \| \mathcal{K}_x \|_{\bm{L} \to \bm{L}} \le q_x < 1.
\]

Conversely, suppose that \(\rho(\mathcal{K}_{\bm{M}})<1\). To apply the holomorphic implicit-function theorem, we complexify the real vector space \(\sym_p(\R)\). Its natural complexification is
\[
\sym_p(\C) \coloneqq \{\bm{X} \in \C^{p \times p} \colon \bm{X}^\top = \bm{X} \}.
\]
Let \(\mathcal{K}_{\bm{M}}^{\C} \colon\sym_p(\C)\to\sym_p(\C)\) denote the complexification of \(\mathcal{K}_{\bm{M}}\), namely
\[
\mathcal{K}_{\bm{M}}^\C [\bm{H}] = \frac{1}{\alpha} \bm{M} \E_{\bm{y}} \left [ \bm{Q}_{\bm{M}} (\bm{y})  \bm{H} \bm{Q}_{\bm{M}} (\bm{y}) \right ] \bm{M},
\]
and set 
\[
\mathcal{L}_{\bm{M}}^{\C} = \operatorname{Id}-\mathcal{K}_{\bm{M}}^{\C}. 
\]
With respect to any real basis of \(\sym_p(\R)\), \(\mathcal{K}_{\bm{M}}\) and its complexification \(\mathcal{K}_{\bm{M}}^{\C}\) are represented by the same real matrix. Hence they have the same characteristic polynomial and, in particular,
\[
\rho(\mathcal{K}_{\bm{M}}^{\C}) = \rho(\mathcal{K}_{\bm{M}})<1.
\]
It follows that \(\mathcal{L}_{\bm{M}}^{\C}\) is invertible. Consider the map
\[
\mathcal{G} (z,\bm{X}) \coloneqq \bm{X} - \bm{\Psi}_{\alpha,z} (\bm{X}).
\]
Since \(\bm{M}\) is an algebraic solution at \(x\), \(\mathcal{G} (x,\bm{M}) =\bm{0}_{p \times p}\). Moreover, regularity of \(\bm{M}\) and boundedness of \(\bm{T}\) imply that the matrices
\[
\bm{I}_p + \frac{1}{\alpha} \bm{X} \bm{T} (\bm{y})
\]
remain uniformly invertible for \(\bm{X}\) in a sufficiently small complex neighborhood of \(\bm{M}\). Indeed, this follows from the regularity bound at \(\bm{M}\) and a Neumann-series argument. Therefore \(\mathcal{G}\) is holomorphic in a neighborhood of \((x,\bm{M}) \in \C \times \sym_p(\C)\), and
\[
D_{\bm{X}} \mathcal{G} (x,\bm{M}) = \mathcal{L}_{\bm{M}}^\C.
\]
The holomorphic implicit-function theorem therefore yields a neighborhood \(U\) of \(x\) and a unique holomorphic map \(\widehat{\bm{M}} \colon U \to \sym_p(\C)\) such that
\[
\widehat{\bm{M}}(x) =\bm{M} \qquad  \mathrm{and} \qquad \mathcal{G} (z,\widehat{\bm{M}}(z)) =\bm{0}_{p \times p}, \quad z \in U.
\]
After shrinking \(U\), we may assume that it is invariant under complex conjugation.

We next compute the derivative of this local branch. Differentiating \(\mathcal{G} (x,\bm{M}) =\bm{0}_{p \times p}\) at \(z=x\) and applying the chain rule gives
\[
D_z \mathcal{G} (x,\bm{M})[1] + D_{\bm{X}} \mathcal{G} (x,\bm{M}) [\widehat{\bm{M}}'(x)] = \bm{0}_{p \times p}.
\]
Since
\[
D_z \bm{\Psi}_{\alpha,z}(\bm{X})[1] = \bm{\Psi}_{\alpha,z}(\bm{X})^2,
\]
and \(\bm{\Psi_}{\alpha,x}(\bm{M})=\bm{M}\) by~\eqref{eq:fixed_point_eq}, we have
\[
D_z \mathcal{G} (x,\bm{M})[1] = -\bm{M}^2.
\]
Consequently,
\[
\mathcal{L}_{\bm{M}}^{\C} [\widehat{\bm{M}}'(x)] = \bm{M}^2,
\]
and hence
\[
\widehat{\bm{M}}'(x) = (\mathcal{L}_{\bm{M}}^\C)^{-1} [\bm{M}^2].
\]
Since \(\bm{M}^2 \in \sym_p(\R)\) and \(\mathcal{L}_{\bm{M}}^\C\) is the complexification of \(\mathcal{L}_{\bm{M}}\), this is equivalently
\[
\widehat{\bm{M}}'(x) = \mathcal{L}_{\bm{M}}^{-1}[\bm{M}^2].
\]

Since \(\rho (\mathcal{K}_{\bm{M}})<1\), the Neumann series converges in operator norm and gives
\[
\mathcal{L}_{\bm{M}}^{-1} = \sum_{k=0}^\infty \mathcal{K}_{\bm{M}}^k.
\]
The operator \(\mathcal{K}_{\bm{M}}\) preserves the positive semidefinite cone. Moreover, \(\bm{M}\) is invertible, and therefore \(\bm{M}^2 \succ \bm{0}\). It follows that
\[
\widehat{\bm{M}}' (x) = \sum_{k=0}^\infty  \mathcal{K}_{\bm{M}}^k [\bm{M}^2] = \bm{M}^2 + \sum_{k=1}^\infty  \mathcal{K}_{\bm{M}}^k [\bm{M}^2] \succ \bm{0}.
\]
Taylor expansion at \(x\) therefore yields 
\[
\Im \widehat{\bm{M}} (x+i\eta) = \eta \widehat{\bm{M}}'(x) + O(\eta^2) \succ \bm{0},
\]
for every sufficiently small \(\eta>0\). Thus \(\widehat{\bm{M}} (x+i\eta) \in \mathbb{H}_p\). By construction, \(\widehat{\bm{M}} (x+i\eta)\) is a fixed point of \(\bm{\Psi}_{\alpha, x + i\eta}\). By pointwise uniqueness from Lemma~\ref{lem:pointwise_existence_uniqueness}, it follows that
\[
\widehat{\bm{M}} (x+i\eta) = \bm{M}_\alpha(x+i\eta),
\]
for every sufficiently small \(\eta>0\). By the identity theorem, \(\widehat{\bm{M}}(z)=\bm{M}_\alpha(z)\) for \(z \in U \cap \C_+\).

It remains to identify the behavior of the local branch on the real axis. Define \(\widetilde{\bm{M}} \coloneqq \overline{\widehat{\bm{M}}(\overline{z})}\) for \(z \in U\). The map \(z \mapsto \widetilde{\bm{M}}(z)\) is analytic on \(U\). Since \(\bm{T}(\bm{y})\) is real symmetric, 
\[
\overline{\bm\Psi_{\alpha,\overline z}(\bm{X})} =\bm\Psi_{\alpha,z}(\overline{\bm{X}}).
\]
It follows that
\[
\widetilde{\bm{M}}(z)=\bm\Psi_{\alpha,z}\bigl(\widetilde{\bm{M}}(z)\bigr), \qquad z \in U.
\]
Moreover, 
\[
\widetilde{\bm{M}}(x) = \overline{\widehat{\bm{M}}(x)} = \overline{\bm{M}} = \bm{M}.
\]
Thus \(\widetilde{\bm{M}}\) and \(\widehat{\bm{M}}\) are two local analytic solution branches passing through \((x,\bm{M})\). By the local uniqueness from the implicit-function theorem, 
\[
\widetilde{\bm{M}} (z) = \widehat{\bm{M}} (z), \quad z \in U.
\]
Therefore, 
\[
\widehat{\bm{M}}(z) = \overline{\widehat{\bm{M}}(\overline{z})}.
\]
In particular, \(\widehat{\bm{M}}(t)\in\sym_p(\R)\) for every \(t \in U \cap \R\).

Consequently, \(\widehat{\bm{M}}\) is a real-symmetric analytic continuation of the physical branch through the interval \(U\cap\R\). Lemma~\ref{lem:analytic_continuation_support} therefore implies that \(\bm\Omega_\alpha\) has no mass on a neighborhood of \(x\). Hence
\[
x\notin\supp(\mu_\alpha).
\]
Finally, evaluating the continuation at \(x\) gives
\[
\bm{M} = \widehat{\bm{M}} (x) = \bm{M}_\alpha (x).
\]
\end{proof}

\begin{cor} \label{cor:physical_stability_right}
For every \(x>\lambda_+(\alpha)\), the physical solution is regular and satisfies \(\rho(\mathcal{K}_x)<1\).
\end{cor}

\begin{proof}
By Remark~\ref{rmk:physical_solution}, the physical solution \(\bm{M}_\alpha(x)\) is a regular algebraic solution at \(x\). Applying Lemma~\ref{lem:physical_sol} with \(\bm{M} = \bm{M}_\alpha(x)\) yields \(\rho(\mathcal{K}_x)<1\).
\end{proof}


\section{Proof of the bulk law} \label{section:bulk}

The proof of Proposition~\ref{prop:bulk} has two steps. First, we derive an approximate self-consistent equation for 
\[
\bm{M}_n(z) =\tr_{d-r}^{(p)} \left ( \bm{G}_n(z) \right), \qquad z \in \C_+,
\]
where \(\bm{G}_n(z)\) is the resolvent of the random matrix \(\bm{P}_n\) defined in~\eqref{eq:P_n}. The approximation produces three errors: a Gaussian quadratic-form error, a self-averaging error, and a leave-one-out error. We show that all three vanish almost surely when \(\Im z\) is sufficiently large. Second, a normal-family argument identifies every subsequential limit with the normalized Herglotz solution \(\bm{M}_\alpha\) of the limiting matrix Dyson equation~\eqref{eq:MDE}. This yields the almost-sure convergence of the empirical spectral distribution.\\

Fix \(z \in \C_+\). We start from the defining identity of the resolvent, 
\[
z \bm{G}_n(z) + \bm{I}_{p(d-r)} = \bm{P}_n \bm{G}_n(z),
\]
and take the normalized partial trace~\eqref{eq:partial_trace} on both sides. This yields
\begin{equation} \label{eq:resolvent-id} 
\begin{split}
z \bm{M}_n (z) + \bm{I}_p & = \tr_{d-r}^{(p)} (\bm{P}_n \bm{G}_n(z)) \\
& = \frac{1}{n} \sum_{i=1}^n \tr_{d-r}^{(p)} \left ( \left( \bm{T}(\bm{y}_i) \otimes \bm{u}_i \bm{u}_i^\top\right) \bm{G}_n(z)\right) \\
& = \frac{1}{n} \sum_{i=1}^n  \bm{T} (\bm{y}_i) \frac{1}{d-r} (\bm{I}_p \otimes \bm{u}_i)^\top \bm{G}_n (z) (\bm{I}_p \otimes \bm{u}_i).
\end{split}
\end{equation}
For \(i \in [n]\), introduce the leave-one-out matrix
\begin{equation*}
\bm{P}_n^{(i)} \coloneqq \bm{P}_n - \frac{1}{n} (\bm{I}_p \otimes \bm{u}_i) \bm{T} (\bm{y}_i) (\bm{I}_p \otimes \bm{u}_i)^\top,
\end{equation*} 
its resolvent \(\bm{G}_n^{(i)} (z)\coloneqq (\bm{P}_n^{(i)} - z \bm{I}_{p(d-r)})^{-1}\) and the associated normalized partial trace \(\bm{M}_n^{(i)} (z)  \coloneqq \tr_{d-r}^{(p)} (\bm{G}_n^{(i)}(z))\). We also introduce the quadratic forms
\begin{equation} \label{eq:tilde_M}
\begin{split}
\bm{\widetilde{M}}_{n,i} (z) & \coloneqq  \frac{1}{d-r}(\bm{I}_p \otimes \bm{u}_i)^\top\bm{G}_n(z) (\bm{I}_p \otimes \bm{u}_i) ,\\
\bm{\widetilde{M}}^{(i)}_{n,i} (z) & \coloneqq  \frac{1}{d-r}(\bm{I}_p \otimes \bm{u}_i)^\top\bm{G}_n^{(i)}(z) (\bm{I}_p \otimes \bm{u}_i) .
\end{split}
\end{equation}
By construction, \(\bm{P}_n^{(i)} \) is independent of both \(\bm{y}_i\) and \(\bm{u}_i\). Before applying the Woodbury identity, note that the inverse appearing below is well defined for every \(z\in\C_+\). Indeed, both \(\bm{P}_n - z \bm{I}_{p(d-r)}\) and \(\bm{P}_n^{(i)}-z\bm{I}_{p(d-r)}\) are invertible because the two matrices are Hermitian. The matrix determinant lemma therefore gives
\[
\det \left(\bm{I} _p+\gamma_n \bm{T} (\bm{y}_i) \widetilde{\bm{M}}_{n,i}^{(i)} (z) \right) =
\frac{\det(\bm{P}_n- z \bm{I}_{p(d-r)})}{\det (\bm{P}_n^{(i)}- z\bm{I}_{p(d-r)})} \neq 0.
\]
Hence the Woodbury identity yields
\begin{equation} \label{eq:Woodbury}
\begin{split} 
\bm{G}_n(z) & = \left ( \bm{P}_n - z \bm{I}_{p (d-r)} \right)^{-1} \\
& = \left(\bm{P}^{(i)}_n + \frac{1}{n}(\bm{I}_p \otimes \bm{u}_i) \bm{T} (\bm{y}_i) (\bm{I}_p \otimes \bm{u}_i)^\top - z \bm{I}_{p(d-r)} \right )^{-1} \\
&= \bm{G}^{(i)}_n(z) - \frac{1}{n} \bm{G}^{(i)}_n(z)  (\bm{I}_p \otimes \bm{u}_i) \left ( \bm{I}_p + \gamma_n \bm{T} (\bm{y}_i) \bm{\widetilde{M}}^{(i)}_{n,i} (z) \right)^{-1} \bm{T} (\bm{y}_i) (\bm{I}_p \otimes \bm{u}_i)^\top\bm{G}_n^{(i)}(z) ,
\end{split}
\end{equation}
where we write \(\gamma_n \coloneqq (d-r)/n\). Combining~\eqref{eq:Woodbury} and~\eqref{eq:tilde_M}, we obtain
\begin{equation} \label{eq:Mtilde}
\begin{split}
\bm{\widetilde{M}}_{n,i} (z) & = \bm{\widetilde{M}}_{n,i}^{(i)} (z) - \gamma_n \bm{\widetilde{M}}_{n,i}^{(i)} (z) \left( \bm{I}_p + \gamma_n \bm{T}(\bm{y}_i)   \bm{\widetilde{M}}_{n,i}^{(i)} (z) \right)^{-1} \bm{T}(\bm{y}_i)  \bm{\widetilde{M}}_{n,i}^{(i)} (z) \\
&  = \bm{\widetilde{M}}_{n,i}^{(i)} (z) \left( \bm{I}_p + \gamma_n \bm{T} (\bm{y}_i)  \bm{\widetilde{M}}_{n,i}^{(i)} (z) \right)^{-1} .
\end{split}
\end{equation}
The second equality follows from \(\bm{I} - (\bm{I}+ \bm{A})^{-1} \bm{A} = (\bm{I}+ \bm{A})^{-1}\). Substituting this identity into \eqref{eq:resolvent-id}, we obtain
\begin{equation} \label{eq:expansion}
z \bm{M}_n (z) + \bm{I}_p  = \frac{1}{n} \sum_{i=1}^n \bm{T} (\bm{y}_i)   \bm{\widetilde{M}}_{n,i}^{(i)} (z) \left( \bm{I}_p + \gamma_n \bm{T} (\bm{y}_i)  \bm{\widetilde{M}}_{n,i}^{(i)} (z) \right)^{-1} .
\end{equation}
To streamline the notation, for \(\bm{T} \in \sym_p(\R)\), \(\gamma \ge0\), and \(\bm{X} \in \C^{p \times p}\) such that \(\bm{I}_p + \gamma \bm{T} \bm{X}\) is invertible, define 
\[
\Phi_{\bm{T},\gamma}(\bm{X}) \coloneqq \bm{T} \bm{X} (\bm{I}_p + \gamma \bm{T} \bm{X})^{-1} \in \C^{p \times p}.
\]
Then~\eqref{eq:expansion} becomes
\begin{equation}\label{eq:phi-expansion}
z \bm{M}_n(z) + \bm{I}_p = \frac{1}{n} \sum_{i=1}^n \Phi_{\bm{T} (\bm{y}_i),\gamma_n} \left (\widetilde{\bm{M}}_{n,i}^{(i)} (z) \right).
\end{equation}
We now rewrite~\eqref{eq:phi-expansion} as an approximate self-consistent equation by adding and subtracting suitable intermediate terms. Here and below, \(\E_{\bm{y}}\) denotes expectation only with respect to an independent random vector \(\bm{y}\) having the common distribution of the \(\bm{y}_i\)'s. We obtain
\begin{equation} \label{eq:approximate_MDE}
z \bm{M}_n(z) + \bm{I}_p = \E_{\bm{y}} \left [\Phi_{\bm{T}(\bm{y}), \gamma_n} \left ( \bm{M}_n (z)\right) \right] + \mathcal{E}_{1,n} (z) + \mathcal{E}_{2,n} (z)  + \mathcal{E}_{3,n} (z) ,
\end{equation}
where
\[
\begin{split}
\mathcal{E}_{1,n} (z) & \coloneqq \frac{1}{n} \sum_{i=1}^n \left ( \Phi_{\bm{T} (\bm{y}_i),\gamma_n} \left (\widetilde{\bm{M}}_{n,i}^{(i)} (z) \right) - \Phi_{\bm{T} (\bm{y}_i),\gamma_n} \left (\bm{M}_n^{(i)} (z) \right) \right), \\
\mathcal{E}_{2,n}(z) & \coloneqq \frac{1}{n} \sum_{i=1}^n \left ( \Phi_{\bm{T} (\bm{y}_i),\gamma_n} \left (\bm{M}_n^{(i)} (z) \right) - \E_{\bm{y}} \left [ \Phi_{\bm{T}(\bm{y}), \gamma_n} \left ( \bm{M}_n^{(i)} (z)\right) \right ] \right) ,\\
\mathcal{E}_{3,n} (z) & \coloneqq \frac{1}{n} \sum_{i=1}^n \left ( \E_{\bm{y}} \left [ \Phi_{\bm{T}(\bm{y}), \gamma_n} \left ( \bm{M}_n^{(i)} (z)\right) \right ] - \E_{\bm{y}} \left [ \Phi_{\bm{T}(\bm{y}), \gamma_n} \left ( \bm{M}_n (z)\right) \right ] \right) .
\end{split}
\]
The three error terms have different origins. The term \(\mathcal{E}_{1,n}\) measures the deviation of the leave-one-out quadratic forms from their normalized block traces, \(\mathcal{E}_{2,n}\) is a self-averaging error, and \(\mathcal{E}_{3,n}\) measures the effect of replacing \(\bm{M}_n^{(i)}\) by \(\bm{M}_n\). We control these terms in the following lemmas. 

\begin{lem} \label{lem:stability-phi}
Fix \(\eta_0>0\) sufficiently large. Then, almost surely, for all sufficiently large \(n\), all \(i\in[n]\), and all \(z\in\C_+\) with \(\Im z>\eta_0\),
\[
\left \| \Phi_{\bm{T} (\bm{y}_i),\gamma_n} (\widetilde{\bm{M}}_{n,i}^{(i)}(z)) - \Phi_{\bm{T} (\bm{y}_i),\gamma_n} (\bm{M}_n^{(i)} (z) ) \right\|_{\mathrm{op}} \le C_{\Phi} (\eta_0) \left \| \widetilde{\bm{M}}_{n,i}^{(i)} (z) - \bm{M}_n^{(i)} (z) \right\|_{\mathrm{op}},
\]
and similarly
\[
\left \| \E_{\bm{y}} \left [ \Phi_{\bm{T} (\bm{y}),\gamma_n} (\bm{M}_n^{(i)}(z)) \right] - \E_{\bm{y}} \left [ \Phi_{ \bm{T}(\bm{y}),\gamma_n} (\bm{M}_n (z) ) \right] \right\|_{\mathrm{op}} \le C_{\Phi} (\eta_0) \left \| \bm{M}_n^{(i)} (z) - \bm{M}_n(z) \right\|_{\mathrm{op}}.
\]
\end{lem}

\begin{proof}
We first show that, almost surely for all sufficiently large \(n\), the matrices \(\gamma_n \bm{T} (\bm{y}_i) \widetilde{\bm{M}}_{n,i}^{(i)}(z)\) and \(\gamma_n \bm{T} (\bm{y}_i) \bm{M}_n^{(i)}(z)\) have operator norm strictly less than one, so that the corresponding inverses are well-defined and uniformly bounded. Since \(\bm{M}_n^{(i)}(z) = \tr_{d-r}^{(p)}(\bm{G}_n^{(i)}(z))\) we have 
\begin{equation}\label{eq:Mi-unif-Lip}
\|\bm{M}_n^{(i)}(z)\|_{\mathrm{op}} = \|\tr_{d-r}^{(p)}(\bm{G}_n^{(i)}(z))\|_{\mathrm{op}} \le \|\bm{G}_n^{(i)}(z)\|_{\mathrm{op}} \le (\Im z)^{-1},
\end{equation}
where the first inequality is the operator-norm contraction in Lemma~\ref{lem:partial_trace_bound}. According to~\eqref{eq:tilde_M} and since \(\|\bm{I}_p\otimes \bm{u}_i\|_{\mathrm{op}}=\|\bm{u}_i\|_2\), we obtain 
\begin{equation*}
\|\widetilde{\bm{M}}_{n,i}^{(i)}(z)\|_{\mathrm{op}} \le \frac{1}{d-r}\| \bm{u}_i\|^2\,\|\bm{G}_n^{(i)}(z)\|_{\mathrm{op}} \le \frac{1}{d-r} \|\bm{u}_i\|^2 (\Im z)^{-1}.
\end{equation*}
Define
\[
\kappa_n\coloneqq \max_{1\le i\le n} \frac{\|\bm{u}_i\|^2}{d-r}. 
\]
Since \(\bm{u}_i \sim \mathcal{N} (\bm{0}_{d-r},\bm{I}_{d-r})\) are i.i.d.\ and \(n /d \to \alpha\), chi-square concentration combined with Borel--Cantelli imply \(\kappa_n\to 1\) almost surely. In particular, \(\kappa_n\le 2\) almost surely for all sufficiently large \(n\). Hence, for all large \(n\),
\begin{equation}\label{eq:Mtilde-unif-Lip}
\max_{1\le i\le n}\|\widetilde{\bm{M}}_{n,i}^{(i)}(z)\|_{\mathrm{op}}\le \frac{2}{\Im z} \quad \mathrm{a.s.}
\end{equation}
From~\eqref{eq:Mi-unif-Lip} and~\eqref{eq:Mtilde-unif-Lip} and \(\|\bm{T} (\bm{y}_i)\|_{\mathrm{op}}\le C_T\) a.s.\ by Assumption~\ref{hyp:preprocessing}, for all large \(n\),
\[
\max_{1\le i\le n} \|\gamma_n \bm{T} (\bm{y}_i) \widetilde{\bm{M}}_{n,i}^{(i)}(z)\|_{\mathrm{op}} \le 2 \gamma_n C_T (\Im z)^{-1} , \qquad
\max_{1\le i\le n} \|\gamma_n \bm{T} (\bm{y}_i) \bm{M}_n^{(i)}(z)\|_{\mathrm{op}} \le \gamma_n C_T (\Im z)^{-1} \qquad \mathrm{a.s.}
\]
Since \(\gamma_n=(d-r)/n\to \alpha^{-1}\), we have \(\gamma_n\le 2/\alpha\) for sufficiently large \(n\). Choose \(\eta_0>0\) such that 
\[
c(\eta_0) \coloneqq\frac{4C_T}{\alpha \eta_0} < 1.
\]
Then for all sufficiently large \(n\) and all \(z\) with \(\Im z > \eta_0\), we have
\[
\max_{1\le i\le n} \|\gamma_n \bm{T} (\bm{y}_i) \widetilde{\bm{M}}_{n,i}^{(i)}(z)\|_{\mathrm{op}} \vee \max_{1\le i\le n} \|\gamma_n \bm{T} (\bm{y}_i) \bm{M}_n^{(i)}(z)\|_{\mathrm{op}} \le c(\eta_0) < 1 \quad \mathrm{a.s.}
\]
Hence, \(\bm{I}_p + \gamma_n \bm{T} (\bm{y}_i) \widetilde{\bm{M}}_{n,i}^{(i)}(z)\) and \(\bm{I}_p + \gamma_n \bm{T} (\bm{y}_i) \bm{M}_n^{(i)}(z)\) are invertible, and by the Neumann series, for all sufficiently large \(n\) and all \(z\) with \(\Im z > \eta_0\),
\begin{equation}\label{eq:bound-Lip}
\max_{1\le i\le n} \left\|(\bm{I}_p + \gamma_n \bm{T} (\bm{y}_i) \widetilde{\bm{M}}_{n,i}^{(i)}(z))^{-1}\right\|_{\mathrm{op}}
\vee
\max_{1\le i\le n}\left\|(\bm{I}_p + \gamma_n \bm{T} (\bm{y}_i) \bm{M}_n^{(i)}(z))^{-1} \right\|_{\mathrm{op}} \le \frac{1}{1-c(\eta_0)} \quad \mathrm{a.s.}
\end{equation}
By Lemma~\ref{lem:Lipschitz} and~\eqref{eq:bound-Lip}, almost surely, for all sufficiently large \(n\), all \(i \in [n]\), and all \(z\) with \(\Im z > \eta_0\), 
\begin{equation} \label{eq:Lipschitz-phi}
\left \|\Phi_{\bm{T} (\bm{y}_i),\gamma_n} \left (\widetilde{\bm{M}}_{n,i}^{(i)} (z)\right)-\Phi_{\bm{T} (\bm{y}_i),\gamma_n} \left(\bm{M}_n^{(i)}(z)\right) \right\|_{\mathrm{op}} \le C_{\Phi} (\eta_0) \left \|\widetilde{\bm{M}}_{n,i}^{(i)}(z)-\bm{M}_n^{(i)}(z) \right\|_{\mathrm{op}} \quad \mathrm{a.s.},
\end{equation}
where \(C_{\Phi} (\eta_0) = C_T (1-c(\eta_0))^{-2}\). 

It remains to prove the second estimate. Since  \(\|\bm{M}_n^{(i)}(z)\|_{\mathrm{op}}\le(\Im z)^{-1}\) and  \(\|\bm{M}_n(z)\|_{\mathrm{op}}\le(\Im z)^{-1}\), and \(\| \bm{T}(\bm{y})\|_{\mathrm{op}} \le C_T\) almost surely, for all sufficiently large \(n\) and almost every \(\bm{y}\),
\[
\left \| \gamma_n\bm{T} (\bm{y}) \bm{M}_n^{(i)}(z) \right\|_{\mathrm{op}} \vee \left\| \gamma_n\bm{T} (\bm{y}) \bm{M}_n(z) \right\|_{\mathrm{op}} \le \frac{2C_T}{\alpha\eta_0} \le c(\eta_0) < 1.
\]
Therefore,
\[
\left \| \left( \bm{I}_p+ \gamma_n \bm{T}(\bm{y})\bm{M}_n^{(i)}(z) \right)^{-1} \right \|_{\mathrm{op}} \vee \left \| \left( \bm{I}_p+ \gamma_n\bm{T}(\bm{y})\bm{M}_n(z) \right)^{-1} \right\|_{\mathrm{op}} \le \frac{1}{1-c(\eta_0)}
\]
almost surely with respect to \(\bm{y}\). Applying Lemma~\ref{lem:Lipschitz} pointwise in \(\bm{y}\), we obtain
\[
\left \| \Phi_{\bm{T}(\bm{y}),\gamma_n} \left( \bm{M}_n^{(i)}(z) \right) - \Phi_{\bm{T}(\bm{y}),\gamma_n} \left( \bm{M}_n(z) \right) \right \|_{\mathrm{op}} \le C_\Phi(\eta_0) \left| \bm{M}_n^{(i)}(z)-\bm{M}_n(z) \right|_{\mathrm{op}}
\]
for almost every \(\bm{y}\). Taking expectations and using \(\left \| \E_{\bm{y}}[\bm{A}(\bm{y})] \right\|_{\mathrm{op}} \le \E_{\bm{y}} \left[ \|\bm{A}(\bm{y})\|_{\mathrm{op}} \right]\), we obtain the desired result. This completes the proof.
\end{proof}

\begin{lem} \label{lem:quadratic_form} 
For every fixed \(z \in \C_+\),
\[
\varepsilon_n(z) \coloneqq \max_{1 \le i \leq n} \lVert \widetilde{\bm{M}}_{n,i}^{(i)} (z) - \bm{M}_n^{(i)}(z) \rVert_{\mathrm{op}}  \to 0, \quad \mathrm{a.s.}
\]
\end{lem}
\begin{proof}
Let \(\mathcal{F}_i \coloneqq \sigma(\{\bm{u}_j\}_{j \neq i}, \{\bm{y}_j\}_{j \neq i})\). By construction, \(\bm{G}_n^{(i)}(z)\) is \(\mathcal{F}_i\)-measurable, while \(\bm{u}_i \sim \mathcal{N}(\bm{0}, \bm{I}_{d-r})\) is independent of \(\mathcal{F}_i\). Conditionally on \(\mathcal{F}_i\), each block \(\bm{G}_{\mu \nu}^{(i)} (z)\) is therefore deterministic. We use the Hanson--Wright inequality; see, for example, \cite[Theorem~6.2.1]{vershyninbook}. Since the matrix \(\bm{G}_{\mu\nu}^{(i)}(z)\) is complex, we apply the real Hanson--Wright inequality separately to the symmetric parts of its real and imaginary parts. This only changes the universal constants and gives, for every \(t\ge0\),
\begin{equation} \label{eq:HW}
\mathbb{P} \left ( \left | \bm{u}_i^\top \bm{G}_{\mu \nu}^{(i)} (z) \bm{u}_i -  \Tr \bm{G}_{\mu \nu}^{(i)}(z)\right| \ge C \left( \lVert \bm{G}_{\mu \nu}^{(i)} (z)\rVert_{\mathrm{F}} \sqrt{t} + \lVert \bm{G}_{\mu \nu}^{(i)} (z)\rVert_{\mathrm{op}}t\right)  \Big | \mathcal{F}_i\right) \le 4 e^{-ct}.
\end{equation}
Indeed, symmetrization and taking real or imaginary parts do not increase either the Frobenius norm or the operator norm. Since \(\bm{P}_n^{(i)}\) is real symmetric, we have \(\|\bm{G}_n^{(i)}(z)\|_{\mathrm{op}} \le (\Im z)^{-1}\). Hence, for every \(\mu,\nu \in [p]\),
\[
\|\bm{G}_{\mu\nu}^{(i)}(z)\|_{\mathrm{op}} \le (\Im z)^{-1},
\qquad
\|\bm{G}_{\mu\nu}^{(i)}(z)\|_{\mathrm{F}} \le \sqrt{d-r}\,(\Im z)^{-1}.
\]
Using
\[
(\widetilde{\bm{M}}_{n,i}^{(i)}(z))_{\mu\nu} = \frac{1}{d-r} \bm{u}_i^\top \bm{G}_{\mu \nu}^{(i)} (z) \bm{u}_i \quad \textnormal{and} \quad (\bm{M}_n^{(i)}(z))_{\mu\nu}=\frac{1}{d-r} \Tr \bm{G}_{\mu \nu}^{(i)}(z),
\]
we deduce from~\eqref{eq:HW} that
\[
\mathbb{P} \left ( \left | (\widetilde{\bm{M}}_{n,i}^{(i)}(z))_{\mu\nu} - (\bm{M}_n^{(i)}(z))_{\mu\nu}\right| \ge \frac{C}{\Im z} \left( \sqrt{\frac{t}{d-r}} + \frac{t}{d-r} \right)  \Big | \mathcal{F}_i\right) \le 4 e^{-ct}.
\]
Taking expectations, then a union bound over \(i\in[n]\) and \((\mu,\nu)\in[p]^2\), and using \(\|\bm{A}\|_{\mathrm{op}}\le p\max_{\mu,\nu}|A_{\mu\nu}|\), we obtain
\[
\mathbb{P} \left ( \varepsilon_n(z) \ge \frac{Cp}{\Im z} \left( \sqrt{\frac{t}{d-r}} + \frac{t}{d-r} \right) \right) \le 4np^2e^{-ct}.
\]
Choose \(t=K\log n\) with \(cK>2\). Since \(d-r\asymp n\), the deterministic threshold tends to zero, while \(4np^2e^{-ct}=4p^2n^{1-cK}\) is summable. The conclusion follows from the Borel--Cantelli lemma.
\end{proof}

\begin{lem} \label{lem:leave-one-out-stability} 
For every fixed \(z \in \C_+\),
\[
\delta_n(z) \coloneqq \max_{1 \le i \leq n} \lVert \bm{M}_n^{(i)}(z) - \bm{M}_n (z)\rVert_{\mathrm{op}}  \to 0, \quad \mathrm{a.s.}
\]
\end{lem}
\begin{proof}
Fix \(i\in[n]\). By the resolvent identity,
\begin{equation} \label{eq:res_id}
\bm{G}_n^{(i)}(z) - \bm{G}_n(z)=\bm{G}_n^{(i)}(z)(\bm{P}_n-\bm{P}_n^{(i)})\bm{G}_n(z).
\end{equation}
Since \(\|\bm{G}_n(z)\|_{\mathrm{op}},\|\bm{G}_n^{(i)}(z)\|_{\mathrm{op}}\le(\Im z)^{-1}\) and
\[
\bm{P}_n-\bm{P}_n^{(i)}= \frac{1}{n} (\bm{I}_p\otimes\bm{u}_i)\bm{T}(\bm{y}_i)(\bm{I}_p\otimes\bm{u}_i)^\top,
\]
Assumption~\ref{hyp:preprocessing} gives
\begin{equation}\label{eq:norm_res_id}
\|\bm{G}_n^{(i)}(z)-\bm{G}_n(z)\|_{\mathrm{op}}
\le \frac{C_T}{n(\Im z)^2}\|\bm{u}_i\|_2^2
\qquad\mathrm{a.s.}
\end{equation}
Moreover,~\eqref{eq:res_id} implies
\[
\operatorname{rank}(\bm{G}_n^{(i)}(z)-\bm{G}_n(z))
\le \operatorname{rank}(\bm{P}_n-\bm{P}_n^{(i)})\le p.
\]
Applying Lemma~\ref{lem:partial_trace_bound} with \(m=d-r\), we obtain
\[
\|\bm{M}_n^{(i)}(z)-\bm{M}_n(z)\|_{\mathrm{op}}
\le \frac{p}{d-r}\|\bm{G}_n^{(i)}(z)-\bm{G}_n(z)\|_{\mathrm{op}}
\le \frac{pC_T}{(d-r)n(\Im z)^2}\|\bm{u}_i\|_2^2.
\]
Therefore,
\[
\delta_n(z) \le \frac{pC_T}{n(\Im z)^2} \max_{1\le i\le n}\frac{\|\bm{u}_i\|_2^2}{d-r}.
\]
As in the proof of Lemma~\ref{lem:stability-phi}, the maximum on the right is eventually bounded almost surely. Hence \(\delta_n(z)\to0\) almost surely.
\end{proof}

We next prove the self-averaging estimate using McDiarmid's bounded-differences inequality; see, for example,~\cite[Theorem 2.9.1]{vershyninbook}.

\begin{lem} \label{lem:self_avg}
Fix \(\eta_0>0\) as in Lemma~\ref{lem:stability-phi}. Then, for every fixed \(z \in \C_+\) with \(\Im z > \eta_0\),
\[
\left \| \frac{1}{n} \sum_{i=1}^n \Phi_{\bm{T} (\bm{y}_i), \gamma_n} (\bm{M}_n^{(i)} (z)) - \frac{1}{n} \sum_{i=1}^n  \E_{\bm{y}} \left [  \Phi_{\bm{T}(\bm{y}), \gamma_n} (\bm{M}_n^{(i)} (z))\right ] \right \|_{\mathrm{op}} \to 0 \quad \textnormal{a.s.}
\]
\end{lem}

\begin{proof}
Set
\[
S_n (z) \coloneqq \frac{1}{n} \sum_{i=1}^n \left ( \Phi_{\bm{T}_i, \gamma_n} (\bm{M}_n^{(i)} (z)) - \E_{\bm{y}} \left [  \Phi_{\bm{T}(\bm{y}), \gamma_n} (\bm{M}_n^{(i)} (z))\right ]\right),
\]
where \(\bm{T}_i=\bm{T}(\bm{y}_i)\). Since \(p\) is fixed and \(\| S_n (z) \|_{\mathrm{op}} \le p \max_{\mu,\nu \in [p]} |(S_n (z))_{\mu \nu}|\), it suffices to prove almost-sure convergence to zero for each fixed entry. Fix \(\mu,\nu\in[p]\), and define
\[
f_n(\bm{Y}) \coloneqq (S_n(z))_{\mu\nu}, 
\]
where \(\bm{Y} = [\bm{y}_1,\ldots,\bm{y}_n]\). We condition throughout on \((\bm{u}_1,\ldots,\bm{u}_n)\). For every \(i\), the matrix \(\bm{M}_n^{(i)}(z)\), is independent of \(\bm{y}_i\). Therefore,
\[
\E \left[ \Phi_{\bm{T}_i,\gamma_n}(\bm{M}_n^{(i)}(z)) \,\middle|\, \sigma(\{\bm{y}_j\}_{j\neq i},\{\bm{u}_j\}_{j\le n}) \right] = \E_{\bm{y}} \left[ \Phi_{\bm{T}(\bm{y}),\gamma_n}(\bm{M}_n^{(i)}(z)) \right].
\]
Taking the \((\mu,\nu)\)-entry, summing over \(i\), and using the tower property gives
\[
\E \left[f_n(\bm{Y}) \mid \bm{u}_1,\ldots,\bm{u}_n\right]=0.
\]

We next prove a bounded-difference estimate. Let
\[
\bm{Y}^{(k)} = [\bm{y}_1,\ldots,\bm{y}_{k-1},\bm{y}_k', \bm{y}_{k+1},\ldots,\bm{y}_n],
\]
where \(\bm{y}_k'\) is an independent copy of \(\bm{y}_k\), and set \(\bm{T}_k' = \bm{T} (\bm{y}_k')\). For \(i\in[n]\), define
\[
\bm{D}_i (\bm{Y}) \coloneqq \Phi_{\bm{T}_i,\gamma_n}(\bm{M}_n^{(i)}(z)) - \E_{\bm{y}} \left[ \Phi_{\bm{T} (\bm{y}), \gamma_n}(\bm{M}_n^{(i)}(z)) \right].
\]
Thus
\[
\begin{split}
\left | f_n(\bm{Y})-f_n(\bm{Y}^{(k)} )\right| & \le \frac{1}{n} \sum_{i=1}^n  \left | \left ( \bm{D}_i(\bm{Y}) - \bm{D}_i(\bm{Y}^{(k)}) \right)_{\mu \nu} \right| \\
& \le \frac{1}{n} \left | \left ( \bm{D}_k(\bm{Y}) - \bm{D}_k(\bm{Y}^{(k)}) \right)_{\mu \nu} \right| +  \frac{1}{n} \sum_{i \neq k} \left | \left ( \bm{D}_i(\bm{Y}) - \bm{D}_i(\bm{Y}^{(k)}) \right)_{\mu \nu} \right| \\
& \le \frac{1}{n} \left \|  \bm{D}_k(\bm{Y}) - \bm{D}_k(\bm{Y}^{(k)})\right \|_{\mathrm{op}} +  \frac{1}{n} \sum_{i \neq k} \left \|  \bm{D}_i(\bm{Y}) - \bm{D}_i(\bm{Y}^{(k)})  \right \|_{\mathrm{op}}.
\end{split}
\]
We first consider the term \(i=k\). Since \(\bm{M}_n^{(k)}(z)\) does not depend on \(\bm{y}_k\), we have \(\bm{M}_n^{(k)}(\bm{Y}; z) = \bm{M}_n^{(k)} (\bm{Y}^{(k)};z)\). The expectation terms in \(\bm{D}_k (\bm{Y})\) and \(\bm{D}_k (\bm{Y}^{(k)})\) therefore cancel, and hence
\[
\bm{D}_k (\bm{Y}) -\bm{D}_k (\bm{Y}^{(k)}) = \Phi_{\bm{T}_k,\gamma_n} \left (\bm{M}_n^{(k)}(z) \right) - \Phi_{\bm{T}_k',\gamma_n} \left (\bm{M}_n^{(k)}(z)\right).
\]
For all sufficiently large \(n\), \(\gamma_n\le\frac{2}{\alpha}\). Moreover, \( \|\bm{M}_n^{(k)}(z)\|_{\mathrm{op}} \le (\Im z)^{-1} \le \eta_0^{-1}\). Thus, for every matrix \(\bm{T}\) satisfying \(\|\bm{T}\|_{\mathrm{op}} \le C_T\),
\[
\|\gamma_n \bm{T} \bm{M}_n^{(k)}(z)\|_{\mathrm{op}} \le \frac{2C_T}{\alpha\eta_0} \eqqcolon c_0<1.
\]
It then follows that
\[
\left \| \left( \bm{I}_p+\gamma_n\bm{T} \bm{M}_n^{(k)}(z) \right)^{-1} \right\|_{\mathrm{op}} \le \frac{1}{1-c_0}.
\]
Consequently,
\[
\left \| \Phi_{\bm{T},\gamma_n} \left (\bm{M}_n^{(k)}(z) \right) \right \|_{\mathrm{op}} \le \frac{C_T}{\eta_0(1-c_0)}.
\]
Applying this estimate to both \(\bm{T}_k\) and \(\bm{T}_k'\), we obtain
\[
\left \| \bm{D}_k (\bm{Y}) - \bm{D}_k(\bm{Y}^{(k)}) \right \|_{\mathrm{op}} \le \frac{2C_T}{\eta_0(1-c_0)} \eqqcolon C_0.
\]
After multiplication by the prefactor \(1/n\), the contribution of the term \(i=k\) to \(\lvert f_n(\bm{Y}) - f_n(\bm{Y}^{(k)})\rvert\) is therefore bounded by \(C_0/n\).

Now consider \(i\neq k\). We write
\[
\bm{P}_n^{(i)} (\bm{Y}) \coloneqq \frac{1}{n} \sum_{j \neq i} \bm{T}_j \otimes \bm{u}_j \bm{u}_j^\top, \qquad \bm{P}_n^{(i)} (\bm{Y}^{(k)}) \coloneqq \frac{1}{n} \sum_{j \neq i,k} \bm{T}_j \otimes \bm{u}_j \bm{u}_j^\top + \frac{1}{n}  \bm{T}_k' \otimes \bm{u}_k \bm{u}_k^\top.
\]
Thus
\[
\bm{P}_n^{(i)} (\bm{Y})  - \bm{P}_n^{(i)} (\bm{Y}^{(k)}) = \frac{1}{n} (\bm{T}_k-\bm{T}_k')\otimes \bm{u}_k \bm{u}_k^\top .
\]
Since \(\bm{T}_k\) and \(\bm{T}_k'\) are \(p\times p\) matrices and \(\|\bm{T}_k\|,\|\bm{T}_k'\|\le C_T\), we have
\[
\operatorname{rank} \left (\bm{P}_n^{(i)} (\bm{Y})  - \bm{P}_n^{(i)} (\bm{Y}^{(k)}) \right) \le p, \qquad \left \| \bm{P}_n^{(i)} (\bm{Y})  - \bm{P}_n^{(i)} (\bm{Y}^{(k)}) \right \|_{\mathrm{op}} \le \frac{2C_T}{n}\|\bm{u}_k\|_2^2 .
\]
Let \(\bm{G}_n^{(i)}(\bm{Y};z)\) and \(\bm{G}_n^{(i)}(\bm{Y}^{(k)};z)\) denote the corresponding leave-one-out resolvents. By the resolvent identity,
\[
\bm{G}_n^{(i)}(\bm{Y};z) - \bm{G}_n^{(i)}(\bm{Y}^{(k)};z) = - \bm{G}_n^{(i)}(\bm{Y};z) \bm{A}_k \bm{G}_n^{(i)}(\bm{Y}^{(k)};z),
\]
where \(\bm{A}_k \coloneqq \bm{P}_n^{(i)} (\bm{Y})  - \bm{P}_n^{(i)} (\bm{Y}^{(k)})\). Since \(\Im z>\eta_0\),
\[
\|\bm{G}_n^{(i)}(\bm{Y};z)\|_{\mathrm{op}}, \, \|\bm{G}_n^{(i)}(\bm{Y}^{(k)};z)\|_{\mathrm{op}} \le \frac{1}{\Im z} \le \frac{1}{\eta_0}.
\]
Therefore, 
\[
\left \| \bm{G}_n^{(i)}(\bm{Y};z) - \bm{G}_n^{(i)}(\bm{Y}^{(k)};z) \right \|_{\mathrm{op}} \le \frac{2C_T}{n \eta_0^2} \|\bm{u}_k\|_2^2.
\]
Applying Lemma~\ref{lem:partial_trace_bound}, we find
\[
\begin{split}
\left\| \bm{M}_n^{(i)}(\bm{Y};z) - \bm{M}_n^{(i)}(\bm{Y}^{(k)};z) \right\|_{\mathrm{op}} & \le \frac{1}{d-r} \operatorname{rank} \left ( \bm{G}_n^{(i)}(\bm{Y};z) - \bm{G}_n^{(i)}(\bm{Y}^{(k)};z) \right) \left \| \bm{G}_n^{(i)}(\bm{Y};z) - \bm{G}_n^{(i)}(\bm{Y}^{(k)};z) \right \|_{\mathrm{op}} \\
& \le \frac{C}{(d-r) n \eta_0^2}\|\bm{u}_k\|_2^2.
\end{split}
\]
Define the event
\[
\Omega_n \coloneqq \left \{ \max_{1\le k\le n} \frac{\|\bm{u}_k\|_2^2}{d-r} \le C_u \right\},
\]
where \(C_u>0\) is chosen large enough. Since \(\bm{u}_k\sim\mathcal{N}(0,\bm{I}_{d-r})\) and \(d-r \asymp n\), standard \(\chi^2\) concentration gives
\[
\sum_{n=1}^\infty \mathbb{P} (\Omega_n^c)<\infty.
\]
Hence \(\Omega_n\) holds eventually almost surely by Borel--Cantelli. On \(\Omega_n\), for all \(i\neq k\),
\[
\left\| \bm{M}_n^{(i)}(\bm{Y};z) - \bm{M}_n^{(i)}(\bm{Y}^{(k)};z) \right\|_{\mathrm{op}} \le \frac{C}{n \eta_0^2}.
\]
Using the Lipschitz estimate for \(\Phi\) from Lemma~\ref{lem:Lipschitz}, we get, for \(i\neq k\),
\[
\begin{split}
\|\bm{D}_i(\bm{Y})-\bm{D}_i(\bm{Y}^{(k)})\|_{\mathrm{op}} & \le \left \| \Phi_{\bm{T}_i,\gamma_n}(\bm{M}_n^{(i)}(\bm{Y};z)) - \Phi_{\bm{T}_i,\gamma_n}(\bm{M}_n^{(i)}(\bm{Y}^{(k)};z)) \right\|_{\mathrm{op}} \\
& \quad+ \left \| \E_{\bm{y}} \left[ \Phi_{\bm{T}(\bm{y}),\gamma_n} (\bm{M}_n^{(i)}(\bm{Y};z))  \right ]- \E_{\bm{y}} \left [ \Phi_{\bm{T}(\bm{y}),\gamma_n} (\bm{M}_n^{(i)}(\bm{Y}^{(k)};z)) \right ] \right \|_{\mathrm{op}} \\
& \le 2C_\Phi(\eta_0) \left \| \bm{M}_n^{(i)}(\bm{Y};z) - \bm{M}_n^{(i)}(\bm{Y}^{(k)};z) \right\|_{\mathrm{op}} \\
& \le \frac{C}{n}.
\end{split}
\]
Thus, on \(\Omega_n\),
\[
|f_n(\bm{Y})-f_n(\bm{Y}^{(k)})|   \le \frac{1}{n} \left( C_\Phi(\eta_0) + \sum_{i\neq k}\frac{C_0}{n} \right)  \le \frac{C}{n}.
\]

All the preceding bounded-difference estimates are deterministic on \(\Omega_n\) and uniform over the possible values of the coordinates \(\bm{y}_1,\ldots,\bm{y}_n\) for which \(\|\bm{T}(\bm{y}_i)\|_{\mathrm{op}}\le C_T\). Thus, conditional on \((\bm{u}_1,\ldots,\bm{u}_n)\), on the event \(\Omega_n\), the function \(f_n\) satisfies the bounded-difference condition with constants \(c_k=C/n\) for every \(k \in [n]\). Since \(f_n\) is complex-valued, we apply McDiarmid's inequality (see e.g.~\cite[Theorem 2.9.1]{vershyninbook}) separately to its real and imaginary parts. Using \(\E[f_n(\bm{Y})\mid \bm{u}_1,\ldots,\bm{u}_n]=0\) and \(\sum_{k=1}^n c_k^2\le \frac{C}{n}\), we obtain, on \(\Omega_n\),
\[
\mathbb{P} \left( |f_n(\bm{Y})|\ge t \,\middle|\, \bm{u}_1,\ldots,\bm{u}_n \right)  \le 4\exp(-c n t^2).
\]
Since \(\Omega_n\) is measurable with respect to \((\bm{u}_1,\ldots,\bm{u}_n)\), integrating the conditional estimate gives
\[
\mathbb{P} \left( |f_n(\bm{Y})|\ge t, \Omega_n \right)  \le 4\exp(-c n t^2).
\]
Choose \(t_n=K\sqrt{\log n/n}\) where \(K>0\) is sufficiently large so that \(cK^2 >1\). Then
\[
\sum_{n=1}^\infty \mathbb{P} \left( |f_n(\bm{Y})|\ge t_n, \Omega_n \right) <\infty.
\]
Since \(\sum_{n=1}^\infty \mathbb{P}(\Omega_n^\mathrm{c})<\infty\), the Borel--Cantelli lemma yields
\[
f_n(\bm{Y})\to0 \qquad\text{a.s.}
\]
This holds for each of the finitely many entries\((\mu,\nu)\in[p]^2\). Intersecting the corresponding probability-one events, we conclude that \(\|S_n(z)\|_{\mathrm{op}}\) vanishes almost surely.
\end{proof}

We are now in a position to conclude the proof of Proposition~\ref{prop:bulk}.

\begin{proof}[Proof of Proposition~\ref{prop:bulk}]
Choose \(\eta_0>0\) as in Lemma~\ref{lem:stability-phi}. Define
\[
\mathcal{H}_{\eta_0} \coloneqq \{z \in\C_+ \colon \Im z>\eta_0\}.
\]
Let
\[
\Omega_T \coloneqq \bigcap_{n\geq1} \bigcap_{1\leq i\leq n} \left \{ \|\bm{T}(\bm{y}_i)\|_{\mathrm{op}} \leq C_T \right\}.
\]
By Assumption~\ref{hyp:preprocessing} and countable subadditivity, \(\mathbb{P}(\Omega_T)=1\). Let also
\[
\Omega_u \coloneqq \left \{ \max_{1\leq i\leq n} \frac{ \|\bm{u}_i\|_2^2}{d-r}\leq 2 \text{ for all sufficiently large }n \right \}.
\]
The chi-square concentration estimate and the Borel--Cantelli lemma imply \(\mathbb{P}(\Omega_u)=1\). The proof of Lemma~\ref{lem:stability-phi} shows that on
\(\Omega_T\cap\Omega_u\), its conclusions hold, for all sufficiently large \(n\), simultaneously for every \(i\in[n]\) and every \(z\in\mathcal{H}_{\eta_0}\). Let \(\mathcal{D} \subset \mathcal{H}_{\eta_0}\) be a countable dense subset. For every \(z\in\mathcal{D}\),let \(\Omega_z^{\mathrm{qf}}\), \(\Omega_z^{\mathrm{loo}}\), and \(\Omega_z^{\mathrm{sa}}\) denote probability-one events on which the conclusions of Lemmas~\ref{lem:quadratic_form},~\ref{lem:leave-one-out-stability}, and~\ref{lem:self_avg}, respectively, hold at \(z\). Define
\[ 
\Omega_0 \coloneqq \Omega_T\cap\Omega_u \cap \bigcap_{z\in\mathcal{D}} \left( \Omega_z^{\mathrm{qf}} \cap \Omega_z^{\mathrm{loo}} \cap \Omega_z^{\mathrm{sa}} \right). 
\]
Since \(\mathcal{D}\) is countable, \(\mathbb{P}(\Omega_0)=1\). In the remainder of the proof, we fix an outcome in \(\Omega_0\). All subsequent limits and \(o(1)\) terms are understood for this fixed outcome.

Fix \(z\in\mathcal{D}\). By Lemmas~\ref{lem:stability-phi} and~\ref{lem:quadratic_form},
\[
\| \mathcal{E}_{1,n} (z) \|_{\mathrm{op}} \le C_{\Phi} (\eta_0) \frac{1}{n} \sum_{i=1}^n \|\widetilde{\bm{M}}_{n,i}^{(i)} (z) - \bm{M}_n^{(i)} (z)\|_{\mathrm{op}} \le C_{\Phi} (\eta_0) \varepsilon_n (z) = o(1).
\]
Moreover, Lemma~\ref{lem:self_avg} gives
\[
\| \mathcal{E}_{2,n} (z) \|_{\mathrm{op}}\to 0.
\]
Finally, Lemmas~\ref{lem:stability-phi} and~\ref{lem:leave-one-out-stability} yield
\[
\left \| \mathcal{E}_{3,n} (z) \right\|_{\mathrm{op}} \le C_{\Phi} (\eta_0) \frac{1}{n} \sum_{i=1}^n \|\bm{M}_n^{(i)}(z)-\bm{M}_n(z)\|_{\mathrm{op}}  \le C_{\Phi} (\eta_0) \delta_n(z) = o(1).
\]
Consequently, combining these estimates in~\eqref{eq:approximate_MDE}, for every \(z\in\mathcal{D}\), we obtain
\begin{equation}\label{eq:phi-expansion2}
z \bm{M}_n(z) + \bm{I}_p = \E_{\bm{y}} \left[ \Phi_{\bm{T}(\bm{y}),\gamma_n}\left(\bm{M}_n(z)\right) \right] + o(1) .
\end{equation}

We next identify all subsequential limits. For every \(n\), the function \(\bm{M}_n(z) = \tr_{d-r}^{(p)} (\bm{G}_n (z))\) is analytic on \(\C_+\). Since \(\bm{P}_n\) is Hermitian, \(\|\bm{G}_n(z)\|_{\mathrm{op}} \le (\Im z )^{-1}\), and hence \(\|\bm{M}_n(z)\|_{\mathrm{op}} \le (\Im z)^{-1}\). Thus, the family \(\{\bm{M}_n(z)\}_n\) is locally uniformly bounded on \(\mathbb{C}_+\). By Montel's theorem, applied entrywise, every subsequence admits a further subsequence, still denoted by \(n_k\), and an analytic function \(\widetilde{\bm{M}} \colon \C_+ \to \C^{p \times p}\) such that \(\bm{M}_{n_k} (z) \to \widetilde{\bm{M}} (z)\) locally uniformly on \(\C_+\). In particular, \(\| \widetilde{\bm{M}} (z) \|_{\mathrm{op}} \le (\Im z)^{-1}\).

Since \(\gamma_{n_k} \to \alpha^{-1}\) by Assumption~\ref{hyp:prop_limit}, for all sufficiently large \(k\), \(\gamma_{n_k}\leq 2\alpha^{-1}\). Therefore, for \(z \in \mathcal{D}\),
\[
\| \gamma_{n_k} \bm{T} (\bm{y}) \bm{M}_{n_k} (z)\|_{\mathrm{op}} \le \frac{2C_T}{\alpha \Im z} \le \frac{2 C_T}{\alpha \eta_0} < 1, \quad \mathbb{P}_{\bm{y}}\text{-a.s.}
\]
for all sufficiently large \(k\). The corresponding inverses are therefore uniformly bounded. Passing to the limit in\eqref{eq:phi-expansion2}, using dominated convergence, gives
\begin{equation} \label{eq:limiting-phi-equation-dense} 
z \widetilde{\bm{M}}(z)+\bm{I}_p = \E_{\bm{y}} \left[ \Phi_{\bm{T}(\bm{y}),\alpha^{-1}} \left (\widetilde{\bm{M}}(z)\right) \right], \quad z\in\mathcal{D}. 
\end{equation}
Define, on \(\mathcal{H}_{\eta_0}\), the matrix-valued function 
\[ 
\bm{F}(z) \coloneqq z \widetilde{\bm{M}}(z)+\bm{I}_p - \E_{\bm{y}} \left[ \Phi_{\bm{T}(\bm{y}),\alpha^{-1}} \left(\widetilde{\bm{M}}(z)\right) \right]. 
\] 
The bound 
\[ 
\left\| \frac{1}{\alpha} \bm{T}(\bm{y})\widetilde{\bm{M}}(z) \right\|_{\mathrm{op}} \leq \frac{C_T}{\alpha\eta_0} <1 
\] 
shows that the expectation in the definition of \(\bm{F}\) is well-defined and analytic on \(\mathcal{H}_{\eta_0}\). Hence \(\bm{F}\) is analytic there. By~\eqref{eq:limiting-phi-equation-dense}, \( \bm{F}(z)=\bm{0}\) for every \(z\in\mathcal{D}\). Since \(\mathcal{D}\) is dense in \(\mathcal{H}_{\eta_0}\), the identity theorem, applied entrywise, implies \(\bm{F}(z)=\bm{0}\) for every \(z \in \mathcal{H}_{\eta_0}\). Thus 
\begin{equation} \label{eq:limiting-phi-equation} 
z\widetilde{\bm{M}}(z)+\bm{I}_p = \E_{\bm{y}} \left[ \Phi_{\bm{T}(\bm{y}),\alpha^{-1}} \left(\widetilde{\bm{M}}(z)\right) \right] ,
\end{equation} 
for every \(z\in\mathcal{H}_{\eta_0}\). Using \(\bm{X} \left( \bm{I}_p+\bm{T}\bm{X} \right)^{-1} = \left( \bm{I}_p+\bm{X}\bm{T} \right)^{-1} \bm{X}\),~\eqref{eq:limiting-phi-equation} becomes 
\[ 
\bm{I}_p = \left[ -z\bm{I}_p + \E_{\bm{y}} \left[ \bm{T}(\bm{y}) \left( \bm{I}_p + \frac{1}{\alpha} \widetilde{\bm{M}}(z)\bm{T}(\bm{y}) \right)^{-1} \right] \right] \widetilde{\bm{M}}(z). 
\] 
Therefore \(\widetilde{\bm{M}}(z)\) has a left inverse. Since it is a square matrix, it is invertible, and 
\[ 
\widetilde{\bm{M}}(z)^{-1} = -z\bm{I}_p + \E_{\bm{y}} \left[ \bm{T}(\bm{y}) \left( \bm{I}_p + \frac{1}{\alpha} \widetilde{\bm{M}}(z)\bm{T}(\bm{y}) \right)^{-1} \right]. 
\] 
Hence \(\widetilde{\bm{M}}\) satisfies the matrix Dyson equation~\eqref{eq:MDE} on \(\mathcal{H}_{\eta_0}\).

We next verify that \(\widetilde{\bm{M}}\) belongs to the normalized matrix-valued Herglotz class. Since 
\[ 
\Im\bm{G}_n(z) = (\Im z)\bm{G}_n(z)\bm{G}_n(z)^\ast \succeq0 
\] 
and the normalized partial trace preserves positive semidefiniteness, \(\Im\bm{M}_n(z)\succeq0\). Passing to the locally uniform limit gives \( \Im\widetilde{\bm{M}}(z)\succeq0\) for every \(z\in\C_+\). It remains to prove the normalization. For sufficiently large \(\eta\), 
\[ 
\left\| \frac{1}{\alpha} \bm{T}(\bm{y})\widetilde{\bm{M}}(i\eta) \right\|_{\mathrm{op}} \leq \frac{C_T}{\alpha\eta} \leq \frac{1}{2} \quad \mathbb{P}_{\bm{y}}\text{-a.s.}
\] 
Therefore, 
\[ 
\left\| \left( \bm{I}_p + \frac{1}{\alpha} \bm{T}(\bm{y})\widetilde{\bm{M}}(i\eta) \right)^{-1} \right\|_{\mathrm{op}} \leq2 \quad \mathbb{P}_{\bm{y}}\text{-a.s.} 
\] 
Using~\eqref{eq:limiting-phi-equation} and the inequality \(\|\E_{\bm{y}} [\bm{A} (\bm{y})]\|_{\mathrm{op}} \leq\E_{\bm{y}}[ \|\bm{A}(\bm{y})\|_{\mathrm{op}}]\), we obtain
\[ 
\left\| i\eta\widetilde{\bm{M}}(i\eta)+\bm{I}_p \right\|_{\mathrm{op}} \leq \E_{\bm{y}} \left[ \|\bm{T}(\bm{y})\|_{\mathrm{op}} \|\widetilde{\bm{M}}(i\eta)\|_{\mathrm{op}} \left\| \left( \bm{I}_p + \frac{1}{\alpha} \bm{T}(\bm{y})\widetilde{\bm{M}}(i\eta) \right)^{-1} \right\|_{\mathrm{op}} \right] \leq \frac{2C_T}{\eta}.
\] 
Hence, \(\lim_{\eta \to \infty} \| i\eta\widetilde{\bm{M}}(i\eta)+\bm{I}_p \|_{\mathrm{op}} = 0\) and therefore \(\lim_{\eta\to\infty} i\eta\widetilde{\bm{M}}(i\eta) = -\bm{I}_p\). Thus \(\widetilde{\bm{M}}\) belongs to the normalized matrix-valued Herglotz class. By Corollary~\ref{cor:MDE_recognition}, \(\widetilde{\bm{M}}\) coincides with the unique normalized Herglotz solution \(\bm{M}_\alpha\) of~\eqref{eq:MDE} on all of \(\C_+\). We have shown that every subsequence of \((\bm{M}_n)\) admits a further subsequence converging locally uniformly to the same deterministic function \(\bm{M}_\alpha\). Therefore, 
\[ 
\bm{M}_n \to \bm{M}_\alpha 
\] 
locally uniformly on \(\C_+\). Finally, Proposition~\ref{prop:existence_uniqueness} gives a probability measure \(\mu_\alpha\) whose Stieltjes transform is 
\[ 
m_{\mu_\alpha}(z) = \frac{1}{p}\Tr\bm{M}_\alpha(z). 
\] 
For every \(z\in\C_+\), 
\[ 
m_n(z) = \frac{1}{p}\Tr \bm{M}_n(z) \to m_{\mu_\alpha}(z). 
\] 
Since \(m_n\) is the Stieltjes transform of \(\hat{\mu}_{\bm{P}_n}\), the continuity theorem for Stieltjes transforms yields
\[ 
\hat{\mu}_{\bm{P}_n} \to \mu_\alpha \quad \text{a.s.} 
\] 
Moreover, Lemma~\ref{lem:continuation_off_support} shows that \(\mu_\alpha\) is compactly supported.
\end{proof}

We now prove the first part of Theorem~\ref{thm:bulk}. The proof of the second part is postponed to the end of the next section.

\begin{proof}[Proof of the first assertion of Theorem~\ref{thm:bulk}]
As shown in Subsection~\ref{subsection:preliminaries}, \(\bm{D}_n\) and \(\bm{D}_n'\) have the same eigenvalues. It therefore suffices to compare \(\bm{D}_n'\) with the block embedding of \(\bm{P}_n\). Define the \(pd\times pd\) matrix
\[
\bm{P}'_n \coloneqq
\begin{pmatrix}
\bm{0}_{pr\times pr} & \bm{0} \\
\bm{0} & \bm{P}_n
\end{pmatrix}.
\]
The eigenvalues of \(\bm{P}_n'\) consist of the eigenvalues of \(\bm{P}_n\), together with \(pr\) additional zero eigenvalues. Thus
\[
\hat{\mu}_{\bm{P}'_n} = \left( 1-\frac{r}{d} \right) \hat{\mu}_{\bm{P}_n} + \frac{r}{d}\delta_0.
\]
Since \(r\) remains fixed and \(\hat{\mu}_{\bm{P}_n}\) converges weakly almost surely to \(\mu_\alpha\) by Proposition~\ref{prop:bulk}, it follows that 
\[ 
\hat{\mu}_{\bm{P}'_n} \Rightarrow \mu_\alpha \quad \mathrm{a.s.} 
\]
Moreover, 
\[
\bm{D}_n' - \bm{P}'_n =
\begin{pmatrix}
\bm{A}_n & \bm{Q}_n^\top \\
\bm{Q}_n & \bm{0}
\end{pmatrix},
\]
and hence \(\mathrm{rank} (\bm{D}_n' - \bm{P}'_n) \leq 2pr\). Let \(d_{\mathrm{KS}}\) denote the Kolmogorov--Smirnov distance between probability measures. By the rank inequality for empirical spectral distributions (see e.g.~\cite[Theorem~A.43]{bai2010}),
\[
d_{\mathrm{KS}} \left( \hat{\mu}_{\bm{D}'_n}, \hat{\mu}_{\bm{P}'_n} \right) \leq \frac{1}{pd} \mathrm{rank} (\bm{D}_n' - \bm{P}'_n) \le \frac{2r}{d} ,
\]
which vanishes as \(d \to \infty\). Let \(d_{\mathrm{L}}\) denote the L\'evy distance. By e.g.~\cite[Theorem~B.18]{bai2010}, we have
\[
d_{\mathrm{L}}\left( \hat{\mu}_{\bm{D}'_n}, \hat{\mu}_{\bm{P}'_n} \right) \le d_{\mathrm{KS}} \left( \hat{\mu}_{\bm{D}'_n}, \hat{\mu}_{\bm{P}'_n} \right).
\]
The triangle inequality therefore gives 
\[ 
d_{\mathrm{L}} ( \hat{\mu}_{\bm{D}'_n}, \mu_\alpha) \leq d_{\mathrm{L}} ( \hat{\mu}_{\bm{D}'_n}, \hat{\mu}_{\bm{P}'_n}) + d_{\mathrm{L}} ( \hat{\mu}_{\bm{P}'_n}, \mu_\alpha) \le \frac{2r}{d} + d_{\mathrm{L}} ( \hat{\mu}_{\bm{P}'_n}, \mu_\alpha),
\] 
and since the L\'evy distance metrizes weak convergence on \(\R\), \(d_{\mathrm{L}} (\hat{\mu}_{\bm{P}'_n}, \mu_\alpha) \to 0\) almost surely. Hence,
\[ 
\hat{\mu}_{\bm{D}'_n} \Rightarrow \mu_\alpha \quad\text{a.s.} 
\] 
The same conclusion holds for \(\hat{\mu}_{\bm{D}_n}\) since \(\bm{D}_n\) and \(\bm{D}'_n\) have the same eigenvalues.
\end{proof}

\section{The upper spectral edge} \label{section:upper_edge}

The aim of this section is to prove Proposition~\ref{prop:upper_edge}, namely that \(\bm{P}_n\) has no eigenvalues asymptotically above the upper edge of its limiting bulk spectrum. We proceed in two steps. First, we compare \(\bm{P}_n\) with an associated free Gaussian operator by applying the spectral comparison results of Bandeira, Cipolloni, Schr\"{o}der, and van Handel~\cite{bandeira2026}, conditionally on the preprocessing variables. Second, we derive a finite-dimensional variational upper bound for the spectral edge of the free model using a Fock-space argument inspired by Lehner~\cite{lehner1999}, and evaluate this bound at a solution of the limiting matrix Dyson equation.

\subsection{Comparison with an associated free Gaussian model}

If the preprocessing map is positive semidefinite, then \(\bm{P}_n\) is a sample covariance matrix, and the quadratic spectral comparison theorem~\cite[Theorem~2.5]{bandeira2026} can be applied directly to the associated Gaussian data matrix. Under Assumption~\ref{hyp:preprocessing}, however, the matrix \(\bm{T}(\bm{y})\) may have both positive and negative eigenvalues. In that case, \(\bm{P}_n\) has a signed Gram matrix representation \(\bm{P}_n = \bm{X}_n \bm{J}_n \bm{X}_n^\top\), where \(\bm{J}_n\) is a self-adjoint involution. Since such an indefinite quadratic form is not directly covered by the quadratic comparison theorem, we pass to a self-adjoint linearization. The resulting matrix is affine in the Gaussian variables and therefore falls within the scope of~\cite[Theorem~2.2]{bandeira2026}.

For a symmetric matrix \(\bm{A} \in \sym_p(\R)\), define
\[
|\bm{A}| \coloneqq(\bm{A}^2)^{1/2}.
\]
We also define
\[
\operatorname{sgn} (\bm{A}) \coloneqq \mathbf{1}_{[0,\infty)} (\bm{A}) - \mathbf{1}_{(-\infty,0)} (\bm{A}).
\]
Thus, zero eigenvalues are assigned sign \(+1\). With this convention, \(\operatorname{sgn}(\bm{A})^\top=\operatorname{sgn}(\bm{A})\), \(\operatorname{sgn}(\bm{A})^2 = \bm{I}_p\), and \(|\bm{A}|^{1/2} \operatorname{sgn}(\bm{A}) |\bm{A}|^{1/2} = \bm{A}\). 

Define
\[
\bm{X}_n \coloneqq
\begin{bmatrix}
\widetilde{\bm{X}}_1 & \cdots & \widetilde{\bm{X}}_n
\end{bmatrix}
\in \R^{p(d-r)\times pn},
\]
where
\[
\widetilde{\bm{X}}_i \coloneqq \frac{1}{\sqrt{n}} \left( |\bm{T} (\bm{y}_i)|^{1/2} \otimes \bm{u}_i \right) \in\R^{p(d-r)\times p}.
\]
Moreover, let
\[
\bm{J}_n \coloneqq \operatorname{diag} \left( \mathrm{sgn} (\bm{T}(\bm{y}_1)),\ldots, \mathrm{sgn} (\bm{T}(\bm{y}_n)) \right) \in\R^{pn\times pn}.
\]
By construction, \(\bm{J}_n = \bm{J}_n^\top\) and \(\bm{J}_n^2 = \bm{I}_{pn}\). Therefore,
\begin{equation} \label{eq:signed_gram_representation}
\bm{P}_n=\bm{X}_n \bm{J}_n \bm{X}_n^\top .
\end{equation}

Write \( \bm{u}_i = \sum_{k=1}^{d-r} g_{ik} \bm{e}_k\) with \(g_{ik}\stackrel{\mathrm{i.i.d.}}{\sim}\mathcal{N}(0,1)\), and define
\[
\bm{A}_{ik} \coloneqq  \frac{1}{\sqrt{n}} \left(|\bm{T} (\bm{y}_i)|^{1/2}\otimes \bm{e}_k \right) \left( \bm{e}_i\otimes\bm{I}_p \right)^\top.
\]
Then
\[
\bm{X}_n = \sum_{i=1}^n\sum_{k=1}^{d-r} g_{ik}\bm{A}_{ik}.
\]
Let \(\mathcal{F}_n \coloneqq \sigma (\bm{y}_1,\ldots,\bm{y}_n)\). Conditionally on \(\mathcal{F}_n\), the coefficient matrices \((\bm{A}_{ik})_{i,k}\) and \(\bm{J}_n\) are deterministic, whereas the random variables \((g_{ik})_{i,k}\) remain independent standard Gaussians.

Let \((s_{ik})_{i,k}\) be a free semicircular family in a tracial \(C^\ast\)-probability space \((\mathcal{A},\tau)\), with \(\tau\) faithful, and define the associated free Gaussian matrix
\[
\bm{X}_n^{\mathrm{free}} \coloneqq \sum_{i=1}^n \sum_{k=1}^{d-r} \bm{A}_{ik} \otimes s_{ik}.
\]
Conditionally on \(\mathcal{F}_n\), define the associated free signed covariance operator by
\begin{equation} \label{eq:signed_free_model}
\bm{P}_n^{\mathrm{free}} \coloneqq \bm{X}_n^{\mathrm{free}} \left( \bm{J}_n \otimes 1_{\mathcal{A}} \right) \left( \bm{X}_n^{\mathrm{free}} \right)^\ast.
\end{equation}
Thus, conditionally on \(\mathcal{F}_n\), the operator \(\bm{P}_n^{\mathrm{free}}\) is a fixed element of \(M_{p(d-r)}(\mathcal{A}) \). Since \(\bm{J}_n\otimes1_{\mathcal{A}}\) is self-adjoint, \(\bm{P}_n^{\mathrm{free}}\) is self-adjoint, although it need not be positive. 

Our first objective is to compare the largest eigenvalue of \(\bm{P}_n\) with the upper spectral edge 
\[
\lambda_{+,n}^{\mathrm{free}} \coloneqq \sup \operatorname{sp}(\bm{P}_n^{\mathrm{free}})
\]
of the associated free model.

\begin{lem} \label{lem:signed_free_comparison}
There exists a deterministic sequence \(\varepsilon_n\downarrow0\) such that
\[
\mathbb{P} \left( \lambda_1(\bm{P}_n) > \lambda_{+,n}^{\mathrm{free}} + \varepsilon_n \right) \to 0.
\]
\end{lem}

To prove Lemma~\ref{lem:signed_free_comparison}, we introduce self-adjoint linearizations of the classical and free signed covariance models. For \(E\in\R\), let
\begin{equation} \label{eq:signed_linearization}
\bm{L}_n(E) \coloneqq \begin{pmatrix}
-E\bm{I}_{p(d-r)}&\bm{X}_n\\
\bm{X}_n^\top & -\bm{J}_n
\end{pmatrix},
\end{equation}
and
\begin{equation} \label{eq:signed_free_linearization}
\bm{L}_n^{\mathrm{free}}(E) \coloneqq
\begin{pmatrix}
-E\bm{I}_{p(d-r)}\otimes1_{\mathcal{A}} & \bm{X}_n^{\mathrm{free}}\\
(\bm{X}_n^{\mathrm{free}})^\ast&-\bm{J}_n\otimes1_{\mathcal{A}} 
\end{pmatrix}.
\end{equation}
The following deterministic lemma provides the link between a signed quadratic form and its linearization. In addition to the usual Schur complement equivalence, it gives a quantitative lower bound on the spectral gap of the linearization.

\begin{lem}\label{lem:signed_linearization_gap}
Let \(\mathcal{H}\) and \(\mathcal{K}\) be Hilbert spaces, let \(X\colon\mathcal{K} \to \mathcal{H}\) be a bounded operator, and let \(J\colon\mathcal{K}\to\mathcal{K}\) be a self-adjoint unitary. Set \(P=XJX^\ast\) and
\[
L (E) \coloneqq \begin{pmatrix}
-EI_{\mathcal{H}} & X\\
X^\ast & -J
\end{pmatrix}.
\]
Then
\[
E \in\operatorname{sp}(P) \quad\Longleftrightarrow\quad 0\in\operatorname{sp}(L(E)).
\]
Moreover, for every \(E\notin\operatorname{sp}(P)\),
\[
\operatorname{dist} \left ( 0,\operatorname{sp}(L(E)) \right) \geq \frac{\min \left \{1,\operatorname{dist}(E,\operatorname{sp}(P))\right\}}{(1+\|X\|_{\mathrm{op}})^2}.
\]
\end{lem}

Lemma~\ref{lem:signed_linearization_gap} applies both to the classical model, with \((X,J,P,L) = (\bm{X}_n,\bm{J}_n,\bm{P}_n,\bm{L}_n)\), and to the free model, with \((X,J,P,L) = (\bm{X}_n^{\mathrm{free}}, \bm{J}_n\otimes1_{\mathcal{A}}, \bm{P}_n^{\mathrm{free}},\bm{L}_n^{\mathrm{free}})\).

\begin{proof}
Since \(J^{-1}=J\), the lower-right block of \(L(E)\) is invertible. Its Schur complement in \(L(E)\) is
\[
-E I_{\mathcal{H}} - X (-J)^{-1} X^\ast = - E I_{\mathcal{H}} + X J X^\ast = P -  E I_{\mathcal{H}}.
\]
The Schur complement criterion gives
\[
L(E) \enspace \textnormal{is invertible} \quad\Longleftrightarrow \quad P -  E I_{\mathcal{H}}\enspace \textnormal{is invertible}.
\]
This proves the asserted spectral equivalence.

Now suppose that \(E \notin \operatorname{sp} (P)\) and write \(G(E) \coloneqq (P-EI_{\mathcal{H}})^{-1}\). Then
\[
L(E)^{-1} = \begin{pmatrix}
G(E)& G(E) XJ\\
J X^\ast G(E) & -J + J X^\ast G(E) X J
\end{pmatrix}
= \begin{pmatrix}
I& 0\\
J X^\ast & I 
\end{pmatrix}
\begin{pmatrix}
G(E) & 0\\
0 & -J 
\end{pmatrix}
\begin{pmatrix}
I& XJ\\
0 & I 
\end{pmatrix}.
\]
Thus,
\[
\left \| L(E)^{-1}\right \|_{\mathrm{op}} \le (1 + \|X\|_{\mathrm{op}})^2 \max \{ \|G(E)\|_{\mathrm{op}},1\},
\]
where we used that \(\|J\|_{\mathrm{op}}=1\). Since \(P\) and \(L(E)\) are self-adjoint, \(\|G(E)\|_{\mathrm{op}}^{-1} = \operatorname{dist} (E, \operatorname{sp} (P))\) and \(\|L(E)^{-1}\|_{\mathrm{op}}^{-1} = \operatorname{dist} (0, \operatorname{sp} (L(E)))\). Therefore,
\[
\operatorname{dist} (0, \operatorname{sp} (L(E))) \ge \frac{1}{(1 + \|X\|_{\mathrm{op}})^2 \max \{ \|G(E)\|_{\mathrm{op}},1\}} = \frac{\min \left \{1,\operatorname{dist}(E,\operatorname{sp}(P))\right\}}{(1+\|X\|_{\mathrm{op}})^2}.
\]
\end{proof}

Conditionally on \(\mathcal{F}_n\), the linearization \(\bm{L}_n(E)\) defined in~\eqref{eq:signed_linearization} is a self-adjoint Gaussian matrix. More precisely,
\[
\bm{L}_n(E) = 
\begin{pmatrix}
-E \bm{I}_{p(d-r)} & \bm{0} \\
\bm{0} &-\bm{J}_n
\end{pmatrix}
+
\sum_{i=1}^n \sum_{k=1}^{d-r}
g_{ik}
\begin{pmatrix}
\bm{0} &\bm{A}_{ik} \\
\bm{A}_{ik}^\top& \bm{0}
\end{pmatrix}.
\]
The associated free Gaussian matrix in the spectral comparison theorem is precisely
\(\bm{L}_n^{\mathrm{free}}(E)\). We may therefore apply the self-adjoint spectral comparison result of~\cite[Theorem 2.2]{bandeira2026}, conditionally on \(\mathcal{F}_n\).

We recall the form of the estimate that will be used. Let \(\bm{Z}\) be an \(m\times m\) self-adjoint Gaussian matrix, and let \(\bm{Z}^{\mathrm{free}}\) denote its associated free Gaussian model. Then, for every \(t\ge0\),
\[
\mathbb{P} \left(d_{\mathrm{H}} \left( \operatorname{sp} (\bm{Z}), \operatorname{sp} (\bm{Z}^{\mathrm{free}}) \right) >  C \tilde{v}(\bm{Z}) (\log m)^{3/4} + C \sigma_\ast (\bm{Z}) t\right) \le e^{-t^2},
\]
where \(C>0\) is a universal constant. Here, 
\[
d_{\mathrm{H}}(A,B) \coloneqq \inf \left \{  \varepsilon > 0 \colon A \subseteq B + [-\varepsilon, \varepsilon] \enspace \mathrm{and} \enspace B \subseteq A + [-\varepsilon, \varepsilon]  \right \}
\]
denotes the Hausdorff distance between subsets \(A,B\subseteq\R\), 
\[
\begin{split}
\sigma(\bm{Z})^2 & \coloneqq  \left \|\E \left [ \left ( \bm{Z} - \E \bm{Z} \right)^\ast \left ( \bm{Z} - \E \bm{Z} \right) \right]\right\|_{\mathrm{op}}
\vee \left\| \E \left [\left ( \bm{Z} - \E \bm{Z} \right) \left ( \bm{Z} - \E \bm{Z} \right)^\ast \right ]\right \|_{\mathrm{op}}, \\
v(\bm{Z})^2 &\coloneqq \left \|\operatorname{Cov}(\bm{Z})\right \|_{\mathrm{op}}, \\
\sigma_\ast (\bm{Z})^2 & \coloneqq \sup_{\|\bm{v}\| = \|\bm{w}\| =1} \E \left[ \left| \langle\bm{v},\left ( \bm{Z} - \E \bm{Z} \right)\bm{w} \rangle\right|^2 \right], \\
\widetilde{v} (\bm{Z})^2 & \coloneqq v(\bm{Z})\sigma(\bm{Z}).
\end{split}
\]
In our application, we use \(\bm{Z}=\bm{L}_n(E)\), \(\bm{Z}^{\mathrm{free}} = \bm{L}_n^{\mathrm{free}}(E)\), and \(m=p(d-r+n)\).

\begin{lem} \label{lem:signed_linearized_comparison}
For every fixed finite set \(\mathcal{E} \subset \R\), there exists a deterministic sequence \(\delta_n\downarrow0\) such that
\[
\mathbb{P} \left( \max_{E\in\mathcal{E}} d_{\mathrm{H}} \left( \operatorname{sp} (\bm{L}_n(E)), \operatorname{sp} (\bm{L}_n^{\mathrm{free}}(E)) \right) > \delta_n
\right) \to 0.
\]
\end{lem}

\begin{proof}
We apply Theorem~2.2 of~\cite{bandeira2026} to \(\bm{L}_n(E)\), conditionally on \(\mathcal{F}_n=\sigma(\bm{y}_1,\ldots,\bm{y}_n)\). All Gaussian expectations and Gaussian parameters below are therefore understood conditionally on \(\mathcal{F}_n\). In the following, set 
\[
\bm{B}_{ik} \coloneqq \begin{pmatrix}
\bm{0} &\bm{A}_{ik} \\
\bm{A}_{ik}^\top& \bm{0}
\end{pmatrix} ,
\]
so that
\[
\bm{L}_n(E) - \E [\bm{L}_n(E)] = \sum_{i=1}^n \sum_{k=1}^{d-r} g_{ik} \bm{B}_{ik}.
\]

We first estimate \(\sigma(\bm{L}_n(E))\). By definition,
\[
\sigma(\bm{L}_n(E))^2 = \left \| \sum_{i=1}^n \sum_{k=1}^{d-r} \bm{B}_{ik}^2 \right \|_{\mathrm{op}} = \left \| \sum_{i=1}^n \sum_{k=1}^{d-r} \bm{A}_{ik} \bm{A}_{ik} ^\top\right \|_{\mathrm{op}}  \vee \left \| \sum_{i=1}^n \sum_{k=1}^{d-r} \bm{A}_{ik}^\top \bm{A}_{ik} \right \|_{\mathrm{op}}  .
\]
A direct computation yields
\[
 \sum_{i=1}^n \sum_{k=1}^{d-r} \bm{A}_{ik} \bm{A}_{ik} ^\top = \left ( \frac{1}{n} \sum_{i=1}^n |\bm{T}(\bm{y}_i)|\right) \otimes \bm{I}_{d-r}
\]
and 
\[
 \sum_{i=1}^n \sum_{k=1}^{d-r} \bm{A}_{ik}^\top \bm{A}_{ik} = \operatorname{diag} \left ( \gamma_n |\bm{T}(\bm{y}_1)|, \ldots, \gamma_n |\bm{T}(\bm{y}_n)| \right),
\]
where \(\gamma_n = (d-r)/n\). By Assumption~\ref{hyp:preprocessing}, it follows that 
\[
\sigma(\bm{L}_n(E))^2 = \left \| \frac{1}{n} \sum_{i=1}^n |\bm{T}(\bm{y}_i)| \right \|_{\mathrm{op}} \vee \left ( \gamma_n \max_{1 \le i \le n} \| \bm{T}(\bm{y}_i)\|_{\mathrm{op}} \right )\le C_T (1 \vee \gamma_n). 
\]
Thus, there exists a deterministic constant \(K_\sigma < \infty\) such that, uniformly in \(E\), 
\begin{equation}\label{eq:signed_L_sigma}
\sigma(\bm{L}_n(E))\leq K_\sigma
\end{equation}
for all sufficiently large \(n\). We next estimate \(\sigma_\ast(\bm{L}_n(E))\). Let \(\bm{z} = (\bm{v},\bm{w})^\top\) and \(\bm{z}' = (\bm{v}',\bm{w}')^\top\) be unit vectors in \(\R^{p(d-r)} \oplus \R^{pn}\). Then 
\[
\left \langle\bm{z}, \bm{B}_{ik} \bm{z}' \right \rangle = \langle\bm{v},\bm{A}_{ik}\bm{w}' \rangle + \langle\bm{w},\bm{A}_{ik}^\top \bm{v}'\rangle.
\]
For arbitrary \(\bm{a} = (\bm{a}_1,\ldots, \bm{a}_{d-r}) \in \R^{p(d-r)}\) and \(\bm{b} = (\bm{b}_1,\ldots, \bm{b}_n) \in \R^{pn}\), one has
\[
\begin{split}
\sum_{i=1}^n \sum_{k=1}^{d-r} |\langle \bm{a}, \bm{A}_{ik} \bm{b} \rangle|^2 & = \frac{1}{n} \sum_{i=1}^n \sum_{k=1}^{d-r} \left | \langle \bm{a}_k, |\bm{T}(\bm{y}_i)|^{1/2} \bm{b}_i \rangle \right |^2 \\
& \le \frac{C_T}{n} \left ( \sum_{k=1}^{d-r} \|\bm{a}_k\|_2^2 \right ) \left ( \sum_{i=1}^n \|\bm{b}_i\|_2^2 \right ) \\
& = \frac{C_T}{n} \|\bm{a}\|_2^2 \|\bm{b}\|_2^2.
\end{split}
\]
Hence, using \(|x+y|^2 \le 2 |x|^2 + 2 |y|^2\), 
\[
\sum_{i=1}^n \sum_{k=1}^{d-r} \left | \left \langle\bm{z}, \bm{B}_{ik} \bm{z}' \right \rangle \right |^2 \le \frac{2C_T}{n} \left (  \|\bm{v}\|_2^2 \|\bm{w}'\|_2^2 + \|\bm{w}\|_2^2 \|\bm{v}'\|_2^2  \right ) \le \frac{2C_T}{n}.
\]
Therefore
\begin{equation}\label{eq:signed_L_sigma_star}
\sigma_\ast(\bm{L}_n(E)) \leq \sqrt{\frac{2C_T}{n}},
\end{equation}
uniformly in \(E\). It remains to estimate \(v(\bm{L}_n(E))\). The matrices \((\bm{A}_{ik})_{i,k}\) are pairwise orthogonal for the Hilbert--Schmidt inner product. Indeed, different values of \(i\) correspond to disjoint column blocks, while different values of \(k\) correspond to orthogonal row directions. Furthermore,
\[
\left \| \bm{B}_{ik} \right \|_{\mathrm{F}}^2 = 2 \| \bm{A}_{ik} \|_{\mathrm{F}}^2 = \frac{2}{n} \Tr \left |\bm{T}(\bm{y}_i)\right | \le \frac{2 p C_T}{n}.
\]
Thus
\begin{equation*}
v(\bm{L}_n(E))^2   = \left \| \sum_{i=1}^n \sum_{k=1} ^{d-r} \operatorname{vec}(\bm{B}_{ik}) \operatorname{vec}(\bm{B}_{ik})^\top \right \|_{\mathrm{op}} = \max_{i,k} \|\bm{B}_{ik}\|_{\mathrm{F}}^2\le \frac{2p C_T}{n}.
\end{equation*}
Combining this with~\eqref{eq:signed_L_sigma}, we obtain a deterministic constant \(K_v < \infty\) such that 
\begin{equation} \label{eq:v_tilde}
\widetilde{v} (\bm{L}_n(E)) \le K_v n^{-1/4}
\end{equation}
uniformly in \(E\), for all sufficiently large \(n\). 

Apply~\cite[Theorem 2.2]{bandeira2026} conditionally on \(\mathcal{F}_n\), with \(m = p(d-r)+pn\) and \(t_n = \sqrt{2 \log m}\). By~\eqref{eq:signed_L_sigma_star} and~\eqref{eq:v_tilde}, there are deterministic constants \(C_1,C_2<\infty\) such that, for every \(E\),
\[
\mathbb{P} \left ( d_{\mathrm{H}} \left ( \operatorname{sp} (\bm{L}_n(E)), \operatorname{sp} (\bm{L}_n^{\mathrm{free}}(E)) \right) > a_n  \, \Big | \, \mathcal{F}_n\right ) \le m^{-2},
\]
where 
\[
a_n = C_1 n^{-1/4} \left ( \log m \right)^{3/4} + C_2 n^{-1/2}\left ( \log m \right)^{1/2}.
\]
Under the proportional-growth Assumption~\ref{hyp:prop_limit}, \(m = O(n)\), and therefore \(a_n\to 0\). The estimates do not depend on \(E\), because \(E\) appears only in the deterministic mean of \(\bm{L}_n(E)\). A union bound gives
\[
\mathbb{P} \left ( \max_{E \in \mathcal{E}} d_{\mathrm{H}} \left ( \operatorname{sp} (\bm{L}_n(E)), \operatorname{sp} (\bm{L}_n^{\mathrm{free}}(E)) \right) > a_n  \, \Big | \, \mathcal{F}_n\right ) \le |\mathcal{E}| m^{-2}.
\]
Taking expectations yields the same unconditional bound. Finally, replacing \(a_n\) by its decreasing envelope \(\delta_n = \sup_{k \ge n} a_k\) gives a deterministic sequence \(\delta_n \downarrow 0\), and 
\[
\mathbb{P} \left ( \max_{E \in \mathcal{E}} d_{\mathrm{H}} \left ( \operatorname{sp} (\bm{L}_n(E)), \operatorname{sp} (\bm{L}_n^{\mathrm{free}}(E)) \right) > \delta_n \right ) \le |\mathcal{E}| m^{-2} \to 0.
\]
\end{proof}

Combining the quantitative gap estimate in Lemma~\ref{lem:signed_linearization_gap} with the spectral comparison in Lemma~\ref{lem:signed_linearized_comparison} yields the desired comparison of upper edges.

\begin{proof}[Proof of Lemma~\ref{lem:signed_free_comparison}]
We first show that, for every fixed \(\eta > 0\),
\begin{equation} \label{eq:fixed_eta_signed}
\mathbb{P} \left( \lambda_1(\bm{P}_n) > \lambda_{+,n}^{\mathrm{free}}+\eta \right) \to 0.
\end{equation}
We begin with two uniform boundedness observations. From the free Gaussian Khintchine inequality, we have \(\|\bm{X}_n^{\mathrm{free}}\|_{\mathrm{op}} \le 2 \sigma (\bm{X}_n)\) and since \(\sigma (\bm{X}_n) = \sigma (\bm{L}_n(E)) \le K_\sigma\) by~\eqref{eq:signed_L_sigma}, for all sufficiently large \(n\), there exists a deterministic constant \(K<\infty\) such that
\begin{equation}\label{eq:free_X_uniform_bound}
\|\bm{X}_n^{\mathrm{free}} \|_{\mathrm{op}} \le K
\end{equation}
for all sufficiently large \(n\), almost surely with respect to the conditioning variables. 

Moreover, Lemma~\ref{lem:signed_linearized_comparison} applied to the singleton \(\mathcal{E}=\{0\}\) gives
\[
d_{\mathrm{H}} \left ( \operatorname{sp} (\bm{L}_n(0)), \operatorname{sp} (\bm{L}_n^{\mathrm{free}}(0)) \right ) \stackrel{\mathbb{P}}{\to} 0.
\]
By~\eqref{eq:free_X_uniform_bound},
\[
\| \bm{L}_n^{\mathrm{free}} (0) \|_{\mathrm{op}} \le 1 + \|\bm{X}_n^{\mathrm{free}} \|_{\mathrm{op}} \le 1+K.
\]
This implies that \(\| \bm{L}_n (0) \|_{\mathrm{op}} = O_{\mathbb{P}}(1)\). Since \(\bm{X}_n\) is an off-diagonal compression of \(\bm{L}_n(0)\), 
\[
\| \bm{X}_n \|_{\mathrm{op}} \le \| \bm{L}_n (0) \|_{\mathrm{op}} = O_{\mathbb{P}}(1).
\]
Thus, for every \(\rho >0\), there exist \(R<\infty\) such that
\begin{equation}\label{eq:X_tight}
\limsup_{n\to\infty} \mathbb{P}( \|\bm{X}_n\|_{\mathrm{op}}>R) \leq\rho.
\end{equation}

On the event \(\{\|\bm{X}_n\|_{\mathrm{op}} \le R\}\), 
\[
\lambda_1(\bm{P}_n) \le \|\bm{P}_n\|_{\mathrm{op}} \le \|\bm{X}_n\|_{\mathrm{op}}^2 \|\bm{J}_n\|_{\mathrm{op}} = \|\bm{X}_n\|_{\mathrm{op}}^2 \le R^2.
\]
Define \(c_\eta \coloneqq \min \{1,\eta/2\}/(1+K)^2>0\). Choose \(0 < h < \min \{\eta/2, c_\eta/4\}\), and let \(\mathcal{E}_{\eta,R}\) be a deterministic finite \(h\)-net of \([-R^2-1,R^2+1]\). Thus every \(x \in [-R^2,R^2]\) is within distance \(h\) of some \(E \in \mathcal{E}_{\eta,R}\). By Lemma~\ref{lem:signed_linearized_comparison}, there is a deterministic sequence \(\delta_n \downarrow 0\) such that the event 
\[
\Omega_n \coloneqq \left \{ \max_{E \in \mathcal{E}_{\eta,R}} d_{\mathrm{H}} \left (\operatorname{sp} (\bm{L}_n(E)), \operatorname{sp} (\bm{L}_n^{\mathrm{free}}(E)) \right ) \le \delta_n\right \}
\]
satisfies \(\mathbb{P} (\Omega_n^{\mathrm{c}}) \to 0\). For all sufficiently large \(n\), we may assume that \(\delta_n < c_\eta /2\). We claim that, on \(\Omega_n \cap \{\|\bm{X}_n\|_{\mathrm{op}} \le R\}\), one has
\begin{equation} \label{eq:claim_upper_edge_free}
\lambda_1 (\bm{P}_n) \le \lambda_{+,n}^{\mathrm{free}} + \eta.
\end{equation}
Suppose by contradiction that~\eqref{eq:claim_upper_edge_free} does not hold. Since \(\lambda_1(\bm{P}_n) \le R^2\), we can choose \(E \in \mathcal{E}_{\eta,R}\) such that \(|E - \lambda_1(\bm{P}_n)| \le h\). Then
\[
E \ge \lambda_1 (\bm{P}_n) - h > \lambda_{+,n}^{\mathrm{free}} + \eta - h \ge  \lambda_{+,n}^{\mathrm{free}}  + \frac{\eta}{2}.
\]
Hence
\[
\operatorname{dist} \left (E, \operatorname{sp} (\bm{P}_n^{\mathrm{free}}) \right ) = E -\lambda_{+,n}^{\mathrm{free}} \ge \frac{\eta}{2}.
\]
Applying Lemma~\ref{lem:signed_linearization_gap} to the free model and using~\eqref{eq:free_X_uniform_bound}, we obtain
\[
\operatorname{dist} \left (0, \operatorname{sp} (\bm{L}_n^{\mathrm{free}}(E)) \right ) \ge c_\eta.
\]
On the event \(\Omega_n\), 
\[
\operatorname{dist} \left (0, \operatorname{sp} (\bm{L}_n(E)) \right )  \ge \operatorname{dist} \left (0, \operatorname{sp} (\bm{L}_n^{\mathrm{free}}(E)) \right ) - \delta_n \ge c_\eta - \delta_n > \frac{c_\eta}{2}.
\]
On the other hand, because \(\lambda_1 (\bm{P}_n) \in \operatorname{sp} (\bm{P}_n)\), the spectral equivalence of Lemma~\ref{lem:signed_linearization_gap} gives that \(0 \in \operatorname{sp} (\bm{L}_n(\lambda_1(\bm{P}_n)))\). Furthermore,
\[
\| \bm{L}_n(E) - \bm{L}_n (\lambda_1(\bm{P}_n)) \|_{\mathrm{op}} = |E - \lambda_1(\bm{P}_n)| \le h.
\]
Therefore,
\[
\operatorname{dist} \left (0, \operatorname{sp} (\bm{L}_n(E)) \right ) \le h < \frac{c_\eta}{4}.
\]
This is a contradiction and hence~\eqref{eq:claim_upper_edge_free} holds:
\[
\left \{ \lambda_1 (\bm{P}_n ) > \lambda_{+,n}^{\mathrm{free}} + \eta\right \} \subseteq \left \{\|\bm{X}_n \|_{\mathrm{op}} > R\right \} \cup \Omega_n^{\mathrm{c}},
\]
for all sufficiently large \(n\). Therefore, 
\[
\limsup_{n \to \infty} \mathbb{P} \left  ( \lambda_1(\bm{P}_n) > \lambda_{+,n}^{\mathrm{free}} + \eta \right ) \le \rho
\]
by~\eqref{eq:X_tight}. Since \(\rho>0\) was arbitrary, this proves~\eqref{eq:fixed_eta_signed}.

It remains to replace the fixed \(\eta\) by a deterministic sequence tending to zero. For every \(k \ge 1\),~\eqref{eq:fixed_eta_signed} with \(\eta = k^{-1}\) yields an integer \(N_k\) such that
\[
\mathbb{P} \left  ( \lambda_1(\bm{P}_n) > \lambda_{+,n}^{\mathrm{free}} + \frac{1}{k} \right ) \le \frac{1}{k},
\]
for every \(n \ge N_k\). We may choose \(N_k\) strictly increasing. Define \(\varepsilon_n = k^{-1}\) for \(N_k \le n < N_{k+1}\) and \(\varepsilon_n =1\) for \(n < N_1\). Then \(\varepsilon_n \downarrow0\), and if \(N_k \le n < N_{k+1}\),
\[
\mathbb{P} \left  ( \lambda_1(\bm{P}_n) > \lambda_{+,n}^{\mathrm{free}} + \varepsilon_n \right ) \le \frac{1}{k}.
\]
The right-hand side tends to zero as \(n\to\infty\), proving the result.
\end{proof}

\subsection{Variational upper bound for the free spectral edge}

We next control the upper spectral edge of the conditional free model. An exact finite-\(n\) characterization of \(\lambda_{+,n}^{\mathrm{free}}\) is not needed for Proposition~\ref{prop:upper_edge}; it is enough to obtain a tractable upper bound that converges to the deterministic bulk edge. When the preprocessing matrices are positive semidefinite, an exact variational formula can be obtained from~\cite[Theorem~1.2]{parmaksiz2025}. In our case, however, \(\bm{P}_n^{\mathrm{free}}\) need not be positive and is thus not of the form \(xx^\ast\). We therefore construct a Fock-space realization with the same \(M_p(\C)\)-valued distribution and use a completion-of-squares argument inspired by Lehner~\cite{lehner1999}. \\

Recall that \(\gamma_n = (d-r)/n\). Define the admissible set 
\[
\mathcal{D}_n \coloneqq \left \{ \bm{Z} \in \mathrm{Sym}_p^{++}(\R) \colon \bm{Z}^{-1} - \gamma_n \bm{T}(\bm{y}_i) \succ \bm{0}_{p \times p} \enspace \text{for every }1\leq i\leq n \right \}.
\]
For \(\bm{Z}\in\mathcal{D}_n\), set
\[
\mathcal{G}_n(\bm{Z}) \coloneqq \bm{Z}^{-1} + \frac{1}{n} \sum_{i=1}^n \bm{T}(\bm{y}_i) \left( \bm{I}_p - \gamma_n \bm{Z} \bm{T}(\bm{y}_i) \right)^{-1}.
\]
The positivity condition defining \(\mathcal{D}_n\) ensures that all of the inverses appearing in this expression are well defined. The following lemma bounds the upper spectral edge of the free model by a finite-dimensional optimization problem.

\begin{lem} \label{lem:finite_variational_formula} 
Almost surely, 
\[ 
\lambda_{+,n}^{\mathrm{free}} \le \inf_{\bm{Z} \in \mathcal{D}_n} \lambda_1 \left (\mathcal{G}_n(\bm{Z}) \right).
\]
\end{lem}

\begin{proof}
Fix \(n\) and condition throughout on \(\mathcal{F}_n = \sigma (\bm{y}_1,\ldots, \bm{y}_n)\), so that the matrices \(\bm{T}(\bm{y}_1), \ldots, \bm{T}(\bm{y}_n)\) are deterministic. 

We first construct a Fock-space operator that will be shown to have the same \(M_p(\C)\)-valued distribution as \(\bm{P}_n^{\mathrm{free}}\). We use the standard full Fock-space realization of a free semicircular family; see, for example,~\cite[Chapter 7]{nica2006}. Let 
\[
\mathscr{F} (\C^n) \coloneqq \bigoplus_{k=0}^\infty (\C^n)^{\otimes k} = \C \Omega \oplus \bigoplus_{k=1}^\infty (\C^n)^{\otimes k} 
\]
be the full Fock space over \(\C^n\), where \(\Omega\) is the vacuum vector. We denote by
\[
\omega_\Omega (X) \coloneqq \langle \Omega, X \Omega \rangle , \quad X \in \mathcal{B} (\mathscr{F} (\C^n))
\]
the vacuum expectation state. Let \(\bm{f}_1,\ldots,\bm{f}_n\) be the standard orthonormal basis of \(\C^n\). For \(i\in[n]\), let \(l_i\) be the left creation operator associated with \(\bm{f}_i\), defined by 
\[
l_i \Omega = \bm{f}_i
\]
and 
\[
l_i (\bm{a}_1 \otimes \cdots \otimes \bm{a}_k) \coloneqq \bm{f}_i \otimes \bm{a}_1 \otimes \cdots \otimes \bm{a}_k,
\]
for all \(\bm{a}_1, \ldots, \bm{a}_k \in \C^n\) and \(k \ge 1\). The creation operators satisfy
\begin{equation} \label{eq:vacuum_identity}
l_i^\ast l_j = \delta_{ij}\bm{I}_{\mathscr{F} (\C^n)}, \quad \sum_{i=1}^n l_i l_i^\ast = \bm{I}_{\mathscr{F} (\C^n)} - p_\Omega,
\end{equation}
where \(p_\Omega\) denotes the orthogonal projection onto the vacuum space \(\C \Omega\). Define
\begin{equation}\label{eq:fock_free_model}
\widehat{\bm{P}}_n^{\mathrm{free}} \coloneqq \bm{A}_{0,n} \otimes \bm{I}_{\mathscr{F} (\C^n)} + \sum_{i=1}^n \bm{C}_{i,n} \otimes (l_i + l_i^\ast) + \sum_{i=1}^n \bm{D}_{i,n} \otimes l_i l_i^\ast,
\end{equation}
where
\[
\bm{A}_{0,n} \coloneqq \frac{1}{n} \sum_{i=1}^n \bm{T}(\bm{y}_i), \quad \bm{C}_{i,n} \coloneqq \sqrt{\frac{\gamma_n}{n}} \bm{T} (\bm{y}_i), \quad \bm{D}_{i,n} \coloneqq \gamma_n \bm{T}(\bm{y}_i).
\]
Thus, \(\widehat{\bm{P}}_n^{\mathrm{free}} \in M_p(\C)\otimes \mathcal{B} (\mathscr{F} (\C^n))\) is a bounded self-adjoint operator.

We now claim that \(\bm{P}_n^{\mathrm{free}}\) and \(\widehat{\bm{P}}_n^{\mathrm{free}}\) have the same \(M_p(\C)\)-valued distribution. We begin by computing the \(M_p(\C)\)-valued free cumulants of \(\bm{P}_n^{\mathrm{free}}\). Define the \(M_p(\C)\)-valued expectation
\[
\E_n \colon M_{p(d-r)} (\mathcal{A}) \cong M_p(\C) \otimes M_{d-r} (\C) \otimes \mathcal{A} \to M_p (\C)
\]
by
\[
\E_n \coloneqq \operatorname{id}_{M_p} \otimes \tr_{d-r} \otimes \tau.
\]
Let \(\kappa_m^{(n)}\) denote the \(m\)th \(M_p(\C)\)-valued free cumulant map of \(\bm{P}_n^{\mathrm{free}}\) with respect to \(\E_n\), that is,
\[
\kappa_1^{(n)} \coloneqq \E_n [\bm{P}_n^{\mathrm{free}}],
\]
and, for \(m \ge 2\), 
\[
\kappa_m^{(n)} (\bm{B}_1,\ldots, \bm{B}_{m-1}) \coloneqq \kappa_m^{\E_n} \left ( \bm{P}_n^{\mathrm{free}} \bm{B}_1,\bm{P}_n^{\mathrm{free}} \bm{B}_2, \ldots, \bm{P}_n^{\mathrm{free}} \bm{B}_{m-1},\bm{P}_n^{\mathrm{free}} \right),
\]
for \(\bm{B}_1, \ldots, \bm{B}_{m-1} \in M_p(\C)\). Under the canonical identification above, the operator \(\bm{P}_n^{\mathrm{free}}\) defined in~\eqref{eq:signed_free_model} can be written as
\begin{equation} \label{eq:P_free_W_decomposition}
\bm{P}_n^{\mathrm{free}} = \frac{1}{n} \sum_{i=1}^n\bm{T}(\bm{y}_i)\otimes\bm{W}_i,
\end{equation}
where
\[
\bm{W}_i \coloneqq \sum_{1\le a,b\le d-r} \bm{e}_a \bm{e}_b^\top \otimes s_{ia} s_{ib} \in
M_{d-r}(\C) \otimes\mathcal{A}.
\]
Set \(\varphi \coloneqq \tr_{d-r} \otimes \tau\). We first compute the scalar free cumulants of the family \((\bm{W}_i)_{i=1}^n\) with respect to \(\varphi\). Clearly,
\[
\kappa_1^\varphi (\bm{W}_i) = \varphi (\bm{W}_i)=1.
\]
For \(m\geq2\), the free-cumulant formula with product as arguments (see e.g.~\cite[Chapter 2, Theorem 13]{mingo2017}) expresses \(\kappa_m^\varphi (\bm{W}_{i_1},\ldots,\bm{W}_{i_m})\) as a sum over noncrossing pairings of the \(2m\) semicircular factors that connect the blocks of the interval partition
\[
\rho_m = \left \{ \{1,2\},\{3,4\},\ldots,\{2m-1,2m\} \right \}.
\]
Since \((s_{ia})_{i,a}\) is a free semicircular family, the only contributing connected pairing is
\[
\pi_m = \left \{ \{2,3\},\{4,5\},\ldots, \{2m-2,2m-1\},\{1,2m\} \right \}.
\]
This pairing forces \(i_1=\cdots=i_m\). Moreover, the matrix-trace constraints leave \(m\) free index sums, while the normalized trace contributes a factor \((d-r)^{-1}\), yielding the factor \((d-r)^{m-1}\). For instance, the case \(m=2\) makes the index contraction explicit. We have
\[
\varphi (\bm{W}_i \bm{W}_j) =\frac{1}{d-r} \sum_{a,b=1}^{d-r} \tau (s_{ia} s_{ib} s_{jb} s_{ja}) = 1 + (d-r) \mathbf{1}_{\{i=j\}},
\]
where the first term corresponds to the pairing \(\{\{1,2\},\{3,4\}\}\), while the second corresponds to the connected pairing \(\{\{1,4\},\{2,3\}\}\). Since \(\varphi (\bm{W}_i) = \varphi (\bm{W}_j) =1\), it follows that
\[
\kappa_2^\varphi (\bm{W}_i,\bm{W}_j) = (d-r) \mathbf{1}_{\{i=j\}}.
\]
For general \(m \ge 2\), the same contraction pattern associated with \(\pi_m\) forces \(i_1 = \cdots =i_m\) and leaves \(m-1\) free matrix indices, yielding
\[
\kappa_m^\varphi (\bm{W}_{i_1},\ldots,\bm{W}_{i_m}) = (d-r)^{m-1} \mathbf{1}_{\{i_1=\cdots=i_m\}}, \qquad m \ge 2.
\]
By the relation between scalar-valued and matrix-valued cumulants (see e.g.~\cite[Chapter 9, Proposition 13]{mingo2017}), together with multilinearity, we obtain
\[
\kappa_1^{(n)} =\E_n [\bm{P}_n^{\mathrm{free}}] = \frac{1}{n} \sum_{i=1}^n \bm{T} ( \bm{y}_i),
\]
and, for every \(m \ge 2\),
\[
\begin{split}
\kappa_m^{(n)} (\bm{B}_1,\ldots,\bm{B}_{m-1}) 
& = \frac{1}{n^m} \sum_{1 \le i_1,\ldots, i_m \le n}  \bm{T}(\bm{y}_{i_1}) \bm{B}_1 \bm{T}(\bm{y}_{i_2}) \cdots  \bm{B}_{m-1} \bm{T}(\bm{y}_{i_m}) \kappa_m^\varphi (\bm{W}_{i_1}, \ldots, \bm{W}_{i_m}) \\
& = \frac{\gamma_n^{m-1}}{n} \sum_{i=1}^n \bm{T}(\bm{y}_i) \bm{B}_1 \bm{T}(\bm{y}_i)\cdots \bm{B}_{m-1} \bm{T}(\bm{y}_i).
\end{split}
\]
We next compute the \(M_p(\C)\)-valued cumulants of \(\widehat{\bm{P}}_n^{\mathrm{free}}\). Introduce the \(M_p(\C)\)-valued expectation
\[
\widehat{\E}_n \colon M_p (\C) \otimes \mathcal{B} (\mathscr{F} (\C^n)) \to Mp(\C), \qquad \widehat{\E}_n \coloneqq \operatorname{id}_{M_p} \otimes \omega_\Omega,
\]
and define the cumulant maps  \(\widehat{\kappa}_m^{(n)}\) analogously to \(\kappa_m^{(n)}\). Let \(\bm{Q}_i \coloneqq \bm{f}_i \bm{f}_i^\ast\) be the orthogonal projection onto \(\C \bm{f}_i\). The corresponding gauge operator associated is 
\[
\Lambda (\bm{Q}_i) = l_i l_i^\ast.
\]
Since \(\omega_\Omega(l_i) = \omega_\Omega (l_i^\ast) = \omega_\Omega (\Lambda (\bm{Q}_i) )=0\), we have
\[
\widehat{\kappa}_1^{(n)} = \bm{A}_{0,n} = \frac{1}{n} \sum_{i=1}^n\bm{T} (\bm{y}_i) .
\]
For \(m \ge 3\), the Fock-space cumulant formula~\cite[Proposition~13.5]{nica2006}, gives
\[
\begin{split}
\widehat{\kappa}_m^{(n)} (\bm{B}_1,\ldots, \bm{B}_{m-1})  & = \sum_{1 \le i_1, \ldots, i_m \le n} \bm{C}_{i_1,n}  \bm{B}_1 \bm{D}_{i_2,n} \bm{B}_2 \cdots \bm{D}_{i_{m-1},n} \bm{B}_{m-1} \bm{C}_{i_m,n} \langle \bm{f}_{i_1}, \bm{Q}_{i_2} \cdots \bm{Q}_{i_{m-1}} \bm{f}_{i_m}\rangle \\
&  =\sum_{i=1}^n \bm{C}_{i,n}  \bm{B}_1 \bm{D}_{i,n} \bm{B}_2 \cdots \bm{D}_{i,n} \bm{B}_{m-1} \bm{C}_{i,n}\\
& = \frac{\gamma_n^{m-1}}{n} \sum_{i=1}^n \bm{T}(\bm{y}_i) \bm{B}_1 \bm{T}(\bm{y}_i)  \cdots \bm{B}_{m-1} \bm{T} (\bm{y}_i),
\end{split}
\]
because the projections \((\bm{Q}_i)_{i=1}^n\) are mutually orthogonal and \(\langle \bm{f}_{i_1}, \bm{Q}_{i_2} \cdots \bm{Q}_{i_{m-1}} \bm{f}_{i_m}\rangle = \mathbf{1}_{\{i_1 = \cdots = i_m\}}\). For \(m=2\), the gauge factors are absent, so therefore
\[
\widehat{\kappa}_2^{(n)} (\bm{B}_1) = \sum_{1 \le i,j \le n} \bm{C}_{i,n} \bm{B}_1 \bm{C}_{j,n} \langle \bm{f}_i,\bm{f}_j\rangle = \sum_{i=1}^n  \bm{C}_{i,n} \bm{B}_1 \bm{C}_{i,n} = = \frac{\gamma_n}{n} \sum_{i=1}^n \bm{T}(\bm{y}_i) \bm{B}_1 \bm{T}(\bm{y}_i)  .
\]
Thus, for every \(m \ge 1\),
\[
\widehat{\kappa}_m^{(n)} = \kappa_m^{(n)}.
\]
By the operator-valued moment-cumulant formula, equality of these cumulant maps for every \(m \ge 1\) implies that the two operators have the same \(M_p(\C)\)-valued distribution. 

Define the scalar states \(\varphi_n \coloneqq \tr_p \circ \E_n\) and \(\widehat{\varphi}_n \coloneqq \tr_p \circ \widehat{\E}_n\). Equality of the \(M_p(\C)\)-valued distributions implies that \(\bm{P}_n^{\mathrm{free}}\) and \(\widehat{\bm{P}}_n^{\mathrm{free}}\) have the same scalar moments with respect to \(\varphi_n\) and \(\widehat{\varphi}_n\), respectively. Since both operators are bounded and self-adjoint, their scalar distributions coincide. We denote the common distribution by \(\mu_n^{\mathrm{free}}\).

The state \(\varphi_n\) is faithful because the normalized matrix traces and \(\tau\) are faithful. Consequently, the support of the spectral distribution of \(\bm{P}_n^{\mathrm{free}}\) coincides with its spectrum:
\begin{equation}\label{eq:support_equals_free_spectrum}
\supp (\mu_n^{\mathrm{free}}) = \operatorname{sp}(\bm{P}_n^{\mathrm{free}}).
\end{equation}
Indeed, the inclusion
\[
\supp(\mu_n^{\mathrm{free}}) \subseteq \operatorname{sp}(\bm{P}_n^{\mathrm{free}})
\]
holds for the spectral distribution of any self-adjoint element. For the reverse inclusion, let \(U \subseteq \R\) be open and suppose that \(U \cap \operatorname{sp}(\bm{P}_n^{\mathrm{free}}) \neq \emptyset\). By the continuous functional calculus, there exists a nonzero function \(f \in C\left(  \operatorname{sp}(\bm{P}_n^{\mathrm{free}}) \right)\) such that \(0\le f\le1\) and \(\supp (f) \subseteq U\). Then
\[
f(\bm{P}_n^{\mathrm{free}})\succeq 0, \qquad f(\bm{P}_n^{\mathrm{free}})\neq 0.
\]
Since \(\varphi_n\) is faithful, we have
\[
0 < \varphi_n \left[ f(\bm{P}_n^{\mathrm{free}}) \right] = \int_\R  f(\lambda) \mu_n^{\mathrm{free}}(\mathrm{d} \lambda) \le \mu_n^{\mathrm{free}} (U).
\]
Thus every open set intersecting \(\operatorname{sp}(\bm{P}_n^{\mathrm{free}})\) has positive \(\mu_n^{\mathrm{free}}\)-measure, which proves~\eqref{eq:support_equals_free_spectrum}. On the Fock-space side, the vacuum state \(\omega_\Omega\), and hence \(\widehat{\varphi}_n\), need not be faithful. We therefore only have
\[
\supp (\mu_n^\mathrm{free}) \subseteq \operatorname{sp} ( \widehat{\bm{P}}_n^{\mathrm{free}}).
\]
In particular, 
\begin{equation} \label{eq:lambda_free_UB}
\lambda_{+,n}^{\mathrm{free}} = \sup \operatorname{sp} (\bm{P}_n^{\mathrm{free}}) = \sup \operatorname{supp} (\mu_n^\mathrm{free})   \le \sup \operatorname{sp} ( \widehat{\bm{P}}_n^{\mathrm{free}}).
\end{equation}

The last step is to provide an upper bound for \(\sup \operatorname{sp} ( \widehat{\bm{P}}_n^{\mathrm{free}})\). For every \(\bm{Z} \in \mathcal{D}_n\) and \(i \in [n]\), define
\[
\bm{K}_{i,n} \coloneqq \bm{Z}^{-1} - \gamma_n \bm{T} (\bm{y}_i) \succ \bm{0}_p.
\]
We then write
\[
\mathcal{G}_n(\bm{Z}) = \bm{A}_{0,n} + \bm{Z}^{-1} + \sum_{i=1}^n \bm{C}_{i,n} \bm{K}_{i,n}^{-1} \bm{C}_{i,n},
\]
where we used that
\[
\bm{T}(\bm{y}_i) \left (\bm{I}_p - \gamma_n \bm{Z} \bm{T}(\bm{y}_i) \right )^{-1} = \bm{T} (\bm{y}_i) + \gamma_n \bm{T} (\bm{y}_i) \bm{K}_{i,n}^{-1} \bm{T} (\bm{y}_i) .
\]
In particular, this representation shows that \(\mathcal{G}_n(\bm{Z})\) is self-adjoint. 
Define
\[
\bm{Y}_{i,n}  \coloneqq \bm{C}_{i,n} \bm{K}_{i,n}^{-1/2} \otimes I_{\mathscr{F}(\C^n)} - \bm{K}_{i,n}^{1/2} \otimes l_i.
\]
Expanding the product \(\bm{Y}_{i,n} \bm{Y}_{i,n}^\ast\) gives
\[
\bm{Y}_{i,n} \bm{Y}_{i,n}^\ast =  \bm{C}_{i,n} \bm{K}_{i,n}^{-1}\bm{C}_{i,n}  \otimes I_{\mathscr{F}(\C^n)}  - \bm{C}_{i,n} \otimes (l_i + l_i^\ast) + \bm{K}_{i,n} \otimes l_i l_i^\ast.
\]
Since \(\bm{K}_{i,n} = \bm{Z}^{-1} - \bm{D}_{i,n}\) and \(p_\Omega + \sum_{i=1}^n l_i l_i^\ast= I_{\mathscr{F}(\C^n)} \) by~\eqref{eq:vacuum_identity}, we obtain the identity
\[
\mathcal{G}_n(\bm{Z}) \otimes I_{\mathscr{F}(\C^n)}  - \widehat{\bm{P}}_n^{\mathrm{free}} = \bm{Z}^{-1} \otimes p_\Omega + \sum_{i=1}^n \bm{Y}_{i,n} \bm{Y}_{i,n}^\ast.
\]
The right-hand side of the above identity is positive semidefinite. Therefore
\[
\widehat{\bm{P}}_n^{\mathrm{free}} \preceq \mathcal{G}_n(\bm{Z}) \otimes I_{\mathscr{F}(\C^n)} ,
\]
and hence
\[
\sup \operatorname{sp} ( \widehat{\bm{P}}_n^{\mathrm{free}}) \le \lambda_1 (\mathcal{G}_n (\bm{Z})).
\]
Combining this inequality with~\eqref{eq:lambda_free_UB} and taking the infimum over \(\bm{Z} \in\mathcal{D}_n\), we obtain \(\lambda_{+,n}^{\mathrm{free}} \le \inf_{\bm{Z} \in \mathcal{D}_n} \lambda_1 (\mathcal{G}_n (\bm{Z}))\). Since the argument holds for almost every realization of \((\bm{y}_1,\ldots,\bm{y}_n)\), the claim follows.
\end{proof}

Thus, to control \(\lambda_{+,n}^{\mathrm{free}}\), it suffices to evaluate the finite-\(n\) variational bound at a suitable deterministic test matrix derived from the limiting Dyson equation.

\begin{lem} \label{lem:free_upper_edge}
For every \(\varepsilon > 0\),
\[
\mathbb{P} \left ( \lambda_{+,n}^{\mathrm{free}} \ge \lambda_+(\alpha) + \varepsilon \right ) \to 0.
\]
\end{lem}

\begin{proof}
Fix \(\varepsilon > 0\) and set \(E \coloneqq \lambda_+(\alpha)+\frac{\varepsilon}{2}\). Define \(\bm{Z}_E \coloneqq - \bm{M}_\alpha (E)\). The proof proceeds in three steps. We first show that \(\bm{Z}_E\) is positive definite and satisfies a uniform strict admissibility condition. We then evaluate the limiting matrix Dyson equation at \(E\). Finally, we use \(\bm{Z}_E\) as a test matrix in the finite-\(n\) variational formula of Lemma~\ref{lem:finite_variational_formula} and pass to the limit.

We first show that \(\bm{Z}_E \succ \bm{0}_p\). By Lemma~\ref{lem:continuation_off_support}, \(\bm{M}_\alpha\) admits an analytic continuation to \(\C \setminus \supp(\mu_\alpha)\), is self-adjoint on \(\R \setminus \supp(\mu_\alpha)\), and has the representation
\[
\bm{M}_{\alpha}(z) = \int_\R \frac{1}{x-z} \bm{\Omega}_{\alpha}(\mathrm{d}x),
\]
where \(\bm{\Omega}_{\alpha}\) is a positive semidefinite matrix-valued measure satisfying \(\bm{\Omega}_{\alpha}(\R)=\bm{I}_p\) and \( \supp(\bm{\Omega}_\alpha) = \supp(\mu_\alpha)\). Hence, for every nonzero \(\bm{v}\in\C^p\),
\[
\bm{v}^\ast \bm{M}_\alpha(E) \bm{v} = \int_\R \frac{\bm{v}^\ast \bm{\Omega}_\alpha (\mathrm{d}x) \bm{v}}{x-E} <0.
\]
Indeed, the scalar measure \(\bm{v}^\ast \bm{\Omega}_\alpha (\cdot) \bm{v}\) is nonzero, since its total mass is \(\|\bm{v}\|_2^2\) and \(x-E<0\) throughout its support. Therefore, \(\bm{M}_{\alpha}(E)\prec\bm{0}_p\) and so \(\bm{Z}_E \succ \bm{0}_p\). Moreover, 
\begin{equation} \label{eq:ZE_bounds}
\bm{0}_p \prec \bm{Z}_E \preceq \frac{1}{E-\lambda_+(\alpha)} \bm{I}_p.
\end{equation}

We next prove that \(\bm{Z}_E\) is uniformly strictly admissible. This is the key point that allows the same deterministic test matrix to be used in the finite-\(n\) variational formula for all sufficiently
large \(n\). For \(t \ge E\), define 
\[
\bm{Z}_t \coloneqq - \bm{M}_\alpha (t),
\]
and
\[
\bm{H}_t (\bm{y}) \coloneqq \bm{Z}_t^{-1} - \frac{1}{\alpha} \bm{T}(\bm{y}).
\]
By the same argument used for \(\bm{Z}_E\), we have \(\bm{Z}_t\succ\bm{0}_p\) for every \(t>\lambda_+(\alpha)\). In particular, \(\bm{H}_t(\bm{y})\) is self-adjoint. Choose \(E_0>E\) sufficiently large that 
\[
E_0 - \lambda_+(\alpha) > \frac{C_T}{\alpha},
\]
where \(C_T\) is the constant from Assumption~\ref{hyp:preprocessing}. By~\eqref{eq:ZE_bounds},
\[
\bm{H}_{E_0} (\bm{y}) \succeq \left (E_0 - \lambda_+(\alpha) \right) \bm{I}_p - \frac{1}{\alpha} \bm{T}(\bm{y}) \succeq \left (E_0-\lambda_+(\alpha)-\frac{C_T}{\alpha} \right) \bm{I}_p \succ \bm{0}_p
\]
almost surely. We then apply Lemma~\ref{lem:uniform_denominator_bound} on a compact off-support neighborhood containing \([E,E_0]\). Outside a single null set, the matrices
\[
\bm{I}_p + \frac{1}{\alpha} \bm{M}_\alpha(t) \bm{T} (\bm{y}) = \bm{I}_p - \frac{1}{\alpha} \bm{Z}_t \bm{T} (\bm{y}) 
\]
are invertible for every \(t \in [E,E_0]\). Since
\[
\bm{H}_t (\bm{y}) = \bm{Z}_t^{-1} \left ( \bm{I}_p - \frac{1}{\alpha} \bm{Z}_t \bm{T}(\bm{y}) \right ),
\]
the invertibility of \(\bm{I}_p-\alpha^{-1}\bm{Z}_t\bm{T}(\bm{y})\) implies that
\(\bm{H}_t(\bm{y})\) is invertible. For fixed \(\bm{y}\) outside the null set, the map \(t \mapsto\bm{H}_t(\bm{y})\) is continuous and self-adjoint. Since \(\bm{H}_{E_0} (\bm{y}) \succ \bm{0}_p\), if \(\bm{H}_E (\bm{y})\) were not positive definite, then some eigenvalue would have to vanish at an intermediate point \(t \in [E,E_0]\), contradicting the invertibility of \(\bm{H}_t(\bm{y})\). Therefore
\[
\bm{H}_E (\bm{y}) \succ \bm{0}_p \quad \text{a.s.}
\]
The positivity is in fact uniform. Indeed,
\[
\bm{H}_E (\bm{y})^{-1} = \left ( \bm{I}_p - \frac{1}{\alpha} \bm{Z}_E \bm{T} (\bm{y}) \right )^{-1} \bm{Z}_E.
\]
Lemma~\ref{lem:uniform_denominator_bound} gives
\[
\operatorname*{ess \,sup}_{\bm{y}} \| \bm{H}_E (\bm{y})^{-1} \|_{\mathrm{op}} <\infty.
\]
Set \(c_E \coloneqq \left ( \operatorname*{ess \,sup}_{\bm{y}} \| \bm{H}_E (\bm{y})^{-1} \|_{\mathrm{op}}\right)^{-1} >0\). Then
\begin{equation}\label{eq:strict_admissibility_E}
\bm{H}_E(\bm{y})=\bm{Z}_E^{-1} - \frac{1}{\alpha} \bm{T} (\bm{y}) \succeq c_E \bm{I}_p \quad \mathrm{a.s.}
\end{equation}

We now evaluate the matrix Dyson equation~\eqref{eq:MDE} at \(E\). On \(\C_+\),
\[
\bm{M}_\alpha(z)^{-1} = -z \bm{I}_p + \E_{\bm{y}} \left[ \bm{T}(\bm{y}) \left( \bm{I}_p+
\frac{1}{\alpha} \bm{M}_\alpha(z) \bm{T}(\bm{y}) \right)^{-1} \right].
\]
Lemma~\ref{lem:uniform_denominator_bound}, together with \(\|\bm{T}(\bm{y})\|_{\mathrm{op}}\le C_T\) from Assumption~\ref{hyp:preprocessing}, provides an integrable uniform bound for the Dyson-equation integrand in a complex neighborhood of \(E\). Since \(\bm{M}_\alpha(z)\to\bm{M}_\alpha(E)\) as \(z\to E\) from
\(\C_+\), dominated convergence yields
\begin{equation}\label{eq:Galpha_at_ZE}
\bm{Z}_E^{-1} + \E_{\bm{y}} \left[ \bm{T} (\bm{y}) \left( \bm{I}_p - \frac{1}{\alpha} \bm{Z}_E \bm{T} (\bm{y}) \right)^{-1} \right] = E\bm{I}_p.
\end{equation}

We next use \(\bm{Z}_E\) as a test matrix in the finite-\(n\) variational formula of Lemma~\ref{lem:finite_variational_formula}. By~\eqref{eq:strict_admissibility_E},
\[
\bm{Z}_E^{-1} - \gamma_n \bm{T}(\bm{y}_i) = \bm{H}_E (\bm{y}_i) - \left (\gamma_n - \frac{1}{\alpha} \right) \bm{T} (\bm{y}_i) \succeq \left (c_E - C_T |\gamma_n - \alpha^{-1}| \right ) \bm{I}_p. 
\]
Since \(\gamma_n\to\alpha^{-1}\), it follows that, almost surely, for all sufficiently large \(n\),
\[
\bm{Z}_E^{-1} - \gamma_n \bm{T}(\bm{y}_i)  \succeq \frac{c_E}{2}\bm{I}_p
\]
simultaneously for every \(1\le i\le n\). Hence \(\bm{Z}_E \in \mathcal{D}_n\) eventually almost surely. Lemma~\ref{lem:finite_variational_formula} then gives
\begin{equation}\label{eq:finite_var_formula2}
\lambda_{+,n}^{\mathrm{free}} \le \lambda_1 \left (\mathcal{G}_n (\bm{Z}_E) \right)
\end{equation}
for all sufficiently large \(n\), almost surely.

It remains to show that \(\mathcal{G}_n(\bm{Z}_E) \to E \bm{I}_p\) almost surely. For \(\gamma>0\), define
\[
\bm{R}_\gamma (\bm{y}) \coloneqq \left ( \bm{I}_p  - \gamma \bm{Z}_E \bm{T} (\bm{y})\right)^{-1}.
\]
For \(\gamma\) sufficiently close to \(\alpha^{-1}\),
\[
\bm{Z}_E^{-1} - \gamma\bm{T} (\bm{y}) \succeq \frac{c_E}{2} \bm{I}_p
\]
almost surely. Since
\[
\bm{R}_\gamma (\bm{y}) =  \left( \bm{Z}_E^{-1} - \gamma \bm{T} (\bm{y}) \right)^{-1} \bm{Z}_E^{-1},
\]
the family \(\bm{R}_\gamma (\bm{y})\) is uniformly bounded in operator norm. The resolvent identity gives
\[
\bm{R}_\gamma (\bm{y}) - \bm{R}_{\alpha^{-1}} ( \bm{y}) = \left ( \gamma-\alpha^{-1} \right) \bm{R}_\gamma (\bm{y}) \bm{Z}_E \bm{T}(\bm{y}) \bm{R}_{\alpha^{-1}} (\bm{y}).
\]
Let \(\bm{F}_\gamma (\bm{y}) \coloneqq \bm{T} (\bm{y}) \bm{R}_\gamma (\bm{y})\). Then, for some deterministic \(C_E<\infty\),
\[
\operatorname*{ess \, sup}_{\bm{y}} \left \| \bm{F}_{\gamma_n} (\bm{y}) - \bm{F}_{\alpha^{-1}} (\bm{y})
\right \|_{\mathrm{op}} \le C_E \left| \gamma_n-\alpha^{-1} \right| \to 0.
\]
Consequently,
\[
\begin{split} 
\left \| \frac{1}{n} \sum_{i=1}^n \bm{F}_{\gamma_n} (\bm{y}_i)  - \E_{\bm{y}} \left [ \bm{F}_{\alpha^{-1}} (\bm{y})  \right ] \right \|_{\mathrm{op}} & \le \operatorname*{ess \, sup}_{\bm{y}}  \left \| \bm{F}_{\gamma_n}(\bm{y}) - \bm{F}_{\alpha^{-1}}(\bm{y}) \right \|_{\mathrm{op}} \\
& \quad + \left \| \frac{1}{n} \sum_{i=1}^n  \bm{F}_{\alpha^{-1}} (\bm{y}_i)  - \E_{\bm{y}} \left [ \bm{F}_{\alpha^{-1}} (\bm{y})  \right ] \right \|_{\mathrm{op}}  .
\end{split}
\]
The first term tends to zero. The second tends to zero almost surely by the strong law of large numbers applied entrywise (here we use that \(p\) is fixed and that \(\bm{F}_{\alpha^{-1}}(\bm{y})\) is uniformly bounded). It follows from~\eqref{eq:Galpha_at_ZE} that
\[
\mathcal{G}_n(\bm{Z}_E) \to  E \bm{I}_p \quad \mathrm{a.s.}
\]
Combining this convergence with~\eqref{eq:finite_var_formula2}, we find that, almost surely, for all sufficiently large \(n\),
\[
\lambda_{+,n}^{\mathrm{free}} <  E+\frac{\varepsilon}{2} = \lambda_+(\alpha)+\varepsilon.
\]
Thus, \(\mathbf{1}_{\{ \lambda_{+,n}^{\mathrm{free}} \ge \lambda_+(\alpha)+\varepsilon\}} \to 0\) almost surely. Since these indicators are bounded by \(1\), taking the expectations and applying the dominated convergence theorem completes the proof.
\end{proof}

\begin{rmk}[Exact limiting variational formula]
The preceding proof uses the variational bound only at the particular test matrices \(\bm{Z}_E=-\bm{M}_\alpha(E)\) for \(E>\lambda_+(\alpha)\). In Appendix~\ref{appendix:limiting_variational_edge}, we show that the bound is
in fact exact at the limiting level. A related variational characterization of limiting spectral edges was obtained recently by Montanari and Saeed~\cite{montanarisaeed2026}.
\end{rmk}

Finally, we combine the comparison with the associated free model from
Lemma~\ref{lem:signed_free_comparison} and the preceding control of its upper spectral edge.

\begin{proof}[Proof of Proposition~\ref{prop:upper_edge}]
The lower bound follows from Proposition~\ref{prop:bulk}, so it remains only to prove the matching upper bound. Fix \(\varepsilon>0\), and let \((\varepsilon_n)_{n \ge 1}\) be the deterministic sequence provided by Lemma~\ref{lem:signed_free_comparison}. Since \(\varepsilon_n \to 0\), for all sufficiently large \(n\),
\[
\varepsilon_n<\frac{\varepsilon}{2}.
\]
For such \(n\), we have
\[
\left \{ \lambda_1 (\bm{P}_n) \ge \lambda_+ (\alpha) + \varepsilon \right \} \subseteq \left \{ \lambda_1 (\bm{P}_n) > \lambda_{+,n}^{\mathrm{free}} + \varepsilon_n \right \} \cup \left \{   \lambda_{+,n}^{\mathrm{free}} \ge \lambda_+ (\alpha) + \frac{\varepsilon}{2} \right \}.
\]
Taking probabilities and applying the union bound, the first term tends to zero by Lemma~\ref{lem:signed_free_comparison}, while the second tends to zero by Lemma~\ref{lem:free_upper_edge}. This proves the proposition.
\end{proof}

We conclude this section by completing the proof of Theorem~\ref{thm:bulk}.

\begin{proof}[Proof of the second assertion of Theorem~\ref{thm:bulk}]
The first part of the theorem was proved in Section~\ref{section:bulk}. It thus remains to control the eigenvalues above the upper bulk edge. By the block decomposition~\eqref{eq:D_n_prime}, \(\bm{P}_n\) is a principal submatrix of \(\bm{D}_n'\), obtained by deleting \(pr\) rows and columns. Cauchy's interlacing theorem (see e.g.~\cite[Theorem 4.3.15]{horn13}) gives
\[ 
\lambda_i(\bm{D}'_n) \ge \lambda_i(\bm{P}_n) \ge \lambda_{i+pr}(\bm{D}_n'), \quad i\in[p(d-r)].
\]
Since \(\bm{D}_n\) and \(\bm{D}_n'\) have the same eigenvalues (see Subsection~\ref{subsection:preliminaries}), we have \(\lambda_i(\bm{P}_n) \ge \lambda_{i+pr}(\bm{D}_n)\). In particular,
\[
\lambda_{pr+1}(\bm{D}_n) \leq \lambda_1(\bm{P}_n).
\]
Combining this inequality with Proposition~\ref{prop:upper_edge}, we obtain, for every fixed \(\varepsilon>0\),
\[
\mathbb{P}\left( \lambda_{pr+1}(\bm{D}_n) \leq \lambda_+(\alpha)+\varepsilon \right) \to 1.
\]
Hence, with high probability, at most \(pr\) eigenvalues of \(\bm{D}_n\) lie above \(\lambda_+(\alpha)+\varepsilon\). This completes the proof.
\end{proof}

\section{Proofs of the spectral phase transition} \label{section:phase_transition}

This section establishes the spectral phase transition for the largest eigenvalue of \(\bm{D}_n\). The argument has three main steps. We first prove Proposition~\ref{prop:outlier_equation}, which identifies the deterministic finite-dimensional equation governing eigenvalues above the bulk. We then use the monotonicity of this equation to prove Theorem~\ref{thm:BBP_transition} and the weak-recovery statement. Finally, we prove Proposition~\ref{prop:crit_thresholds} by showing subcriticality for sufficiently small sampling ratios and supercriticality for sufficiently large ones.

\subsection{Proof of the outlier equation}

We begin by proving Proposition~\ref{prop:outlier_equation}. By the Schur complement identity~\eqref{eq:schur_det_lemma}, any eigenvalue of \(\bm{D}_n\) lying outside the spectrum of \(\bm{P}_n\) is characterized by the singularity of the matrix \(\bm{H}_n(z)\) defined in~\eqref{eq:H_n}: 
\[
\bm{H}_n(z) = \bm{A}_n - \bm{Q}_n^\top \left(\bm{P}_n - z \bm{I}_{p(d-r)} \right)^{-1} \bm{Q}_n - z \bm{I}_{pr}\in \R^{pr \times pr},
\]
where the block matrices \(\bm{A}_n, \bm{Q}_n\), and \(\bm{P}_n\) are defined in~\eqref{eq:A_n},~\eqref{eq:Q_n}, and~\eqref{eq:P_n}, respectively. Proposition~\ref{prop:upper_edge} ensures that, with probability tending to one, this characterization applies throughout any compact subset of \((\lambda_+(\alpha),\infty)\). We begin by computing the deterministic equivalents of \(\bm{A}_n\) and \( \bm{Q}_n^\top \left(\bm{P}_n - z \bm{I}_{p(d-r)} \right)^{-1} \bm{Q}_n\) in the proportional asymptotic regime \(n,d \to \infty\), \(n/d \to \alpha\). 

Throughout this subsection, \(\bm{y}\) has the marginal distribution induced by
\[
\bm{s} \sim \mathcal{N}(\bm{0}_r, \bm{I}_r), \quad \bm{y} \mid \bm{s} \sim P_\ast (\, \cdot \mid \bm{s}).
\]
We also recall the definition~\eqref{eq:cond_second_moment}, i.e., \(\bm{C}(\bm{y}) = \E \left [ \bm{s} \bm{s}^\top \mid \bm{y}\right] \in \sym_r^+(\R)\).

\begin{lem} \label{lem:A_n}
Suppose that Assumptions~\ref{hyp:prop_limit} and~\ref{hyp:preprocessing} hold. Then
\[
\lim_{n\to \infty} \bm{A}_n = \E_{\bm{y}} \left[ \bm{T}(\bm{y})\otimes \bm{C} (\bm{y}) \right] \quad \mathrm{a.s.}
\]
\end{lem}

\begin{proof}
Recall that \(\bm{A}_n = \frac{1}{n} \sum_{i=1}^{n} \bm{X}_i  \in \R^{pr \times pr}\), where \(\bm{X}_i \coloneqq \bm{T}(\bm{y}_i) \otimes \bm{s}_i\bm{s}_i^\top\). By definition, the pairs \((\bm{s}_i,\bm{y}_i)\) are i.i.d., hence so are \((\bm{X}_i)_{i\ge1}\). Furthermore, using \(\|\bm{A}\otimes \bm{B}\|_{\mathrm{op}} = \|\bm{A}\|_{\mathrm{op}} \|\bm{B}\|_{\mathrm{op}}\) and \(\|\bm{s}\bm{s}^\top\|_{\mathrm{op}}=\|\bm{s}\|_2^2\), we have from Assumption~\ref{hyp:preprocessing} that
\[
\| \bm{X}_i\|_{\mathrm{op}} \le \|\bm{T} (\bm{y}_i)\|_{\mathrm{op}} \|\bm{s}_i\bm{s}_i^\top\|_{\mathrm{op}} \le C_T \|\bm{s}_i\|_2^2, \quad \mathrm{a.s.}
\]
Since \(\bm{s}_i \stackrel{\mathrm{i.i.d.}}{\sim} \mathcal{N}(\bm{0}_r,\bm{I}_r)\), it follows that \(\|\bm{s}_i\|_2^2 \sim \chi_r^2\) and \(\E [ \|\bm{s}_i\|_2^2 ] =r<\infty\), hence \(\E [ \|\bm{X}_i\|_{\mathrm{op}}] \le C_T r<\infty\). In particular, each entry of \(\bm{X}_i\) is integrable. Therefore, by the strong law of large numbers applied entrywise,
\[
\bm{A}_n=\frac{1}{n} \sum_{i=1}^n \bm{X}_i \to \E[ \bm{X}_1 ] \quad \text{almost surely entrywise.}
\]
Because \(p\) and \(r\) are fixed, entrywise convergence implies convergence in any matrix norm, in particular
\[
\|\bm{A}_n-\E[\bm{X}_1]\|_{\mathrm{op}}\to 0 \quad \mathrm{a.s.}
\]
Finally, by the tower property of conditional expectation and the fact that \(\bm{T}(\bm{y})\) is measurable with respect to \(\sigma\)-algebra \(\sigma (\bm{y})\),
\[
\E \left [\bm{X}_1 \right] = \E \left [\bm{T}(\bm{y}) \otimes \bm{s}\bm{s}^\top \right] = \E_{\bm{y}} \left[ \bm{T}(\bm{y}) \otimes \E[\bm{s}\bm{s}^\top \mid \bm{y}] \right] = \E_{\bm{y}} \left[ \bm{T}(\bm{y})\otimes \bm{C} (\bm{y}) \right],
\]
which completes the proof.
\end{proof}

We next turn to the deterministic equivalent of the second term \( \bm{Q}_n^\top \bm{G}_n(z) \bm{Q}_n\), where \(\bm{G}_n(z) = \left(\bm{P}_n - z \bm{I}_{p(d-r)} \right)^{-1} \) denotes the resolvent of \(\bm{P}_n\). 

\begin{lem} \label{lem:Q_n} 
Suppose that Assumptions~\ref{hyp:prop_limit} and~\ref{hyp:preprocessing} hold. Let \(K\subset(\lambda_+ (\alpha),\infty)\) be compact and \(\mathcal{E}_n(K) \coloneqq \left \{ K \cap \mathrm{sp}(\bm{P}_n)=\emptyset \right \}\). Then, for every \(\varepsilon>0\),
\[ 
\mathbb{P} \left ( \mathcal{E}_n (K) \cap \left \{ \sup_{z\in K} \left\| \bm{Q}_n^\top \bm{G}_n(z) \bm{Q}_n - \bm{\Gamma}_\alpha (z)  \right\|_{\mathrm{op}} > \varepsilon \right \} \right) \to 0,
\] 
where
\begin{equation} \label{eq:Gamma_alpha}
\bm{\Gamma}_\alpha (z) \coloneqq \frac{1}{\alpha} \E_{\bm{y}} \left[ \left( \bm{T}(\bm{y}) \bm{M}_\alpha(z) \left( \bm{I}_p+\frac{1}{\alpha}\bm{T}(\bm{y})\bm{M}_\alpha(z) \right)^{-1} \bm{T}(\bm{y}) \right) \otimes \bm{C}(\bm{y}) \right].
\end{equation}
\end{lem}

The proof of Lemma~\ref{lem:Q_n} has two distinct components, which we isolate below. From~\eqref{eq:Q_n} and the symmetry of \(\bm{T} (\bm{y})\),
\[
\bm{Q}_n^\top \bm{G}_n(z) \bm{Q}_n  = \frac{1}{n^2} \sum_{1 \le i,j \le n}\bm{T} (\bm{y}_i) (\bm{I}_p \otimes \bm{u}_i^\top) \bm{G}_n(z) (\bm{I}_p \otimes \bm{u}_j) \bm{T}(\bm{y}_j)  \otimes \bm{s}_i \bm{s}_j^\top.
\]
We decompose \(\bm{Q}_n^\top \bm{G}_n(z) \bm{Q}_n\) in terms of diagonal and off-diagonal terms, i.e., \(\bm{Q}_n^\top \bm{G}_n(z) \bm{Q}_n = \bm{\Gamma}_n^{\mathrm{diag}} (z) + \bm{\Gamma}_n^{\mathrm{off}} (z)  \), where
\[
\bm{\Gamma}_n^{\mathrm{diag}} (z) \coloneqq \frac{1}{n^2} \sum_{i=1}^n
\bm{T} (\bm{y}_i) (\bm{I}_p\otimes \bm{u}_i^\top) \bm{G}_n(z) (\bm{I}_p \otimes \bm{u}_i) \bm{T} (\bm{y}_i) \otimes \bm{s}_i \bm{s}_i^\top,
\]
and
\begin{equation} \label{eq:off_diagonal_term}
\bm{\Gamma}_n^{\mathrm{off}} (z) \coloneqq \frac{1}{n^2} \sum_{i\neq j} \bm{T} (\bm{y}_i) (\bm{I}_p \otimes \bm{u}_i^\top) \bm{G}_n(z) (\bm{I}_p \otimes \bm{u}_j) \bm{T} (\bm{y}_j) \otimes \bm{s}_i \bm{s}_j^\top.
\end{equation}

The diagonal term has a non-trivial deterministic limit, whereas the off-diagonal term vanishes. We record these two facts separately.

\begin{lem} \label{lem:Q_n_diagonal}
Under the assumptions of Lemma~\ref{lem:Q_n}, for every compact \(K\subset(\lambda_+(\alpha),\infty)\) and every \(\varepsilon>0\),
\[
\mathbb{P} \left(\mathcal{E}_n(K)\cap\left\{\sup_{z\in K}\|\bm\Gamma_n^{\mathrm{diag}}(z)-\bm\Gamma_\alpha(z)\|_{\mathrm{op}}>\varepsilon\right\}\right) \to 0.
\]
\end{lem}

\begin{proof}
\emph{Step 1: Reduction to the bulk deterministic equivalent.} Recall the definition of \(\widetilde{\bm{M}}_{n,i}^{(i)}(z)\) introduced in~\eqref{eq:tilde_M}. Using~\eqref{eq:Mtilde}, we rewrite \(\bm{\Gamma}_n^{\mathrm{diag}} (z)\) as
\begin{equation} \label{eq:diag_term}
\bm{\Gamma}_n^{\mathrm{diag}} (z)   = \frac{\gamma_n}{n} \sum_{i=1}^n \Psi_{\bm{T} (\bm{y}_i),\gamma_n} (\widetilde{\bm{M}}_{n,i}^{(i)}(z)) \otimes \bm{s}_i \bm{s}_i^\top  ,
\end{equation}
where \(\gamma_n \coloneqq (d-r)/n \ge0\) and 
\[
\Psi_{\bm{T},\gamma} (\bm{X}) \coloneqq \bm{T} \bm{X} \left ( \bm{I}_p + \gamma \bm{T} \bm{X}\right)^{-1} \bm{T} \in \C^{p \times p}.
\]
As in the proof of Proposition~\ref{prop:bulk}, we decompose \(\bm{\Gamma}_n^{\mathrm{diag}} (z)\) as 
\begin{equation}\label{eq:Gamma_diag_decomposition}
\bm{\Gamma}_n^{\mathrm{diag}}(z) - \gamma_n \E_{\bm{y}} \left [ \Psi_{\bm{T} (\bm{y}),\gamma_n} ( \bm{M}_n(z)) \otimes \bm{C} (\bm{y})  \right ] = \gamma_n \left ( \bm{R}_{1,n}(z)+\bm{R}_{2,n}(z)+\bm{R}_{3,n}(z) \right),
\end{equation}
where
\begin{equation} \label{eq:error_Gamma_n}
\begin{split}
\bm{R}_{1,n}(z) & = \frac{1}{n} \sum_{i=1}^n \left [ \Psi_{\bm{T}(\bm{y}_i),\gamma_n} (\widetilde{\bm{M}}_{n,i}^{(i)}(z)) - \Psi_{\bm{T}(\bm{y}_i) ,\gamma_n} ( \bm{M}_n^{(i)}(z)) \right ] \otimes \bm{s}_i \bm{s}_i^\top ,\\
\bm{R}_{2,n}(z) & =\frac{1}{n} \sum_{i=1}^n  \Psi_{\bm{T}(\bm{y}_i) ,\gamma_n} ( \bm{M}_n^{(i)}(z))   \otimes \bm{s}_i \bm{s}_i^\top  -  \frac{1}{n} \sum_{i=1}^n  \E_{\bm{y}} \left [\Psi_{\bm{T} (\bm{y}) ,\gamma_n} ( \bm{M}_n^{(i)}(z))   \otimes \bm{C} (\bm{y}) \right ] ,\\
\bm{R}_{3,n}(z) & =\frac{1}{n} \sum_{i=1}^n  \E_{\bm{y}} \left [\Psi_{\bm{T} (\bm{y}) ,\gamma_n} ( \bm{M}_n^{(i)}(z))   \otimes \bm{C} (\bm{y}) \right ] -  \E_{\bm{y}} \left [ \Psi_{\bm{T} (\bm{y}),\gamma_n} ( \bm{M}_n(z)) \otimes \bm{C} (\bm{y}) \right ].
\end{split}
\end{equation}

\emph{Step 2: Control of the three error terms.} Choose \(\eta_0>0\) as in Lemma~\ref{lem:stability-phi}. The proof of that lemma applies without change to the maps \(\Psi_{\bm{T},\gamma}\): almost surely, for all sufficiently large \(n\), all \(i\in[n]\), and all \(z\in\C_+\) with \(\Im z>\eta_0\),
\[
\left \| \Psi_{\bm{T}(\bm{y}_i),\gamma_n}(\bm{X}) - \Psi_{\bm{T}(\bm{y}_i),\gamma_n}(\bm{Y}) \right \|_{\mathrm{op}} \leq C_\Psi(\eta_0) \|\bm{X} -\bm{Y} \|_{\mathrm{op}}
\]
for the matrices appearing above. Indeed,
\[
\Psi_{\bm{T},\gamma}(\bm{X}) - \Psi_{\bm{T},\gamma}(\bm{Y}) = \bm{T} \left( \bm{I}_p + \gamma \bm{X} \bm{T}\right)^{-1} (\bm{X} - \bm{Y}) \left(\bm{I}_p + \gamma \bm{T} \bm{Y} \right)^{-1} \bm{T},
\]
and the inverse factors are uniformly bounded by the estimates already obtained in Lemma~\ref{lem:stability-phi}. By Lemma~\ref{lem:quadratic_form}, 
\[
\varepsilon_n (z) = \max_{1 \le i \le n} \|\widetilde{\bm{M}}_{n,i}^{(i)}(z) - \bm{M}_n^{(i)}(z)\|_{\mathrm{op}}\to 0,
\]
almost surely for every fixed \(z\in\C_+\). Therefore,
\[
\| \bm{R}_{1,n}(z)\|_{\mathrm{op}} \leq C_\Psi (\eta_0) \varepsilon_n(z) \frac{1}{n} \sum_{i=1}^n \|\bm{s}_i\|_2^2 \to 0
\]
almost surely, since \(\frac{1}{n} \sum_{i=1}^n \|\bm{s}_i\|_2^2 \to \E [\|\bm{s}\|_2^2] = r\) almost surely. The same Lipschitz estimate gives
\[
\left \| \E_{\bm{y}} \left [  \left ( \Psi_{\bm{T}(\bm{y}),\gamma_n}(\bm{X}) -   \Psi_{\bm{T}(\bm{y}),\gamma_n}(\bm{Y}) \right ) \otimes \bm{C} (\bm{y})\right]\right \|_{\mathrm{op}} \leq r C_\Psi(\eta_0) \|\bm{X} -\bm{Y} \|_{\mathrm{op}} ,
\]
where we used
\[
\E_{\bm{y}} \left [ \| \bm{C}(\bm{y})\|_{\mathrm{op}} \right ] \leq \E_{\bm{y}} \left [\Tr \bm{C}(\bm{y}) \right]= \E \|\bm{s}\|_2^2 = r.
\]
Consequently, Lemma~\ref{lem:leave-one-out-stability} yields
\[
\| \bm{R}_{3,n}(z)\|_{\mathrm{op}} \leq r C_\Psi (\eta_0) \max_{1\leq i\leq n}
\|\bm{M}_n^{(i)}(z)-\bm{M}_n(z) \|_{\mathrm{op}}  \to 0,
\]
almost surely. It remains to control \(\bm{R}_{2,n}(z)\). We use the bounded-differences argument of Lemma~\ref{lem:self_avg}, with an additional truncation of the Gaussian factors. More precisely, set \(b_n=A\log(n)\), with \(A>0\) sufficiently large, and define
\[
\bm{C}_{b_n}(\bm{y}) \coloneqq \E \left[ \bm{s} \bm{s}^\top \mathbf{1}_{\{\|\bm{s}\|_2^2\leq b_n\}} \mid \bm{y} \right].
\]
Consider the truncated centered sum
\[
\bm{R}_{2,n}^{(b)}(z) \coloneqq \frac{1}{n} \sum_{i=1}^n \left \{ \Psi_{\bm{T} (\bm{y}_i),\gamma_n} \left (\bm{M}_n^{(i)}(z) \right) \otimes \bm{s}_i \bm{s}_i^\top \mathbf{1}_{\{\|\bm{s}_i\|_2^2\leq b_n\}} -  \E_{\bm{y}} \left[ \Psi_{\bm{T} (\bm{y}),\gamma_n} \left (\bm{M}_n^{(i)}(z)\right) \otimes \bm{C}_{b_n}(\bm{y}) \right] \right \}.
\]
Conditionally on \((\bm{u}_1,\ldots,\bm{u}_n)\), the pairs \((\bm{s}_i,\bm{y}_i)\) remain independent. Moreover, \(\bm{A}_n^{(i)}(z)\) is independent of \((\bm{s}_i,\bm{y}_i)\); conditioning further on all coordinates except the \(i\)-th therefore shows that the \(i\)-th summand has conditional mean zero. In particular,
\[
\E \left[\bm{R}_{2,n}^{(b)}(z)\mid \bm{u}_1,\ldots,\bm{u}_n\right]=\bm{0}.
\]
The bounded-differences argument of Lemma~\ref{lem:self_avg} now applies. On the high-probability event controlling the Gaussian vectors \(\bm{u}_i\), changing one coordinate \((\bm{s}_k,\bm{y}_k)\) changes each fixed entry of \(\bm{R}_{2,n}^{(b)}(z)\) by at most \(C b_n /n\). Therefore, for every fixed entry \((a,b)\),
\[
\mathbb{P} \left( \left| \left (\bm{R}_{2,n}^{(b)}(z) \right)_{ab} \right|>t \right) \leq
4 \exp\left( -\frac{cnt^2}{b_n^2} \right) + o(n^{-2}).
\]
Taking \(t_n=Kb_n\sqrt{\frac{\log(n)}{n}}\) with \(K\) sufficiently large, the resulting probabilities are summable and \(t_n\to0\). Hence
\[
\|\bm{R}_{2,n}^{(b)}(z) \|_{\mathrm{op}} \to 0, \quad \mathrm{a.s.}
\]
It remains to remove the truncation. Since \(\|\bm{s}_i\|_2^2\sim\chi_r^2\), choosing \(A\) sufficiently large gives
\[
\sum_{n=1}^\infty
\mathbb{P}\!\left(\max_{1\le i\le n}\|\bm{s}_i\|_2^2>b_n\right)<\infty.
\]
Thus, by Borel--Cantelli, almost surely
\(\max_{1\le i\le n}\|\bm{s}_i\|_2^2\le b_n\) for all sufficiently large \(n\). On this event, the empirical terms in \(\bm{R}_{2,n}\) and \(\bm{R}_{2,n}^{(b)}\) coincide. For the deterministic centering terms, the uniform bound on \(\Psi_{\bm{T},\gamma_n}\) gives
\[
\begin{split}
&\left \| \frac{1}{n}\sum_{i=1}^n \E_{\bm{y}} \left[ \Psi_{\bm{T}(\bm{y}),\gamma_n}(\bm{M}_n^{(i)}(z)) \otimes \bigl(\bm{C}(\bm{y})-\bm{C}_{b_n}(\bm{y})\bigr) \right] \right\|_{\mathrm{op}}\le
C_z \E \left[ \|\bm{s}\|_2^2 \mathbf{1}_{\{\|\bm{s}\|_2^2>b_n\}} \right],
\end{split}
\]
and the right-hand side vanishes. As a result,
\[
\|\bm{R}_{2,n}(z)-\bm{R}_{2,n}^{(b)}(z)\|_{\mathrm{op}}\to0
\qquad\text{a.s.},
\]
and hence \(\|\bm{R}_{2,n}(z)\|_{\mathrm{op}}\to0\) almost surely.

\emph{Step 3: Passage to the limiting MDE.} Combining the estimates for the three error terms in~\eqref{eq:Gamma_diag_decomposition}, we obtain, for every fixed
\(z\in\C_+\) with \(\Im z>\eta_0\),
\begin{equation}\label{eq:Gamma_diag_approximation}
\left \|\bm{\Gamma}_n^{\mathrm{diag}} (z)   - \gamma_n \E_{\bm{y}} \left [ \Psi_{\bm{T} (\bm{y}),\gamma_n} ( \bm{M}_n(z)) \otimes \bm{C} (\bm{y})  \right ] \right \|_{\mathrm{op}} \to 0,
\end{equation}
almost surely. By Proposition~\ref{prop:bulk}, \(\bm{M}_n(z)\to\bm{M}_\alpha(z)\) almost surely, and \(\gamma_n\to\alpha^{-1}\). Moreover, the uniform inverse estimates imply
\[
\left \| \Psi_{\bm{T}(\bm{y}),\gamma_n} (\bm{M}_n(z)) \otimes \bm{C}(\bm{y}) \right \|_{\mathrm{op}} \leq C_z \|\bm{C}(\bm{y})\|_{\mathrm{op}},
\]
where the right-hand side is integrable. Dominated convergence applied to~\eqref{eq:Gamma_diag_approximation} therefore gives
\begin{equation} \label{eq:Gamma_diag_pointwise}
\bm{\Gamma}_n^{\mathrm{diag}}(z) \to  \bm{\Gamma}_\alpha(z)
\quad\text{a.s.}
\end{equation}
for every fixed \(z\in\C_+\) with \(\Im z>\eta_0\), where \( \bm{\Gamma}_\alpha(z)\) is given in~\eqref{eq:Gamma_alpha}.

\emph{Step 4: Uniformity above the bulk.} We finally extend this convergence to compact subsets above the upper spectral edge. Let \(K\subset(\lambda_+(\alpha),\infty)\) be compact, and choose
\(\delta>0\) such that \(\inf K>\lambda_+(\alpha)+4\delta\). Define
\[
\mathcal{F}_n (\delta) \coloneqq \left \{ \lambda_1(\bm{P}_n) \leq \lambda_+(\alpha)+\delta \right\}
\]
and
\[
\Omega_\delta \coloneqq \left \{ z\in \C \colon \Re z>\lambda_+(\alpha)+2\delta \right \}.
\]
By Proposition~\ref{prop:upper_edge},
\[
\mathbb{P} \left (\mathcal{F}_n(\delta)^\mathrm{c} \right)\to 0.
\]
We claim that, for every \(\varepsilon>0\),
\[
\mathbb{P} \left( \mathcal{F}_n (\delta)\cap \left \{ \sup_{z\in K} \|\bm{\Gamma}_n^{\mathrm{diag}}(z) -\bm{\Gamma}_\alpha(z) \|_{\mathrm{op}} > \varepsilon \right \} \right) \to 0.
\]
To prove the claim, consider an arbitrary subsequence. It contains a further subsequence, still denoted by \((n_k)\), such that
\[
\sum_{k=1}^\infty \mathbb{P} \left (\mathcal{F}_{n_k} (\delta)^\mathrm{c} \right) < \infty.
\]
By the Borel--Cantelli lemma, almost surely \(\mathcal{F}_{n_k}(\delta)\) holds for all sufficiently large \(k\). On this event, \(\bm{\Gamma}_{n_k}^{\mathrm{diag}}\) is analytic on
\(\Omega_\delta\). Moreover, for every compact \(L\subset\Omega_\delta\),
\[
\begin{split}
\sup_{z\in L} \| \bm{\Gamma}_{n_k}^{\mathrm{diag}}(z) \|_{\mathrm{op}} & \leq C_L C_T^2 \frac{1}{n_k^2} \sum_{i=1}^{n_k} \|\bm{u}_i\|_2^2 \|\bm{s}_i\|_2^2\leq C_LC_T^2 \gamma_{n_k} \left( \max_{1\leq i\leq n_k} \frac{\|\bm{u}_i\|_2^2}{d-r} \right) \left( \frac{1}{n_k}\sum_{i=1}^{n_k} \|\bm{s}_i\|_2^2 \right),
\end{split}
\]
which is almost surely bounded. Thus, almost surely, the sequence
\(\{\bm{\Gamma}_{n_k}^{\mathrm{diag}}\}_{k\geq 1}\) is locally uniformly
bounded on \(\Omega_\delta\). Let \(\mathcal{D}\) be a countable subset of \(\Omega_\delta \cap \{z \in \C \colon \Im z>\eta_0\}\) having an accumulation point in \(\Omega_\delta\). For every \(z\in\mathcal{D}\), the pointwise convergence in~\eqref{eq:Gamma_diag_pointwise} holds on an event of probability one. Since \(\mathcal{D}\) is countable, there exists an event \(\mathcal{A}_{\mathcal{D}}\) of probability one such that, on \(\mathcal{A}_{\mathcal{D}}\),
\[
\bm{\Gamma}_{n_k}^{\mathrm{diag}}(z) \to  \bm{\Gamma}_\alpha(z)
\]
for every \( z\in\mathcal{D}\).

We next verify that \(\bm{\Gamma}_\alpha\) is analytic on \(\Omega_\delta\). By Lemma~\ref{lem:continuation_off_support}, the map \(z\mapsto\bm{M}_\alpha(z)\) admits an analytic continuation to \(\C \setminus \supp(\mu_\alpha)\). Since \(\Omega_\delta \subset
\C \setminus \supp(\mu_\alpha)\), the map \(\bm{M}_\alpha\) is analytic on \(\Omega_\delta\). Moreover, by Lemma~\ref{lem:uniform_denominator_bound}, for every compact set \(L\subset\Omega_\delta\), there exists a constant \(C_L<\infty\) such
that
\[
\sup_{z\in L} \left\| \left( \bm{I}_p+\frac{1}{\alpha} \bm{T}(\bm{y})\bm{M}_\alpha(z) \right)^{-1} \right\|_{\mathrm{op}} \leq C_L
\]
for almost every \(\bm{y}\). Consequently, for almost every \(\bm{y}\), the map
\[
z \mapsto \bm{T}(\bm{y})\bm{M}_\alpha(z) \left( \bm{I}_p+\frac{1}{\alpha} \bm{T}(\bm{y})\bm{M}_\alpha(z) \right)^{-1} \bm{T}(\bm{y}) \otimes\bm{C}(\bm{y})
\]
is analytic on \(\Omega_\delta\). Furthermore, for every compact set \(L\subset\Omega_\delta\), its operator norm is bounded uniformly over \(z\in L\) by an integrable multiple of \(\|\bm{C}(\bm{y})\|_{\mathrm{op}}\). Indeed,
\[
\sup_{z\in L} \left\| \bm{T}(\bm{y})\bm{M}_\alpha(z) \left( \bm{I}_p+\frac{1}{\alpha} \bm{T}(\bm{y})\bm{M}_\alpha(z) \right)^{-1} \bm{T}(\bm{y}) \otimes\bm{C}(\bm{y}) \right\|_{\mathrm{op}} \leq C_L' \|\bm{C}(\bm{y})\|_{\mathrm{op}} \le C_L' r,
\]
where \(C_L'<\infty\) is deterministic. Thus, the dominating function is integrable. It follows that \(\bm{\Gamma}_\alpha\) is analytic on \(\Omega_\delta\).

We now fix an outcome belonging to the probability-one event on which both the local uniform boundedness of \(\{\bm{\Gamma}_{n_k}^{\mathrm{diag}}\}_{k\geq 1}\) and the pointwise
convergence on \(\mathcal{D}\) hold. By Vitali's convergence theorem,
applied entrywise, the sequence \(\{\bm{\Gamma}_{n_k}^{\mathrm{diag}}\}_{k\geq 1}\) converges locally uniformly on \(\Omega_\delta\) to an analytic matrix-valued function \(\widetilde{\bm{\Gamma}}\). For every \(z\in\mathcal{D}\),
\[
\widetilde{\bm{\Gamma}}(z) = \bm{\Gamma}_\alpha(z).
\]
Since \(\Omega_\delta\) is connected and \(\mathcal{D}\) has an accumulation point in \(\Omega_\delta\), the identity theorem, applied entrywise, yields
\[
\widetilde{\bm{\Gamma}}(z) = \bm{\Gamma}_\alpha(z), \quad z\in\Omega_\delta.
\]
Therefore,
\[
\bm{\Gamma}_{n_k}^{\mathrm{diag}}(z) \to \bm{\Gamma}_\alpha(z)
\]
locally uniformly on \(\Omega_\delta\), almost surely. In particular, since \(K\subset\Omega_\delta\),
\[
\sup_{z\in K} \left\| \bm{\Gamma}_{n_k}^{\mathrm{diag}}(z) - \bm{\Gamma}_\alpha(z) \right\|_{\mathrm{op}} \to  0, \quad \mathrm{a.s.}
\]
Since every subsequence admits a further subsequence along which the preceding almost-sure convergence holds, the subsequence characterization of convergence in probability proves the claim.

Finally, \(\mathcal{F}_n (\delta) \subseteq \mathcal{E}_n(K)\), and hence
\[
\begin{split}
&\mathbb{P} \left( \mathcal{E}_n(K) \cap \left \{ \sup_{z\in K} \|\bm{\Gamma}_n^{\mathrm{diag}}(z) -\bm{\Gamma}_\alpha(z) \|_{\mathrm{op}} > \varepsilon \right \} \right) \\
&\leq  \mathbb{P} \left (\mathcal{F}_n(\delta)^\mathrm{c} \right)+ \mathbb{P}\left( \mathcal{F}_n(\delta)\cap \left \{ \sup_{z\in K} \|\bm{\Gamma}_n^{\mathrm{diag}}(z) -\bm{\Gamma}_\alpha(z) \|_{\mathrm{op}} > \varepsilon \right \} \right) \to 0.
\end{split}
\]
This proves the desired uniform deterministic equivalent for the diagonal term.
\end{proof}

Next we prove that the off-diagonal term~\eqref{eq:off_diagonal_term} is negligible. 
\begin{lem} \label{lem:off_diagonal}
Suppose that Assumptions~\ref{hyp:prop_limit} and~\ref{hyp:preprocessing} hold. Let
\(K\subset(\lambda_+(\alpha),\infty)\) be compact, and define \(\mathcal{E}_n(K) \coloneqq \{K\cap \mathrm{sp}(\bm{P}_n)=\emptyset\}\). Then, for every \(\varepsilon>0\),
\[
\mathbb{P} \left ( \mathcal{E}_n(K) \cap \left \{ \sup_{z \in K} \| \bm{\Gamma}_n^{\mathrm{off}} (z) \|_{\mathrm{op}} > \varepsilon \right \} \right) \to 0 .
\]
\end{lem}

\begin{proof}
\emph{Step 1: Leave-two-out representation.} For \(i \neq j\), define the leave-two-out matrix by 
\begin{equation*}
\bm{P}_n^{(ij)} \coloneqq \bm{P}_n - \frac{1}{n} \bm{T}_i\otimes  \bm{u}_i \bm{u}_i^\top - \frac{1}{n} \bm{T}_j \otimes  \bm{u}_j \bm{u}_j^\top, 
\end{equation*} 
and its resolvent by \(\bm{G}_n^{(ij)} (z)\coloneqq (\bm{P}_n^{(ij)} - z \bm{I}_{p(d-r)})^{-1}\). Furthermore, let
\[
\bm{U}_{ij} \coloneqq [\bm{I}_p \otimes \bm{u}_i \quad \bm{I}_p \otimes \bm{u}_j] \in \R^{p(d-r) \times 2p}, 
\]
and
\[
\bm{T}_{ij} \coloneqq \mathrm{diag} (\bm{T}_i , \bm{T}_j  ) \in \R^{2p \times 2p}. 
\]
Woodbury's identity yields 
\begin{equation} \label{eq:woodbury_ij}
\bm{G}_n (z) = \bm{G}_n^{(ij)} (z) - \frac{1}{n} \bm{G}_n^{(ij)} (z) \bm{U}_{ij} \left ( \bm{I}_{2p} + \bm{T}_{ij} \widetilde{\bm{M}}_{ij}(z) \right)^{-1} \bm{T}_{ij} \bm{U}_{ij}^\top \bm{G}_n^{(ij)} (z) ,
\end{equation}
where 
\[
\widetilde{\bm{M}}_{ij} (z) \coloneqq \frac{1}{n} \bm{U}_{ij}^\top \bm{G}_n^{(ij)}(z) \bm{U}_{ij} =
\begin{pmatrix}
\bm{A}_{ij}(z) & \bm{B}_{ij}(z)\\
\bm{C}_{ij}(z) & \bm{D}_{ij}(z)
\end{pmatrix}.
\]
Here,
\[
\bm{A}_{ij}(z)=\frac{1}{n}(\bm{I}_p\otimes \bm{u}_i^\top)\bm{G}_n^{(ij)}(z)(\bm{I}_p\otimes \bm{u}_i), \qquad \bm{B}_{ij}(z)=\frac{1}{n}(\bm{I}_p\otimes \bm{u}_i^\top)\bm{G}_n^{(ij)}(z)(\bm{I}_p\otimes \bm{u}_j),
\]
\[
\bm{C}_{ij}(z)=\frac{1}{n}(\bm{I}_p\otimes \bm{u}_j^\top)\bm{G}_n^{(ij)}(z)(\bm{I}_p\otimes \bm{u}_i), \qquad \bm{D}_{ij}(z)=\frac{1}{n}(\bm{I}_p\otimes \bm{u}_j^\top)\bm{G}_n^{(ij)}(z)(\bm{I}_p\otimes \bm{u}_j).
\]
From~\eqref{eq:woodbury_ij} and after some manipulations, we have
\begin{equation} \label{eq:woodbury_ij_2}
\frac{1}{n} \bm{U}_{ij}^\top \bm{G}_n(z) \bm{U}_{ij} = \widetilde{\bm{M}}_{ij} (z) \left ( \bm{I}_{2p} + \bm{T}_{ij} \widetilde{\bm{M}}_{ij} (z) \right)^{-1}.
\end{equation}
Define 
\[
\bm{F}_{ij} (z) \coloneqq \frac{1}{n} (\bm{I}_p \otimes \bm{u}_i^\top) \bm{G}_n(z) (\bm{I}_p \otimes \bm{u}_j).
\]
We note that \(\bm{F}_{ij} (z)\) is the \((1,2)\) block of the matrix \(\frac{1}{n} \bm{U}_{ij}^\top \bm{G}_n(z) \bm{U}_{ij}\). That is, from~\eqref{eq:woodbury_ij_2},
\[
\bm{F}_{ij} (z)  = \left [\widetilde{\bm{M}}_{ij} (z) \left ( \bm{I}_{2p} + \bm{T}_{ij} \widetilde{\bm{M}}_{ij} (z) \right)^{-1} \right ]_{12}.
\]
A block inversion gives 
\[
\bm{F}_{ij} (z) = (\bm{I}_p + \bm{A}_{ij} (z) \bm{T}_i)^{-1} \bm{B}_{ij}(z) \bm{S}_{ij}(z)^{-1},
\]
where
\[
\bm{S}_{ij} (z) \coloneqq \bm{I}_p + \bm{T}_j \bm{D}_{ij} (z) - \bm{T}_j \bm{C}_{ij} (z) (\bm{I}_p + \bm{T}_i \bm{A}_{ij} (z) )^{-1} \bm{T}_i \bm{B}_{ij} (z).
\]
Therefore, 
\begin{equation} \label{eq:norm_bound_F_ij}
\| \bm{F}_{ij} (z) \|_{\mathrm{op}} \le \| (\bm{I}_p + \bm{A}_{ij} (z) \bm{T}_i)^{-1}\|_{\mathrm{op}} \| \bm{B}_{ij}(z)\|_{\mathrm{op}} \|\bm{S}_{ij}(z)^{-1}\|_{\mathrm{op}}.
\end{equation}
The two inverse factors can be bounded directly from resolvent identities. Adding the \(i\)-th observation to \(\bm{P}_n^{(ij)}\) (or, equivalently, removing only the \(j\)-th observation to \(\bm{P}_n\)) yields
\[
\bm{G}_n^{(j)} (z) = \left ( \bm{P}_n^{(ij)} + \frac{1}{n} (\bm{I}_p \otimes \bm{u}_i) \bm{T}_i (\bm{I}_p \otimes \bm{u}_i^\top) - z \bm{I}_{p(d-r)} \right )^{-1}.
\]
Then Woodbury implies
\[
(\bm{I}_p + \bm{A}_{ij} (z) \bm{T}_i)^{-1} = \bm{I}_p - \frac{1}{n} (\bm{I}_p \otimes \bm{u}_i^\top) \bm{G}_n^{(j)} (z)  (\bm{I}_p \otimes \bm{u}_i) \bm{T}_i.
\]
Hence
\begin{equation} \label{eq:norm_bound_A_ij}
\| (\bm{I}_p + \bm{A}_{ij} (z) \bm{T}_i)^{-1} \|_{\mathrm{op}} \le 1 + \frac{C_T}{n} \|\bm{I}_p \otimes \bm{u}_i\|_{\mathrm{op}}^2 \|\bm{G}_n^{(j)} (z)\|_{\mathrm{op}} \le 1 + \frac{C_T \|\bm{u}_i\|_2^2}{n (\Im z)}.
\end{equation}
Moreover, since \(\bm{S}_{ij}(z)^{-1}\) is the \((2,2)\)-block of \((\bm{I}_{2p} + \bm{T}_{ij} \widetilde{\bm{M}}_{ij} (z))^{-1}\), we have
\[
\| \bm{S}_{ij}(z)^{-1} \|_{\mathrm{op}} \le \left \| (\bm{I}_{2p} + \bm{T}_{ij} \widetilde{\bm{M}}_{ij} (z))^{-1}\right\|_{\mathrm{op}}.
\]
Using Woodbury we have
\[
(\bm{I}_{2p} + \bm{T}_{ij} \widetilde{\bm{M}}_{ij} (z))^{-1} = \bm{I}_{2p} - \frac{1}{n} \bm{T}_{ij} \bm{U}_{ij}^\top \bm{G}_n(z) \bm{U}_{ij},
\]
so that
\begin{equation} \label{eq:norm_bound_S_ij_inverse}
\| \bm{S}_{ij}(z)^{-1} \|_{\mathrm{op}} \le 1 + \frac{C_T}{n} \|\bm{U}_{ij}\|_{\mathrm{op}}^2 \|\bm{G}_n(z)\|_{\mathrm{op}} \le 1 + \frac{C_T \left (\|\bm{u}_i\|_2^2 + \|\bm{u}_j\|_2^2 \right)}{n (\Im z)}.
\end{equation}
Combining~\eqref{eq:norm_bound_A_ij} and~\eqref{eq:norm_bound_S_ij_inverse} into~\eqref{eq:norm_bound_F_ij} and using the fact that \(\max_i \frac{\|\bm{u}_i\|_2^2}{n} = O(1)\) almost surely, for every fixed \(z \in \C_+\), there exists a deterministic constant \(C_z < \infty\) such that
\begin{equation}\label{eq:norm_bound_F_ij_2}
\|\bm{F}_{ij}(z)\|_{\mathrm{op}} \leq C_z\|\bm{B}_{ij}(z)\|_{\mathrm{op}},
\end{equation}
uniformly over \(i \neq j\), eventually almost surely. More explicitly, one can take 
\[
C_z = \left(1+\frac{C_T L}{\Im z} \right)\left(1+\frac{2C_T L}{\Im z} \right),
\]
on the almost-sure event \(\max_{1 \le i \le n} \frac{\|\bm{u}_i\|_2^2}{n} \le L\).

\emph{Step 2: Second moment.} Fix \(\mu,\nu\in[p]\) and \(a,b \in [r]\), and define
\[
X_{ij} (z) \coloneqq \left (\bm{e}_\mu^\top \bm{T}_i \bm{F}_{ij}(z) \bm{T}_j \bm{e}_\nu \right) (\bm{s}_i)_a(\bm{s}_j)_b.
\]
Then
\begin{equation} \label{eq:off_diag_entrywise}
\left ( \bm{\Gamma}_n^{\mathrm{off}} (z)  \right)_{(\mu,a),(\nu,b)} = \frac{1}{n} \sum_{i\neq j} X_{ij} (z).
\end{equation}
Let \(\mathcal{F}_{ij} \coloneqq \sigma \left ( (\bm{u}_k,\bm{s}_k,\bm{y}_k) \colon k\notin\{i,j\} \right)\). Then \(\bm{G}_n^{(ij)}(z)\) is \(\mathcal{F}_{ij}\)-measurable, while \(\bm{u}_i\) and \(\bm{u}_j\) are independent standard Gaussian vectors, independent of \(\mathcal{F}_{ij}\). Indeed, each entry of \(\bm{B}_{ij}(z)\) is of the form
\[
\frac{1}{n} \bm{u}_i^\top \bm{G}_{\mu\nu}^{(ij)}(z)\bm{u}_j,
\]
where, conditionally on \(\mathcal{F}_{ij}\), the two Gaussian vectors are independent and the matrix in the middle is deterministic. Standard Gaussian moment bounds for bilinear forms therefore give, for every fixed integer \( k \ge 1\),
\[
\E \left[ \|\bm{B}_{ij}(z)\|_{\mathrm{op}}^{2k} \mid \mathcal{F}_{ij} \right] \leq \frac{C_{k,z}}{n^k}.
\]
Here, we used that $p$ is fixed, \(\|\bm{G}_n^{(ij)}(z)\|_{\mathrm{F}} \le \sqrt{d-r} (\Im z)^{-1}\), and \(d-r\asymp n\). Moreover, the bounds~\eqref{eq:norm_bound_A_ij} and~\eqref{eq:norm_bound_S_ij_inverse} imply
\[
\|\bm{F}_{ij}(z)\|_{\mathrm{op}} \leq \left( 1+\frac{C_T\|\bm{u}_i\|_2^2}{n\Im z} \right) \left( 1 + \frac{C_T(\|\bm{u}_i\|_2^2+\|\bm{u}_j\|_2^2)}{n\Im z} \right) \|\bm{B}_{ij}(z)\|_{\mathrm{op}}.
\]
Since the normalized Gaussian norms \(\|\bm{u}_i\|_2^2/n\) have uniformly bounded moments of every fixed order, Hölder's inequality gives
\[
\E \left [ \|\bm{F}_{ij}(z) \|_{\mathrm{op}}^2\right ] \leq \frac{C_z}{n}.
\]
Finally, using the boundedness of \(\bm{T} (\bm{y}_i)\), the independence of the Gaussian vectors \(\bm{u}_i,\bm{u}_j\) from \((\bm{s}_i,\bm{y}_i,\bm{s}_j,\bm{y}_j)\), and the finiteness of all Gaussian moments, we obtain
\begin{equation} \label{eq:moments_X_ij}
\E [|X_{ij}(z)|^2] \leq \frac{C_z}{n}.
\end{equation}
We next use the exact sign symmetry of the model. For arbitrary \(\varepsilon_1,\ldots, \varepsilon_n \in \{-1,1\}\), replace \(\bm{u}_k\) by \(\varepsilon_k \bm{u}_k\). The matrix \(\bm{P}_n\), and hence \(\bm{G_n}(z)\), is unchanged because it depends on \(\bm{u}_k\) only through \(\bm{u}_k \bm{u}_k^\top\). We note that in this case \(\bm{F}_ij (z)\) is replaced by \(\varepsilon_i \varepsilon_j \bm{F}_{ij} (z)\) and similarly \(X_{ij}(z)\) by \(\varepsilon_i \varepsilon_j X_{ij}(z) \). Since the joint distribution is invariant under these sign changes, 
\[
\E [X_{ij}  \overline{X_{k\ell}}] =0,
\]
unless \(\{i,j\} = \{k,\ell\}\). Indeed, otherwise one of the indices appears exactly once, and flipping the sign of the corresponding Gaussian vector changes the sign of the product while preserving
its distribution. Hence using~\eqref{eq:moments_X_ij} and the Cauchy--Schwarz inequality,
\[
\E \left [ \left| \frac{1}{n} \sum_{i\neq j} X_{ij} (z) \right|^2 \right ]  =  \frac{1}{n^2} \sum_{i\neq j} \sum_{k\neq \ell}  \E \left [X_{ij} \overline{X_{k\ell}} \right ] \le \frac{C}{n^2} \sum_{i \neq j} \E [|X_{ij}|^2] \le \frac{C_z}{n}.
\]
Thus each entry of \(\bm{\Gamma}_n^{\mathrm{off}} (z)\) converges to zero in \(L^2\), and therefore in probability. Since \(p\) and \(r\) are fixed, there are only finitely many entries, and hence for every fixed \(z \in \C_+\), \(\| \bm{\Gamma}_n^{\mathrm{off}} (z)\|_{\mathrm{op}} \to 0\) in probability.

\emph{Step 3: Uniformity above the bulk.} It remains to upgrade the preceding pointwise convergence to uniform convergence on \(K\). As in the proof of Lemma~\ref{lem:Q_n}, choose \(\delta>0\) such that \(\inf K>\lambda_+(\alpha)+4\delta\), and define
\[
\mathcal{F}_n(\delta) \coloneqq \left\{ \lambda_1(\bm{P}_n) \leq \lambda_+(\alpha)+\delta \right\}, \quad
\Omega_\delta \coloneqq \left\{ z\in\C \colon \Re z>\lambda_+(\alpha)+2\delta \right\}.
\]
By Proposition~\ref{prop:upper_edge},
\[
\mathbb{P}\left(\mathcal{F}_n(\delta)^{\mathrm{c}}\right) \to 0.
\]
On the event \(\mathcal{F}_n(\delta)\), the function \(\bm{\Gamma}_n^{\mathrm{off}}\) is analytic on \(\Omega_\delta\). We also need local uniform boundedness. Since
\[
\bm{\Gamma}_n^{\mathrm{off}}(z)
=
\bm{Q}_n^\top\bm{G}_n(z)\bm{Q}_n-\bm{\Gamma}_n^{\mathrm{diag}}(z),
\]
it is enough to bound the first term. The matrix \(\bm{Q}_n\) is an off-diagonal block of \(\bm{D}_n'\), hence
\[
\|\bm{Q}_n\|_{\mathrm{op}}\le \|\bm{D}_n\|_{\mathrm{op}}.
\]
Since \(-C_T\bm{I}_p\preceq\bm{T}(\bm{y}_i)\preceq C_T\bm{I}_p\),
\[
-C_T(\bm{I}_p\otimes\bm{S}_n)
\preceq
\bm{D}_n
\preceq
C_T(\bm{I}_p\otimes\bm{S}_n),
\qquad
\bm{S}_n=\frac1n\sum_{i=1}^n\bm{x}_i\bm{x}_i^\top,
\]
and therefore
\[
\|\bm{D}_n\|_{\mathrm{op}}
\le C_T\|\bm{S}_n\|_{\mathrm{op}}.
\]
The Gaussian sample-covariance norm is almost surely bounded for all sufficiently large \(n\); see, e.g.,~\cite{yin1988}. Moreover, on \(\mathcal{F}_n(\delta)\), for every compact \(L\subset\Omega_\delta\),
\[
\sup_{z\in L}\|\bm{G}_n(z)\|_{\mathrm{op}}\le C_L.
\]
Together with the local bound already proved for \(\bm\Gamma_n^{\mathrm{diag}}\), this shows that
\(\{\bm\Gamma_n^{\mathrm{off}}\}_n\) is locally uniformly bounded on \(\Omega_\delta\) along subsequences on which \(\mathcal{F}_n(\delta)\) holds eventually. For every
fixed \(z\in\C_+\), we have already proved that
\[
\bm{\Gamma}_n^{\mathrm{off}}(z) \stackrel{\mathbb{P}}{\to}
\bm{0}.
\]
We may therefore repeat the subsequence--Vitali argument used for the diagonal term. More precisely, from any subsequence one may extract a further subsequence along which
\(\mathcal{F}_n(\delta)\) holds eventually almost surely and
\[
\bm{\Gamma}_n^{\mathrm{off}}(z) \to  \bm{0},
\]
almost surely for every \(z\) in a countable subset of \(\Omega_\delta\cap\C_+\) having an accumulation point. Vitali's theorem, applied entrywise, then yields
\[
\bm{\Gamma}_n^{\mathrm{off}}(z) \to  \bm{0},
\]
locally uniformly on \(\Omega_\delta\) along this further subsequence. Since \(K\subset\Omega_\delta\), it follows that
\[
\sup_{z\in K} \|\bm{\Gamma}_n^{\mathrm{off}}(z)\|_{\mathrm{op}} \to 0,
\]
almost surely along the further subsequence. By the subsequence characterization of convergence in probability,
\[
\mathbb{P} \left( \mathcal{F}_n(\delta) \cap \left \{ \sup_{z\in K} \|\bm{\Gamma}_n^{\mathrm{off}}(z)\|_{\mathrm{op}} >\varepsilon \right\} \right) \to 0.
\]
Finally, since \(\mathcal{F}_n(\delta)\subseteq \mathcal{E}_n(K)\),
\[
\begin{split}
&\mathbb{P}\left( \mathcal{E}_n(K) \cap \left\{ \sup_{z\in K} \|\bm{\Gamma}_n^{\mathrm{off}} (z)\|_{\mathrm{op}}
>\varepsilon \right\} \right)\\
&\leq \mathbb{P} \left(\mathcal{F}_n(\delta)^{\mathrm{c}}\right) + \mathbb{P}\left( \mathcal{F}_n(\delta) \cap \left \{ \sup_{z\in K} \|\bm{\Gamma}_n^{\mathrm{off}}(z)\|_{\mathrm{op}} >\varepsilon \right\} \right)
\to 0.
\end{split}
\]
\end{proof}

\begin{proof}[Proof of Lemma~\ref{lem:Q_n}]
Fix \(K\subset(\lambda_+(\alpha),\infty)\) as in the statement. By construction,
\[
\bm Q_n^\top\bm G_n(z)\bm Q_n
=
\bm\Gamma_n^{\mathrm{diag}}(z)
+
\bm\Gamma_n^{\mathrm{off}}(z).
\]
Therefore, for every \(\varepsilon>0\),
\[
\begin{split}
&\mathbb P\!\left(
\mathcal E_n(K)\cap
\left\{
\sup_{z\in K}
\|\bm Q_n^\top\bm G_n(z)\bm Q_n-\bm\Gamma_\alpha(z)\|_{\mathrm{op}}
>\varepsilon
\right\}
\right)\\
& \le
\mathbb{P}\left(
\mathcal E_n(K)\cap
\left\{
\sup_{z\in K}
\|\bm\Gamma_n^{\mathrm{diag}}(z)-\bm\Gamma_\alpha(z)\|_{\mathrm{op}}
>\frac{\varepsilon}{2}
\right\}
\right) +
\mathbb{P} \left(
\mathcal E_n(K)\cap
\left\{
\sup_{z\in K}
\|\bm\Gamma_n^{\mathrm{off}}(z)\|_{\mathrm{op}}
>\frac{\varepsilon}{2}
\right\}
\right),
\end{split}
\]
and both terms tend to zero by Lemmas~\ref{lem:Q_n_diagonal} and~\ref{lem:off_diagonal}.
\end{proof}

\begin{proof} [Proof of Proposition~\ref{prop:outlier_equation}]
Let \(K\subset(\lambda_+(\alpha),\infty)\) be compact, and set \(a_K \coloneqq \inf K>\lambda_+(\alpha)\). If \(K\cap \mathrm{sp} (\bm{P}_n)\neq\emptyset\), then necessarily \(\lambda_1(\bm{P}_n) \geq a_K\). Therefore, Proposition~\ref{prop:upper_edge} gives
\[
\mathbb{P} \left (\mathcal{E}_n (K)^\mathrm{c} \right) \leq \mathbb{P} \left (\lambda_1 (\bm{P}_n)\geq a_K \right) \to 0.
\]
In particular, with probability tending to one, \(\bm{H}_n(z)\) is well-defined for every \(z\in K\).

By Lemmas~\ref{lem:A_n} and~\ref{lem:Q_n}, for every \(\varepsilon>0\),
\begin{equation} \label{eq:Hn_intermediate_limit}
\mathbb{P} \left ( \mathcal{E}_n (K) \cap \left \{ \sup_{z\in K} \left \| \bm{H}_n(z) - \left( \E_{\bm{y}} \left [\bm{T}(\bm{y}) \otimes \bm{C}(\bm{y})  \right ] - \bm{\Gamma}_\alpha (z)- z\bm{I}_{pr} \right) \right \|_{\mathrm{op}} > \varepsilon \right\} \right )\to 0.
\end{equation}
We now identify the deterministic matrix appearing above. Some algebraic manipulations give
\[
\E_{\bm{y}} \left [\bm{T}(\bm{y}) \otimes \bm{C}(\bm{y})  \right ] - \bm{\Gamma}_\alpha (z) = \E_{\bm{y}} \left[
\bm{T} (\bm{y}) \left( \bm{I}_p+ \frac{1}{\alpha} \bm{M}_\alpha(z) \bm{T}(\bm{y}) \right)^{-1} \otimes \bm{C} (\bm{y}) \right].
\]
Hence,
\[
\E_{\bm{y}} \left [\bm{T}(\bm{y}) \otimes \bm{C}(\bm{y})  \right ] - \bm{\Gamma}_\alpha (z)  -z\bm{I}_{pr} = \bm{H}_\alpha(z).
\]
Combining this identity with~\eqref{eq:Hn_intermediate_limit} yields, for every \(\varepsilon>0\),
\[
\mathbb{P} \left ( \mathcal{E}_n (K) \cap \sup_{z\in K} \left \{ \left\| \bm{H}_n(z) - \bm{H}_\alpha(z) \right \|_{\mathrm{op}} > \varepsilon \right \} \right ) \to 0.
\]

We finally prove the assertion concerning the outlying eigenvalues. Suppose that \(\theta_n\in \mathrm{sp} (\bm{D}_n)\) such that \(\theta_n \stackrel{\mathbb{P}}{\to}\theta\) with \(\theta > \lambda_+(\alpha)\). Choose \(\delta>0\) such that
\[
K_\theta \coloneqq [\theta-\delta,\theta+\delta] \subset (\lambda_+(\alpha),\infty).
\]
Since \(\theta_n\to\theta\) in probability and \(\mathbb{P} (\mathcal{E}_n (K_\theta))\to 1\), 
\[
\mathbb{P} \left ( \{ \theta_n \in K_\theta \}\cap \mathcal{E}_n (K_\theta) \right ) \to 1.
\]
On the event \(\{ \theta_n \in K_\theta \}\cap \mathcal{E}_n (K_\theta)\), we have \(\theta_n\notin \mathrm{sp}(\bm{P}_n)\), and hence the Schur complement identity gives \(\det (\bm{H}_n(\theta_n) )=0\), or equivalently,
\[
s_{\min} (\bm{H}_n(\theta_n) )=0,
\]
where \(s_{\min}\) denotes the smallest singular value. Moreover, on \(\{ \theta_n \in K_\theta \}\cap \mathcal{E}_n (K_\theta)\),
\[
\left \| \bm{H}_n (\theta_n)-\bm{H}_\alpha(\theta) \right \|_{\mathrm{op}} \leq \sup_{z\in K_\theta} \left \|
\bm{H}_n(z) - \bm{H}_\alpha(z) \right \|_{\mathrm{op}} + \left \| \bm{H}_\alpha(\theta_n) -\bm{H}_\alpha(\theta) \right \|_{\mathrm{op}}.
\]
The first term converges to zero in probability on \(\mathcal{E}_n(K_\theta)\) by the uniform convergence established above. The second term converges to zero in probability because \(\theta_n\to\theta\) in probability and \(\bm{H}_\alpha\) is continuous on \(K_\theta\). Therefore, for every \(\varepsilon>0\),
\[
\mathbb{P} \left( \{ \theta_n \in K_\theta \}\cap \mathcal{E}_n (K_\theta) \cap \left \{ \|\bm{H}_n(\theta_n) - \bm{H}_\alpha (\theta) \|_{\mathrm{op}}> \varepsilon \right \}\right) \to 0.
\]
Finally, the smallest singular value is \(1\)-Lipschitz with respect to the operator norm. Hence, on \(\{ \theta_n \in K_\theta \}\cap \mathcal{E}_n (K_\theta)\),
\[
s_{\min} (\bm{H}_\alpha(\theta)) \leq s_{\min} (\bm{H}_n(\theta_n)) + \left \| \bm{H}_n(\theta_n)-\bm{H}_\alpha(\theta)
\right \|_{\mathrm{op}} = \left \| \bm{H}_n(\theta_n)-\bm{H}_\alpha(\theta) \right \|_{\mathrm{op}} .
\]
Since the left-hand side is deterministic, while \(\mathbb{P} (\{ \theta_n \in K_\theta \}\cap \mathcal{E}_n (K_\theta))\to1\) and the right-hand side converges to zero in probability on \(\{ \theta_n \in K_\theta \}\cap \mathcal{E}_n (K_\theta)\), we conclude that \(s_{\min} (\bm{H}_\alpha(\theta))=0\). Therefore, \(\bm{H}_\alpha(\theta)\) is singular, and hence \(\det (\bm{H}_\alpha(\theta))=0\).
\end{proof}

\subsection{Largest-eigenvalue phase transition and weak recovery}

We first prove the transition for the largest eigenvalue. 

\begin{proof}[Proof of Theorem~\ref{thm:BBP_transition}]
We begin by proving the continuity and strict monotonicity of \(h_\alpha \colon (\lambda_+(\alpha),\infty) \to \R\). For \(z>\lambda_1 (\bm{P}_n)\), differentiation of~\eqref{eq:H_n} gives
\[
\bm{H}_n'(z) = -\bm{I}_{pr} - \bm{Q}_n^\top \left( \bm{P}_n - z \bm{I}_{p(d-r)} \right)^{-2} \bm{Q}_n \preceq -\bm{I}_{pr}.
\]
Therefore, if \(\lambda_1 (\bm{P}_n) < z_1 < z_2\), then
\[
\bm{H}_n(z_2) - \bm{H}_n(z_1) = \int_{z_1}^{z_2} \bm{H}_n' (t) \, \mathrm{d}t \preceq -\int_{z_1}^{z_2} \bm{I}_{pr} \, \mathrm{d}t =- (z_2-z_1) \bm{I}_{pr},
\]
and hence
\begin{equation}\label{eq:Hn_strict_monotonicity}
\bm{H}_n(z_2) \preceq \bm{H}_n(z_1) - ( z_2 - z_1) \bm{I}_{pr}.
\end{equation}
Moreover \(z \mapsto \left( \bm{P}_n - z \bm{I}_{p(d-r)} \right)^{-1}\) is real analytic on every connected component of \(\R \setminus \mathrm{sp} (\bm{P}_n)\). Thus, \(\bm{H}_n\) is real analytic there
and, in particular, continuous and strictly decreasing in the Loewner order on \((\lambda_1(\bm{P}_n),\infty)\). We next pass the monotonicity inequality to the deterministic limit. Fix \(\lambda_+(\alpha)<z_1<z_2\) and let \(K=[z_1,z_2]\). We claim that
\begin{equation}\label{eq:Halpha_strict_monotonicity}
\bm{H}_\alpha(z_2) \preceq \bm{H}_\alpha(z_1) - (z_2-z_1)\bm{I}_{pr}.
\end{equation}
Suppose, by contradiction, that this inequality fails. Then there exist a unit vector \(\bm{v}\in\R^{pr}\) and \(c>0\) such that
\[
\bm{v}^\top \left ( \bm{H}_\alpha(z_2) - \bm{H}_\alpha(z_1) + (z_2-z_1)\bm{I}_{pr} \right) \bm{v} > c.
\]
By Proposition~\ref{prop:upper_edge}, \(\mathbb{P} \left( \lambda_1 (\bm{P}_n)<z_1\right) \to 1\). Furthermore, Proposition~\ref{prop:outlier_equation} implies that, with probability tending to one,
\[
\sup_{z\in K} \|\bm{H}_n(z)-\bm{H}_\alpha(z) \|_{\mathrm{op}} <\frac{c}{4}.
\]
On the intersection of these two events,
\[
\bm{v}^\top \left[\bm{H}_n(z_2) - \bm{H}_n(z_1) + (z_2-z_1)\bm{I}_{pr} \right] \bm{v} > \frac{c}{2},
\]
contradicting~\eqref{eq:Hn_strict_monotonicity}. This proves~\eqref{eq:Halpha_strict_monotonicity}.  The continuity of \(\bm{H}_\alpha\) follows from the same dominated-convergence argument used in the proof of Proposition~\ref{prop:outlier_equation}, together with the continuity of \(\bm{M}_\alpha\) from Lemma~\ref{lem:continuation_off_support} and the local uniform bounds on the inverse appearing in~\eqref{eq:H_alpha}. Thus, \(\bm{H}_\alpha(z)\) is continuous and strictly decreasing in the Loewner order on \((\lambda_+(\alpha),\infty)\). 

For \(z \in \R \setminus \mathrm{sp} (\bm{P}_n)\), define
\[
h_n (z) \coloneqq \lambda_1 \left(\bm{H}_n(z)\right).
\]
On \((\lambda_1(\bm{P}_n),\infty)\),~\eqref{eq:Hn_strict_monotonicity} gives \(h_n(z_2) \leq h_n(z_1)-(z_2-z_1)\), so \(h_n\) is continuous and strictly decreasing on this interval. Similarly, since \(h_\alpha(z) = \lambda_1 \left(\bm{H}_\alpha(z)\right)\),~\eqref{eq:Halpha_strict_monotonicity} yields \(h_\alpha(z_2) \leq h_\alpha(z_1)-(z_2-z_1)\), and hence \(h_\alpha\) is strictly decreasing. Its continuity follows from the continuity of \(\bm{H}_\alpha\) and of the largest-eigenvalue map. Moreover, for every compact interval \(K\subset(\lambda_+(\alpha),\infty)\), Weyl's inequality gives, on \(\mathcal{E}_n(K)\),
\[
\sup_{z \in K} |h_n(z) - h_\alpha(z)| \le \sup_{z \in K} \|\bm{H}_n(z) - \bm{H}_\alpha(z) \|_{\mathrm{op}}.
\]
Thus, Proposition~\ref{prop:outlier_equation} implies that, for every \(\varepsilon>0\),
\begin{equation}\label{eq:hn_uniform_convergence}
\mathbb{P} \left( \mathcal{E}_n(K) \cap \left \{ \sup_{z\in K} |h_n(z)-h_\alpha(z)| > \varepsilon\right\} \right) \le \mathbb{P} \left( \mathcal{E}_n(K) \cap \left \{ \sup_{z\in K} \|\bm{H}_n(z)- \bm{H}_\alpha(z)\|_{\mathrm{op}} > \varepsilon\right\} \right) \to 0.
\end{equation}
We next determine the behavior of \(h_\alpha\) at infinity. According to Lemma~\ref{lem:continuation_off_support},  the matrix-valued measure \(\bm{\Omega}_\alpha\) is supported in \([-R,R]\) and thus for \(x >R\),
\[
\bm{M}_\alpha(x) = \int_\R \frac{1}{\lambda-x} \bm{\Omega}_\alpha (\mathrm{d}\lambda).
\]
The normalization of the matrix-valued Stieltjes transform gives \(\bm{\Omega}_\alpha(\R)=\bm{I}_p\).
Thus,
\[
\bm{M}_\alpha(x)  + \frac{1}{x}\bm{I}_p = \int_{\R}  \frac{\lambda}{x(\lambda-x)}
\bm{\Omega}_\alpha(\mathrm{d}\lambda).
\]
Since \(\| \bm{M}_\alpha(x) + x^{-1} \bm{I}_p\|_{\mathrm{op}} \le \frac{R}{x(x-R)}\), we have
\[
\bm{M}_\alpha(x) = -\frac{1}{x} \bm{I}_p + o(x^{-1})
\quad \text{as } z\to\infty,
\]
and, in particular, \(\|\bm{M}_\alpha(x)\|_{\mathrm{op}} \to 0\). Since \(\bm{T}(\bm{y})\) is uniformly bounded, dominated convergence yields
\[
\Phi_\alpha (x) \to  \lambda_1\left( \E_{\bm{y}} \left[ \bm{T}(\bm{y})\otimes\bm{C}(\bm{y}) \right] \right).
\]
In particular,
\[
h_\alpha(x) =\Phi_\alpha (\bm{M}_\alpha(x) )-x \to -\infty.
\]
Since \(h_\alpha\) is strictly decreasing, it follows that
\[
\Delta_{\bm{T}}(\alpha) = \sup_{x>\lambda_+(\alpha)} h_\alpha(x) = \lim_{x\downarrow\lambda_+(\alpha)} h_\alpha(x).
\]
This proves the first part of the statement.

We now study the largest eigenvalue of \(\bm{D}_n\). We first suppose that \(\Delta_{\bm{T}}(\alpha)\leq0\). Since \(h_\alpha\) is strictly decreasing, it follows that \(h_\alpha(z)<0\) for every \(z>\lambda_+(\alpha)\). Indeed, this is immediate when \(\Delta_{\bm{T}}(\alpha)<0\), while if \(\Delta_{\bm{T}}(\alpha)=0\), strict monotonicity implies that \(h_\alpha(z)\) is strictly smaller than its limit at the left endpoint. Fix \(\varepsilon>0\) and set \(z_\varepsilon \coloneqq \lambda_+(\alpha)+\varepsilon\). Then \(h_\alpha(z_\varepsilon)<0\). By Proposition~\ref{prop:upper_edge},
\[
\mathbb{P} \left( \lambda_1 (\bm{P}_n) < z_\varepsilon \right) \to 1.
\]
Moreover, by~\eqref{eq:hn_uniform_convergence}, applied to the singleton \(K= \{z_\varepsilon \}\), \(h_n(z_\varepsilon)<0\) with probability tending to one. On this event, the monotonicity of \(h_n\) gives \(h_n(z) <0\) for every \(z\geq z_\varepsilon\). Thus \(\bm{H}_n(z)\) is negative definite for every \( z \geq z_\varepsilon\), and in particular it is never singular there. By the Schur complement identity, \(\bm{D}_n\) has no eigenvalue in \([z_\varepsilon,\infty)\). Therefore,
\[
\mathbb{P} \left( \lambda_1 (\bm{D}_n) \leq \lambda_+(\alpha)+\varepsilon \right) \to 1.
\]
On the other hand, since \(\bm{P}_n\) is a principal submatrix of \(\bm{D}_n\), Cauchy's interlacing theorem (see, e.g.,~\cite[Exercise~1.3.14]{tao2012}) gives
\[
\lambda_1 (\bm{D}_n) \geq \lambda_1(\bm{P}_n).
\]
Using \(\lambda_1(\bm{P}_n) \stackrel{\mathbb{P}}{\to} \lambda_+(\alpha)\) by Proposition~\ref{prop:upper_edge}, we obtain
\[ 
\lambda_1(\bm{D}_n) \stackrel{\mathbb{P}}{\to} \lambda_+(\alpha).
\]

We now suppose that \(\Delta_{\bm{T}}(\alpha)>0\). Since \(h_\alpha\) is continuous and strictly decreasing and tends to \(-\infty\), there exists a unique \(\theta_\alpha>\lambda_+(\alpha)\) such that \(h_\alpha(\theta_\alpha)=0\). Fix \(\varepsilon>0\) sufficiently small that
\[
0<\varepsilon<\theta_\alpha-\lambda_+(\alpha).
\]
Strict monotonicity gives \(h_\alpha(\theta_\alpha-\varepsilon)>0\) and \(h_\alpha(\theta_\alpha+\varepsilon)<0\). By Proposition~\ref{prop:upper_edge},
\[
\mathbb{P} \left( \lambda_1(\bm{P}_n) < \theta_\alpha-\varepsilon \right) \to 1.
\]
Furthermore, by~\eqref{eq:hn_uniform_convergence}, with probability
tending to one, \(h_n(\theta_\alpha-\varepsilon)>0\) and \(h_n(\theta_\alpha+\varepsilon)<0\). Since \(h_n\) is continuous and strictly decreasing on \([ \theta_\alpha  -\varepsilon, \theta_\alpha + \varepsilon]\), there exists a unique \(\theta_n \in (\theta_\alpha-\varepsilon, \theta_\alpha+\varepsilon) \) such that \(h_n(\theta_n)=0\). Equivalently,
\[
\lambda_1 \left(\bm{H}_n (\theta_n)\right)=0,
\]
so \(\bm{H}_n(\theta_n)\) is singular. By the Schur complement identity, \(\theta_n\in \mathrm{sp} (\bm{D}_n)\). Moreover, for every \(z>\theta_n\), we have \(h_n(z)<0\), so \(\bm{H}_n(z)\) is negative definite and hence nonsingular. Since \(\theta_n>\lambda_1(\bm{P}_n)\), no eigenvalue of \(\bm{P}_n\) can lie above \(\theta_n\). The Schur complement identity therefore implies that \(\bm{D}_n\) has no eigenvalue larger than \(\theta_n\). Consequently, \(\lambda_1(\bm{D}_n)=\theta_n\) on this event, and thus
\[
\mathbb{P} \left( \left| \lambda_1 (\bm{D}_n) - \theta_\alpha \right| <\varepsilon \right) \to 1.
\]
Since \(\varepsilon>0\) is arbitrary,
\[
\lambda_1 (\bm{D}_n) \stackrel{\mathbb{P}}{\to} \theta_\alpha.
\]
This completes the proof.
\end{proof}

We now prove Theorem~\ref{thm:weak_recovery} by combining Theorem~\ref{thm:BBP_transition} and Proposition~\ref{prop:upper_edge}.

\begin{proof}[Proof of Theorem~\ref{thm:weak_recovery}]
Because the estimator in~\eqref{eq:estimator} is invariant under nonzero rescaling of its defining eigenvector, let \(\bm{v}_1\) be a unit eigenvector associated with \(\lambda_1(\bm{D}_n)\). Let \(\bm{R}\) denote the orthogonal transformation obtained by applying the rotation and permutation introduced in Subsection~\ref{subsection:preliminaries}, so that
\[
\bm{D}_n'=\bm{R}\bm{D}_n\bm{R}^\top = \begin{pmatrix}
\bm{A}_n & \bm{Q}_n^\top\\
\bm{Q}_n & \bm{P}_n
\end{pmatrix}.
\]
Define the corresponding eigenvector of \(\bm{D}_n'\) by \(\bm{v}_1' \coloneqq \bm{R} \bm{v}_1\). Since \(\bm{R}\) is orthogonal, \(\bm{v}_1'\) is a unit eigenvector associated with \(\lambda_1(\bm{D}_n)\). Write
\[
\bm{v}_1' =
\begin{pmatrix}
\bm{a}_n\\
\bm{b}_n
\end{pmatrix},
\qquad
\bm{a}_n\in\R^{pr},
\quad
\bm{b}_n\in\R^{p(d-r)}.
\]
By construction, \(R\) maps the lifted signal subspace \(\R^p\otimes\mathcal{S}_\ast\) onto the first \(pr\) coordinate directions. Therefore, \(\bm{a}_n\) is precisely the projection of \(\bm{v}_1\) onto this subspace. That is,
\[
\|\bm{a}_n\|_2^2 = \| (\bm{I}_p \otimes \bm{\Pi}_\ast) \bm{v}_1 \|_2^2,
\]
where we recall that \(\bm{\Pi}_\ast = \bm{W}_\ast \bm{W}_\ast^\top\) is the orthogonal projector onto \(\mathcal{S}_\ast\). Let \(\bm{V}_1 \coloneqq\operatorname{mat}_{d\times p}(\bm{v}_1)\), so that \(\bm{v}_1=\operatorname{vec}(\bm{V}_1)\). Using \((\bm{I}_p \otimes \bm{\Pi}_\ast) \operatorname{vec} (\bm{V}_1) = \operatorname{vec}(\bm{\Pi}_\ast \bm{V}_1)\), we obtain
\[
\lVert \bm{a}_n \rVert_2^2 = \|\bm{\Pi}_\ast \bm{V}_1 \|_{\mathrm{F}}^2 = \|\bm{W}_\ast^\top \bm{V}_1\|_{\mathrm{F}}^2.
\]
Since \(\|\bm{v}_1\|_2 = \|\bm{V}_1\|_{\mathrm{F}}=1\) and \(\widehat{\bm{W}} = \sqrt{d} \bm{V}_1\), it follows that
\begin{equation} \label{eq:overlap_signal_coordinates} 
\lVert \bm{a}_n \rVert_2^2 = \frac{ \lVert \bm{W}_\ast^\top \widehat{\bm{W}} \rVert_{\mathrm{F}}^2 }{ \lVert \widehat{\bm{W}} \rVert_{\mathrm{F}}^2 } = \operatorname{Ov}(\widehat{\bm{W}},\mathcal{S}_\ast). 
\end{equation}

By Theorem~\ref{thm:BBP_transition} and Proposition~\ref{prop:upper_edge},
\[
\lambda_1(\bm{D}_n) \stackrel{\mathbb{P}}{\to}\theta_\alpha
\quad\text{and}\quad
\lambda_1(\bm{P}_n)\stackrel{\mathbb{P}}{\to}\lambda_+(\alpha),
\]
where \(\theta_\alpha>\lambda_+(\alpha)\). Set \(\delta \coloneqq (\theta_\alpha-\lambda_+(\alpha))/2 >0\). Then 
\begin{equation} \label{eq:weak_recovery_gap_event}
\mathbb{P} \left( \lambda_1 (\bm{D}_n) -\lambda_1(\bm{P}_n)\geq\delta \right)\to1.
\end{equation}
On this event, the lower block of the eigenvector equation \(\bm{D}_n'\bm{v}_1' = \lambda_1(\bm{D}_n) \bm{v}_1'\) yields
\[
\bm{b}_n = \left(\lambda_1 (\bm{D}_n)\bm{I}_{p(d-r)}-\bm{P}_n\right)^{-1} \bm{Q}_n\bm{a}_n.
\]
It follows from~\eqref{eq:weak_recovery_gap_event} that, with probability tending to one,
\begin{equation} \label{eq:norm_bound_Q}
\|\bm{b}_n\|_2 \leq \frac{\|\bm{Q}_n\|_{\mathrm{op}}}{\delta}\|\bm{a}_n\|_2.
\end{equation}
Since \(\bm{Q}_n\) is an off-diagonal compression of \(\bm{D}_n'\), we have
\[
\|\bm{Q}_n\|_{\mathrm{op}} \leq \|\bm{D}_n'\|_{\mathrm{op}} = \|\bm{D}_n\|_{\mathrm{op}}.
\]
Moreover, by Assumption~\ref{hyp:preprocessing}, 
\[
\|\bm{D}_n\|_{\mathrm{op}} \le C_T \|\bm{I}_p \otimes \bm{S}_n\|_{\mathrm{op}} = C_T \|\bm{S}_n\|_{\mathrm{op}},
\]
where \(\bm{S}_n \coloneqq \frac{1}{n} \sum_{i=1}^n \bm{x}_i\bm{x}_i^\top\). By the almost-sure convergence of the largest eigenvalue of a Gaussian sample covariance matrix (see, e.g.,~\cite{yin1988}),
\[
\|\bm{S}_n\|_{\mathrm{op}} \stackrel{\mathrm{a.s.}}{\to} \left( 1 + \alpha^{-1/2} \right)^2.
\]
Thus, for any fixed \(\eta>0\), setting \(C_Q \coloneqq C_T \left( \left ( 1 + \alpha^{-1/2} \right)^2 + \eta \right)\), we obtain
\begin{equation} \label{eq:weak_recovery_Q_bound}
\mathbb{P} \left( \|\bm{Q}_n\|_{\mathrm{op}}\leq C_Q \right) \to 1.
\end{equation}
Therefore, combining~\eqref{eq:weak_recovery_Q_bound} and~\eqref{eq:norm_bound_Q} we obtain with probability tending to one,
\[
\|\bm{b}_n\|_2 \leq \frac{C_Q}{\delta}\|\bm{a}_n\|_2.
\]
Since \(\|\bm{a}_n\|_2^2+\|\bm{b}_n\|_2^2=1\), we conclude that
\[
\|\bm{a}_n\|_2^2 \geq \left(1+\frac{C_Q^2}{\delta^2}\right)^{-1}
\]
with probability tending to one. The statement follows from~\eqref{eq:overlap_signal_coordinates} with \(\epsilon =\left(1+C_Q^2 / \delta^2\right)^{-1}\).
\end{proof}

\subsection{Critical sampling thresholds}
We finally prove Proposition~\ref{prop:crit_thresholds}. The argument has two independent aprts: small sampling ratios are uniformly subcritical, whereas large sampling ratios are uniformly supercritical under Assumption~\ref{hyp:population_separation}.

\begin{lem} \label{lem:small_alpha_subcritical}
Under Assumptions~\ref{hyp:prop_limit} and~\ref{hyp:preprocessing}, there exists \(\alpha_->0\) such that \(\Delta_{\bm{T}} (\alpha)\leq 0\) for every \(0<\alpha<\alpha_-\).
\end{lem}

\begin{proof}
Set 
\begin{equation} \label{eq:kappa}
\kappa^2 \coloneqq \E_{\bm{y}} \left [ \| \bm{C} (\bm{y}) - \bm{I}_r \|_{\mathrm{F}}^2\right ].
\end{equation}
This quantity is finite. Indeed, \(\E_{\bm{y}} [\bm{C} (\bm{y})]= \E[\bm{s} \bm{s}^\top ] = \bm{I}_r\), and conditional Jensen's inequality gives
\[
\E_{\bm{y}} \|\bm{C} (\bm{y})\|_{\mathrm{F}}^2 = \E_{\bm{y}} \| \E [\bm{s} \bm{s}^\top \mid \bm{y}] \|_{\mathrm{F}}^2 \le \E \| \bm{s} \bm{s}^\top\|_{\mathrm{F}}^2 = \E \|\bm{s}\|_2^4 = r(r+2).
\]
Thus, 
\[
\kappa^2 = \E_{\bm{y}} \|\bm{C} (\bm{y})\|_{\mathrm{F}}^2 - r \le r(r+1).
\]

Fix \(\alpha >0\) and \(x > \lambda_+ (\alpha)\). Since \(\bm{M}_\alpha\) is a normalized matrix-valued Stieltjes transform,  
\[
\bm{Z} \coloneqq - \bm{M}_\alpha(x) \succ \bm{0}_p. 
\] 
Define
\[
\bm{A} (\bm{y}) \coloneqq \bm{T} (\bm{y}) \left ( \bm{I}_p - \frac{1}{\alpha} \bm{Z} \bm{T}(\bm{y}) \right)^{-1}
\]
and
\[
\bm{B} (\bm{y}) \coloneqq \frac{1}{\alpha} \bm{Z}^{1/2}\bm{T} (\bm{y})\bm{Z}^{1/2} \left ( \bm{I}_p - \frac{1}{\alpha} \bm{Z}^{1/2}\bm{T} (\bm{y})\bm{Z}^{1/2} \right)^{-1}.
\]
The inverses above are well defined by the analysis from Section~\ref{section:MDE}. Moreover, \(\bm{B}(\bm{y})\) is symmetric, and Lemma~\ref{lem:push_through} yields
\begin{equation} \label{eq:identity_A_B}
\bm{Z}^{1/2} \bm{A}(\bm{y}) \bm{Z}^{1/2} = \alpha \bm{B}(\bm{y}).
\end{equation}

We claim that
\[
h_\alpha(x) = \lambda_1 \left ( \E_{\bm{y}} \left [ \bm{A}(\bm{y}) \otimes \bm{C}(\bm{y})\right ]\right) - x < 0.
\]
The matrix Dyson equation~\eqref{eq:MDE} at \(x\) reads
\[
- \bm{Z}^{-1} = -x \bm{I}_p + \E_{\bm{y}} [\bm{A}(\bm{y})],
\]
or equivalently,
\begin{equation} \label{eq:MDE_Z}
x \bm{I}_p = \bm{Z}^{-1} + \E_{\bm{y}} [\bm{A} (\bm{y})].
\end{equation}
Hence
\[
\E_{\bm{y}} \left [\bm{A}(\bm{y}) \otimes \bm{C} (\bm{y}) \right ] - x \bm{I}_{pr} = \E_{\bm{y}} \left [ \bm{A}(\bm{y}) \otimes \left (\bm{C}(\bm{y}) - \bm{I}_r \right) \right] - \bm{Z}^{-1} \otimes \bm{I}_r.
\]
Conjugation by \(\bm{Z}^{1/2} \otimes \bm{I}_r\) and using~\eqref{eq:identity_A_B}, we obtain
\begin{equation}\label{identity_A_B_C}
(\bm{Z}^{1/2} \otimes \bm{I}_r) \left ( \E_{\bm{y}} \left [\bm{A}(\bm{y}) \otimes \bm{C} (\bm{y}) \right ] - x \bm{I}_{pr}  \right) (\bm{Z}^{1/2} \otimes \bm{I}_r) = \alpha \E \left [ \bm{B}(\bm{y}) \otimes \left (\bm{C}(\bm{y}) - \bm{I}_r \right) \right] - \bm{I}_{pr} .
\end{equation}
It is therefore enough to prove
\begin{equation} \label{eq:bound_small_alpha2}
\left \| \alpha \E_{\bm{y}} \left [ \bm{B}(\bm{y}) \otimes \left (\bm{C}(\bm{y}) - \bm{I}_r \right) \right] \right\|_{\mathrm{op}} < 1 .
\end{equation}

By the Cauchy--Schwarz inequality, we have
\begin{equation} \label{eq:bound_B_C}
\begin{split}
\left \| \alpha \E_{\bm{y}} \left [ \bm{B}(\bm{y}) \otimes \left (\bm{C}(\bm{y}) - \bm{I}_r \right) \right] \right\|_{\mathrm{op}}  & \le \alpha \left ( \E_{\bm{y}} \| \bm{B} (\bm{y})\|_{\mathrm{F}}^2 \right)^{1/2} \left ( \E \|\bm{C}(\bm{y}) - \bm{I}_r   \|_{\mathrm{F}}^2 \right)^{1/2} \\
& = \alpha \kappa \left ( \E_{\bm{y}} \| \bm{B} (\bm{y})\|_{\mathrm{F}}^2 \right)^{1/2}.
\end{split}
\end{equation}
We thus need to control \(\alpha\left ( \E_{\bm{y}} \| \bm{B} (\bm{y})\|_{\mathrm{F}}^2 \right)^{1/2} \). For \(\bm{H} \in \sym_p(\R)\), define 
\[
\mathcal{F}_{\alpha,x} [\bm{H}] \coloneqq \alpha \E_{\bm{y}} \left [ \bm{B}(\bm{y}) \bm{H} \bm{B} (\bm{y})\right],
\]
and
\[
\mathcal{C}_{\bm{Z}} [\bm{H}]\coloneqq \bm{Z}^{1/2} \bm{H} \bm{Z}^{1/2}.
\]
Using~\eqref{eq:mathcal_K} and~\eqref{eq:identity_A_B}, we obtain
\[
\begin{split}
\mathcal{C}_{\bm{Z}}^{-1} \mathcal{K}_x \mathcal{C}_{\bm{Z}} [\bm{H}] &= \bm{Z}^{-1/2} \mathcal{K}_x \left [ \bm{Z}^{1/2}\bm{H} \bm{Z}^{1/2} \right ] \bm{Z}^{-1/2} \\
& = \frac{1}{\alpha} \E_{\bm{y}} \left [ \bm{Z}^{1/2}\bm{A}(\bm{y})\bm{Z}^{1/2} \bm{H} \bm{Z}^{1/2}\bm{A}(\bm{y})\bm{Z}^{1/2} \right ] \\
& = \alpha\E_{\bm{y}}\left[ \bm{B}(\bm{y})\bm{H}\bm{B}(\bm{y}) \right] \\
&= \mathcal{F}_{\alpha,x}[\bm{H}].
\end{split}
\]
Thus
\[
\mathcal{F}_{\alpha,x} = \mathcal{C}_{\bm{Z}}^{-1} \mathcal{K}_x \mathcal{C}_{\bm{Z}}. 
\]
In particular, \(\mathcal{F}_{\alpha,x}\) and \(\mathcal{K}_x\) are similar and hence have the same spectrum. Corollary~\ref{cor:physical_stability_right} gives \(\rho(\mathcal{K}_x)<1\) and so \(\rho(\mathcal{F}_{\alpha,x})<1\).

By symmetry of \(\bm{B}(\bm{y})\), \(\mathcal{F}_{\alpha,x}\) is self-adjoint with respect to the Frobenius inner product. Indeed, for every
\(\bm{H},\bm{K}\in\sym_p(\R)\),
\[
\begin{split}
\left\langle
\bm{H}, \mathcal{F}_{\alpha,x}[\bm{K}] \right\rangle_{\mathrm{F}} = \alpha \E_{\bm{y}}
\Tr \left( \bm{H} \bm{B} (\bm{y}) \bm{K} \bm{B} (\bm{y}) \right) = \alpha \E_{\bm{y}} \Tr \left(
\bm{B}(\bm{y}) \bm{H} \bm{B} (\bm{y}) \bm{K} \right) = \left \langle \mathcal{F}_{\alpha,x} [\bm{H}] ,\bm{K} \right \rangle_{\mathrm{F}}.
\end{split}
\]
It then follows that
\[
\|\mathcal{F}_{\alpha,x} \|_{\mathrm{F} \to\mathrm{F}} = \rho(\mathcal{F}_{\alpha,x}) <1.
\]
In particular, this implies that
\[
\|\mathcal{F}_{\alpha,x}[\bm{I}_p]\|_{\mathrm{F}} < \|\bm{I}_p\|_{\mathrm{F}} = \sqrt{p}.
\]
Moreover, \(\mathcal{F}_{\alpha,x} [\bm{I}_p] = \alpha \E_{\bm{y}} [\bm{B}(\bm{y})^2]\), yielding 
\[
\alpha \E_{\bm{y}} \|\bm{B}(\bm{y})\|_{\mathrm{F}}^2 = \Tr \mathcal{F}_{\alpha,z} [\bm{I}_p] \leq \|\bm{I}_p\|_{\mathrm{F}} \|\mathcal{F}_{\alpha,z} [\bm{I}_p]\|_{\mathrm{F}}< p.
\]
Substituting this bound into~\eqref{eq:bound_B_C} gives
\begin{equation} \label{eq:bound_B_C2}
\left \| \alpha \E_{\bm{y}} \left [ \bm{B}(\bm{y}) \otimes \left (\bm{C}(\bm{y}) - \bm{I}_r \right) \right] \right\|_{\mathrm{op}} <\sqrt{\alpha p} \kappa.
\end{equation}

Choose, for instance,
\[
\alpha_- \coloneqq \frac{1}{p(1+\kappa^2)} >0.
\]
Then, for every \(0 < \alpha < \alpha_-\), 
\[
\sqrt{\alpha p} \kappa < \frac{\kappa}{\sqrt{1+\kappa^2}}<1. 
\]
It follows from~\eqref{identity_A_B_C} and~\eqref{eq:bound_B_C2} that
\[
\E_{\bm{y}} \left [ \bm{A}(\bm{y}) \otimes \bm{C}(\bm{y})\right ] - x \bm{I}_{pr} \prec \bm{0}_{pr}.
\]
Therefore, 
\[
h_\alpha(x) = \lambda_1 \left ( \E_{\bm{y}} \left [ \bm{A}(\bm{y}) \otimes \bm{C}(\bm{y})\right ]\right) - x < 0,
\]
for every \(x > \lambda_+(\alpha)\). Taking the supremum over \(x > \lambda_+(\alpha)\) , we conclude that \(\Delta_{\bm{T}}(\alpha) \le 0\).
\end{proof}

We next prove eventual supercriticality. This is the only part of the argument that uses Assumption~\ref{hyp:population_separation}.

\begin{lem} \label{lem:large_alpha_supercritical}
Suppose that Assumptions~\ref{hyp:prop_limit},~\ref{hyp:preprocessing}, and~\ref{hyp:population_separation} hold. Then there exists \(\alpha_+\in (0,\infty)\) such that \(\Delta_{\bm{T}}(\alpha)>0\) for every \(\alpha>\alpha_+\).
\end{lem}

\begin{proof}
Set 
\[
\rho_{\mathrm{noise}} \coloneqq \lambda_1 \left ( \E_{\bm{y}}[\bm{T}(\bm{y})] \right ) \quad \mathrm{and} \quad  
\rho_{\mathrm{sig}} \coloneqq \lambda_1\left( \E_{\bm{y}} \left[ \bm{T}(\bm{y})\otimes\bm{C}(\bm{y}) \right] \right). 
\]
By Assumption~\ref{hyp:population_separation}, \(\rho_{\mathrm{sig}}>\rho_{\mathrm{noise}}\). We first prove the upper-edge estimate:
\begin{equation} \label{eq:lim_supp_alpha}
\limsup_{\alpha \to +\infty} \lambda_+(\alpha) \le \rho_{\mathrm{noise}}.
\end{equation}
One could also prove the matching lower bound and thus show that \(\lambda_+(\alpha) \to  \rho_{\mathrm{noise}}\) as \(\alpha \to +\infty\). Since the lower bound is not needed here, it will be omitted. 

We proceed similarly as in Lemma~\ref{lem:analytic_continuation}. Fix \(\varepsilon>0\), and define
\[
\mathcal{U}_\varepsilon \coloneqq \left \{ z \in \C \colon \Re z > \rho_{\mathrm{noise}} + \varepsilon \right\}, \quad
\mathcal{B}_\varepsilon \coloneqq \left \{ \bm{X} \in \C^{p\times p} \colon \|\bm{X}\|_{\mathrm{op}} \leq \frac{2}{\varepsilon} \right\}.
\]
For \(z\in\mathcal{U}_\varepsilon\), recall the maps from~\eqref{eq:mathcal_S_alpha} and~\eqref{eq:bm_Psi_alpha}:
\[
\Psi_{\alpha,z} (\bm{X}) = \left( -z \bm{I}_p + \mathcal{S}_\alpha(\bm{X})
\right)^{-1},
\]
where
\[
\mathcal{S}_\alpha(\bm{X}) = \E_{\bm{y}} \left[ \bm{T} (\bm{y}) \left( \bm{I}_p + \frac{1}{\alpha} \bm{X} \bm{T} (\bm{y}) \right)^{-1} \right].
\]
We first show that, for all sufficiently large \(\alpha\), \(\Psi_{\alpha,z}\) is a strict contraction from \(\mathcal{B}_\varepsilon\) into itself, uniformly in \(z\in \mathcal{U}_\varepsilon\). Let \(C_T\) be the uniform bound on \(\|\bm{T}(\bm{y})\|_{\mathrm{op}}\) from Assumption~\ref{hyp:preprocessing}. If \(\bm{X} \in \mathcal{B}_\varepsilon\), then for sufficiently large \(\alpha\),
\[
\left \| \frac{1}{\alpha}\bm{X}\bm{T}(\bm{y}) \right\|_{\mathrm{op}} \leq \frac{2C_T}{\alpha\varepsilon} < \frac{1}{2},
\]
uniformly in \(\bm{y}\), and hence
\[
\left \| \left( \bm{I}_p + \frac{1}{\alpha} \bm{X} \bm{T}(\bm{y}) \right)^{-1} \right\|_{\mathrm{op}} \leq 2.
\]
We obtain
\[
\|\mathcal{S}_\alpha(\bm{X}) - \E_{\bm{y}} [\bm{T}(\bm{y})] \|_{\mathrm{op}} \leq \frac{1}{\alpha}\E_{\bm{y}} \left[ \|\bm{T}(\bm{y})\|_{\mathrm{op}} \left \| \left( \bm{I}_p + \frac{1}{\alpha} \bm{X} \bm{T}(\bm{y}) \right)^{-1} \right\|_{\mathrm{op}} \|\bm{X}\|_{\mathrm{op}} \|\bm{T}(\bm{y})\|_{\mathrm{op}} \right] \leq \frac{4C_T^2}{\alpha\varepsilon},
\]
where we used the identity \((\bm{I}_p + \bm{A})^{-1} - \bm{I}_p = -(\bm{I}_p +\bm{A} )^{-1} \bm{A}\). Since \(\bm{T}\) is symmetric and \(\Re z>\rho_{\mathrm{noise}}+\varepsilon\),
\[
\left \| \left (\E_{\bm{y}} [\bm{T}(\bm{y})]-z\bm{I}_p \right)^{-1} \right \|_{\mathrm{op}}
\leq \frac{1}{\varepsilon}.
\]
Consequently, if \(4C_T^2 / (\alpha\varepsilon) \leq \frac{\varepsilon}{2}\), the inverse perturbation bound gives
\[
\|\Psi_{\alpha,z}(\bm{X})\|_{\mathrm{op}} \leq \frac{2}{\varepsilon}.
\]
Thus \(\Psi_{\alpha,z}(\mathcal{B}_\varepsilon) \subseteq \mathcal{B}_\varepsilon\) uniformly in \(z\in\mathcal{U}_\varepsilon\). For \(\bm{X},\bm{Y} \in\mathcal{B}_\varepsilon\), the resolvent identity
gives
\[
\|\mathcal{S}_\alpha(\bm{X}) - \mathcal{S}_\alpha(\bm{Y})\|_{\mathrm{op}} \leq \frac{4C_T^2}{\alpha} \|\bm{X}-\bm{Y}\|_{\mathrm{op}}.
\]
Applying the resolvent identity once more yields
\[
\begin{split}
\|\Psi_{\alpha,z}(\bm{X})-\Psi_{\alpha,z}(\bm{Y})\|_{\mathrm{op}}
& \leq \|\Psi_{\alpha,z}(\bm{X})\|_{\mathrm{op}} \|\mathcal{S}_\alpha(\bm{X})-\mathcal{S}_\alpha(\bm{Y})\|_{\mathrm{op}} \|\Psi_{\alpha,z} (\bm{Y})\|_{\mathrm{op}}\\
&\leq \frac{16C_T^2}{\alpha\varepsilon^2} \|\bm{X}-\bm{Y}\|_{\mathrm{op}}.
\end{split}
\]
Hence, for every sufficiently large \(\alpha\), \(\Psi_{\alpha,z}\) is a strict contraction on \(\mathcal{B}_\varepsilon\), uniformly in \(z\in\mathcal{U}_\varepsilon\). By Banach's fixed-point theorem, for every \(z\in\mathcal{U}_\varepsilon\), there exists a unique \(\bm{G}_\alpha(z) \in \mathcal{B}_\varepsilon\) such that
\[
\bm{G}_\alpha(z) = \Psi_{\alpha,z}(\bm{G}_\alpha(z)),
\]
i.e., \(\bm{G}_\alpha\) solves the fixed-point equation associated with the matrix Dyson equation~\eqref{eq:MDE}. The Picard iterates
\[
\bm{X}_0 (z) = \bm{0}_p, \quad \bm{X}_{k+1}(z) = \Psi_{\alpha,z}(\bm{X}_k(z))
\]
are analytic in \(z\) and converge locally uniformly on \(\mathcal{U}_\varepsilon\). Therefore, \(z \mapsto \bm{G}_\alpha(z)\) is analytic on \(\mathcal{U}_\varepsilon\).

We now identify \(\bm{G}_\alpha\) with \(\bm{M}_\alpha\) in the upper half-plane. On the nonempty open set
\[
\mathcal{V}_\varepsilon \coloneqq \mathcal{U}_\varepsilon \cap \left \{ z \in \C_+ \colon \Im z>\frac{\varepsilon}{2} \right\},
\]
the Herglotz bound gives
\[
\|\bm{M}_\alpha(z)\|_{\mathrm{op}} \leq \frac{1}{\Im z} < \frac{2}{\varepsilon}.
\]
Thus \(\bm{M}_\alpha(z)\in\mathcal{B}_\varepsilon\). Since the matrix Dyson equation holds for \(z\in \C_+\), \(\bm{M}_\alpha(z) = \Psi_{\alpha,z}(\bm{M}_\alpha(z))\) for \(z \in \mathcal{V}_\varepsilon\). By uniqueness of the fixed point in \(\mathcal{B}_\varepsilon\),
\[
\bm{M}_\alpha(z) = \bm{G}_\alpha(z),
\quad z \in \mathcal{V}_\varepsilon.
\]
Both functions are analytic on the connected domain \(\mathcal{U}_\varepsilon \cap \C_+\). Therefore, by the identity theorem,
\[
\bm{M}_\alpha(z)=\bm{G}_\alpha(z),
\quad
z\in\mathcal{U}_\varepsilon \cap \C_+.
\]
Hence \(\bm{G}_\alpha\) is an analytic continuation of \(\bm{M}_\alpha\) to \( \mathcal{U}_\varepsilon \). Finally, since \(\bm{T}(\bm{y})\) is real symmetric, \(\mathcal{S}_\alpha(\bm{X})^\ast =
\mathcal{S}_\alpha(\bm{X}^\ast)\), and therefore \(\Psi_{\alpha,z}(\bm{X})^\ast = \Psi_{\alpha,\overline z} (\bm{X}^\ast)\). It follows from uniqueness of the fixed point that
\[
\bm{G}_\alpha(z)^\ast = \bm{G}_\alpha (\overline{z}).
\]
In particular, for every real \(x>\rho_{\mathrm{noise}}+\varepsilon\), \(\bm{G}_\alpha(x) = \bm{G}_\alpha(x)^\ast\). Thus
\[
g_\alpha(z) \coloneqq \frac{1}{p}\Tr\bm{G}_\alpha(z)
\]
is an analytic continuation of \(m_{\mu_\alpha}\) to \(\mathcal{U}_\varepsilon\), and it is real-valued on \((\rho_{\mathrm{noise}} + \varepsilon,\infty)\). By the converse part of
Lemma~\ref{lem:continuation_off_support},
\[
\supp(\mu_\alpha) \cap (\rho_{\mathrm{noise}}+\varepsilon,\infty) = \emptyset.
\]
Therefore,
\[
\lambda_+(\alpha) \leq \rho_{\mathrm{noise}}+\varepsilon
\]
for all sufficiently large \(\alpha\). Since \(\varepsilon\) is arbitrary, this proves~\eqref{eq:lim_supp_alpha}.

Choose now \(x_\ast\in(\rho_{\mathrm{noise}},\rho_{\mathrm{sig}})\). By~\eqref{eq:lim_supp_alpha}, for all sufficiently large \(\alpha\), \(x_\ast>\lambda_+(\alpha)\). Thus \(h_\alpha(x_\ast)\) is well-defined. We next prove that \(\bm{M}_\alpha(x_\ast)\) is uniformly bounded for large \(\alpha\). Choose \(\delta>0\) such that \(0<\delta<x_\ast-\rho_{\mathrm{noise}}\). Then \(x_\ast\in\mathcal{U}_\delta\). Applying the preceding contraction argument with \(\varepsilon=\delta\), we obtain, for all sufficiently large \(\alpha\), a fixed point \(\bm{G}_\alpha(x_\ast)\in\mathcal{B}_\delta\). Moreover, since \(x_\ast > \lambda_+(\alpha)\), the point \(x_\ast\) lies outside \(\supp(\mu_\alpha)\). Hence the analytic continuation constructed above agrees with the Stieltjes-transform continuation of \(\bm{M}_\alpha\) at \(x_\ast\). Therefore \(\bm{M}_\alpha(x_\ast) = \bm{G}_\alpha(x_\ast)\), and consequently 
\[ 
\sup_{\alpha\geq\alpha_0} \|\bm{M}_\alpha(x_\ast)\|_{\mathrm{op}} \leq \frac{2}{\delta} <\infty ,
\] 
for some sufficiently large \(\alpha_0\). 

Set \(K\coloneqq \sup_{\alpha\geq\alpha_0} \|\bm{M}_\alpha(x_\ast)\|_{\mathrm{op}}\). Then, for all sufficiently large \(\alpha\), 
\[ 
\left\| \frac{1}{\alpha} \bm{M}_\alpha(x_\ast)\bm{T}(\bm{y}) \right\|_{\mathrm{op}}  \le \frac{KC_T}{\alpha} < \frac{1}{2}, 
\] 
and hence
\[ 
\left\| \left( \bm{I}_p+ \frac{1}{\alpha} \bm{M}_\alpha(x_\ast)\bm{T}(\bm{y}) \right)^{-1} \right\|_{\mathrm{op}} \leq 2. 
\] 
Using again the identity \( (\bm{I}+\bm{A})^{-1}-\bm{I} = -(\bm{I}+\bm{A})^{-1}\bm{A}\), we get 
\[ 
\left\| \left( \bm{I}_p+ \frac{1}{\alpha} \bm{M}_\alpha(x_\ast)\bm{T}(\bm{y}) \right)^{-1} - \bm{I}_p \right\|_{\mathrm{op}} \leq \frac{2KC_T}{\alpha}. 
\] 
Therefore
\[
\left\| \E_{\bm{y}}\left[ \bm{T}(\bm{y}) \left( \bm{I}_p+ \frac{1}{\alpha} \bm{M}_\alpha(x_\ast)\bm{T}(\bm{y}) \right)^{-1} \otimes \bm{C}(\bm{y}) \right] - \E_{\bm{y}} \left[ \bm{T}(\bm{y})\otimes\bm{C}(\bm{y}) \right] \right\|_{\mathrm{op}} \le \frac{2KC_T^2}{\alpha} \E_{\bm{y}} \|\bm{C}(\bm{y})\|_{\mathrm{op}} .
\]
Since \(\bm{C}(\bm{y}) \succ \bm{0}_r\), 
\[
\E_{\bm{y}} \|\bm{C}(\bm{y})\|_{\mathrm{op}}  \le \E_{\bm{y}} \Tr \bm{C}(\bm{y}) = \E \|\bm{s}\|_2^2 = r.
\]
By continuity of the largest eigenvalue with respect to the operator norm,
\[ 
\begin{split}
\lim_{\alpha \to \infty} \Phi_\alpha(x_\ast) & = \lim_{\alpha \to \infty}\lambda_1 \left ( \E_{\bm{y}}\left[ \bm{T}(\bm{y}) \left( \bm{I}_p+ \frac{1}{\alpha} \bm{M}_\alpha(x_\ast)\bm{T}(\bm{y}) \right)^{-1} \otimes \bm{C}(\bm{y}) \right]  \right ) \\
& = \lambda_1 \left( \E_{\bm{y}}\left[ \bm{T}(\bm{y})\otimes\bm{C}(\bm{y}) \right] \right) = \rho_{\mathrm{sig}}.
\end{split}
\] 
Hence 
\[ 
\lim_{\alpha \to \infty} h_\alpha(x_\ast) = \lim_{\alpha \to \infty} \Phi_\alpha(x_\ast) - x_\ast =  \rho_{\mathrm{sig}}-x_\ast > 0. 
\] 
Therefore \(h_\alpha(x_\ast) > 0\) for all sufficiently large \(\alpha\). Since \( x_\ast >\lambda_+(\alpha)\), we obtain 
\[ 
\Delta_{\bm{T}} (\alpha) = \sup_{x>\lambda_+(\alpha)} h_\alpha(z) \geq h_\alpha (x_\ast) > 0 
\] 
for all sufficiently large \(\alpha\). This proves the existence of \(\alpha_+<\infty\) such that \(\Delta_{\bm{T}}(\alpha)>0\) for \(\alpha>\alpha_+\).
\end{proof}

Having Lemmas~\ref{lem:small_alpha_subcritical} and~\ref{lem:large_alpha_supercritical} at hand, Proposition~\ref{prop:crit_thresholds} follows straightforwardly.

\begin{proof}[Proof of Proposition~\ref{prop:crit_thresholds}]
Define
\[
\mathcal{S}_{\bm{T}} \coloneqq \{\alpha>0 \colon \Delta_{\bm{T}}(\alpha)>0\}.
\]
By Lemma~\ref{lem:small_alpha_subcritical}, there exists \(\alpha_->0\) such that \((0,\alpha_-) \cap \mathcal{S}_{\bm{T}} = \emptyset\). Hence
\[
\alpha_{\mathrm{c},\min}(\bm{T}) = \inf \mathcal{S}_{\bm{T}} \geq \alpha_->0.
\]
By Lemma~\ref{lem:large_alpha_supercritical}, there exists \(\alpha_+ \in (0,\infty)\) such that \((\alpha_+,\infty)\subseteq \mathcal{S}_{\bm{T}}\). Moreover,
\[
\alpha_{\mathrm{c},\max}(\bm{T}) = \inf \{ \bar{\alpha} > 0 \colon (\bar{\alpha},\infty) \subseteq \mathcal{S}_T\} \leq \alpha_+<\infty.
\]
Since, by definition, \(\alpha_{\mathrm{c},\min}(\bm{T}) \leq \alpha_{\mathrm{c},\max}(\bm{T})\), we obtain
\[
0 < \alpha_- \le \alpha_{\mathrm{c},\min}(\bm{T}) \leq \alpha_{\mathrm{c},\max}(\bm{T}) \le \alpha_+ <\infty.
\]

We now prove the two spectral conclusions. If \(0< \alpha <\alpha_{\mathrm{c},\min}(\bm{T})\), then \(\alpha \notin \mathcal{S}_{\bm{T}}\), so \(\Delta_{\bm{T}} (\alpha) \le 0\). Theorem~\ref{thm:BBP_transition} therefore gives 
\[
\lambda_1 (\bm{D}_n) \stackrel{\mathbb{P}}{\to} \lambda_+(\alpha).
\]
Now suppose that \(\alpha>\alpha_{\mathrm{c},\max}(\bm{T}) = \inf \{ \bar{\alpha} > 0 \colon (\bar{\alpha},\infty) \subseteq \mathcal{S}_T\} \). Since \(\{ \bar{\alpha} > 0 \colon (\bar{\alpha},\infty) \subseteq \mathcal{S}_T\} \neq \emptyset \), there exists \(\bar{\alpha} \in \{ \bar{\alpha} > 0 \colon (\bar{\alpha},\infty) \subseteq \mathcal{S}_T\} \) such that \(\bar{\alpha} < \alpha\). This implies that \(\alpha \in \mathcal{S}_T\), that is, \(\Delta_{\bm{T}}(\alpha) > 0\). Theorem~\ref{thm:BBP_transition} then yields a unique \(\theta_\alpha>\lambda_+(\alpha)\) satisfying \(h_\alpha(\theta_\alpha)=0\), and
\[
\lambda_1 (\bm{D}_n) \stackrel{\mathbb{P}}{\to} \theta_\alpha.
\]
\end{proof}
\section{Optimality of the preprocessing map} \label{section:optimal_preprocessing}

We prove Theorem~\ref{thm:opt} in two steps. First, we establish a universal lower bound on the spectral threshold: no admissible preprocessing map can produce a separating outlier below \(\alpha_{\mathrm{c}}^\ast\). Second, we show that the preprocessing map \(\bm{T}_\ast\) attains this bound.

The first step is the following universal lower bound.

\begin{lem}\label{lem:upperbound_samplecomplexity}
For every fixed \(p\geq1\) and every measurable preprocessing map \(\bm{T}\colon \R^q\to\sym_p(\R) \) satisfying Assumption~\ref{hyp:preprocessing}, 
\[
\alpha_{\mathrm{c},\min}(\bm{T}) \geq \alpha_{\mathrm{c}}^\ast ,
\]
where $\alpha_{\mathrm c}^\ast$ is defined in~\eqref{eq:def_alpha_cast}.
\end{lem}

The proof is given in Section~\ref{sec:proof_upperbound_samplecomplexity}. The second step shows that, for \(\bm{T}_\ast\), an outlier separates from the upper edge as soon as \(\alpha>\alpha_{\mathrm{c}}^\ast\). 

\begin{lem}\label{lem:optimality_Tast}
Suppose that Assumptions~\ref{hyp:prop_limit} and~\ref{hyp:second_order_channel} hold. Then the lower and upper spectral thresholds of the preprocessing map \(\bm{T}_\ast\) coincide:
\[
\alpha_{\mathrm{c},\min}(\bm{T}_\ast) = \alpha_{\mathrm{c},\max}(\bm{T}_\ast) \eqqcolon \alpha_{\mathrm{c}}^\ast.
\]
\end{lem}

The proof is given in Section~\ref{sec:proof_optimality_Tast}.
\subsection{A lower bound on the spectral threshold}
\label{sec:proof_upperbound_samplecomplexity}

We now prove Lemma~\ref{lem:upperbound_samplecomplexity}. Let $\bm{T}$ satisfy Assumption~\ref{hyp:preprocessing}, and suppose that $\Delta_{\bm{T}}(\alpha) > 0$. By the monotonicity analysis in the proof of Theorem~\ref{thm:BBP_transition}, there exists a unique \(\theta>\lambda_+(\alpha)\) such that \(h_\alpha(\theta)=0\). Equivalently,
\begin{equation}\label{eq:outlier_equation}
\det \left (\bm{H}_\alpha (\theta) \right)=0,
\end{equation}
where \(\bm{H}_\alpha\) is defined in~\eqref{eq:H_alpha}. We show that this necessarily implies \(\alpha>\alpha_{\mathrm c}^\ast\).

We first rewrite~\eqref{eq:outlier_equation} in a form that separates the contribution of the preprocessing map from that of the observation channel and is suitable for applying the MDE stability bound. The outlier equation~\eqref{eq:outlier_equation} implies that there exists \(\bm{w} \neq \bm{0}_{pr}\) such that
\[
\bm{H}_{\alpha}(\theta) \bm{w} = \bm{0}_{pr}.
\]
Equivalently, using
\[
\bm{\mathcal{Q}}_\alpha (\theta, \bm{y}) \coloneqq \bm{\mathcal{Q}}_\alpha (\bm{M}_\alpha(\theta), \bm{y}) = \bm{T}(\bm{y}) \left( \bm{I}_p+\frac{1}{\alpha} \bm{M}_\alpha(\theta)\bm{T}(\bm{y}) \right)^{-1} ,
\]
we have
\[
\left( \E_{\bm{y}} \left[ \bm{\mathcal{Q}}_\alpha (\theta, \bm{y}) \otimes\bm{C}(\bm{y})\right] - \theta \bm{I}_{pr} \right) \bm{w} = \bm{0}_{pr}.
\]
On the other hand, evaluating the MDE~\eqref{eq:MDE} at \(z=\theta\) gives
\[
\bm{M}_\alpha(\theta)^{-1} = -\theta\bm{I}_p + \E_{\bm{y}} \left[  \bm{\mathcal{Q}}_\alpha (\theta, \bm{y}) \right].
\]
Subtracting gives directly
\[
\left ( \E \left [  \bm{\mathcal{Q}}_\alpha (\theta, \bm{y}) \otimes \left (\bm{C} (\bm{y}) -\bm{I}_r \right) \right ] + \bm{M}_\alpha(\theta)^{-1} \otimes \bm{I}_r \right) \bm{w} = \bm{0}_{pr}.
\]
Multiplying by \(\alpha^{-1} \bm{M}_\alpha(\theta) \otimes \bm{I}_r\) then yields
\begin{equation} \label{eigenvalue_relation}
\E \left[ \frac{1}{\alpha} \bm{M}_\alpha(\theta) \bm{\mathcal{Q}}_\alpha (\theta, \bm{y})  \otimes \left (\bm{C} (\bm{y}) - \bm{I}_r \right )  \right] \bm{w} = -\frac{1}{\alpha} \bm{w}.
\end{equation}
Now set \(\bm{L}_\alpha (\theta) \coloneqq - \alpha^{-1} \bm{M}_\alpha (\theta)\). Since $\theta>\lambda_+(\alpha)$, the Stieltjes-transform property of $\bm{M}_\alpha$ gives \(\bm{L}_\alpha (\theta) \succ \bm{0}_{p \times p}\). Define 
\begin{align} \label{eq:def_A}
\bm{A}^{(\bm{T})}_{\alpha} (\theta,\bm{y}) \coloneqq \bm{L}_\alpha (\theta)^{1/2} \bm{\mathcal{Q}}_\alpha (\theta, \bm{y})  \bm{L}_\alpha (\theta)^{1/2} .
\end{align}
First note that \(\bm{A}^{(\bm{T})}_{\alpha} (\theta,\bm{y})\) is symmetric by the push-through identity in Lemma~\ref{lem:push_through}. Second note that
\[
\frac{1}{\alpha} \bm{M}_\alpha (\theta) \bm{\mathcal{Q}}_\alpha (\theta, \bm{y}) = - \bm{L}_\alpha (\theta) \bm{\mathcal{Q}}_\alpha (\theta, \bm{y})  = -  \bm{L}_\alpha (\theta) ^{1/2} \bm{A}^{(\bm{T})}_{\alpha} (\theta,\bm{y})  \bm{L}_\alpha (\theta) ^{-1/2}.
\]
Thus the eigenvalue relation~\eqref{eigenvalue_relation} is similar to 
\[
\E_{\bm{y}} \left [ \bm{A}^{(\bm{T})}_{\alpha} (\theta,\bm{y})  \otimes (\bm{C}(\bm{y}) - \bm{I}_r)\right ] \widetilde{\bm{w}} = \frac{1}{\alpha} \widetilde{\bm{w}},
\]
for a nonzero vector \(\widetilde{\bm{w}}\). Therefore \(\alpha^{-1}\) is an eigenvalue of this symmetric matrix and hence
\begin{align}\label{eq:bound_outlier_eq}
\frac{1}{\alpha} \le \|  \E_{\bm{y}} [  \bm{A}^{(\bm{T})}_{\alpha} (\theta,\bm{y}) \otimes (\bm{C}(\bm{y}) - \bm{I}_r) ]\|_{\mathrm{op}} .
\end{align}

\begin{lem} \label{lem:bound_stability_theta}
Let $\theta>\lambda_{+}(\alpha)$. Then
\begin{align*}
\left \| \E_{\bm{y}} \left [ \bm{A}^{(\bm{T})}_{\alpha}(\theta,\bm{y}) \otimes \bm{A}^{(\bm{T})}_{\alpha}(\theta,\bm{y}) \right ] \right \|_{\mathrm{op}} < \frac{1}{\alpha}.
\end{align*}
\end{lem}

\begin{proof}
Since $\theta> \lambda_+(\alpha)$, Corollary~\ref{cor:physical_stability_right} gives \(\rho (\mathcal{K}_\theta) < 1\).  Let 
\[
\bm{\Gamma}_\theta [\bm{H}] \coloneqq  \bm{L}_\alpha (\theta)^{-1/2} \bm{H} \bm{L}_\alpha (\theta)^{-1/2}.
\]
A direct computation gives
\[
\bm{\Gamma}_\theta \mathcal{K}_\theta \bm{\Gamma}_\theta^{-1} [\bm{H}] = \alpha \E_{\bm{y}} \left [ \bm{A}^{(\bm{T})}_{\alpha} (\theta,\bm{y}) \bm{H} \bm{A}^{(\bm{T})}_{\alpha} (\theta,\bm{y})\right ] \eqqcolon \alpha \mathcal{B}_\theta [\bm{H}] .
\]
Hence
\[
\rho \left  (\alpha \mathcal{B}_\theta\right ) = \rho \left (\mathcal{K}_\theta \right ) < 1.
\]
Moreover, \(\mathcal{B}_\theta\) is self-adjoint with respect to the Frobenius inner product becuse
\[
\langle  \bm{H}_1, \mathcal{B}_\theta [\bm{H}_2] \rangle_{\mathrm{F}} = \langle  \mathcal{B}_\theta [\bm{H}_1],  \bm{H}_2\rangle_{\mathrm{F}}.
\]
Therefore, we have 
\[
\|  \mathcal{B}_\theta\|_{\mathrm{op}} = \rho ( \mathcal{B}_\theta) < \frac{1}{\alpha}.
\]
\end{proof}

Next we use a variant of the Cauchy-Schwarz inequality for Kronecker product. 

\begin{lem} \label{lem:tensorCS}
Let $\bm{A} \in \sym_p(\mathbb{R})$ and $\bm{B} \in \sym_p(\mathbb{R})$ be random matrices with finite second moments. Then
\begin{align*}
\| \E [ \bm{A} \otimes \bm{B} ] \|^2_{\mathrm{op}}  \le \| \E [ \bm{A} \otimes \bm{A} ] \|_{\mathrm{op}}  \| \E [ \bm{B} \otimes \bm{B} ] \|_{\mathrm{op}}.
\end{align*}    
\end{lem}

\begin{proof}
By definition of the operator norm, we have
\begin{align*}
\| \E [ \bm{A} \otimes \bm{B} ] \|^2_{\mathrm{op}}  = \sup_{\| \bm{x} \|=1} |\langle\bm{x}, (\bm{A} \otimes \bm{B}) \bm{x}\rangle|.
\end{align*}
Fix $\bm{x} \in \mathbb{R}^{p r}$ with $\| \bm{x} \|_2 =1$, and write its Schmidt decomposition as 
\[
\bm{x} = \sum_{k=1}^m s_k \bm{u}_k \otimes \bm{v}_k,
\]
where \(m \le \min \{p,r\}\), $(\bm{u}_k)_k$ and $(\bm{v}_k)_k$ are orthonormal families, and $s_k \ge 0$ with $\sum_{k=1}^m s_k^2 =1$. This follows from the singular value decomposition of $\mathrm{mat}(\bm{x}) = \bm{U} \bm{S} \bm{V}^\top$ by identifying the $\bm{u}_k$ with the columns of $\bm{U}$, the $\bm{v}_k$ with the columns of $\bm{V}$ and $s_k$ are the singular values and since the norm condition writes $\| \bm{x} \|^2 = \| \mathrm{mat}(\bm{x})\|^2_F = \sum_{k=1}^{\min(p,r)} s_k^2$. Then, using the elementary identities
\begin{align*}
 (\bm{A} \otimes \bm{B})(\bm{u} \otimes \bm{v}) = (\bm{A} \bm{u}) \otimes (\bm{B} \bm{v}) 
    \quad 
    \mbox{and} 
    \quad 
    \langle (\bm{a} \otimes \bm{b}),(\bm{c} \otimes \bm{d}) \rangle = \langle \bm{a}, \bm{c} \rangle \langle  \bm{b}, \bm{d} \rangle,
\end{align*}
we have
\begin{align*}
\langle \bm{x}, \E \left [\bm{A} \otimes \bm{B} \right] \bm{x}\rangle = \E \left[ \sum_{i,j}^m s_i s_j \langle \bm{u}_i, \bm{A} \bm{u}_j \rangle \langle\bm{v}_i, \bm{B} \bm{v}_j\rangle \right].
\end{align*}
Applying the Cauchy--Schwarz inequality on the product of probability space with the finite index set, equipped with the nonnegative weights $s_i s_j\ge 0$, gives
\begin{align} \label{eq:First_CS}
\left | \left \langle \bm{x},  \E \left [\bm{A} \otimes \bm{B} \right] \bm{x} \right \rangle \right |^2 \le \E \left[ \sum_{i,j}^m s_i s_j \langle \bm{u}_i, \bm{A} \bm{u}_j \rangle^2  \right] \E\left[ \sum_{i,j}^m s_i s_j \langle \bm{v}_i, \bm{B} \bm{v}_j \rangle^2  \right].
\end{align}
Set $\bm{x}_a \coloneqq  \sum_{k=1}^m s_k \bm{u}_k \otimes \bm{u}_k \in \R^{p^2}$ and $\bm{x}_b \coloneqq  \sum_{k=1}^m s_k \bm{v}_k \otimes \bm{v}_k \in \R^{r^2} $. It holds that $\| \bm{x}_a \| = \| \bm{x}_b \| =1$. Moreover, from~\eqref{eq:First_CS} we have
\begin{align*}
\mathbb{E} \left[ \sum_{i,j}^m s_i s_j \langle \bm{u}_i, \bm{A} \bm{u}_j \rangle^2  \right] \mathbb{E}\left[ \sum_{i,j}^m s_i s_j \langle \bm{v}_i, \bm{B} \bm{v}_j \rangle^2  \right] 
    &= 
    \mathbb{E}[  \langle \bm{x}_a, (\bm{A}^{\otimes 2}) \bm{x}_a \rangle^2   ] \mathbb{E}[  \langle \bm{x}_b, \bm{B}^{\otimes 2} \bm{x}_b \rangle^2   ] , \\
    &\le
     \| \mathbb{E}[  \bm{A}^{\otimes 2}   ] \|_{\mathrm{op}} \| \mathbb{E}[  \bm{B}^{\otimes 2}   ] \|_{\mathrm{op}}.
\end{align*}
Since \( \E \left [\bm{A} \otimes \bm{B} \right] \) is symmetric, taking the supremum over unit vectors $\bm{x}$ proves the claim. 
\end{proof}

\begin{proof}[Proof of Lemma~\ref{lem:upperbound_samplecomplexity}]
Combining~\eqref{eq:bound_outlier_eq} with Lemma~\ref{lem:tensorCS}, applied to \(\bm{A} =\bm{A}_\alpha^{(\bm{T})}(\theta,\bm{y}) \) and \(\bm{B} = \bm{C} (\bm{y})-\bm{I}_r\), gives
\[
\frac{1}{\alpha^2} \leq \left \| \E_{\bm{y}} \left [ \bm{A}_\alpha^{(\bm{T})}(\theta,\bm{y})^{\otimes 2}\right ] \right \|_{\mathrm{op}}  \left \| \E_{\bm{y}} \left [ \left (\bm{C} (\bm{y}) - \bm{I}_r \right )^{\otimes 2}\right ] \right \|_{\mathrm{op}} .
\]
By Lemma~\ref{lem:bound_stability_theta}, 
\[
\left \| \E_{\bm{y}} \left [ \bm{A}_\alpha^{(\bm{T})}(\theta,\bm{y})^{\otimes 2}\right ] \right \|_{\mathrm{op}}  < \frac{1}{\alpha},
\]
whereas, by the definition of \(\alpha_{\mathrm{c}}^\ast\) given in~\eqref{eq:def_alpha_cast}, 
\[
\left \| \E_{\bm{y}} \left [ \left (\bm{C} (\bm{y}) - \bm{I}_r \right )^{\otimes 2}\right ] \right \|_{\mathrm{op}}  = \frac{1}{\alpha_{\mathrm{c}}^\ast}.
\]
Therefore, 
\[
\frac{1}{\alpha^2} < \frac{1}{\alpha \alpha_{\mathrm{c}}^\ast},
\]
and so \(\alpha > \alpha_{\mathrm{c}}^\ast\). Thus every \(\alpha\) such that \(\Delta_{\bm{T}} (\alpha)>0\) is at least \(\alpha_{\mathrm{c}}^\ast\). Taking the infimum over this set yields 
\[
\alpha_{\mathrm{c},\min}(\bm{T}) \geq \alpha_{\mathrm{c}}^\ast.
\]
\end{proof}

\subsection{The optimal preprocessing map and its MDE} \label{sec:proof_optimality_Tast}
We now prove Lemma~\ref{lem:optimality_Tast}. We first record the
properties of the Perron reduction introduced in
Definition~\ref{defn:optimal_preprocessing_map} that will be used
throughout the proof.

\begin{lem} \label{lem:perron_reduction_properties}
Let \(\bm{H}_\ast \in \operatorname{relint}(\mathscr{P}_\ast)\), let
\(\bm{U}_\ast \in \R^{r\times p_\ast}\) have orthonormal columns spanning
\(\operatorname{range}(\bm{H}_\ast)\), and write \(\bm{H}_\ast=\bm{U}_\ast\bm{D}_\ast\bm{U}_\ast^\top\). Then \(\bm{D}_\ast \in \sym^{++}_{p_\ast}(\R)\) and the following hold almost surely:
\[
\bm{C} (\bm{y}) \bm{U}_\ast = \bm{U}_\ast \bm{C}_\ast(\bm{y}),
\qquad
\bm{C}_\ast(\bm{y})\bm{U}_\ast^\top = \bm{U}_\ast^\top\bm{C}(\bm{y}).
\]
Moreover,
\[
\E_{\bm{y}} \left[
(\bm{C}_\ast(\bm{y})-\bm{I}_{p_\ast})
\bm{D}_\ast
(\bm{C}_\ast(\bm{y})-\bm{I}_{p_\ast})
\right]
= \frac1{\alpha_{\mathrm c}^\ast}\bm{D}_\ast,
\]
and
\[
\left \| \mathcal{A}_\ast \right \|_{\mathrm{op}} = \left \| \mathcal{A} \right \|_{\mathrm{op}}=\frac1{\alpha_{\mathrm c}^\ast},
\]
where \(\mathcal{A}\) is defined in~\eqref{eq:operator_A} and \(\mathcal{A}_\ast\) is with \(\bm{C}(\bm{y})\) replaced by \(\bm{C}_\ast(\bm{y})\).
\end{lem}

\begin{proof}
Since \(\operatorname{range}(\bm{H}_\ast)=\operatorname{range}(\bm{U}_\ast)\),
the matrix \(\bm{D}_\ast=\bm{U}_\ast^\top\bm{H}_\ast\bm{U}_\ast\) is positive definite. Set \(\bm{B}(\bm{y})\coloneqq\bm{C}(\bm{y})-\bm{I}_r\). We first show that \(\operatorname{range}(\bm{H}_\ast)\) is invariant under
\(\bm{B}(\bm{y})\) almost surely. Let \(\bm{v}\in\ker(\bm{H}_\ast)\). Since \(\bm{H}_\ast \in \mathscr{P}_\ast\), we have
\[
\mathcal{A}[\bm{H}_\ast] = \|\mathcal{A}\|_{\mathrm{op}} \bm{H}_\ast= \frac1{\alpha_{\mathrm c}^\ast} \bm{H}_\ast.
\]
Moreover,
\[
0 = \bm{v}^\top\mathcal{A}[\bm{H}_\ast]\bm{v} = \E \left[
\bm{v}^\top \bm{B}(\bm{y})\bm{H}_\ast\bm{B}(\bm{y}) \bm{v} \right].
\]
The integrand is nonnegative, hence
\[
\bm{H}_\ast^{1/2}\bm{B}(\bm{y})\bm{v}=\bm{0}
\]
almost surely. Thus \(\ker(\bm{H}_\ast)\) is invariant under \(\bm{B}(\bm{y})\).
Since \(\bm{B}(\bm{y})\) is symmetric, its orthogonal complement \(\operatorname{range} (\bm{H}_\ast)\) is invariant as well. Therefore
\[
\bm{C}(\bm{y})\bm{U}_\ast = \bm{U}_\ast\bm{C}_\ast(\bm{y}),
\]
and transposing yields
\[
\bm{C}_\ast(\bm{y})\bm{U}_\ast^\top = \bm{U}_\ast^\top\bm{C}(\bm{y}).
\]

Compressing the Perron equation
\[
\mathcal{A}[\bm{H}_\ast] =\frac1{\alpha_{\mathrm c}^\ast}\bm{H}_\ast
\]
to \(\operatorname{range}(\bm{H}_\ast)\) now gives
\[
\E \left[ (\bm{C}_\ast-\bm{I}_{p_\ast}) \bm{D}_\ast (\bm{C}_\ast-\bm{I}_{p_\ast}) \right] = \frac1{\alpha_{\mathrm c}^\ast}\bm{D}_\ast.
\]
Hence
\[
\|\mathcal{A}_\ast \|_{\mathrm{op}} \ge \frac1{\alpha_{\mathrm{c}}^\ast}.
\]
On the other hand, the isometric embedding
\[
\bm{H} \mapsto\bm{U}_\ast\bm{H}\bm{U}_\ast^\top
\]
intertwines \(\mathcal{A}_\ast\) with the restriction of \(\mathcal{A}\) to matrices supported on \(\operatorname{range}(\bm{H}_\ast)\). Therefore
\[
\|\mathcal{A}_\ast \|_{\mathrm{op}} \le \|\mathcal{A}\|_{\mathrm{op}}
 =\frac1{\alpha_{\mathrm c}^\ast},
\]
which proves equality.
\end{proof}

Let $\bm{M}^\ast_\alpha$ denote the physical branch associated with $\bm{T} = \bm{T}_\ast$. For fixed \(\alpha\), define
\begin{align*}
\bm{L} (z) \coloneqq  - \frac{1}{\alpha} \bm{M}^\ast_\alpha(z), \qquad z \in \C_+.
\end{align*}
This rescaling is convenient because the distinguished spectral parameter \(z=1/\alpha\) corresponds to the simple candidate fixed point \(\bm{L}=\bm{I}_{p_\ast}\). In terms of \(\bm{L}\), the matrix Dyson equation~\eqref{eq:MDE} becomes
\begin{align} \label{eq:MDE_for_L}
\bm{L} (z) = \widetilde{\bm{\Psi}}_{\alpha,z} (\bm{L} (z)  ),
\end{align}
where
\begin{align} \label{eq:def_Psi}
\widetilde{\bm{\Psi}}_{\alpha,z}(\bm{L})
\coloneqq 
\left(z 
\alpha  \bm{I}_{p_\ast}
-
\alpha \E_{\bm{y}} \left[
\bm{T}_\ast (\bm{y})
\bigl(\bm{I}_{p_\ast}-\bm{L} \bm{T}_\ast(\bm{y}) \bigr)^{-1}
\right]
\right)^{-1}.
\end{align}
For a real algebraic solution \(\bm{L}\) at \(\lambda\), let
\begin{align}
\label{eq:definition_Tilde_mathcalK}
    \widetilde{\mathcal{K}}_{\bm{L}} \colon \bm{H} \mapsto D \widetilde{\bm{\Psi}}_{\alpha,\lambda}(\bm{L})[\bm{H}].
\end{align}
If \(\bm{M} = -\alpha \bm{L}\), then $ \widetilde{\mathcal{K}}_{\bm{L}} = \mathcal{K}_{\bm{M}}$, where \(\mathcal{K}_{\bm{M}}\) is the stability operator defined in~\eqref{eq:mathcal_K}. Thus the stability characterization of
Lemma~\ref{lem:physical_sol} applies directly to the rescaled
fixed-point equation.

By Lemma~\ref{lem:upperbound_samplecomplexity}, $\alpha_{\mathrm{c},\min}(\bm{T}_\ast) \ge \alpha_{\mathrm{c}}^{\ast}$. It therefore remains to prove that a separating outlier. exists for every \(\alpha > \alpha_{\mathrm{c}}^\ast\). This implies $\alpha_{\mathrm{c},\max}(\bm{T}_\ast) \le \alpha_{\mathrm{c}}^{\ast}$, and together with $\alpha_{\mathrm{c},\min}(\bm{T}_\ast) \le \alpha_{\mathrm{c},\max}(\bm{T}_\ast)$, both thresholds equal \(\alpha_{\mathrm c}^\ast\).

We first analyze the real solutions of the rescaled MDE at and above \(1/\alpha\). This will show that, for \(\alpha>\alpha_{\mathrm c}^\ast\),
\[
\lambda_+(\alpha)<\frac1\alpha,
\]
after which it remains only to verify that \(1/\alpha\) satisfies the
limiting outlier equation. This bound on the upper edge could presumably also be obtained from the variational characterization presented in Appendix~\ref{appendix:limiting_variational_edge}. Instead, we follow an approach closely related to that of~\cite{mergny24a,yang2026}, although the tools used there are different from ours.

\begin{lem} \label{lem:qualitative_behavior_L}
The following statements hold.
\begin{enumerate}
\item For every \(\lambda>1/\alpha\),
\[
\lambda \notin \supp(\mu_\alpha)
\qquad\text{and}\qquad
\bm{0}_{p_\ast \times p_\ast}
\prec
\bm{L}(\lambda)
\prec
\bm{I}_{p_\ast}.
\]
\item At $\lambda = 1/\alpha$:
\begin{enumerate}
\item if $\alpha < \alpha_{\mathrm{c}}^\ast$, then $$\frac1\alpha\notin\supp(\mu_\alpha), \qquad \bm{L}(1/\alpha)=\bm{I}_{p_\ast};$$
\item if $\alpha > \alpha_{\mathrm{c}}^\ast$, then \[
    \frac1\alpha\notin\supp(\mu_\alpha),
    \qquad
    \bm{0}_{p_\ast \times p_\ast}  \prec  \bm{L}(1/\alpha)  \prec \bm{I}_{p_\ast},
    \]
    and
    \[
    \rho \left( \widetilde{\mathcal{K}}_{\bm{L}(1/\alpha)} \right)<1.
    \]
    \end{enumerate}
\end{enumerate}
\end{lem}

The proof is deferred to Section~\ref{sec:proof_of_behavior_L}. Part~(1) of Lemma~\ref{lem:qualitative_behavior_L} implies \((1/\alpha,\infty)\cap\supp(\mu_\alpha)=\emptyset,\) and therefore
\[
\lambda_+(\alpha)\le\frac1\alpha.
\]
If \(\alpha>\alpha_{\mathrm c}^\ast\), Part~(2b) additionally gives \(1/\alpha\notin\supp(\mu_\alpha)\). Since \(\lambda_+(\alpha)\in\supp(\mu_\alpha)\), equality is impossible, and therefore 
\[
\lambda_+(\alpha)<\frac1\alpha.
\]

Next we show that the point \(1/\alpha\) makes the limiting Schur-complement matrix singular. This is enough to prove supercriticality, although it does not by itself identify the top outlier with \(1/\alpha\).

\begin{lem} \label{lem:one_over_alpha_outlier}
    Let $\alpha > \alpha_{\mathrm{c}}^\ast$. Then
    \[
    \det \left ( \bm{H}_\alpha \left ( \frac{1}{\alpha} \right ) \right )=0.
    \]
    Therefore, \(\Delta_{\bm{T}_\ast} (\alpha) >0\).
\end{lem}

\begin{proof}
Set \(\bm{L} \coloneqq\bm{L} (1/\alpha)\). By Lemma~\ref{lem:qualitative_behavior_L}, \(\bm{0}_{p_\ast \times p_\ast}\prec\bm{L}\prec\bm{I}_{p_\ast}\). To ease notations, we write $\bm{T}_\ast \equiv \bm{T}_\ast(\bm{y})$ and $\bm{C}_\ast \equiv \bm{C}_\ast(\bm{y})$. Moreover, set \(\bm{Q}_\ast \coloneqq \bm{T}_\ast \left ( \bm{I}_{p_\ast} - \bm{L} \bm{T}_\ast \right)^{-1}\). Evaluating the rescaled MDE at \(z=\alpha^{-1}\) and right-multiplying by \(\bm{L}\) gives
\begin{align} \label{eq:MDE_atoneoveralpha}
\frac{1}{\alpha} (\bm{I}_{p_\ast} -\bm{L}) =- \E[ \bm{Q}_\ast] \bm{L}  .
\end{align}
On the other hand, using $\bm{T}_\ast \bm{C}_\ast = \bm{C}_\ast - \bm{I}_{p_\ast}$ and $\E[ \bm{T}_\ast \bm{C}_\ast] = \E[ \bm{C}_\ast - \bm{I}_{p_\ast}] = \bm{0}_{p_\ast \times p_\ast}$, we obtain
\begin{equation} \label{eq:identity_operator_B}
\begin{split}
    \E[ \bm{Q}_\ast (\bm{I}_{p_\ast} - \bm{L}) \bm{C}_\ast ] 
    &=
    \E[ \bm{Q}_\ast (\bm{I}_{p_{\ast}} - \bm{L}) ( \bm{T}_\ast \bm{C}_\ast +\bm{I}_{p_\ast}) ] \\
    &= 
     \E[ \bm{Q}_\ast ( \bm{T}_\ast \bm{C}_\ast  + \bm{I}_{p_\ast}- \bm{L}\bm{T}_\ast \bm{C}_\ast - \bm{L})  ] \\
    &= 
    \E[ \bm{Q}_\ast( \bm{I}_{p_\ast}  - \bm{L} \bm{T}_\ast) \bm{C}_\ast   -  \bm{Q}_\ast \bm{L}  ] \\
    &=
    \E[ \bm{T}_\ast \bm{C}_\ast] - \E[\bm{Q}_\ast \bm{L}  ] \\
    &= - \E[\bm{Q}_\ast \bm{L}  ].
\end{split}
\end{equation}
Combining~\eqref{eq:MDE_atoneoveralpha} and~\eqref{eq:identity_operator_B} yields
\begin{align*}
\E[ \bm{Q}_\ast (\bm{I}_{p_\ast} - \bm{L}) \bm{C}_\ast ]  =     \frac{1}{\alpha} (\bm{I}_{p_\ast} -\bm{L}).
\end{align*}

Right-multiplying by $\bm{U}_\ast^\top$ and using the intertwining relation $\bm{C}_\ast \bm{U}_\ast^\top = \bm{U}_\ast^\top \bm{C}$, we get
\begin{align} \label{eq:eigenmatrix_of_H_matform}
\E[  \bm{Q}_\ast (\bm{I}_{p_{\ast}} - \bm{L})\bm{U}_\ast^\top \bm{C} ] -  \frac{1}{\alpha} (\bm{I}_{p_{\ast}} -\bm{L}) \bm{U}_\ast^\top = \bm{0}_{p_\ast \times r} .
\end{align}
Since \(\bm{I}_{p_\ast} - \bm{L} \succ \bm{0}_{p_\ast \times p_\ast}\) and \(\bm{U}_\ast\) has orthonormal columns, \((\bm{I}_{p_{\ast}} - \bm{L})\bm{U}_\ast^\top \neq \bm{0}_{p_\ast \times r}\). Transposing and applying the standard column-vectorization identity \(\operatorname{vec} (\bm{AYB}) =(\bm{B}^\top \otimes \bm{A}) \operatorname{vec}(\bm{Y})\), we get
\begin{align} \label{eq:eigenmatrix_of_H_vecform}
\left(\E[  \bm{Q}_\ast \otimes \bm{C} ] -  \frac{1}{\alpha} \bm{I}_{p_\ast \bm{r}} \right) \bm{w} = \bm{0},
\end{align}
where
\begin{align}
\bm{w} \coloneqq  \mathrm{vec}( \left ( (\bm{I}_{p_{\ast}} - \bm{L})\bm{U}_\ast^\top \right)^\top) \neq \bm{0}_{p_\ast r},
\end{align}
which follows by Part (2b) of Lemma~\ref{lem:qualitative_behavior_L}. By definition of $ \bm{H}_\alpha (\theta) $, 
\[
\bm{H}_\alpha (1/\alpha) = \E [\bm{Q}_\ast \otimes \bm{C}] - \frac{1}{\alpha} \bm{I}_{p_\ast r},
\]
hence $\det \bm{H}_\alpha (1/\alpha) = 0$.

It remains to translate this singularity into the supercriticality criterion. Since \(\lambda_+(\alpha) < 1/\alpha\), the function \(h_\alpha(x) = \lambda_1 (\bm{H}_\alpha (x))\) is well-defined at \(x = 1/\alpha\). The singularity of \(\bm{H}_\alpha (1/\alpha) \) implies that \(h_\alpha(1/\alpha) \ge 0\). Moreover, the strict Loewner monotonicty proved in Theorem~\ref{thm:BBP_transition} gives, for \(\lambda_+ (\alpha) < x < 1/\alpha\), 
\[
h_\alpha (x) \ge h_\alpha (1/\alpha) + \left ( \frac{1}{\alpha}-x\right )>0.
\]
Hence \(\Delta_{\bm{T}_\ast} (\alpha) = \sup_{x > \lambda_+(\alpha)} h_\alpha(x) >0\).
\end{proof} 

We can now conclude the proof of Lemma~\ref{lem:optimality_Tast}. 

\begin{proof}[Proof of Lemma~\ref{lem:optimality_Tast}]
By Lemma~\ref{lem:upperbound_samplecomplexity},
\[
\alpha_{\mathrm c,\min}(\bm{T}_\ast) \ge \alpha_{\mathrm c}^\ast.
\]
Conversely, Lemma~\ref{lem:one_over_alpha_outlier} shows that \(\Delta_{\bm{T}_\ast} (\alpha)>0\) for every \(\alpha > \alpha_{\mathrm{c}}^\ast\). Thus
\[
(\alpha_{\mathrm{c}}^\ast, \infty) \subseteq \{ \alpha > 0 \colon \Delta_{\bm{T}_\ast} (\alpha)>0 \},
\]
and therefore
\[
\alpha_{\mathrm c,\max}(\bm{T}_\ast) \le \alpha_{\mathrm c}^\ast.
\]
Since, by definition, \(\alpha_{\mathrm c,\min}(\bm{T}_\ast) \le \alpha_{\mathrm c,\max}(\bm{T}_\ast)\), we conclude that
\[
\alpha_{\mathrm c,\min}(\bm{T}_\ast) = \alpha_{\mathrm c,\max}(\bm{T}_\ast) =\alpha_{\mathrm c}^\ast.
\]
\end{proof}

\subsection{Fixed-point analysis of the rescaled MDE} \label{sec:proof_of_behavior_L}
Here we prove Lemma~\ref{lem:qualitative_behavior_L}. We start by establishing the properties of $\widetilde{\bm{\Psi}}_{\alpha,\lambda}$ needed afterwards. 

\begin{lem} \label{lem:properties_of_Phi}
Let $\lambda>\frac{1}{\alpha}$. Then, for every $\bm{0}_{p_\ast \times p_\ast}  \preceq \bm{L}_1 \preceq \bm{L}_2 \prec \bm{I}_{p_\ast}$, it holds that 
\[
\bm{0}_{p_\ast \times p_\ast} \prec \widetilde{\bm{\Psi}}_{\alpha,\lambda}(\bm{L}_1) \preceq  \widetilde{\bm{\Psi}}_{\alpha,\lambda}(\bm{L}_2) \prec \bm{I}_{p_{\ast}} .
\]
\end{lem}

\begin{proof}
Fix $\bm{0}_{p_\ast \times p_\ast} \preceq \bm{L} \prec \bm{I}_{p_\ast}$, and set
\begin{align}  \label{eq:def_Q_of_Landy}
\bm{Q}(\bm{y})  \coloneqq \bm{I}_{p_\ast} + ( \bm{I}_{p_\ast} - \bm{L})^{1/2} (\bm{C}_\ast(\bm{y}) -\bm{I}_{p_\ast})  ( \bm{I}_{p_\ast} - \bm{L})^{1/2}.
\end{align}
Equivalently,
\[
\bm{Q}(\bm{y})  =  \bm{L} +  ( \bm{I}_{p_\ast} - \bm{L})^{1/2} \bm{C}_\ast(\bm{y})  ( \bm{I}_{p_\ast} - \bm{L})^{1/2} \succ \bm{0}_{p_\ast \times p_\ast},
\]
since \(\bm{I}_{p_\ast} - \bm{L} \succ \bm{0}_{p_\ast \times p_\ast}\). Using \(\bm{T}_\ast = \bm{I}_{p_\ast} - \bm{C}_\ast^{-1}\), a direct computation gives
\begin{align*}
\bm{T}_\ast(\bm{y})(\bm{I}_{p_\ast} - \bm{L}\bm{T}_\ast(\bm{y}))^{-1} = ( \bm{I}_{p_\ast} - \bm{L})^{1/2} ( \bm{I}_{p_\ast}  - \bm{Q}(\bm{y})^{-1})  ( \bm{I}_{p_\ast} - \bm{L})^{1/2}.
\end{align*}
Moreover, 
\[
\E_{\bm{y}} [\bm{Q}(\bm{y})] = \bm{L} +  ( \bm{I}_{p_\ast} - \bm{L}) = \bm{I}_{p_\ast}.
\]
Since $\bm{L} \mapsto \bm{L}^{-1}$ is convex, Jensen's inequality yields 
\[
\E_{\bm{y}}[\bm{Q}(\bm{y})^{-1}] \succeq \E_{\bm{y}}[\bm{Q}(\bm{y})]^{-1} = \bm{I}_{p_\ast}.
\]
Thus, taking expectations, we obtain
\begin{align*}
\E_{\bm{y}}[\bm{T}_\ast(\bm{y})(\bm{I}_{p_\ast} - \bm{L}\bm{T}_\ast(\bm{y}))^{-1}] \preceq \bm{0}_{p_\ast \times p_\ast}.
\end{align*}
Hence, according to~\eqref{eq:def_Psi},
\begin{align}
\label{eq:last_equation_bound_Phi}
\bm{0}_{p_\ast \times p_\ast} \prec \widetilde{\bm{\Psi}}_{\alpha,\lambda}(\bm{L}) \preceq \frac{1}{\lambda \alpha} \bm{I}_{p_\ast} \prec \bm{I}_{p_\ast}.
\end{align}

It remains to prove monotonicity. For $\bm{H} \in \sym_{p_\ast}(\R)$, we have 
\begin{align} \label{eq:derivative_of_Phi_Lambda}
D\widetilde{\bm{\Psi}}_{\alpha,\lambda}(\bm{L})[\bm{H}]
=
\alpha \, \widetilde{\bm{\Psi}}_{\alpha,\lambda}(\bm{L})
\mathbb E\!\left[
\bm{T}_\ast(\bm{y})(\bm{I}_{p_\ast} - \bm{L}\bm{T}_\ast(\bm{y}))^{-1} \bm{H} \bm{T}_\ast(\bm{y})(\bm{I}_{p_\ast} - \bm{L}\bm{T}_\ast(\bm{y}))^{-1}
\right]
\widetilde{\bm{\Psi}}_{\alpha,\lambda}(\bm{L}).
\end{align}
Since \( \widetilde{\bm{\Psi}}_{\alpha,\lambda}(\bm{L}) \succ \bm{0}\), the derivative in~\eqref{eq:derivative_of_Phi_Lambda} is order-preserving. Integrating the derivative along the segment joining \(\bm{L}_1\) and \(\bm{L}_2\) proves
\[
\widetilde{\bm{\Psi}}_{\alpha,\lambda}(\bm{L}_1) \preceq \widetilde{\bm{\Psi}}_{\alpha,\lambda}(\bm{L}_2).
\]
\end{proof}

Let $(\bm{L}^{0}_{\lambda,n})_n$ be defined by the recursion $\bm{L}^{0}_{\lambda,{n+1}} = \widetilde{\bm{\Psi}}_{\alpha,\lambda}(\bm{L}^{0}_{\lambda,n})$ starting at $\bm{L}^{0}_{\lambda,0} = \bm{0}_{p_{\ast} \times p_{\ast} }$. Since $\widetilde{\bm{\Psi}}_{\alpha,\lambda}$ is order-preserving and bounded by Lemma~\ref{lem:properties_of_Phi}, as $n \to \infty$ the sequence converges to a limit 
\begin{align}
\label{eq:def_Lzerolambda}
    \bm{L}^{0}_\lambda \coloneqq  \lim_{n \to \infty} \bm{L}^{0}_{\lambda,n},
\end{align}
and satisfies
\begin{align}
\label{eq:property_fixed_point}
    \bm{0}_{p_{\ast} \times p_\ast} \prec \bm{L}^{0}_\lambda \prec \bm{I}_{p_\ast},
\end{align}
and
\begin{align}
\label{eq:MDE_for_Lzerolambda}
    \bm{L}^{0}_\lambda = \widetilde{\bm{\Psi}}_{\alpha,\lambda}(\bm{L}^{0}_\lambda).
\end{align}
We show that this fixed point corresponds to the analytical continuation of $\bm{L}$ at $\lambda$. 

\begin{lem} \label{lem:continuation_eq_fixedpoint}
Let $\bm{L}^{0}_\lambda$ be given by~\eqref{eq:def_Lzerolambda}. Then \(\lambda \notin \supp(\mu_\alpha)\) and $\bm{L}({\lambda}) = \bm{L}^{0}_\lambda$. 
\end{lem}
\begin{proof}
By Lemma~\ref{lem:physical_sol}, it is enough to prove that $\rho \left ( \widetilde{\mathcal{K}}_{\bm{L}^{0}_\lambda}\right)<1$ with $\widetilde{\mathcal{K}}$ given in~\eqref{eq:definition_Tilde_mathcalK}. For brevity, write
\begin{align*}
\bm{Q}_\ast^0 \coloneqq  \bm{T}_{\ast}(\bm{I}_{p_\ast}-\bm{L}^{0}_\lambda \bm{T}_{\ast})^{-1},
\end{align*}
and
\begin{align*}
\bm{R}^0 \coloneqq  \bm{L}^{0}_\lambda(\bm{I}_{p_\ast} - \bm{L}^{0}_\lambda) \succ \bm{0}_{p_{\ast} \times p_{\ast} }.
\end{align*}
Since $\bm{L}^{0}_\lambda$ is positive, we can rewrite the fixed point equation~\eqref{eq:MDE_for_Lzerolambda} as 
\begin{align} \label{eq:MDE_for_Lzerolambda_2}
\bm{L}^{0}_\lambda - \alpha \lambda (\bm{L}^{0}_\lambda)^2 = - \alpha \bm{L}^{0}_\lambda \E[\bm{Q}_\ast^0 ] \bm{L}^{0}_\lambda.
\end{align}
Using~\eqref{eq:derivative_of_Phi_Lambda}, we have
\begin{align*}
    \bm{R}^0 - \widetilde{\mathcal{K}}_{\bm{L}^{0}_\lambda} [\bm{R}^0] 
    &=
    \bm{L}^{0}_\lambda(\bm{I} - \bm{L}^{0}_\lambda)   -\alpha \bm{L}^{0}_\lambda \mathbb{E}[\bm{Q}_\ast^0 \bm{L}^{0}_\lambda(\bm{I}_{p_\ast} - \bm{L}^{0}_\lambda) \bm{Q}_\ast^0] \bm{L}^{0}_\lambda, 
\end{align*}
and using~\eqref{eq:MDE_for_Lzerolambda_2}, it further simplifies into
\begin{align*}
     \bm{R}^0 - \widetilde{\mathcal{K}}_{\bm{L}^{0}_\lambda} [\bm{R}^0]
    &=
    (\alpha \lambda -1) \bm{L}^{0}_\lambda + \alpha \bm{L}^{0}_\lambda \E [ \bm{Q}_\ast^0(\bm{I}_{p_\ast}- \bm{L}^{0}_\lambda) \bm{C}_{\ast} (\bm{I}_{p_\ast}- \bm{L}^{0}_\lambda) \bm{Q}_\ast^0 ] \bm{L}^{0}_\lambda.
\end{align*}
The first term is positive definite because \(\lambda > 1/\alpha\), whereas the second term is positive semidefinite. Thus, we have  
\begin{align*}
\bm{R}^0 - \widetilde{\mathcal{K}}_{\bm{L}^{0}_\lambda} [\bm{R}^0]  \succ \bm{0}_{p_\ast \times p_\ast }.
\end{align*}
Set
\[
q \coloneqq \lambda_1 \left ( (\bm{R}^0 )^{-1/2} \widetilde{\mathcal{K}}_{\bm{L}^{0}_\lambda} [\bm{R}^0] (\bm{R}^0 )^{-1/2} \right ) \in [0,1).
\]
Then \(\widetilde{\mathcal{K}}_{\bm{L}^{0}_\lambda} [\bm{R}^0] \preceq q \bm{R}^0\). Since \(\widetilde{\mathcal{K}}_{\bm{L}^{0}_\lambda}\) preserves the positive semidefinite cone, its operator norm in the order-unit norm 
\[
\| \bm{H}\|_{\bm{R}^0} \coloneqq \inf \left \{ c > 0 \colon - c \bm{R}^0 \preceq \bm{H} \preceq c \bm{R}^0 \right \}
\]
is at most \(q\), Therefore
\[
\rho \left (\widetilde{\mathcal{K}}_{\bm{L}^{0}_\lambda} \right ) \le q < 1.
\]
Lemma~\ref{lem:physical_sol} now identifies \(\bm{L}_\lambda^0\) with the physical real continuation and, in particular, gives \(\lambda \notin \supp (\mu_\alpha)\).
\end{proof}

Combining Lemma~\ref{lem:continuation_eq_fixedpoint} with~\eqref{eq:property_fixed_point} concludes the proof of Part (1).\\

At $z=1/\alpha$, the rescaled MDE is given by~\eqref{eq:MDE_atoneoveralpha} from which one can immediately check that $\bm{L}^{0} = \bm{I}_{p_\ast}$ is always a solution of the real MDE, that is $\bm{I}_{p_\ast} = \widetilde{\bm{\Psi}}_{\alpha,1/\alpha}(\bm{I}_{p_\ast})$ and thus one needs to understand if this special point corresponds to the proper analytical continuation of $\bm{L}(.)$. To do so, remark that  the derivative of $\widetilde{\bm{\Psi}}_{\alpha,\lambda}$ is given by~\eqref{eq:derivative_of_Phi_Lambda}, and since $\widetilde{\bm{\Psi}}_{\alpha,1/\alpha}(\bm{I}_{p_\ast}) = \bm{I}_{p_\ast}$ and $\bm{T}_\ast ( \bm{I}_{p_\ast} - \bm{T}_\ast)^{-1} = \bm{C}_\ast -  \bm{I}_{p_\ast} $, we get  
\begin{align}
\label{eq:D_Phi_1overalpha_Identity}
    D\widetilde{\bm{\Psi}}_{\alpha,1/\alpha}( \bm{I}_{p_\ast})[\bm{H}]
=
\alpha 
\mathbb E\!\left[
(\bm{C}_\ast(\bm{y}) -  \bm{I}_{p_\ast}) \bm{H} (\bm{C}_\ast(\bm{y}) -  \bm{I}_{p_\ast}) \right] = \alpha \mathcal{A}_\ast [\bm{H}].
\end{align}
and thus the linear operator $\widetilde{\mathcal{K}}_{\bm{I}_r} : \bm{H} \mapsto   D\widetilde{\bm{\Psi}}_{\alpha,1/\alpha}( \bm{I}_r)[\bm{H}]$ has operator norm
\begin{align}
    \| \widetilde{\mathcal{K}}_{\bm{I}_r}\|_{\mathrm{op}}
    &=
    \alpha \| \E [ (\bm{C}_\ast(\bm{y}) -  \bm{I}_{p_\ast})^{\otimes 2} ]\|_{\mathrm{op}} , \\
    &= \frac{\alpha}{\alpha_{\mathrm{c}}^\ast} < 1 \quad \mbox{for } \alpha < \alpha_{\mathrm{c}}^\ast.
\end{align}
Since \(\mathcal{A}_\alpha\) is self-adjoint and \(\| \mathcal{A}_\ast\|_{\mathrm{op}} = 1/\alpha_{\mathrm{c}}^\ast\) by Lemma~\ref{lem:perron_reduction_properties},
\[
\rho \left ( \widetilde{\mathcal{K}}_{\bm{I}_{p_\ast}}\right ) = \frac{\alpha}{\alpha_{\mathrm{c}}^\ast}.
\]
If $\alpha < \alpha_{\mathrm{c}}^\ast$, this spectral radius is strictly smaller than one. Lemma~\ref{lem:physical_sol} therefore identifies the stable algebraic solution with the physical continuation: $\bm{L}(1/\alpha) = \bm{I}_{p_\ast}$ and $1/\alpha \notin \supp (\mu_\alpha)$. This proves Part (2a).\\

We now turn to Part (2b).
\begin{lem} \label{lem:properties_of_Phi_atoneoveralpha}
For every $\bm{0}_{p_\ast \times p_\ast}  \preceq \bm{L}_1 \preceq \bm{L}_2 \preceq \bm{I}_{p_\ast}$, it holds that 
\[
\bm{0}_{p_\ast \times p_\ast} \prec \widetilde{\bm{\Psi}}_{\alpha,1/\alpha}(\bm{L}_1) \preceq  \widetilde{\bm{\Psi}}_{\alpha,1/\alpha}(\bm{L}_2) \preceq \bm{I}_{p_{\ast}}.
\]
\end{lem}

\begin{proof}
For \(\bm{L} \prec \bm{I}_{p_\ast}\), the proof of Lemma~\ref{lem:properties_of_Phi} applies verbatim with $\lambda=1/\alpha$ and gives 
\[
\bm{0}_{p_\ast \times p_\ast} \prec \widetilde{\bm{\Psi}}_{\alpha,1/\alpha}(\bm{L}) \preceq \bm{I}_{p_{\ast}},
\]
together with order preservation on the half-open interval \(\{\bm{0} \preceq \bm{L} \prec \bm{I}\}\). It remains only to include boundary points with eigenvalue \(1\). For every \(\bm{0} \preceq \bm{L} \preceq \bm{I}\), the matrix \(\bm{I}-\bm{L} \bm{T}_\ast\) is invertible. Hence \(\widetilde{\bm{\Psi}}_{\alpha, 1/\alpha}\) is continuous on the closed order interval \([\bm{0},\bm{I}]\). Approximating \(\bm{L}_j\) by \((1-\varepsilon) \bm{L}_j \prec \bm{I}\) and letting \(\varepsilon \downarrow 0\) extends both the order-preserving property and the upper bound closed interval. Finally, \(\widetilde{\bm{\Psi}}_{\alpha, 1/\alpha} (\bm{I}_{p_\ast}) = \bm{I}_{p_\ast}\), so the claimed bounds follow. 
\end{proof}

Note that Lemma~\ref{lem:properties_of_Phi_atoneoveralpha} differs from Lemma~\ref{lem:properties_of_Phi} only by the last bound not being strict. Consider once again the recursion 
$\bm{L}^{0}_{1/\alpha,{n+1}} = \widetilde{\bm{\Psi}}_{\alpha,1/\alpha}(\bm{L}^{0}_{1/\alpha,n})$ starting at $\bm{L}^{0}_{1/\alpha,0} = \bm{0}_{p_{\ast} \times p_{\ast} }$. Since $\widetilde{\bm{\Psi}}_{\alpha,1/\alpha}$ is order-preserving and bounded by Lemma~\ref{lem:properties_of_Phi_atoneoveralpha}, this sequence converges to $\bm{L}^{0}_{1/\alpha} \coloneqq  \lim_{n\to\infty}  \bm{L}^{0}_{1/\alpha,{n}}$  which is bounded:  $\bm{0}_{p_\ast \times p_\ast} \prec \bm{L}^{0}_{1/\alpha} \preceq \bm{I}_{p_\ast}$. We show that this last bound is strict when $\alpha>\alpha_{\mathrm{c}}^\ast$ by showing that it is dominated by a matrix strictly lower than the identity.

Recall that $\bm{H}_\ast = \bm{U}_\ast\bm{D}_\ast \bm{U}_\ast^\top$ and \(\bm{D}_\ast \succ \bm{0}_{p_\ast \times p_\ast}\), and by Lemma~\ref{lem:perron_reduction_properties},
\begin{align} \label{eq:Perron_Eigenmatrix_D}
    \mathbb{E}[ (\bm{C}_\ast - \bm{I}_{p_\ast}) \bm{D}_\ast (\bm{C}_\ast - \bm{I}_{p_\ast})]
    &=
    \frac{1}{\alpha_{\mathrm{c}}^\ast} \bm{D}_\ast  .
\end{align}
Fix \(\alpha > \alpha_{\mathrm{c}}^\ast\). For $\varepsilon >0$ small enough, we have $\bm{0}_{p_\ast \times p_\ast} \prec \bm{I}_{p_\ast}- \epsilon \bm{D}_\ast \prec \bm{I}_{p_\ast}$. Taylor expansion at the fixed point \(\bm{I}_{p_\ast}\), together with~\eqref{eq:D_Phi_1overalpha_Identity}, gives
\begin{align}
    \widetilde{\bm{\Psi}}_{\alpha,1/\alpha}(\bm{I}_{p_\ast}- \varepsilon \bm{D}_\ast) 
    &=
    \widetilde{\bm{\Psi}}_{\alpha,1/\alpha}(\bm{I}_{p_\ast})  - \varepsilon D\widetilde{\bm{\Psi}}_{\alpha,1/\alpha}(\bm{I}_{p_\ast})[\bm{D}_\ast] + \bm{R}_{\varepsilon}, \\
    &=
    \bm{I}_{p_\ast} - \varepsilon \frac{\alpha}{\alpha_{\mathrm{c}}^\ast} \bm{D}_\ast + \bm{R}_{\varepsilon}
\end{align}
with $\| \bm{R}_{\varepsilon} \|_{\mathrm{op}} = o(\varepsilon)$. Thus we have 
\begin{align}
    (\bm{I}_{p_\ast}- \varepsilon \bm{D}_\ast) - \widetilde{\bm{\Psi}}_{\alpha,1/\alpha}(\bm{I}_{p_\ast}- \varepsilon \bm{D}_\ast) = \varepsilon \left (\frac{\alpha}{\alpha_{\mathrm{c}}^\ast} -1 \right ) \bm{D}_\ast - \bm{R}_{\epsilon} \succ \bm{0}_{p_\ast \times p_\ast}
\end{align}
for all small enough $\varepsilon$. Hence 
\begin{align}
\label{eq:dominated_point_recursion_Phi}
    \widetilde{\bm{\Psi}}_{\alpha,1/\alpha}(\bm{I}_{p_\ast}- \varepsilon \bm{D}_\ast ) \prec (\bm{I}_{p_\ast}- \varepsilon \bm{D}_\ast) \prec \bm{I}_{p_\ast} .
\end{align}
Since  the map $ \widetilde{\bm{\Psi}}_{\alpha,1/\alpha}$ is continuous, order-preserving, and satisfies the bound~\eqref{eq:dominated_point_recursion_Phi}, induction gives
\[
\bm{L}_{1/\alpha,n}^0 \preceq \bm{I}_{p_\ast} - \varepsilon \bm{D}_\ast \prec \bm{I}_{p_\ast},
\]
for every \(n\). Passing to the limit, 
\begin{align*}
    \bm{L}^{0}_{1/\alpha} \prec \bm{I}_{p_\ast} - \varepsilon \bm{D}_\ast  \prec \bm{I}_{p_\ast}.
\end{align*}

Next, set
$\bm{L} \equiv \bm{L}^{0}_{1/\alpha}$ and  $\widetilde{\mathcal{K}} \equiv \widetilde{\mathcal{K}}_{\bm{L}}$ and also introduce  
\begin{align*}
    \bm{R} \coloneqq  \bm{L}(\bm{I} - \bm{L}) \succ \bm{0}_{p_{\ast} \times p_{\ast} },
\end{align*}
and
\begin{align*}
    \bm{Q} \coloneqq  \bm{T}_{\ast}(\bm{I}-\bm{L} \bm{T}_{\ast})^{-1}.
\end{align*}
At $z =1/\alpha$ since $\bm{L}$ is positive, we can rewrite the MDE in~\eqref{eq:MDE_for_L} as 
\begin{align*}
    \bm{R} = - \alpha \bm{L} \mathbb{E}[\bm{Q}] \bm{L},
\end{align*}
such that we have
\begin{align*}
    \bm{R} - \widetilde{\mathcal{K}}[\bm{R}] 
    &=
    - \alpha \bm{L} \mathbb{E}[\bm{Q}] \bm{L} -\alpha \bm{L} \mathbb{E}[\bm{Q}\bm{L}(\bm{I}_{p_\ast} - \bm{L}) \bm{Q}]\bm{L}, \\
    &= 
    \alpha \bm{L} (\mathbb{E}[ -\bm{Q}    - \bm{Q}\bm{L}(\bm{I}_{p_\ast} - \bm{L}) \bm{Q}  ] ) \bm{L}, \\
    &=
    \alpha \bm{L} (\mathbb{E}[ \bm{Q}(\bm{I}_{p_\ast}-\bm{L})\bm{C}_\ast(\bm{I}_{p_\ast}-  \bm{L})\bm{Q}] + \mathbb{E}[\bm{I}_{p_\ast} -\bm{C}_\ast] )\bm{L}, \\
    &=
    \alpha \bm{L} (\mathbb{E}[ \bm{Q}(\bm{I}_{p_\ast}-\bm{L})\bm{C}_\ast(\bm{I}_{p_\ast}-  \bm{L}) \bm{Q} ] )\bm{L},
\end{align*}
where in the last equality we have used $\mathbb{E} [ \bm{C}_\ast(\bm{y}) = \bm{I}_{p_\ast}]$. Thus we get 
\begin{align*}
    \bm{R} - \widetilde{\mathcal{K}}[\bm{R}]  \succeq \bm{0}_{p_{\ast} \times p_{\ast}  }.
\end{align*}
We claim that the latter bound is strict, that is, that $\bm{R} - \widetilde{\mathcal{K}}[\bm{R}]  \succ \bm{0}_{p_{\ast} \times p_{\ast}  }$. For any $\bm{v} \in \mathbb{R}^{p_\ast}$, we have 
\begin{align*}
    \langle \bm{v}, (\bm{R} - \widetilde{\mathcal{K}}[\bm{R}]) \bm{v} \rangle 
    &=
    \alpha \langle \bm{v},   \bm{L} (\mathbb{E}[ \bm{Q}(\bm{I}_{p_\ast}-\bm{L})\bm{C}_\ast(\bm{I}_{p_\ast}-  \bm{Q}) ] )\bm{L} \bm{v} \rangle , \\
     \langle \bm{v}, (\bm{R} - \widetilde{\mathcal{K}}[\bm{R}]) \bm{v} \rangle 
     &=
    \alpha \mathbb{E}[ \| \bm{C}_\ast^{1/2}(\bm{I}_{p_\ast}-\bm{L})\bm{Q} \bm{L} \bm{v} \|^2].
\end{align*}
Hence, the quadratic is null if and only if we have almost surely 
\begin{align}
\label{eq:condition_null_quadraticform}
    \bm{C}_\ast^{1/2}(\bm{I}_{p_\ast}-\bm{L})\bm{Q} \bm{L} \bm{v} = \bm{0}_{p_\ast}.
\end{align}
Since $\bm{Q} = \bm{T}_\ast(\bm{I}_{p_\ast} - \bm{L} \bm{T}_\ast)^{-1} = (\bm{I}_{p_\ast} -  \bm{T}_\ast \bm{L})^{-1} \bm{T}_\ast$ and the matrices $\bm{C}_\ast^{1/2}$, $(\bm{I}_{p_\ast}-\bm{L})$ and $(\bm{I}_{p_\ast} -  \bm{T}_\ast \bm{L})^{-1}$ are all invertible,~\eqref{eq:condition_null_quadraticform} is equivalent to the condition
\begin{align*}
    \bm{T}_\ast \bm{L} \bm{v} = \bm{0}_{p_\ast}
\end{align*}
almost surely. Or, using $\bm{T}_\ast = \bm{C}_\ast^{-1}( \bm{C}_\ast -\bm{I}_{p_\ast})$, to
\begin{align}
\label{eq:condition_null_quadraticform_2}
   ( \bm{C}_\ast -\bm{I}_{p_\ast})\bm{L}\bm{v} 
    &=
      \bm{0}_{p_\ast}
\end{align}
almost surely. Consider now~\eqref{eq:Perron_Eigenmatrix_D} and multiply it on the right by $\bm{L}\bm{v}$, we obtain 
\begin{align*}
    \frac{1}{\alpha_{\mathrm{c}}^\ast}\bm{D}_\ast \bm{L}\bm{v} 
    &=
    \mathbb{E}[ (\bm{C}_\ast - \bm{I}_{p_\ast}) \bm{D}_\ast (\bm{C}_\ast - \bm{I}_{p_\ast}) \bm{L}\bm{v}]=
    \bm{0}_{p_\ast},
\end{align*}
where we used~\eqref{eq:condition_null_quadraticform}. Since both $\bm{D}_\ast$ and $\bm{L}$ are positive definite, this implies that $\bm{v} = \bm{0}_{p_\ast}$. Hence $\bm{R} - \widetilde{\mathcal{K}}[\bm{R}]$ is strictly positive.

Finally, define
\[
q \coloneqq \lambda_1 \left ( \bm{R}^{-1/2} \widetilde{\mathcal{K}}[\bm{R}] \bm{R}^{-1/2}\right ) \in [0,1).
\]
Then \(\widetilde{\mathcal{K}}[\bm{R}] \preceq q \bm{R}\). Since \(\widetilde{\mathcal{K}}\) preserves the positive semidefinite cone, its operator norm in the order-unit norm induced by \(\bm{R}\) is at most \(q\), and therefore
\[
\rho (\widetilde{\mathcal{K}}) \le q < 1.
\]
Lemma~\ref{lem:physical_sol} now gives $\bm{L}(1/\alpha) = \bm{L}_{1/\alpha}^0 \prec \bm{I}_{p_\ast}$ and \(1/\alpha \notin \supp (\mu_\alpha)\). This proves Part (2b) and completes the proof of Lemma~\ref{lem:qualitative_behavior_L}.


\appendix
\section{Heuristic derivation of the optimal spectral method} \label{appendix:AMP_matrix}
In this section, we explain how the optimal preprocessing function appearing in ~\eqref{eq:optimal_preprocessing} can be motivated from a Bayesian information-theoretic perspective, and the so-called Thouless--Anderson--Palmer (TAP) approximation from disordered systems. This discussion is independent of the main results and should be read as additional intuition for why the preprocessing in~\eqref{eq:optimal_preprocessing} is natural. This argument, which appears in different forms in e.g.~\cite{saade2014,maillard2022,defilippis2025}, is written here in a way that can be easily adapted to models beyond the one considered in this paper.\\

We consider the same multi-index model as in the main text except that now we look at a Bayesian variant for the hidden part $\bm{W}_{\ast}$ where its columns $\bm{W}_\ast = [\bm{w}_{\ast,1},\dots,\bm{w}_{\ast,r}]$ are given by $\{ \bm{w}_{\ast,k} \}_{k=1}^r \overset{\mathrm{i.i.d.}}{\sim} \mathsf{N}(\bm{0}_d,\bm{I}_d/d)$. Let $\bm{s}_i = \bm{W}_\ast^\top \bm{x}_i$, we assume that we have the centering condition
\begin{align}
\label{eq:centering_condition}
    \E[ \bm{s}_i | \bm{y}_i ]
    &=
    \bm{0}_r \, ,
\end{align}
such that we are in the non-trivial setting for the inference of $\bm{W}_\ast$. 
For this Bayesian model, the Minimal Mean Squared Error (MMSE) estimator is defined by
\begin{align}
     \bm{\widehat{W}}_{\mathrm{MMSE}}
     \coloneqq 
     \mathrm{argmin}_{\bm{\widehat{W}} = \bm{\widehat{W}}(\bm{X},\bm{y})}
     \E
     \left[
     \left\|
     \bm{W}_{\ast}
     -
     \bm{\widehat{W}}
     \right\|_F^2
     \right],
\end{align}
and is given by the \emph{posterior mean}
\begin{align}
     \bm{\widehat{W}}_{\mathrm{MMSE}}
     =
     \E
     \left[
     \bm{W}_{\ast}
     \mid
     \bm{y},\bm{X}
     \right] 
     =
     \int \mathrm{d}\mathrm{P}(\bm{W} | \bm{y},\bm{X}) \bm{W}\, , 
\end{align}
where the \emph{posterior distribution} is given by Bayes-rule:  
\begin{align}
    \mathrm{d}\mathrm{P}(\bm{W} | \bm{y},\bm{X}) = \frac{1}{Z_{n,d}(\bm{X},\bm{y})} \mathrm{e}^{- \frac{d}{2} \|\bm{W} \|^2_F} \prod_{i=1}^n P(\bm{y}_i | \bm{W}^\top \bm{x}_i) \mathrm{d}\bm{W},
\end{align}
with the partition function
\begin{align}
    Z_{n,d}(\bm{X},\bm{y})\coloneqq  \int \mathrm{e}^{- \frac{d}{2} \|\bm{W} \|^2_F} \prod_{i=1}^n P(\bm{y}_i | \bm{W}^\top \bm{x}_i) \mathrm{d}\bm{W}\, . 
\end{align}
While this expression is semi-explicit, it is not computationally useful in the high-dimensional regime ($n,d\gg 1$) as it involves integration over a $rd$-dimensional posterior distribution, and for example cannot be accurately estimated by naive Monte Carlo sampling. 
\\
\\
Instead, the approach described below consists of a sequence of approximations which lead to two tractable spectral estimators that can be obtained as follows:
\begin{enumerate}
\item[(i)] Reinterpret the MMSE estimator as the maximizer of a suitable variational functional, namely the Gibbs free energy.
\item[(ii)] Approximate the Gibbs free energy by the Thouless--Anderson--Palmer (TAP) free energy.
\item[(iii)] Derive the approximate message passing (AMP) dynamics associated with the TAP free energy.
\item[(iv-a)] Expand the TAP free energy around the non-informative point. At leading order, the relevant estimator is given by the top eigenspace of the TAP Hessian, which yields our first (presumably) optimal spectral method.
\item[(iv-b)] or/and linearize the AMP dynamics obtained in step (iii) around the non-informative fixed point. The resulting linear iteration yields another spectral method, related to the one obtained in (iv-a). 
\end{enumerate}
Note that the first spectral estimator (iv-a) does not need the introduction of the AMP iteration. We now present each step in more detail.

\subsection{Variational formulation of the MMSE estimator and the Gibbs free energy}

We introduce an external field $\bm{H} \in \mathbb{R}^{d \times r}$ and define the \emph{tilted probability measure}
\begin{align}
\label{eq:tilted_posterior}
    \mathrm{d}\mathrm{P}_{\bm{H}}(\bm{W} | \bm{y},\bm{X}) 
    \coloneqq 
    \frac{1}{Z_{n,d}(\bm{H};\bm{X},\bm{y})} \mathrm{e}^{- \frac{d}{2} \|\bm{W} \|^2_F + \langle \bm{H}, \bm{W} \rangle} \prod_{i=1}^n P(\bm{y}_i | \bm{W}^\top \bm{x}_i) \mathrm{d}\bm{W},
\end{align}
with partition function
\begin{align}
\label{eq:tilted_partition_function}
    Z(\bm{H}) \equiv Z_{n,d}(\bm{H};\bm{X},\bm{y})
    \coloneqq 
    \int
    \mathrm{e}^{- \frac{d}{2} \|\bm{W} \|^2_F + \langle \bm{H}, \bm{W} \rangle} \prod_{i=1}^n P(\bm{y}_i | \bm{W}^\top \bm{x}_i) \mathrm{d}\bm{W}.
\end{align}
Let
\begin{align}
\label{eq:tilted_free_energy}
    F(\bm{H})
    \coloneqq 
    \log Z(\bm{H})
\end{align}
be the \emph{tilted free energy}. Then one may easily obtain the relation
\begin{align}
    \left(\nabla_{\bm{H}}F\right)(\bm{H})
    =
    \E_{\bm{H}}
    \left[
    \bm{W}
    \right],
\end{align}
where $\E_{\bm{H}}$ denotes the expectation with respect to the posterior of~\eqref{eq:tilted_posterior}. In particular,
\begin{align}
    \left(\nabla_{\bm{H}} F\right)(\bm{0}_{d \times r})
    =
    \bm{\widehat{W}}_{\mathrm{MMSE}}.
\end{align}
Next, the idea is to work with an imposed prior $\bm{V} = \E_{\bm{H}}[\bm{W}]$ by introducing the Legendre transform of $F$:
\begin{align}
    F_{\mathrm{Gibbs}}(\bm{V})
    \coloneqq 
    \sup_{\bm{H}}
    \left\{
    \left\langle
    \bm{H},
    \bm{V}
    \right\rangle
    -
    F(\bm{H})
    \right\}.
\end{align}
which is the so-called \emph{(convex) Gibbs free energy at fixed magnetization}. By Legendre duality we have 
\begin{align}
   \big( \nabla_{\bm{V}}  F_{\mathrm{Gibbs}}\big)(\bm{V})
    =
    \bm{H},
\end{align}
where $\bm{H}$ is the external field of the tilted measure \eqref{eq:tilted_posterior}. In particular, at zero external field, we have the relation
\begin{align}
\label{eq:exact_stationary_problem}
    \nabla_{\bm{V}} F_{\mathrm{Gibbs}}(\bm{V})
    =
    \bm{0}_{d \times r}
    \qquad
    \Longleftrightarrow
    \qquad
    \bm{V}
    =
    \bm{\widehat{W}}_{\mathrm{MMSE}}.
\end{align}
Hence \emph{the MMSE estimator is obtained as a stationary point of the Gibbs free energy} and since $F_{\mathrm{Gibbs}}$ is convex, we also have
\begin{align}
     \bm{\widehat{W}}_{\mathrm{MMSE}} = \mathrm{argmin}_{\bm{V}} \,  F_{\mathrm{Gibbs}}(\bm{V})
\end{align}
The difficulty is that $F_{\mathrm{Gibbs}}$ is intractable in practice and the TAP construction consists in replacing $F_{\mathrm{Gibbs}}$ by an explicit and tractable approximation that turns out to be exact in the high-dimensional regime. 

\subsection{Thouless-Anderson-Palmer (TAP) approximation of the free energy}

We now briefly explain the idea behind  the TAP approximation and refer the reader to
\cite{thouless1977,fischer1993,bolthausen2014,zdeborova2016} for more details. The TAP approximation can be understood as a \emph{corrected mean-field approximation of the Gibbs free energy} of the previous section. At the crudest level (“\emph{naive mean-field approximation}”) ignoring the structure of the output channel only leaves the Gaussian-prior cost for the free energy : $ F_{\mathrm{NMF}}(\bm{V}) \coloneqq   \frac{d}{2} \| \bm{V} \|_F^2$ and the role of the TAP approximation is to add to this term the effect of the observations through \emph{local output-channel contributions} and then subtract the 
\emph{self-interaction} created by this procedure. 
\\
\\
In other words, while the exact Gibbs free energy asks for the cost of imposing the  \emph{global constraint}  $\E_{\bm{H}}[\bm{W}] = \bm{V}$, the TAP approximation, instead, relaxes this constraint by the sum of costs enforcing fixed values for each induced \emph{local} means $\E_{\bm{H}}[\bm{s}_i] = \bm{V}^\top \bm{x}_i$; since each observation $\bm{y}_i$ only depends on $\bm{W}$ through $\bm{s}_i = \bm{W}^\top \bm{x}_i$. Given an observation $\bm{y}$, consider the tilted measure
\begin{align}
\label{eq:local_tilting}
    \mathrm{d}Q_{\bm{y},\bm{b}}(\bm{s}) \coloneqq  \frac{1}{Z_{\bm{y}}(\bm{b})} P(\bm{y}| \bm{s}) \mathrm{e}^{\langle \bm{b}, \bm{s}\rangle} \mathrm{d}\gamma(\bm{s}),\quad  Z_{\bm{y}}(\bm{b}) 
    \coloneqq 
    \int
    P(\bm{y}| \bm{s}) \mathrm{e}^{\langle \bm{b}, \bm{s}\rangle} \mathrm{d}\gamma(\bm{s})\, ,
\end{align}
where $\gamma(.)$ is the measure of a standard Gaussian random variable $\mathsf{N}(\bm{0}_r,\bm{I}_p)$ and we denote the local free energy and its Legendre transform by
\begin{align}
    \phi_{\bm{y}}(\bm{b}) \coloneqq  \log Z_{\bm{y}}(\bm{b}) 
     \qquad \mbox{and} \qquad
      \phi_{\bm{y}}^\ast(\bm{m}) \coloneqq   \sup_{\bm{b}} \left\{ \langle \bm{b} , \bm{m} \rangle - \phi(\bm{b}) \right\}.
\end{align}
As before, we have $\nabla_{\bm{b}}  \phi_{\bm{y}}(\bm{b}) = \E_{\bm{b}}[\bm{s}]$ where $\E_{\bm{b}}[\cdot]$ denotes expectation with respect to the law given in~\eqref{eq:local_tilting} and the Legendre transform $  \phi_{\bm{y}}^\ast(\bm{m})$ forces the posterior mean of $\bm{s}$ to be $\bm{m}$. 
\\
\\
To be consistent with the case where the output channel carries no information on $\bm{W}$, the TAP free energy should reduce to the naive mean-field cost $\frac{d}{2}
\|
\bm{V}
\|_F^2$. However, a direct Gaussian computation followed by the Legendre transform of a quadratic function gives:
\begin{align}
    \phi^\ast_{\bm{y}}(\bm{m}) = \frac{1}{2} \| \bm{m} \|_2^2 +C  , \qquad \mbox{if } P(\bm{y}| \bm{s}) = P(\bm{y}),
\end{align}
where $C$ is a constant independent of $\bm{m}$, such that the local Legendre term alone would create a spurious quadratic contribution, even though the observation contains no information about the signal. The TAP approximation therefore subtracts this quadratic contribution known as the \emph{Onsager correction term}. 
\\
\\
All in all, this leads to the following form of the TAP free energy:
\begin{align}
\label{eq:TAP_Free_Energy}
    F_{\mathrm{TAP}}(\bm{V}) \coloneqq   \frac{d}{2 n } \| \bm{V} \|_F^2 + \frac{1}{n}\sum_{i=1}^n \left(\phi_{\bm{y}_i}^\ast(\bm{V}^\top\bm{x}_i) - \frac{1}{2} \| \bm{V}^\top \bm{x}_i \|_2^2 \right),
 \end{align}
where the scaling factor $\frac{1}{n}$ has been introduced to make this cost of order $O(1)$ as $n,d\to \infty$. In other words, the TAP approximation of the MMSE estimator is obtained by replacing the exact stationary problem in~\eqref{eq:exact_stationary_problem} by the simpler one 
\begin{align}
\label{eq:TAP_stationary}
    \left(
    \nabla_{\bm{W}}
    F_{\mathrm{TAP}}
    \right)(\bm{W})
    =
    \bm{0}_{d \times r}.
\end{align}
While this approximation may seem strong at first sight, this TAP variational characterization has been shown to properly describe the limiting Gibbs free energy for a wide variety of related models \cite{subag2021,subag2023} and in particular see \cite{troianifundamental} for the Multi-index model studied in this paper.

\subsection{Approximate Message Passing (AMP) iterations}
Now that the problem has been reduced to the TAP variational formulation, the next step is to turn the stationary equation $\eqref{eq:TAP_stationary}$ into the construction of an effective algorithm designed to reach such stationary points. This task is done by (Bayes-Optimal\footnote{Although we present AMP here as an algorithm associated with the TAP free energy, AMP is more general and can also be constructed to minimize other objective functions.}) Approximate Message Passing (AMP). We do not attempt to review AMP algorithms in full generality here, and refer the reader to 
\cite{bolthausen2014,javanmard2013,feng2022,gerbelot2023} and explain here only the idea behind the construction of AMP at a high-level. 
\\
\\
AMP is a two-step iterative procedure : where first, given the current estimate of the weights $\bm{\widehat{W}}^{t}$ one estimates the corresponding local fields $\bm{\widehat{S}}^{t}$ and applies an output-channel denoising step. Second, these output messages $\bm{\widehat{S}}^{t}$ are propagated back to update the estimate of the new  weights $\bm{\widehat{W}}^{t+1}$. Both steps contain Onsager corrections, which as the correction term in the TAP Free energy,  remove the self-interaction generated by reusing the same data at each iteration to propagate through the iteration. In the case of the Multi-index model such iterations write
\begin{align}
\label{eq:AMP_iteration}
\begin{cases}
    \bm{\widehat{S}}^{t}
    &= 
    \bm{X}\bm{\widehat{W}}^t - \mathrm{g}^{t-1}_{\bm{Y}}(\bm{\widehat{S}}^{t-1}) ({\bm{V}^t})^\top, \\ 
    \bm{\widehat{W}}^{t+1}&= 
    \mathrm{f}^{t+1}(\bm{X}^\top \mathrm{g}^{t}_{\bm{Y}}(\bm{\widehat{S}}^{t}) - \bm{\widehat{W}}^{t}{\bm{A}^t}^\top),
\end{cases}
\end{align}
where  we refer the reader to \cite{troianifundamental,defilippis2025} for the explicit expressions of the \emph{denoising functions} $\mathrm{g}^{t}_{\bm{Y}}(\cdot) \equiv \mathrm{g}^{t}(\cdot , \{\bm{y}\}_{i=1}^n )$ and $\mathrm{f}^{t+1}(\cdot)$ as well as the \emph{Onsager Matrices} $\bm{V}_t$ and $\bm{A}_t$. 
\\
\\
The asymptotic performance of AMP (or more precisely of the corresponding overlap matrix between $\bm{\widehat{W}}^t$ and $\bm{W}_\ast$) at any fixed number of iterations is given by the so-called \emph{state evolution equations} which we do not describe here and refer again to \cite{troianifundamental} for its precise form.  An important point to note is that if  the centering condition \eqref{eq:centering_condition} holds, the non-informative point $\bm{0}_{d \times r}$ is always a fixed point of the AMP iterations. \emph{Escaping mediocrity}, that is the convergence in finite time to an informative fixed point  with positive overlap (provided that such fixed point exists) generally requires a \emph{warm start}, an initial estimator $\bm{\widehat{W}}^0$ having an (asymptotic) non-zero overlap with the planted signal $\bm{W}_\ast$.  Spectral methods are a natural candidate for this initial guess and can themselves be derived from linearizing the AMP iterations as described in the next section.

\subsection{Spectral methods from the TAP approach}

To construct a Bayes-inspired Spectral Method, the goal is to find a direction, if it exists, escaping from the null magnetization $\bm{W}=\bm{0}_{d \times r}$. There are two closely related ways to obtain such a spectral method from the objects defined in the previous sections  either (a) expand the TAP free energy itself, which amounts to relaxing the nonlinear variational problem near the informative point $\bm{V} = \bm{0}_{d \times r}$ or (b) linearize the AMP iterations, the algorithm that is designed to converge to the fixed points of this free energy, around their uninformative fixed point. 

\noindent \textbf{(a) TAP-Hessian :} To express this as an ordinary matrix eigenvalue problem, let
\begin{align}
    \bm{w}
    \coloneqq 
    \mathrm{vec}(\bm{W}),
    \qquad
    \mathcal{F}_{\mathrm{TAP}}(\bm{w})
    \coloneqq 
    F_{\mathrm{TAP}}
    \left(
    \mathrm{mat}(\bm{w})
    \right).
\end{align}
As we want to find the direction escaping from mediocrity, we parametrize $\bm{w}=\epsilon \bm{u}$, where $\epsilon$ is a small parameter and the direction $\bm{u}$ we are interested in is of fixed norm $\left\|\bm{u}\right\|=1$. Under mild regularity conditions on $  \mathcal{F}_{\mathrm{TAP}}$, we have by Taylor's theorem:
\begin{align}
    \mathcal{F}_{\mathrm{TAP}}(\epsilon \bm{u})
    =
    \mathcal{F}_{\mathrm{TAP}}(\bm{0}_{rd})
    +
    \epsilon
    \left\langle
    \left(
    \nabla
    \mathcal{F}_{\mathrm{TAP}}
    \right)(\bm{0}_{rd}),
    \bm{u}
    \right\rangle
    +
    \frac{\epsilon^2}{2}
    \left\langle
    \bm{u},
    \left(
    \nabla^2
    \mathcal{F}_{\mathrm{TAP}}
    \right)(\bm{0}_{rd})
    \bm{u}
    \right\rangle
    +
    o(\epsilon^2).
\end{align}
Since $\bm{0}_{rd}$ is a stationary point, the first-order term vanishes. Hence, at leading order,
\begin{align}
   \lim_{\epsilon \to 0^+}  \mathrm{argmin}_{\left\|\bm{u}\right\|=1}
    \mathcal{F}_{\mathrm{TAP}}(\epsilon \bm{u})
    =
    \mathrm{argmin}_{\left\|\bm{u}\right\|=1}
    \left\langle
    \bm{u},
    \left(
    \nabla^2
    \mathcal{F}_{\mathrm{TAP}}
    \right)(\bm{0}_{rd})
    \bm{u}
    \right\rangle.
\end{align}
By the Courant-Fischer variational principle, if we denote by $\bm{u}_\star$ the argmin of the previous equation, any leading maximizer converges to a top eigenvector of the (opposite of the) Hessian
\begin{align}
\bm{u}_\star = \bm{v}_1 \left\{ 
    \left(
    -\nabla^2
    \mathcal{F}_{\mathrm{TAP}}
    \right)(\bm{0}_{rd}) \right\}.
\end{align}
In the present model with the TAP energy given by~\eqref{eq:TAP_Free_Energy}, this Hessian is
\begin{align}
    \left(-
    \nabla^2
    \mathcal{F}_{\mathrm{TAP}}
    \right)(\bm{0}_{rd})
    &= \frac{1}{n}
    \sum_{i=1}^{n}
    \left(
    \bm{I}_{r}
    - 
    \E
    \left[
    \bm{s}
    \bm{s}^\top
    \mid
    \bm{y}_i
    \right]^{-1}
    \right)
    \otimes
    \bm{x}_i
    \bm{x}_i^\top
    -
    \frac{d}{n}\bm{I}_{rd} 
\end{align}
which is, up to a shift that does change the eigenvector, precisely the matrix $\bm{D}_n^{(\bm{T}_{\ast})}$ in the main text with the optimal preprocessing matrix-valued function $\bm{T}_{\ast}(\bm{y}) = \bm{I}_{r} - \bm{C}(\bm{y})^{-1}$ with $ \bm{C}(\bm{y}) \coloneqq  \E[\bm{s}\bm{s}^\top |\bm{y}]$. 

\noindent \textbf{(b) Linearizing AMP :} Similarly, consider the scaling $\bm{\widehat{W}} = \epsilon \bm{U}$ and $\bm{\widehat{S}} = \epsilon \bm{H}$ with $\bm{U}$ of fixed norm. Since $\bm{\widehat{W}} = \bm{0}_{d \times r}$ is a fixed point, note that we have $ \mathrm{g}^{t}_{\bm{Y}}(\bm{0}) =   \mathrm{f}^{t}(\bm{0}) = \bm{0}$. Under mild regularity, the AMP iteration in~\eqref{eq:AMP_iteration} simplifies to first order as:
\begin{align}
    \epsilon \bm{H}^t 
    &=
    \epsilon \bm{X}\bm{U}^t - \big( \epsilon (D\mathrm{g}_{\bm{Y}}^{t}(\bm{0}))[\bm{H}^{t-1}] + O(\epsilon^2) \big) \big( {\bm{V}^t(\bm{0})}^\top  + O(\epsilon) \big), \\
     \epsilon \bm{U}^{t+1} &= \epsilon D f^{t+1}(\bm{0})[\bm{X}^\top (D\mathrm{g}_{\bm{Y}}^{t}(\bm{0}))[\bm{H}^{t}]  - \ \bm{U}^t (\bm{A}^{t}(\bm{0}) )   ]  + O(\epsilon^2). 
 \end{align}
In the Bayes-optimal case with the centering condition~\eqref{eq:centering_condition} and Gaussian priors, one has from \cite{troianifundamental,defilippis2025}: 
\begin{align}
    \bm{A}^{t}(\bm{0}) = \bm{0}_{d \times r}, \bm{V}^{t}(\bm{0}) = \bm{I}_p, \quad D f^{t+1}(\bm{0}) = \frac{1}{d}\mathcal{I}{\mathrm{d}},  \left\{ (D\mathrm{g}_{\bm{Y}}^{t}(\bm{0}))[\bm{H}] \right\}_{i} = \bm{H}_i (\bm{C}(\bm{y}_i) - \bm{I}_p)\, , 
\end{align}
Assuming that this linearized system admits a nonzero equilibrium, the equilibrium equations reduce to the system:
\begin{align}
    \begin{cases}
        \bm{H} &= \bm{X}\bm{U} - \mathcal{G}[\bm{H}], \\ 
        \bm{U}  &= \frac{1}{d}\bm{X}^\top \mathcal{G}[\bm{H}], 
    \end{cases}
\end{align}
with the linear operator $\{ \mathcal{G}[\bm{H}]\}_i =  \bm{H}_i(\bm{C}(\bm{y}_i ) - \bm{I}_p)$. We can then substitute to solve for $ \bm{U}$ which writes in vectorized form $\bm{u} \coloneqq  \mathrm{vec}(\bm{U})$:
\begin{align}
    \bm{u} =  \left( \frac{1}{d}\sum_{i=1}^n (\bm{I}_p - \bm{C}(\bm{y})^{-1}) \otimes \bm{x}_i \bm{x}_i^\top \right) \bm{u}. 
\end{align}
This suggests taking the top eigenvector of $\frac{1}{d}\sum_{i=1}^n (\bm{I}_p - \bm{C}(\bm{y})^{-1}) \otimes \bm{x}_i \bm{x}_i^\top$ as the spectral estimator, which gives (up to a scaling factor $ \frac{n}{d} =\alpha$ that does not change the top eigenvector), precisely the desired pre-processed optimal function. 

\section{Exact variational characterization of the limiting spectral edge}
\label{appendix:limiting_variational_edge}

The finite-dimensional variational bound of Lemma~\ref{lem:finite_variational_formula} was used in the main text only to prove that the conditional free model has no spectrum asymptotically above \(\lambda_+(\alpha)\). In this appendix, we show that the corresponding
limiting variational problem in fact characterizes the upper spectral edge exactly.

The main additional ingredient is the convergence of the spectral distribution of the conditional free model to the deterministic limiting law. We first identify its matrix-valued Stieltjes transform as the solution of an empirical matrix Dyson equation. We then show that this empirical equation converges to the limiting matrix Dyson equation defining \(\bm{M}_\alpha\). Finally, we combine the resulting weak convergence with the finite-\(n\) variational bound.\\

\begin{prop}\label{prop:limiting_variational_edge}
The limiting upper spectral edge admits the variational characterization
\[
\lambda_+(\alpha) = \inf_{\bm{Z} \in\mathcal{D}_\alpha} \lambda_1 \left (\mathcal{G}_\alpha (\bm{Z}) \right),
\]
where
\[
\mathcal{D}_\alpha \coloneqq \left\{ \bm{Z} \in \sym_p^{++}(\R) \colon \operatorname*{ess \,inf}_{\bm{y}} \lambda_{\min} \left(\bm{Z}^{-1} - \alpha^{-1} \bm{T} (\bm{y}) \right)>0 \right \},
\]
and 
\[
\mathcal{G}_\alpha(\bm{Z}) =\bm{Z}^{-1} + \E_{\bm{y}} \left[ \bm{T} (\bm{y}) \left( \bm{I}_p-\alpha^{-1} \bm{Z} \bm{T} (\bm{y}) \right)^{-1} \right].
\]
\end{prop}

A related variational characterization of the spectral edges of block-Wishart matrices of the type considered here was obtained recently by Montanari and Saeed~\cite{montanarisaeed2026}. Their formulation is expressed in terms of the matrix \(K\)-transform and includes an additional positivity condition on its derivative. After reflecting the spectrum and performing a change of variables, the objective function in their variational problem coincides with \(\mathcal{G}_\alpha\). Their approach, however, is different from the one used here.

\begin{rmk}
An analogous variational characterization holds for the lower spectral edge. Indeed, replacing the preprocessing map \(\bm{T}\) by \(-\bm{T}\) replaces the random matrix \(\bm{P}_n\) by \(-\bm{P}_n\) and reflects the limiting spectral distribution about the origin. Applying Proposition~\ref{prop:limiting_variational_edge} to \(-\bm{T}\) therefore yields a variational characterization of \(\lambda_-(\alpha)=\inf\supp(\mu_\alpha)\). 
\end{rmk}

Recall the free operator \(\bm{P}_n^{\mathrm{free}}\) defined in~\eqref{eq:signed_free_model}. Conditionally on \(\sigma (\bm{y}_1,\ldots, \bm{y}_n)\), it is a fixed element of \(M_{p(d-r)} (\mathcal{A})\), where \((\mathcal{A},\tau)\) is a tracial \(C^\ast\)-probability space \((\mathcal{A},\tau)\) with \(\tau\) faithful. Recall also the \(M_p(\C)\)-valued expectation \(\E_n =  \operatorname{id}_{M_p(\C)} \otimes \operatorname{tr}_{d-r} \otimes\tau\). For \(z \in \C_+\), define
\[
\bm{M}_n^{\mathrm{free}}(z) \coloneqq \E_n  \left[ \left(\bm{P}_n^{\mathrm{free}}
-z\bm{I}_{p(d-r)} \otimes 1_{\mathcal{A}} \right)^{-1} \right] .
\]

\begin{lem} \label{lem:exact_empirical_MDE_free}
For almost every realization of \(\sigma(\bm{y}_1,\ldots,\bm{y}_n)\), \(\bm{M}_n^{\mathrm{free}}\) is the unique solution in the normalized matrix-valued Herglotz class of 
\begin{equation} \label{eq:empirical_MDE}
\left ( \bm{M}_n^{\mathrm{free}} (z)  \right )^{-1}   = -z\bm{I}_p + \frac{1}{n} \sum_{i=1}^n \bm{T}(\bm{y}_i)\left(\bm{I}_p+\gamma_n \bm{M}_n^{\mathrm{free}}(z)\bm{T} (\bm{y}_i) \right)^{-1}.
\end{equation}
Moreover, \(m_n^{\mathrm{free}}(z) \coloneqq \frac{1}{p} \Tr \bm{M}_n^{\mathrm{free}}(z)\) is the Stieltjes transform of the spectral distribution \(\mu_n^{\mathrm{free}}\) of \(\bm{P}_n^{\mathrm{free}}\) with respect to the normalized trace \(\varphi_n = \operatorname{tr}_{p(d-r)} \otimes \tau\). 
\end{lem}

\begin{proof}
Fix \(n\) and a realization of \((\bm{y}_1,\ldots,\bm{y}_n)\) in the full-probability event on which Assumption~\ref{hyp:preprocessing} on \(\bm{T}(\bm{y}_i)\) hold. Throughout the proof, we regard the matrices \(\bm{T}(\bm{y}_1), \ldots, \bm{T}(\bm{y}_n)\) as deterministic. 

In the proof of Lemma~\ref{lem:finite_variational_formula}, we computed the \(M_p(\C)\)-valued free cumulants of \(\bm{P}_n^{\mathrm{free}}\). They satisfy
\[
\kappa_1^{(n)} = \frac{1}{n} \sum_{i=1}^n \bm{T} (\bm{y}_i), 
\]
and for every \(m \ge 2\),
\[
\kappa_m^{(n)} (\bm{B}_1, \ldots, \bm{B}_{m-1})= \frac{\gamma_n^{m-1}}{n} \sum_{i=1}^n \bm{T} (\bm{y}_i) \bm{B}_1 \bm{T} (\bm{y}_i) \cdots \bm{B}_{m-1} \bm{T} (\bm{y}_i).
\]
Recall from~\cite[Chapter 9, Definition 10]{mingo2017} that the operator-valued \(R\)-transform is
\begin{equation*}
\mathcal{R}_n (\bm{B}) = \kappa_1^{(n)} + \sum_{m \ge 2} \kappa_m^{(n)} (\bm{B}, \ldots, \bm{B}) = \frac{1}{n} \sum_{i=1}^n  \sum_{k=0}^\infty  \gamma_n^k \bm{T}(\bm{y}_i) \left (\bm{B} \bm{T}(\bm{y}_i) \right)^k.
\end{equation*}
Define the associated \(M_p(\C)\)-valued Cauchy transform by
\[
\bm{G}_n^{\mathrm{free}} (z) \coloneqq \E_n \left [ \left (z \bm{I}_{p(d-r)} \otimes 1_{\mathcal{A}} - \bm{P}_n^{\mathrm{free}} \right)^{-1} \right ] .
\]
By the operator-valued \(R\)-transform relation~\cite[Chapter~9, Theorem~11]{mingo2017}, for \(\Im z\) sufficiently large, 
\[
\left (\bm{G}_n^{\mathrm{free}} (z) \right )^{-1} = z \bm{I}_p - \mathcal{R}_n \left ( \bm{G}_n^{\mathrm{free}} (z) \right ).
\]
Moreover, since \(\|\bm{G}_n^{\mathrm{free}} (z)\|_{\mathrm{op}} \le (\Im z)^{-1}\) and \(\|\bm{T}(\bm{y}_i)\|_{\mathrm{op}} \le C_T\) by Assumption~\ref{hyp:preprocessing}, we have
\[
\left \|  \gamma_n\bm{G}_n^{\mathrm{free}} (z) \bm{T}(\bm{y}_i) \right \|_{\mathrm{op}} \le \frac{\gamma_n C_T}{\Im z}.
\]
Thus, for \(\Im z\) such that \(\Im z > \gamma_n C_T\), the series defining \(\mathcal{R}_n \left ( \bm{G}_n^{\mathrm{free}} (z) \right)\) can be summed using the Neumann-series identity, yielding
\[
\begin{split}
\left (\bm{G}_n^{\mathrm{free}} (z) \right )^{-1} &= z \bm{I}_p - \frac{1}{n} \sum_{i=1}^n  \bm{T}(\bm{y}_i) \sum_{k=0}^\infty  \left ( \gamma_n\bm{G}_n^{\mathrm{free}} (z) \bm{T}(\bm{y}_i) \right)^k\\
& = z \bm{I}_p - \frac{1}{n} \sum_{i=1}^n  \bm{T}(\bm{y}_i) \left ( \bm{I}_p - \gamma_n\bm{G}_n^{\mathrm{free}} (z) \bm{T}(\bm{y}_i) \right)^{-1}.
\end{split}
\]
Since \(\bm{M}_n^{\mathrm{free}} (z) = - \bm{G}_n^{\mathrm{free}} (z) \), it follows that \(\bm{M}_n^{\mathrm{free}} (z)\) satisfies~\eqref{eq:empirical_MDE} for \(\Im z\) sufficiently large.

We next extend this identity to every \(z \in \C_+\). Since \(\bm{P}_n^{\mathrm{free}}\) is self-adjoint, \(\Im \bm{M}_n^{\mathrm{free}} (z) \succ \bm{0}_p\) for every \(z \in \C_+\). In particular \(\bm{M}_n^{\mathrm{free}}(z)\) is invertible. Moreover, for every \(i\),
\[
\bm{I}_p + \gamma_n \bm{M}_n^{\mathrm{free}} (z) \bm{T}(\bm{y}_i) = \bm{M}_n^{\mathrm{free}} (z) \left ( \left (\bm{M}_n^{\mathrm{free}} (z)\right)^{-1} + \gamma_n \bm{T}(\bm{y}_i) \right ). 
\]
The second factor is invertible since
\[
\Im \left ( \left (\bm{M}_n^{\mathrm{free}} (z)\right)^{-1} + \gamma_n \bm{T}(\bm{y}_i)\right ) = - \left (\bm{M}_n^{\mathrm{free}} (z)\right)^{-\ast} \left (\Im \bm{M}_n^{\mathrm{free}} (z) \right)\left (\bm{M}_n^{\mathrm{free}} (z)\right)^{-1} \prec \bm{0}_p.
\]
Thus all the inverses appearing in~\eqref{eq:empirical_MDE} are well defined throughout \(\C_+\). Hence both sides of~\eqref{eq:empirical_MDE} are analytic on \(\C_+\). Since they coincide on the nonempty open set \(\{z \in \C_+ \colon \Im z > \gamma_n C_T\}\), the identity theorem implies that~\eqref{eq:empirical_MDE} holds throughout \(\C_+\).

We now verify the Herglotz properties. The resolvent representation already shows that \(\bm{M}_n^{\mathrm{free}}\) is analytic on \(\C_+\) and that \(\Im \bm{M}_n^{\mathrm{free}} (z) \succ \bm{0}_p\). Moreover,
\[
i \eta \left( \bm{P}_n^{\mathrm{free}}-i\eta(\bm{I}_{p(d-r)} \otimes 1_{\mathcal{A}} ) \right)^{-1} + \bm{I}_{p(d-r)} \otimes 1_{\mathcal{A}}  = \bm{P}_n^{\mathrm{free}} \left( \bm{P}_n^{\mathrm{free}}-i\eta (\bm{I}_{p(d-r)} \otimes 1_{\mathcal{A}} ) \right)^{-1},
\]
and the right-hand side converges to zero in norm as \(\eta\to\infty\). Hence
\[
\lim_{\eta\to\infty} i\eta\bm{M}_n^{\mathrm{free}}(i\eta) = -\bm{I}_p.
\]
Thus \(\bm{M}_n^{\mathrm{free}}\) belongs to the normalized matrix-valued Herglotz class. For the fixed realization under consideration, \(\nu_n \coloneqq \frac{1}{n}\sum_{i=1}^n\delta_{\bm{y}_i}\) is a deterministic probability measure, and~\eqref{eq:empirical_MDE} is precisely the matrix Dyson equation associated with \((\nu_n, \gamma_n^{-1})\). Uniqueness in the normalized matrix-valued Herglotz class therefore follows from Proposition~\ref{prop:existence_uniqueness}.

Finally, recall that \(\varphi_n=\tr_{p (d-r)}\otimes\tau = \tr_p \circ \E_n \) defines the spectral distribution \(\mu_n^{\mathrm{free}}\) of \(\bm{P}_n^{\mathrm{free}}\) through 
\[
\varphi_n \left[ f(\bm{P}_n^{\mathrm{free}}) \right] = \int_\R f(x)\,\mu_n^{\mathrm{free}}(\mathrm{d}x).
\]
Taking \(f(x)=(x-z)^{-1}\) gives
\[
\begin{split}
\int_\R \frac{1}{x-z}\, \mu_n^{\mathrm{free}}(\mathrm{d}x) = \varphi_n \left[ \left( \bm{P}_n^{\mathrm{free}}-z\bm{I}_{p(d-r)} \otimes 1_{\mathcal{A}} \right)^{-1} \right] = \frac{1}{p}
 \Tr  \bm{M}_n^{\mathrm{free}}(z) = m_n^{\mathrm{free}}(z).
\end{split}
\]
Thus \(m_n^{\mathrm{free}}\) is the Stieltjes transform of the spectral distribution \(\mu_n^{\mathrm{free}}\).
\end{proof}

We next show that the empirical matrix Dyson equation converges to the limiting matrix Dyson equation defining \(\bm{M}_\alpha\). 
\begin{lem} \label{lem:free_MDE_convergence}
Almost surely,
\[
\bm{M}_n^{\mathrm{free}}(z) \to \bm{M}_\alpha(z)
\]
locally uniformly in \(z\in\C_+\). Thus 
\[
\mu_n^{\mathrm{free}} \Rightarrow \mu_\alpha
\]
weakly almost surely.
\end{lem}

\begin{proof}
For \(\gamma>0\), \(\bm{X} \in\C^{p\times p}\), and \(\bm{y} \in \R^q\), define
\[
\bm{Q}_\gamma(\bm{X},\bm{y}) = \bm{T}(\bm{y}) \left( \bm{I}_p+\gamma\bm{X}\bm{T}(\bm{y}) \right)^{-1},
\]
whenever the inverse exists. The limiting and empirical matrix Dyson equations~\eqref{eq:MDE} and~\eqref{eq:empirical_MDE} can then be written as the fixed-point equations
\[
\bm{M}_\alpha(z) = \bm{\Psi}_{\alpha,z} (\bm{M}_\alpha(z)), \qquad 
\bm{M}_n^{\mathrm{free}}(z) = \bm{\Psi}_{n,z} (\bm{M}_n^{\mathrm{free}}(z)),
\]
where
\[
\bm{\Psi}_{\alpha,z}(\bm{X}) \coloneqq  \left ( -z\bm{I}_p + \E_{\bm{y}} \left[ \bm{Q}_{\alpha^{-1}}(\bm{X},\bm{y}) \right] \right)^{-1}
\]
and
\[
\bm{\Psi}_{n,z}(\bm{X}) \coloneqq \left( -z\bm{I}_p + \frac{1}{n} \sum_{i=1}^n \bm{Q}_{\gamma_n}(\bm{X},\bm{y}_i) \right)^{-1}.
\]

We first prove convergence on an upper half-plane on which the two fixed-point maps are uniform contractions. Since \(\gamma_n\to\alpha^{-1}\), there exists \(\Gamma<\infty\) such that \(\gamma_n \leq\Gamma\) for every \(n \ge 1\) and \(\alpha^{-1} \le \Gamma\). For \(z \in \C_+\), define
\[
\mathcal{B}_z \coloneqq \left\{ \bm{X}\in \C^{p\times p} \colon \|\bm{X}\|_{\mathrm{op}} \leq \frac{2}{\Im z} \right\}.
\]
Choose \(\eta_0>0\) sufficiently large that
\[
\frac{2\Gamma C_T}{\eta_0}\leq\frac{1}{2},
\qquad
\frac{2C_T}{\eta_0} \le \frac{1}{2},
\qquad
\rho \coloneqq \frac{16\Gamma C_T^2}{\eta_0^2} <1.
\]
Let \(\Im z\geq\eta_0\), \(\bm{X}\in\mathcal{B}_z\), and \(0<\gamma\leq\Gamma\). By Assumption~\ref{hyp:preprocessing}
\[
\left\| \gamma\bm{X}\bm{T}(\bm{y}) \right\|_{\mathrm{op}} \leq \frac{2\Gamma C_T}{\Im z} \leq \frac{2\Gamma C_T}{\eta_0} \leq \frac{1}{2}.
\]
Hence
\[
\left\| \left( \bm{I}_p+\gamma\bm{X}\bm{T}(\bm{y}) \right)^{-1} \right\|_{\mathrm{op}} \leq2,
\]
and therefore \(\left\| \bm{Q}_\gamma(\bm{X},\bm{y}) \right\|_{\mathrm{op}} \leq 2C_T\). It follows that
\[
\left\| \frac{1}{z} \frac{1}{n} \sum_{i=1}^n \bm{Q}_{\gamma_n}(\bm{X},\bm{y}_i) \right\|_{\mathrm{op}} \leq \frac{2 C_T}{|z|}
\leq \frac{1}{2}.
\]
The Neumann series therefore gives
\[
\left\| \bm{\Psi}_{n,z}(\bm{X}) \right\|_{\mathrm{op}} \leq \frac{2}{|z|} \leq \frac{2}{\Im z}.
\]
The same argument applies to \(\bm{\Psi}_{\alpha,z}\). Thus, for every \(\Im z \ge \eta_0\), both maps send \(\mathcal{B}_z\) into itself. 

We next show that these maps are uniform contractions. For \(\bm{X},\bm{Y}\in\mathcal{B}_z\), the resolvent identity yields
\[
\bm{Q}_\gamma(\bm{X},\bm{y}) - \bm{Q}_\gamma(\bm{Y},\bm{y}) = \gamma\bm{T}(\bm{y}) \left( \bm{I}_p+\gamma\bm{X}\bm{T}(\bm{y})
\right)^{-1} (\bm{Y}-\bm{X})\bm{T}(\bm{y}) \left( \bm{I}_p+\gamma\bm{Y}\bm{T}(\bm{y})\right)^{-1}.
\]
Therefore,
\[
\left\| \bm{Q}_\gamma(\bm{X},\bm{y}) - \bm{Q}_\gamma(\bm{Y},\bm{y}) \right\|_{\mathrm{op}} \leq 4\gamma C_T^2 \|\bm{X}-\bm{Y}\|_{\mathrm{op}}.
\]
Applying the resolvent identity once more to the outer inverse gives
\[
\left\| \bm{\Psi}_{n,z}(\bm{X}) - \bm{\Psi}_{n,z}(\bm{Y}) \right\|_{\mathrm{op}} \leq \frac{16 \Gamma C_T^2}{(\Im z)^2}  \|\bm{X}-\bm{Y}\|_{\mathrm{op}} \le \rho  \|\bm{X}-\bm{Y}\|_{\mathrm{op}}.
\]
The same estimate holds for \(\bm{\Psi}_{\alpha,z}\). Hence both maps are uniform contractions on \(\mathcal{B}_z\), uniformly for \(\Im z\geq\eta_0\).

We now compare the empirical and limiting fixed-point maps. Set
\[
\mathcal{B}_0 \coloneqq \left\{ \bm{X} \in \C^{p\times p} \colon \|\bm{X}\|_{\mathrm{op}} \leq \frac{2}{\eta_0} \right\}.
\]
The family \( \left \{ \bm{Q}_{\alpha^{-1}}(\bm{X},\bm{y}) \colon \bm{X} \in \mathcal{B}_0 \right \}\) is uniformly bounded in operator norm and uniformly Lipschitz in \(\bm{X}\), uniformly in \(\bm{y}\). Since \(p\) is fixed, \(\mathcal{B}_0\) is compact in a finite-dimensional space. A finite-net argument, combined with the strong law of large numbers applied entrywise, therefore yields 
\begin{equation} \label{eq:LLN_app}
\sup_{\bm{X}\in\mathcal{B}_0} \left\| \frac{1}{n} \sum_{i=1}^n \bm{Q}_{\alpha^{-1}}(\bm{X},\bm{y}_i) - \E_{\bm{y}} \left[ \bm{Q}_{\alpha^{-1}}(\bm{X},\bm{y}) \right] \right\|_{\mathrm{op}} \to 0
\end{equation}
almost surely. We also compare the maps corresponding to \(\gamma_n\) and \(\alpha^{-1}\). For \(\gamma,\gamma'>0\), another resolvent identity gives
\[
\begin{split}
& \bm{Q}_\gamma(\bm{X},\bm{y}) - \bm{Q}_{\gamma'} (\bm{X},\bm{y}) \\
& = (\gamma - \gamma') \bm{T}(\bm{y}) \left( \bm{I}_p+\gamma\bm{X}\bm{T}(\bm{y})
\right)^{-1} \bm{X} \bm{T}(\bm{y}) \left( \bm{I}_p+\gamma' \bm{X}\bm{T}(\bm{y})\right)^{-1}.
\end{split}
\]
For \(\gamma,\gamma' \le \Gamma\) and \(\bm{X} \in \mathcal{B}_0\), all the inverses above have operator norm at most \(2\). Hence, for some deterministic constant \(C<\infty\),
\[
\sup_{\bm{X} \in \mathcal{B}_0} \operatorname*{ess\,sup}_{\bm{y}} \left\| \bm{Q}_{\gamma_n}(\bm{X},\bm{y}) - \bm{Q}_{\alpha^{-1}}(\bm{X},\bm{y}) \right\|_{\mathrm{op}} \leq C \left| \gamma_n-\alpha^{-1} \right| \to 0.
\]
Define
\[
\Delta_n \coloneqq \sup_{\bm{X}\in\mathcal{B}_0} \left\| \frac{1}{n} \sum_{i=1}^n \bm{Q}_{\gamma_n}(\bm{X},\bm{y}_i) - \E_{\bm{y}} \left[ \bm{Q}_{\alpha^{-1}}(\bm{X},\bm{y}) \right] \right\|_{\mathrm{op}}.
\]
By~\eqref{eq:LLN_app} and \(\gamma_n \to \alpha^{-1}\), \(\Delta_n \to 0\) almost surely. Applying the resolvent identity to the outer inverses, we obtain
\[
\sup_{\substack{ \Im z\geq\eta_0\\ \bm{X}\in\mathcal{B}_z }} \left\| \bm{\Psi}_{n,z}(\bm{X}) - \bm{\Psi}_{\alpha,z}(\bm{X}) \right\|_{\mathrm{op}} \leq \frac{4}{\eta_0^2}\Delta_n.
\]
Moreover, by the matrix-valued Stieltjes representations of \(\bm{M}_n^{\mathrm{free}}\) and \(\bm{M}_\alpha\), we have \(\|\bm{M}_n^{\mathrm{free}} (z)\|_{\mathrm{op}} \le (\Im z)^{-1}\) and \(\|\bm{M}_\alpha (z)\|_{\mathrm{op}} \le (\Im z)^{-1}\) for every \(z \in \C_+\). Thus both fixed points belong to \(\mathcal{B}_z\). Using the fixed-point identities and the uniform contraction estimate, for \(\Im z \ge \eta_0\), we obtain
\[
\begin{split}
& \left\| \bm{M}_n^{\mathrm{free}}(z) - \bm{M}_\alpha(z) \right\|_{\mathrm{op}} \\
& \leq \left\| \bm{\Psi}_{n,z} (\bm{M}_n^{\mathrm{free}}(z))
- \bm{\Psi}_{n,z} (\bm{M}_\alpha(z)) \right\|_{\mathrm{op}}  + \left\| \bm{\Psi}_{n,z} (\bm{M}_\alpha(z)) - \bm{\Psi}_{\alpha,z} (\bm{M}_\alpha(z) ) \right\|_{\mathrm{op}} \\
& \leq \rho \left\| \bm{M}_n^{\mathrm{free}}(z) - \bm{M}_\alpha(z) \right\|_{\mathrm{op}} + \frac{4}{\eta_0^2}\Delta_n.
\end{split}
\]
Therefore
\[
\sup_{\Im z\geq\eta_0} \left\| \bm{M}_n^{\mathrm{free}}(z) - \bm{M}_\alpha(z)
\right\|_{\mathrm{op}} \leq \frac{4}{(1-\rho)\eta_0^2} \Delta_n \to 0
\]
almost surely. In particular, \(\bm{M}_n^{\mathrm{free}}(z) \to \bm{M}_\alpha(z)\) uniformly on \(\left\{ z \in \C_+ \colon \Im z  > \eta_0\right \}\) almost surely.

It remains to extend the convergence to all of \(\C_+\). Since \( \bm{P}_n^{\mathrm{free}} \) is self-adjoint, \(\left\| \bm{M}_n^{\mathrm{free}}(z) \right\|_{\mathrm{op}} \leq (\Im z)^{-1}\) for every \(z \in \C_+\). Thus, the sequence \(\{ \bm{M}_n^{\mathrm{free}} \}_{n\geq1}\) is locally uniformly bounded on \(\C_+\). For every \(a,b\in[p]\), the scalar functions \(z \mapsto \left( \bm{M}_n^{\mathrm{free}}(z) \right)_{ab} \) are analytic and locally uniformly bounded on \(\C_+\), and they converge to \((\bm{M}_\alpha(z))_{a,b}\) on the nonempty open set \(\left\{ z \in \C_+ \colon \Im z  > \eta_0\right \}\). Vitali's convergence theorem therefore yields locally uniform convergence on \(\C_+\) to an analytic limit. By the identity theorem, this limit coincides with \(\bm{M}_\alpha\). Thus
\[
\bm{M}_n^{\mathrm{free}}(z) \to \bm{M}_\alpha(z)
\]
locally uniformly on \(\C_+\), almost surely.

Taking normalized traces gives, almost surely, for every \(z \in \C_+\),
\[
m_n^{\mathrm{free}}(z) = \frac{1}{p} \Tr \bm{M}_n^{\mathrm{free}}(z) \to  \frac{1}{p} \Tr \bm{M}_\alpha(z) = m_\alpha(z)
\]
for every \(z\in\C_+\). By Lemma~\ref{lem:exact_empirical_MDE_free}, \(m_n^{\mathrm{free}}\) is the Stieltjes transform of \(\mu_n^{\mathrm{free}}\), whereas \(m_\alpha\) is the Stieltjes transform of \(\mu_\alpha\). The Stieltjes continuity theorem therefore implies
\[
\mu_n^{\mathrm{free}} \Rightarrow \mu_\alpha
\]
weakly almost surely.
\end{proof}

\begin{proof}[Proof of Proposition~\ref{prop:limiting_variational_edge}]
Fix \(\bm{Z} \in\mathcal{D}_\alpha\). By definition of \(\mathcal{D}_\alpha\), there exists \(c_{\bm{Z}}\) such that 
\[
\bm{Z}^{-1}-\gamma_n \bm{T}(\bm{y}) \succeq c_{\bm{Z}} \bm{I}_p \quad \mathrm{a.s.}
\]
Since \(\gamma_n\to\alpha^{-1}\) and \(\|\bm{T}(\bm{y})\|_{\mathrm{op}} \leq C_T\), it follows that, almost surely, for all sufficiently large \(n\),
\[
\bm{Z}^{-1} - \gamma_n \bm{T} (\bm{y}_i) \succeq \frac{c_{\bm{Z}}}{2}\bm{I}_p,
\qquad 1\leq i\leq n.
\]
Hence \(\bm{Z}\in\mathcal{D}_n\) eventually. Moreover, by the same argument as in the proof of Lemma~\ref{lem:free_upper_edge},
\[
\mathcal{G}_n(\bm{Z}) \to \mathcal{G}_\alpha(\bm{Z}) \quad\text{a.s.}
\]
Therefore, Lemma~\ref{lem:finite_variational_formula} gives
\[
\limsup_{n\to\infty} \lambda_{+,n}^{\mathrm{free}} \leq \lambda_1 \left (\mathcal{G}_\alpha (\bm{Z}) \right)
\quad\text{a.s.}
\]
On the other hand, Lemma~\ref{lem:free_MDE_convergence} gives \( \mu_n^{\mathrm{free}} \Rightarrow \mu_\alpha\) weakly almost surely. Since \(\supp(\mu_n^{\mathrm{free}}) =\operatorname{sp}(\bm{P}_n^{\mathrm{free}})\) by~\eqref{eq:support_equals_free_spectrum}, weak convergence implies
\[
\lambda_+(\alpha) \leq \liminf_{n\to\infty}\lambda_{+,n}^{\mathrm{free}} \quad \text{a.s.}
\]
Consequently, \(\lambda_+(\alpha) \leq \lambda_1 \left (\mathcal{G}_\alpha (\bm{Z}) \right)\). Since \(\bm{Z}\in\mathcal{D}_\alpha\) was arbitrary,
\[
\lambda_+(\alpha) \leq \inf_{\bm{Z} \in\mathcal{D}_\alpha} \lambda_1 \left ( \mathcal{G}_\alpha (\bm{Z}) \right).
\]

For the reverse inequality, fix \(E>\lambda_+(\alpha)\) and set \(\bm{Z}_E=-\bm{M}_\alpha(E)\). As shown in the proof of Lemma~\ref{lem:free_upper_edge},
\(\bm{Z}_E \in \mathcal{D}_\alpha\) and the limiting matrix Dyson equation
gives \(\mathcal{G}_\alpha (\bm{Z}_E)=E\bm{I}_p\). It follows that
\[
\inf_{\bm{Z} \in\mathcal{D}_\alpha} \lambda_1 \left (\mathcal{G}_\alpha (\bm{Z}) \right) \leq E.
\]
Letting \(E\downarrow\lambda_+(\alpha)\) yields the reverse inequality and completes the proof.
\end{proof}
\section{Auxiliary results} \label{appendix:auxiliary}

This appendix collects auxiliary matrix identities, partial-trace estimates, and Lipschitz bounds that are used repeatedly in the proofs of the main results.

\begin{lem} \label{lem:push_through}
Let \(\bm{A},\bm{B}\in\C^{p\times p}\). If \(\bm{I}_p+\bm{A}\bm{B}\) is invertible, then \(\bm{I}_p+\bm{B}\bm{A}\) is invertible and
\[
(\bm{I}_p+\bm{A}\bm{B})^{-1} \bm{A} = \bm{A} (\bm{I}_p+\bm{B}\bm{A})^{-1}.
\]
\end{lem}

\begin{proof}
The identity
\[
(\bm{I}_p+\bm{A}\bm{B})\bm{A} = \bm{A}(\bm{I}_p+\bm{B}\bm{A})
\]
implies the claim after multiplication by the corresponding inverses. The invertibility of \( \bm{I}_p + \bm{B} \bm{A}\) follows, for instance, from Sylvester's determinant identity \(\det (\bm{I}_p +\bm{A}\bm{B}) = \det(\bm{I}_p+\bm{B} \bm{A})\).
\end{proof}

\begin{lem} \label{lem:partial_trace_bound}
Let \(m\ge1\) and \(\bm{A} \in \C^{pm\times pm}\). Let \(\tr_m^{(p)}(\bm{A})\) denote the normalized partial trace of \(\bm{A}\) defined in~\eqref{eq:partial_trace}. Then
\[
\bigl \|\tr_m^{(p)}(\bm{A})\bigr\|_{\mathrm{op}} \le \|\bm{A}\|_{\mathrm{op}}
\]
and
\[
\bigl \|\tr_m^{(p)}(\bm{A})\bigr\|_{\mathrm{op}} \le \frac{\operatorname{rank}(\bm{A})}{m}\|\bm{A}\|_{\mathrm{op}} \le \|\bm{A}\|_{\mathrm{op}}.
\]
\end{lem}

\begin{proof}
The normalized partial trace is the map \(\operatorname{id}_p \otimes \tr_m\), where \(\tr_m = m^{-1} \Tr\) is the normalized trace. It is unital and completely positive, hence contractive in operator norm. This proves the first bound.

For the second bound, we claim that
\begin{equation} \label{eq:nuclear-norm-claim}
\|\tr_m^{(p)} (\bm{A})\|_{\mathrm{op}} \le \|\tr_m^{(p)} (\bm{A})\|_\ast \le \frac{1}{m} \|\bm{A}\|_\ast \le \frac{\mathrm{rank}(\bm{A})}{m} \|\bm{A}\|_{\mathrm{op}},
\end{equation}
where the nuclear norm \(\|\cdot \|_\ast\) is defined by
\[
\|\bm{A} \|_\ast \coloneqq \sum_{j=1}^{\mathrm{rank}(\bm{A})} \mu_j (\bm{A}).
\]
Here, \(\mu_1 (\bm{A}) \ge \cdots \ge \mu_{\mathrm{rank}}(\bm{A})>0\) are the singular values of \(\bm{A}\). The first and last inequalities are immediate. For the second inequality, duality and the definition of the normalized partial trace give
\[
\begin{aligned}
\bigl\|\tr_m^{(p)}(\bm{A})\bigr\|_\ast
&=\sup_{\|\bm{B}\|_{\mathrm{op}}\le1}
\left|\Tr\!\left(\bm{B}^\ast\tr_m^{(p)}(\bm{A})\right)\right|\\
&=\frac1m\sup_{\|\bm{B}\|_{\mathrm{op}}\le1}
\left|\Tr\!\left((\bm{B}^\ast\otimes\bm{I}_m)\bm{A}\right)\right|\\
&\le \frac1m\|\bm{A}\|_\ast,
\end{aligned}
\]
since \(\|\bm{B}\otimes\bm{I}_m\|_{\mathrm{op}}=\|\bm{B}\|_{\mathrm{op}}\). This proves the second inequality in~\eqref{eq:nuclear-norm-claim}.
\end{proof}

\begin{lem}\label{lem:Lipschitz}
Let \(\bm{T} \in R^{p\times p}\) satisfy \(\|\bm{T}\|_{\mathrm{op}}\le C_T\) and let \(\gamma \ge 0\). Let \(\bm{X},\bm{Y}\in \C^{p\times p}\) and assume that \((\bm{I}_p+\gamma \bm{T} \bm{X})\) and \((\bm{I}_p+\gamma \bm{T} \bm{Y})\) are invertible and satisfy
\[
\|(\bm{I}_p + \gamma \bm{T} \bm{X})^{-1}\|_{\mathrm{op}}\le K,\quad \|(\bm{I}_p +\gamma \bm{T} \bm{Y})^{-1}\|_{\mathrm{op}}\le K,
\]
for some \(K<\infty\). Define
\[
\Phi_{\bm{T},\gamma}(\bm{X}) \coloneqq \bm{T} \bm{X}(\bm{I}_p + \gamma \bm{T} \bm{X})^{-1}.
\]
Then
\[
\|\Phi_{\bm{T},\gamma}(\bm{X})-\Phi_{\bm{T},\gamma}(\bm{Y})\|_{\mathrm{op}}
\le C_T K^2  \|\bm{X}-\bm{Y}\|_{\mathrm{op}}.
\]
\end{lem}
\begin{proof}
A direct resolvent calculation gives
\[
\Phi_{\bm{T},\gamma}(\bm{X})-\Phi_{\bm{T},\gamma}(\bm{Y})
=
(\bm{I}_p+\gamma\bm{T}\bm{Y})^{-1}\bm{T}(\bm{X}-\bm{Y})(\bm{I}_p+\gamma\bm{T}\bm{X})^{-1}.
\]
Therefore,
\[
\|\Phi_{\bm{T},\gamma}(\bm{X})-\Phi_{\bm{T},\gamma}(\bm{Y})\|_{\mathrm{op}}
\le C_TK^2\|\bm{X}-\bm{Y}\|_{\mathrm{op}}.
\]
\end{proof}

\begin{lem}\label{lem:Lipschitz_2}
Let \(\bm{T} \in R^{p\times p}\) satisfy \(\|\bm{T}\|_{\mathrm{op}}\le C_T\) and let \(\gamma \ge 0\). Let \(\bm{X},\bm{Y}\in \C^{p\times p}\) and assume that \((\bm{I}_p+\gamma \bm{T} \bm{X})\) and \((\bm{I}_p+\gamma \bm{T} \bm{Y})\) are invertible and satisfy
\[
\|(\bm{I}_p + \gamma \bm{X} \bm{T} )^{-1}\|_{\mathrm{op}}\le K,\quad \|(\bm{I}_p +\gamma \bm{Y} \bm{T} )^{-1}\|_{\mathrm{op}}\le K,
\]
for some \(K<\infty\). Define
\[
\bm{\mathcal{S}}_{\bm{T},\gamma}(\bm{X}) \coloneqq \bm{T} (\bm{I}_p + \gamma \bm{X} \bm{T} )^{-1}.
\]
Then
\[
\| \bm{\mathcal{S}}_{\bm{T},\gamma}(\bm{X}) - \bm{\mathcal{S}}_{\bm{T},\gamma}(\bm{Y})\|_{\mathrm{op}} \le \gamma C_T^2 K^2  \|\bm{X}-\bm{Y}\|_{\mathrm{op}}.
\]
\end{lem}
\begin{proof}
By the resolvent identity,
\[
\begin{split}
\bm{\mathcal{S}}_{\bm{T},\gamma} (\bm{X})  - \bm{\mathcal{S}}_{\bm{T},\gamma} (\bm{Y}) &= \bm{T} \left [(\bm{I}_p + \gamma \bm{X} \bm{T} )^{-1} - (\bm{I}_p + \gamma \bm{Y} \bm{T} )^{-1} \right ] \\
&= \gamma \bm{T} (\bm{I}_p + \gamma \bm{X} \bm{T} )^{-1}  (\bm{Y}-\bm{X})  \bm{T} (\bm{I}_p + \gamma \bm{Y} \bm{T} )^{-1}.
\end{split}
\]
Therefore
\[
\begin{split}
\lVert \bm{\mathcal{S}}_{\bm{T},\gamma} (\bm{X})  - \bm{\mathcal{S}}_{\bm{T},\gamma} (\bm{Y})  \rVert_{\mathrm{op}} 
&  \le  \gamma C_T^2 K^2 \|\bm{X}-\bm{Y}\|_{\mathrm{op}}.
\end{split}
\]
\end{proof}

\printbibliography
\end{document}

%% file: bibliography_setup.tex
\usepackage[giveninits=true,url=false,doi=false,isbn=false,eprint=true,datamodel=mrnumber,sorting=nyt,sortcites=false,maxcitenames=4,maxbibnames=99,backref=false,block=space,backend=biber,style=phys, biblabel=brackets]{biblatex} 
\AtEveryBibitem{\clearfield{month}}
\AtEveryCitekey{\clearfield{month}} 
\renewbibmacro{in:}{}
\DeclareFieldFormat[article]{title}{\emph{#1}} 
\DeclareFieldFormat{mrnumber}{\ifhyperref{\href{http://www.ams.org/mathscinet-getitem?mr=#1}{\nolinkurl{MR#1}}}{\nolinkurl{#1}}}
\DeclareFieldFormat{pmid}{\ifhyperref{\href{https://www.ncbi.nlm.nih.gov/pubmed/#1}{\nolinkurl{PMID#1}}}{\nolinkurl{#1}}}
\DeclareFieldFormat{eprint}{\ifhyperref{\href{https://arxiv.org/abs/#1}{\nolinkurl{arXiv:#1}}}{\nolinkurl{#1}}}
\renewbibmacro*{doi+eprint+url}{%
  \iftoggle{bbx:doi}{\printfield{doi}}{}
  \newunit\newblock%
  \printfield{mrnumber}%
  \newunit\newblock%
  \printfield{pmid}%
  \newunit\newblock%
  \printfield{eprint}%
  \iftoggle{bbx:url}{\usebibmacro{url+urldate}}{}}